\documentclass[a4paper,11pt,reqno,noindent]{amsart}
\usepackage[centertags]{amsmath}
\usepackage{amsfonts,amssymb,amsthm,amscd,dsfont,cases,esint,enumerate,color,mathdots}
\usepackage[T1]{fontenc}
\usepackage[english]{babel}
\usepackage{tikzsymbols}
\usepackage[applemac]{inputenc}
\usepackage[body={15cm,21.5cm},centering]{geometry} 
\usepackage{fancyhdr}
\numberwithin{equation}{section}
\usepackage{tikz}
\usetikzlibrary{arrows.meta,patterns}

\theoremstyle{plain}
\newtheorem{theor10}{Theorem}
\newenvironment{theor1}
  {\pushQED{\qed}\begin{theor10}}
  {\popQED\end{theor10}}
\newtheorem{prop10}[theor10]{Proposition}
\newenvironment{prop1}
  {\pushQED{\qed}\begin{prop10}}
  {\popQED\end{prop10}}
\newtheorem{cor10}[theor10]{Corollary}
\newenvironment{cor1}
  {\pushQED{\qed}\begin{cor10}}
  {\popQED\end{cor10}}
\newtheorem{lem10}[theor10]{Lemma}

\newtheorem{theor0}{Theorem}[section]
\newenvironment{theor}
  {\pushQED{\qed}\begin{theor0}}
  {\popQED\end{theor0}}
\newtheorem{lem0}[theor0]{Lemma}
\newenvironment{lem}
  {\pushQED{\qed}\begin{lem0}}
  {\popQED\end{lem0}}
\newtheorem{prop0}[theor0]{Proposition}
\newenvironment{prop}
  {\pushQED{\qed}\begin{prop0}}
  {\popQED\end{prop0}}
\newtheorem{cor0}[theor0]{Corollary}

\theoremstyle{definition}
\newtheorem{rems0}[theor0]{Remarks}
\newenvironment{rems}
  {\pushQED{\qed}\begin{rems0}}
  {\popQED\end{rems0}}
\newtheorem{rem0}[theor0]{Remark}
\newenvironment{rem}
  {\pushQED{\qed}\begin{rem0}}
  {\popQED\end{rem0}}

\theoremstyle{plain}
\newtheorem{as0}[theor0]{Assumption}

\newtheorem*{asn0*}{\assumptionnumber}
  \providecommand{\assumptionnumber}{}
\makeatletter
\newenvironment{asn0}[2]
   {\renewcommand{\assumptionnumber}{Assumption \!#1 {\normalfont--- #2}}
    \begin{asn0*}
    \protected@edef\@currentlabel{{\normalfont#1}}}
   {\end{asn0*}}
\makeatother
\newenvironment{asn}
  {\pushQED{\qed}\begin{asn0}}
  {\popQED\end{asn0}}
\makeatletter
\newenvironment{asn01}[1]
   {\renewcommand{\assumptionnumber}{Assumption \!#1}
    \begin{asn0*}
    \protected@edef\@currentlabel{{\normalfont#1}}}
   {\end{asn0*}}
\makeatother

\newcommand{\N}{\mathbb N}

\newcommand{\e}{\varepsilon}

\newcommand{\Pc}{\mathcal{P}}

\newcommand{\Fc}{\mathcal{F}}
\newcommand{\Ec}{\mathcal{E}}
\newcommand{\Kc}{\mathcal{K}}
\newcommand{\calC}{\mathcal{C}}
\newcommand{\calR}{\mathcal{R}}
\newcommand{\Lg}{\mathfrak{L}}
\newcommand{\Cg}{\mathfrak{C}}
\newcommand{\Kg}{\mathfrak{K}}
\newcommand{\Gg}{\mathfrak{G}}

\newcommand{\Hg}{\mathfrak{H}}

\newcommand{\Sc}{\mathcal S}
\newcommand{\Nb}{N}
\newcommand{\calM}{\mathcal M}
\newcommand{\Jc}{\mathcal J}
\newcommand{\Qc}{\mathcal Q}

\newcommand{\Lc}{\mathcal L}
\newcommand{\Gc}{\mathcal G}

\newcommand{\Qd}{Q_{L,\rho}}
\newcommand{\R}{\mathbb R}
\newcommand{\Z}{\mathbb Z}
\newcommand{\Ic}{\mathcal I}
\newcommand{\calH}{\mathcal H}

\newcommand{\pv}{\operatorname{p.v.}}
\newcommand{\diam}{\operatorname{diam}}
\newcommand{\dist}{\operatorname{dist}}
\newcommand{\Md}{\mathbb M}

\newcommand{\loc}{{\operatorname{loc}}}
\newcommand{\Id}{\operatorname{Id}}
\newcommand{\E}{\mathbb{E}}
\newcommand{\per}{{\operatorname{per}}}
\newcommand{\D}{\operatorname{D}}

\newcommand{\del}{\delta}
\newcommand{\Bb}{\bar{\mathbf B}}

\newcommand{\Bh}{\hat{\mathbf B}}
\newcommand{\Bt}{\tilde{\mathbf B}}
\newcommand{\Cc}{\bar{\mathbf C}}

\newcommand{\inter}{{\operatorname{int}}}
\newcommand{\Ld}{\operatorname{L}}
\newcommand{\supp}{\operatorname{supp}}
\newcommand{\Div}{{\operatorname{div}}}
\newcommand{\Sym}{{\operatorname{sym}}}
\newcommand{\Skew}{{\operatorname{skew}}}
\newcommand{\Tr}{\operatorname{tr}}
\newcommand{\step}[1]{\noindent \textit{Step} #1.}
\newcommand{\substep}[1]{\noindent \textit{Substep} #1.}
\newcommand{\Pm}{\mathbb{P}}
\newcommand{\pr}[1]{\mathbb{P}\left[ #1 \right]}
\newcommand{\prm}[1]{\mathbb{P}\big[ #1 \big]}
\newcommand{\expec}[1]{\mathbb{E}\left[ #1 \right]}
\newcommand{\expecm}[1]{\mathbb{E}\big[ #1 \big]}
\newcommand{\expecM}[1]{\mathbb{E}\bigg[ #1 \bigg]}

\newcommand{\var}[1]{\mathrm{Var}\left[#1\right]}
\newcommand{\cov}[2]{\mathrm{Cov}\left[#1;#2\right]}
\newcommand{\covM}[2]{\mathrm{Cov}\bigg[#1;#2\bigg]}

\usepackage[colorlinks,citecolor=black,urlcolor=black]{hyperref}

\title[On Einstein's effective viscosity formula]{On Einstein's effective viscosity formula}

\author[M. Duerinckx]{Mitia Duerinckx}
\address[Mitia Duerinckx]{Universit\'e Libre de Bruxelles, D\'epartement de Math\'ematique, 1050~Brussels, Belgium \& Universit\'e Paris-Saclay, CNRS, Laboratoire de Math\'ematiques d'Orsay, 91405~Orsay, France}
\email{mitia.duerinckx@ulb.be}
\author[A. Gloria]{Antoine Gloria}
\address[Antoine Gloria]{Sorbonne Universit\'e, CNRS, Universit\'e de Paris, Laboratoire Jacques-Louis Lions, 75005~Paris, France \& Institut Universitaire de France \& Universit\'e Libre de Bruxelles, D\'epartement de Math\'ematique, 1050~Brussels, Belgium}
\email{gloria@ljll.math.upmc.fr}

\begin{document}

\maketitle

\begin{abstract}
In his PhD thesis, Einstein derived an explicit first-order expansion for the effective viscosity of a Stokes fluid with a suspension of small rigid particles at low density.
His formal derivation relied on two implicit assumptions: (i) there is a scale separation between the size of the particles and the observation scale;
and (ii) at first order, dilute particles do not interact with one another.
In mathematical terms, the first assumption amounts to the validity of a homogenization result defining the effective viscosity tensor, which is now well understood.
Next, the second assumption allowed Einstein to approximate this effective viscosity at low density by considering particles as being isolated.
The rigorous justification is, in fact, quite subtle as the effective viscosity is a nonlinear nonlocal function of the ensemble of particles and as hydrodynamic interactions have borderline integrability.
In the present memoir, we establish Einstein's effective viscosity formula in the most general setting.
In addition, we pursue the low-density expansion to arbitrary order in form of a cluster expansion, where the summation of hydrodynamic interactions crucially requires suitable renormalizations. In particular, we justify a celebrated result by Batchelor and Green on the second-order correction and we explicitly describe all higher-order renormalizations for the first time.
In some specific settings, we further address the summability of the whole cluster expansion. Our approach relies on
a combination of combinatorial arguments, variational analysis, elliptic regularity, probability theory, and diagrammatic integration methods.

\bigskip\noindent
{\sc MSC-class:}
76T20; 35R60; 76M50; 35Q35; 76D03; 76D07; 60G55.
\end{abstract}

\setcounter{tocdepth}{1}
\tableofcontents

\vspace{-1cm}
\section{General overview}

\subsection{Historical context} 
At the dawn of the 20th century, the debate was still raging on the existence of atoms,
and Einstein's PhD thesis ``A New Determination of Molecular Dimensions''~\cite{Einstein-06} aimed to support the atomic theory.
This was the second of his five celebrated~1905 contributions and constitutes his most cited work.
The main part was devoted to the hydrodynamic derivation of a formula for the effective viscosity of a fluid with a dilute suspension of rigid particles: the so-called Einstein formula in fluid mechanics, which is the focus of the present memoir. Next, in the same work, Einstein derived a relation between the diffusion constant for suspended particles and their mobility: the so-called Einstein relation in kinetic theory.
He then applied these two relations to sugar dissolved in water: using available empirical data, he deduced an estimate of the Avogadro number and of the size of sugar molecules (after eliminating a calculation error~\cite{Einstein-11}).
We refer to~\cite{Straumann-05} for an inspiring account of this seminal work.
As discussed by Perrin in his extensive report \cite{Perrin-11} at the first Solvay conference in 1911 in Brussels, these discoveries were confirmed by further experiments and shown to agree with other methods to determine the Avogadro number, which sealed the triumph of the atomic theory.

We briefly describe Einstein's argument to estimate the effective viscosity of a dilute suspension.
Viscosity of a fluid is usually measured by shear-flow experiments: a cylindrical vessel is filled with the fluid, a rotating spindle is immersed in it, and one measures the torque needed to make it rotate at constant angular speed.
Assume now that the fluid contains a suspension of small rigid spherical particles and consider their influence on the measured viscosity.
As particles are rigid, they act as obstacles and hinder the fluid flow, thus effectively increasing
the measured viscosity. 
A first challenging question concerns the dynamics of the particles: do they reach a statistical steady state?
If this is the case and if one indeed measures a constant-in-time effective viscosity,
then the latter depends on the steady state, hence possibly on the speed of the spindle itself, which corresponds to possible non-Newtonian behaviors~\cite[Section~7]{GM-11}. Einstein's main idea in~\cite{Einstein-06} was that, in the low-density regime, for spherical particles, the first-order effective change in viscosity should only depend on the volume fraction of the particles and not on their distribution. In particular, this universality would relegate non-Newtonian effects to higher-order corrections.
More precisely, in 3D,
given a fluid with isotropic viscosity $\Id$ and given suspended spherical particles with small volume fraction~$\varphi\ll1$, Einstein's formula for the effective viscosity takes the form
\begin{equation}\label{eq:Einstein}
\Bb\,=\,\Id\big(1+ \tfrac{5}{2}\varphi+o(\varphi)\big).
\end{equation}
Heuristically, the argument is as follows: at low density, particles are scarce and typically well separated,
hence their interactions are negligible to leading order.
The first-order effect on the viscosity should thus be proportional to the volume fraction and correspond to the energy dissipation of a single isolated particle in the fluid. The latter can be computed explicitly for spherical particles and leads to the celebrated $\frac 52$ factor in~\eqref{eq:Einstein};
we refer to Section~\ref{sec:explicitEinstein} below, where this classical calculation is reproduced.

This type of low-density expansions was not new in the physics community at the time, but was very much in line with other work on the micromechanics of heterogeneous media of the late 19th century. Einstein's formula is indeed comparable to the Clausius--Mossotti formula for the effective dielectric constant~\cite{Mossotti-36,Mossotti-50,Clausius-79}, to Maxwell's formula for the effective conductivity in electrostatics~\cite{Maxwell-81}, or to the Lorentz--Lorenz formula for the effective refractive index in optics~\cite{Lorenz-80,Lorentz-09};
we refer to~\cite{Markov-00} for an account of the historical context.

Einstein's formula
triggered a lot of long-lasting activity in fluid mechanics: the large-scale rheology of suspensions was soon considered as a topic in its own right~\cite{KRM-67,FA-67,JA-76}.
Various works have aimed
at understanding to what extent Einstein's formula is robust and accurate.
Robustness has been addressed in particular by establishing corresponding formulas for particles of different shapes, as e.g.~the explicit formulas by Jeffery~\cite{Jeffery-22} for suspensions of ellipsoids (see also~\cite{LH-71,HL-72}).
Accuracy is a more subtle issue and essentially amounts to capturing the next-order term in the low-density expansion.
While particle interactions are neglected at first order,
the next-order correction consists of including the effects of pairwise interactions.
Due to their long-range nature, the sum of pairwise contributions is not summable and some renormalization is therefore needed. This was first achieved by Batchelor and Green~\cite{BG-72},
and we refer to~\cite{Hinch-77,Obrien-79,Almog-Brenner} for other formal renormalization ideas.
A related, yet different, topic concerns the sedimentation of suspended particles under gravity and the computation of their effective settling speed, which happens to require a similar renormalization: the above-mentioned contribution by Batchelor and Green~\cite{BG-72} was indeed inspired by Batchelor's work~\cite{Batchelor-72} on sedimentation. Interestingly, the renormalization of higher-order corrections to the effective viscosity had remained  open  in the physics community.

We also refer to~\cite{ZAB-83,NK-84,AGKL-12} for the asymptotic analysis of the effective viscosity for dilute \emph{periodic} arrays of suspended particles and, in a more mathematical spirit, we mention the pioneering work by S\'anchez-Palencia et al.~\cite{SP-85,LSP-85} using formal two-scale expansions for locally periodic suspensions.

\subsection{Mathematical reformulation and objectives}
As described above, Einstein's formal derivation of~\eqref{eq:Einstein} in~\cite{Einstein-06} relies the following two implicit hypotheses:
\begin{enumerate}[(E1)]
\item {\it Scale separation.}
There is a {scale separation} between the ``microscopic'' particle size and the ``macroscopic'' observation scale. Therefore, the suspension behaves on the observation scale like an ``effective'' fluid with some effective viscosity tensor~$\Bb$ that can then be measured by shear-flow experiments.
\smallskip\item {\it Particle interactions are negligible.}
In the low-density regime, particles are typically well-separated and therefore, to leading order, they do not interact and can be treated as being isolated.
\end{enumerate}
We briefly discuss the validity of these two working hypotheses and then turn to describing the literature and our objectives in the present memoir.

\subsubsection{Einstein's hypothesis~{\rm(E1)}: scale separation}
This first hypothesis concerns the definition of a notion of effective viscosity for suspensions when the particle size $O(\e)$ is much smaller than the observation scale $O(1)$.
Consider a shear-flow experiment to measure the viscosity, say using a rotational viscosimeter.
Let $D$ denote the fluid domain in this device and let $\{x_{\e,n}^t\}_n\subset D$ stand for positions of suspended particles at time~$t$, which evolve over time with the fluid flow.
If inertia is neglected, the dynamics is greatly simplified: given particle positions at a given time, the fluid velocity satisfies steady Stokes equations, which determine instantaneous particle velocities.
In this context, the emergence of an effective viscosity can be split into two parts:
\begin{enumerate}[---]
\item {\it Steady-state microstructure.} As the measured effective viscosity is expected not to depend on time, it implicitly requires particle positions to reach a statistical steady state in the long run.
Focussing on a portion of the fluid in the bulk, we may consider without much loss of generality that the statistical ensemble is stationary (henceforth, ``stationarity'' stands for statistical spatial homogeneity).
In other words, the point set $\{x_{\e,n}^t\}_n$ can be approximately replaced by a random point set
$\{\e x_n:\e x_n\in D\}$
that is the $\e$-rescaling of some stationary random point process $\Pc=\{x_n\}_n$.
The law of this steady state may depend itself on the prescribed shear flow in the viscosimeter, which leads to possible non-Newtonian effects~\cite[Section~7]{GM-11}.
\smallskip\item {\it Steady homogenization problem.} Given a statistical ensemble of particle positions, under an ergodicity assumption, the steady Stokes equations for the fluid velocity are expected to homogenize on the macroscopic observation scale and can be replaced by effective steady Stokes equations with some effective viscosity tensor $\Bb$.
\end{enumerate}
While the rigorous analysis of the steady-state flow-induced microstructure remains a fully open problem at this time, the steady homogenization problem, in contrast, has been extensively studied under various assumptions in our recent series of articles~\cite{DG-19,D-20a,DG-21a,DG-21b} and is by now very well understood. Given a statistical ensemble of particle positions, this provides a rigorous definition of the effective viscosity together with a homogenization result.
More precisely, considering the system at the particle scale,
we denote by $\Ic=\cup_n I_n$ the random ensemble of particles (not necessarily spherical), centered at the points of a point process $\Pc=\{x_n\}_n$, say in the $d$-dimensional Euclidean space $\R^d$ for generality.
The effective viscosity tensor $\Bb$ is defined as a quadratic form on the set $\Md_0^\Sym\subset\R^{d\times d}$ of trace-free symmetric matrices,
\begin{equation}\label{def:eff-vis}
E:\Bb E\,:=\,\expecm{|\!\D(\psi_{E})+E|^2}\,=\,|E|^2+\expecm{|\!\D(\psi_{E})|^2},\qquad E\in\Md_0^\Sym,
\end{equation}
where $\D(\psi_E)$ is the unique stationary symmetric gradient solution, with bounded second moment and vanishing expectation, of the  corrector problem
\begin{equation}\label{eq:cor}
\left\{\begin{array}{ll}
-\triangle\psi_{E}+\nabla\Sigma_{E}=0,&\text{in $\R^d\setminus\Ic$},\\
\Div( \psi_{E})=0,&\text{in $\R^d\setminus\Ic$},\\
\D(\psi_{E}+Ex)=0,&\text{in $\Ic$},\\
\int_{\partial I_{n}}\sigma_{E}\nu=0,&\forall n,\\
\int_{\partial I_{n}}\Theta(x-x_{n})\cdot\sigma_{E}\nu=0,&\forall n,~\forall \Theta\in\Md^\Skew,
\end{array}\right.
\end{equation}
in terms of the associated Cauchy stress tensor
\begin{equation}\label{eq:Cauchystress}
 \sigma_{E}\,:=\,\sigma(\psi_E+Ex,\Sigma_E)\,:=\,2\D(\psi_{E}+Ex)-\Sigma_{E}\Id,
\end{equation}
where $\Md^\Skew\subset\R^{d\times d}$ is the set of skew-symmetric matrices.
Throughout this work, we assume for simplicity that the plain fluid has isotropic viscosity $\Id$.
Equation~\eqref{eq:cor} can be viewed as describing the velocity field $\psi_E+Ex$ of a Stokes fluid in the whole space in presence of rigid suspended particles $\{I_n\}_n$ with linear strain imposed at infinity, \mbox{$\psi_E+Ex\sim Ex$} as $|x|\uparrow\infty$.
The last two boundary conditions in~\eqref{eq:cor} correspond to the balance of forces and torques on each particle.
Note that, if $\Ic$ contains an unbounded chain of touching particles, then the rigidity constraint \mbox{$\D(\psi_E+Ex)|_{\Ic}=0$}  entails that the field~$\psi_E$ would grow linearly along this chain, which would prevent $\D(\psi_E)$ from having vanishing expectation: it shows that this corrector problem can only be well-posed provided that some suitable non-clustering assumption is made.
Different sets of sufficient assumptions are recalled in Section~\ref{sec:main-Einstein} below and we refer to our previous work~\cite{DG-19,D-20a,DG-21a,DG-21b} for a detailed account.

\subsubsection{Einstein's hypothesis~{\rm(E2)}: interactions are negligible.}
As it appears from~\eqref{eq:cor}, the corrector $\psi_E$ depends nonlocally and nonlinearly on the set $\Ic$ of particles via boundary conditions: this corresponds to the multibody nature of hydrodynamic interactions.
Einstein's second hypothesis can be reinterpreted
as claiming that $\psi_E$ can be approximated around each inclusion $I_n$ by the unique decaying solution $\psi_{E}^{\{n\}}$ of the single-particle problem
\begin{equation}\label{eq:singlepart}
\left\{\begin{array}{ll}
-\triangle\psi_{E}^{\{n\}}+\nabla\Sigma_{E}^{\{n\}}=0,&\text{in $\R^d\setminus I_n$},\\
\Div( \psi_{E}^{\{n\}})=0,&\text{in $\R^d\setminus I_n$},\\
\D(\psi_{E}^{\{n\}}+Ex)=0,&\text{in $I_n$},\\
\int_{\partial I_{n}}\sigma(\psi_{E}^{\{n\}}+Ex,\Sigma_{E}^{\{n\}})\nu=0,& \\
\int_{\partial I_{n}}\Theta(x-x_{n})\cdot\sigma(\psi_{E}^{\{n\}}+Ex,\Sigma_{E}^{\{n\}})\nu=0,&~\forall \Theta\in\Md^\Skew.
\end{array}\right.
\end{equation}
This amounts to neglecting the effect of other particles on $\psi_E$ around $I_n$, thus precisely neglecting the multibody nature of the problem.
To give a more precise statement, consider the Voronoi tessellation $\{V_n\}_n$ associated with the set of particles $\{I_n\}_n$, that is,
\[V_n\,:=\,\Big\{x:\dist(x,I_n)<\inf_{m:m\ne n}\dist(x,I_m)\Big\}.\]
The relevant approximation of $\psi_E$ then takes the form
\begin{equation}\label{eq:approx-Einstein-Vor}
\D(\psi_E)~~\approx~~\Psi_E^{\text{Einstein}}\,:=\,\sum_n \D(\psi_E^{\{n\}})\mathds1_{V_n}.
\end{equation}
Inserting this into the definition~\eqref{def:eff-vis} of the effective viscosity
yields after straightforward calculations, in case of spherical particles,
\begin{eqnarray}\label{e.Einstein-intro}
E:\Bb E=|E|^2+\expecm{|\!\D(\psi_E)|^2}&\approx& |E|^2+ \expecm{|\Psi_E^{\text{Einstein}}|^2}\nonumber\\
&=& |E|^2\big(1+\tfrac{d+2}2 \varphi+o(\varphi)\big),
\end{eqnarray}
in terms of the particle volume fraction
\begin{equation}\label{eq:volfrac}
\varphi\,:=\,\varphi(\Ic)\,:=\,\lim_{R\uparrow\infty}R^{-d}|\Ic\cap RQ|,
\end{equation}
where in 3D we recover the celebrated $\frac52$ factor, cf.~\eqref{eq:Einstein};
we refer to Section~\ref{sec:explicitEinstein} for the detailed computation.
Corrections to Einstein's formula are obtained by taking into account that $\psi_E$ does, in fact, depend on the positions of all particles at once. As we shall see, in the low-density regime, this is naturally written in form of a cluster expansion:
the next-order correction, known in the physics literature as the Batchelor--Green correction~\cite{BG-72}, involves the two-particle problem, and so on.

\subsubsection{Objectives}
In this memoir, we focus on the rigorous analysis of Einstein's hypothesis~(E2): we start from the relevant notion of effective viscosity~\eqref{def:eff-vis} as defined by homogenization theory and we study its asymptotic behavior at low density, aiming to justify Einstein's formula~\eqref{e.Einstein-intro} and to describe all higher-order corrections.

The early works~\cite{SP-85,LSP-85,Haines-Mazzu} focussed on Einstein's formula for locally periodic dilute arrays of particles.
It was extended in~\cite{Niethammer-Schubert-19,Hillairet-Wu-19} to the dilute disordered setting under the simplifying assumption that the minimal interparticle distance is large enough (that is, $\ell(\Pc)\gg1$ with the notation~\eqref{eq:def-ell(P)} below).
The next-order Batchelor--Green correction was captured in~\cite{GVH,GVM-20} in the same setting.
The uniform separation assumption is particularly convenient as it allows to exploit the reflection method and rigorously neglect many-particle interactions, e.g.~\cite{Jabin-Otto-04,Hofer-Velazquez-18,Hofer-18,Niethammer-Schubert-19,Hofer-19},
but it is physically quite restrictive and unsatisfactory.
More recently, it was replaced in~\cite{GV-Hofer-20} by some weaker non-concentration condition in the proof of Einstein's formula, however still requiring some control on the minimal interparticle distance. In this context, we shall address the following two main points:
\begin{enumerate}[---]
\item We shall justify Einstein's formula under the weakest assumptions under which homogenization is known to hold, in particular covering the case of the general subcritical percolation condition in~\cite{DG-21a}. At the same time, we aim at optimal error estimates: the error $o(\varphi)$ in~\eqref{eq:Einstein} was often claimed to be $O(\varphi^2)$, but we shall see that it actually strongly depends on the structure of the random ensemble of particles.
\smallskip\item We shall describe higher-order corrections to Einstein's formula in form of a cluster expansion. Due to the long-range nature of hydrodynamic interactions, renormalizations are needed to make sense of cluster contributions. In the physics literature, formal renormalizations were actually still lacking beyond the second-order Batchelor--Green correction. On the rigorous side, even the justification of the latter was restricted to some specific regimes~\cite{GVH,GVM-20,GV-20}.
\end{enumerate}
In terms of techniques, previous results on the topic relied on deterministic analysis, more precisely on various forms of the reflection method. In the present memoir, we rather take inspiration from our work~\cite{DG-16a} on the Clausius--Mossotti conductivity formula based on the triad consisting of: (1)~finite-volume approximation; (2)~cluster expansion; (3)~uniform \mbox{$\ell^1-\ell^2$} energy estimates. Substantially refining on this analysis, we go far beyond~\cite{DG-16a} by covering general dilute regimes (beyond the case of explicit dilution by random deletion), and we shall further describe the explicit renormalization of cluster coefficients.

\subsection{Cluster expansion formalism}
While Einstein's formula~\eqref{e.Einstein-intro} is obtained by considering dilute particles as being isolated, next-order corrections amount to taking into account many-particle interactions and the multibody structure of the corrector field $\psi_E$.
At low density, particles are scarce and one might want to consider contributions of  finite subsets of particles only.
As in  \cite{DG-16a}, taking inspiration from statistical mechanics, see e.g.~\cite[Chapter~19]{Torquato-02}, this is naturally performed by means of cluster expansions, which provide natural asymptotic series at low density.
We recall the formalism, discuss the accuracy of cluster expansions, and describe the key difficulty to apply it to the effective viscosity problem: the long-range nature of hydrodynamic interactions.

\subsubsection{Cluster expansions of multibody quantities}
We recall the cluster expansion formalism in the form that we introduced in~\cite{DG-16a}. As particles are indexed by natural numbers, we denote by~$P(\N)$ the set of subsets of the index set $\N$ and we consider the space~$M(\N)$ of set functions from~$P(\N)$ to a given vector space $V$.
Starting from the corrector problem~\eqref{eq:cor}, for any index subset~\mbox{$H\in P(\N)$}, we may consider\footnote{The corrector problem~\eqref{eq:cor} is, in fact, not well-posed in general for a given deterministic infinite subset~$H$ of particles. In the sequel, we shall rather consider finite-volume approximations of the corrector problem, for which well-posedness is trivial. We skip this detail at the level of the present discussion.}
the associated corrector $\psi_E^H$ obtained by replacing the full set~$\Ic$ of particles by its corresponding subset \mbox{$\Ic^H:=\cup_{n\in H} I_n$}. The map $\psi_E^\#:H\mapsto \psi_E^H$ is then viewed as an element of~$M(\N)$, where $\psi_E^\varnothing\equiv 0$ and where~$\psi_E^\N\equiv\psi_E$ is the original corrector defined in~\eqref{eq:cor}.

In this setting, for all $n\in \N$, we introduce a \emph{difference operator} $\delta^{\{n\}}:M(\N) \to M(\N)$, defined for all~$\Phi\in M(\N)$ by
\[\del^{\{n\}}\Phi^H~:=~\del^{\{n\}}\Phi^{H\cup\#}~:=~\Phi^{H\cup\{n\}}-\Phi^H,\qquad H\subset\N,\]
which provides a natural measure of the sensitivity of $\Phi$ with respect to the index~$n$
(it plays the role of a discrete derivative). Note that for all $n\ne m$,
\[(\del^{\{n\}})^2=-\del^{\{n\}},\qquad \del^{\{n\}}\del^{\{m\}}=\del^{\{m\}}\del^{\{n\}}.\]
For any finite $F\subset\N$, we also define the higher-order difference operator
\[\del^F\,:=\,\prod_{n\in F}\del^{\{n\}},\]
which acts as follows: for all $\Phi \in M(\N)$,
\begin{equation}\label{eq:def-diff}
\del^{F}\Phi^H \,=\,
\sum_{G\subset F}(-1)^{|F\setminus G|}\Phi^{G\cup H},\qquad H\subset \N.
\end{equation}
We take the natural convention $\del^\varnothing\Phi^H:=\Phi^H$.
These difference operators are the building blocks to construct the so-called {\it cluster expansions}, e.g.~\cite[Chapter~19]{Torquato-02}: to order $k$, the cluster expansion of $\Phi\in M(\N)$
takes the form 
\begin{equation*}
\Phi^\N \,\sim\, \Phi^{\varnothing}+\sum_{n}\del^{\{n\}}\Phi^\varnothing+\tfrac1{2!}\sum_{n_1,n_2}^{\ne}\del^{\{n_1,n_2\}}\Phi^\varnothing
+\ldots+\tfrac1{k!}\sum_{n_1,\ldots,n_k}^{\ne}\del^{\{n_1,\ldots,n_k\}}\Phi^\varnothing,
\end{equation*}
 where we use the short-hand notation $\sum_{n_1,\ldots,n_j}^{\ne}$ for sums over $j$-tuples $(n_1,\ldots,n_j)$ of distinct indices. This can be rewritten in the more compact form
\begin{align}\label{eq:clusterphiT}
\Phi^\N\,\sim\, \sum_{j=0}^k\sum_{\sharp F=j}\del^F\Phi^\varnothing,
\end{align}
where $\sum_{\sharp F=j}$ stands for the sum over all sets $F$ of $j$ distinct indices.
This expansion is particularly relevant in the low-density regime when particles are very scarce: the $0$th-order term corresponds to the situation without any particle, the $1$st-order term corresponds to contributions of isolated particles, the $2$nd-order term to contributions of pairs of particles, etc.
Formally, it can be viewed as a Taylor expansion associated with the difference operator~$\delta$,
where under suitable assumptions higher-order terms will be shown to be indeed of higher order at low density.
Note that, if $\Phi\in M(\N)$ only depends on indices in a finite subset $K \subset \N$ in the sense that $\Phi^{H} = \Phi^{H \cap K}$ for all $H\subset\N$, then the expansion~\eqref{eq:clusterphiT} is always a finite sum and is actually {\it equal} to~$\Phi^\N$ provided~$k\ge\sharp K$.

\subsubsection{Multi-point intensities}
The general estimation of the terms in the cluster expansion~\eqref{eq:clusterphiT} naturally leads to the notion of multi-point intensities, which appear as refined measures of diluteness and seem new to the literature.
Given an ergodic stationary point process $\Pc=\{x_n\}_n$,
we start by recalling the standard notion of {\it intensity} of the point process (or {\it one-point intensity} in our terminology below),
\[\lambda(\Pc)\,:=\,\lambda_1(\Pc)\,:=\,\expec{\sharp(\Pc\cap Q)},\]
where $Q:=[-\frac12,\frac12]^d$ stands for the unit cube.
By the ergodic theorem, we have almost surely
\begin{equation}\label{eq:conv-lambda}
\lambda(\Pc)\,=\,\lim_{R\uparrow\infty}R^{-d}\,\sharp\{n:x_n\in RQ\}.
\end{equation}
In particular, provided that random shapes satisfy $|I_n^\circ|\simeq1$ almost surely for all $n$, this relates to the particle volume fraction~\eqref{eq:volfrac} via
\begin{equation}\label{eq:link-philambda}
\varphi(\Ic)\,\simeq\,\lambda(\Pc),
\end{equation}
so that the low-density regime $\varphi(\Ic)\ll1$ is equivalently characterized by the condition~$\lambda(\Pc)\ll1$.
Yet, as we consider nonlinear functions of the point process (like the effective viscosity $\Bb$), this linear notion of diluteness is not strong enough and we need to introduce refined notions of ``multi-point intensities''.

For that purpose, we start by introducing a notation for the {\it minimal distance} of the point process $\Pc$,
\begin{equation}\label{eq:def-ell(P)}
\ell\,:=\,\ell(\Pc)\,:=\, \inf_{n\ne m} |x_n-x_m|_\infty,
\end{equation}
which is almost surely a deterministic characteristic length of $\Pc$.
 The point process is called hardcore if $\ell(\Pc)>0$, which is the case of all the processes considered in this memoir, cf.~\ref{H0} below.
For all $j\ge 1$, provided $\ell=\ell(\Pc)>0$, we then define the {\it$j$-point intensity}
\begin{equation}\label{eq:high-intens}
\qquad\lambda_{j}(\Pc)\,:=\,\sup_{z_1,\ldots,z_j}\expecM{\sum_{n_1,\ldots,n_j}^{\ne}\ell^{-d}\mathds1_{Q_{\ell}(z_1)}(x_{n_1}) \ldots \ell^{-d} \mathds1_{Q_{\ell}(z_j)}(x_{n_j})},
\end{equation}
where $Q_r(z):=z+r Q$ stands for the cube of sidelength $r$ centered at $z$. Note that, by definition~\eqref{eq:def-ell(P)}, each cube $Q_\ell(z)$ contains at most one point of~$\Pc$.
This definition corresponds to the maximum expected number of $j$-tuples of points of $\Pc$ that lie in the $\ell$-neighborhood of an element of $(\R^d)^j$, properly normalized by $\ell$.
Alternatively, recalling that the {\it$j$-point density}~$f_{j}$ associated with~$\Pc$ is the non-negative function defined 
by the following relation,
\begin{equation}\label{eq:def-fj}
\qquad\expecM{\sum_{n_1,\ldots, n_j}^{\ne} \zeta(x_{n_1},\ldots,x_{n_j})} \,=\, \int_{(\R^d)^j}\zeta f_j\qquad\text{for all $\zeta \in C^\infty_c((\R^d)^j)$},
\end{equation}
the definition~\eqref{eq:high-intens} of $j$-point intensity can be reformulated as
\begin{equation}\label{eq:high-intens/Re}
\lambda_j(\Pc)\,=\,\sup_{z_1,\ldots,z_j}\fint_{Q_{\ell}(z_1)\times\ldots\times Q_\ell(z_j)}f_j.
\end{equation}
In the case $\ell(\Pc)=0$, this definition is naturally extended to $\lambda_j(\Pc)=\|f_j\|_{\Ld^\infty((\R^d)^j)}$ for completeness.
In view of upcoming arguments, it is convenient to further introduce the following quantities,
\begin{equation}\label{e.mul-int-eff}
\underline \lambda_j(\Pc) \,:=\, \min_{ \sum_{i} j_i=j} \prod_{i} \lambda_{j_i}(\Pc)~~\le~~\overline \lambda_j(\Pc) \,:=\, \max_{ \sum_{i} j_i=j} \prod_{i} \lambda_{j_i}(\Pc).
\end{equation}
For a Poisson point process,
these quantities are, in fact, equivalent since independence yields $\lambda_j(\Pc)=\lambda(\Pc)^j$ for all $j\ge1$, hence $\underline\lambda_j(\Pc)=\overline\lambda_j(\Pc)=\lambda(\Pc)^j$.
For a hardcore Poisson point process, we similarly find
$\lambda_j(\Pc)\simeq_j\lambda(\Pc)^j$.
In other words, the one-point intensity~$\lambda(\Pc)$ is enough to fully describe
low-density regimes in those cases. However, multi-point intensities are non-trivial in general: for any $\beta\in[0,1]$, one can construct examples of point processes with $\lambda_2(\Pc)\simeq\lambda(\Pc)^{1+\beta}$ (see last paragraph of~Section~\ref{sec:cluster-recap}). For instance, given $e\in\R^d$, the point process $\Pc_e:=\Pc\cup(\Pc+e)$ consists of pairs of points $\{x_n,x_n+e\}$ and thus satisfies $\lambda_2(\Pc_e)\simeq\ell(\Pc_e)^{-d}\lambda(\Pc_e)$, hence $\lambda_2(\Pc_e)\simeq\lambda(\Pc_e)$ provided $\Pc_e$ is hardcore.
The following lemma states some general properties.

\begin{lem}[Multi-point intensities]$ $\label{lem:gen-cond-lambd}
Let $\Pc=\{x_n\}_n$ be an ergodic stationary random point process.
\begin{enumerate}[(i)]
\item For all $j\ge1$, we have
\[\lambda_{j+1}(\Pc)\le\ell(\Pc)^{-d} \lambda_j(\Pc).\]
\item If $\Pc$ is strongly mixing, then for all $j\ge1$ we have
\[\lambda(\Pc)^j \,=\,\underline \lambda_j(\Pc) \,\le\, \overline \lambda_j(\Pc) \,=\,  \lambda_j(\Pc).\]
(The same holds for all $j\le n$ under the mixing assumption~\emph{\ref{Mix-om-n}} introduced in Section~\ref{sec:explicit-prel} provided the rate $\omega$ decays at infinity.)
\smallskip\item Given $j\ge2$ and $\theta\in[0,1]$, for any nonnegative function $\phi\in C^\infty_c((\R^d)^j)$ that satisfies $\phi(z_1,\ldots,z_j)\le C\phi(z_1',\ldots,z_j')$ for all $z_1,z_1',\ldots,z_j,z_j'$ provided $\textstyle\max_i|z_i-z_i'|_\infty\le\theta\ell(\Pc)$ and $\textstyle\min_{j\ne i}|z_i-z_{j}|_\infty\ge\ell(\Pc)$,
we have
\[\int_{(\R^d)^k}\phi f_k\,\le\,C\theta^{-dk}\lambda_k(\Pc)\int_{(\R^d)^k}\phi.\qedhere\]
\end{enumerate}
\end{lem}
\begin{proof}
As each cube $Q_\ell(z)$ contains at most one point of $\Pc$, we find $\sum_n\ell^{-d}\mathds1_{Q_\ell(z)}(x_n)\le\ell^{-d}$ for all $z$, so that item~(i) readily follows from definition~\eqref{eq:high-intens}.

\medskip\noindent
We turn to the proof of~(ii). Given $j\ge1$, for any partition $0=k_1< k_2<\ldots<k_l=j$, setting $j_i:=k_{i+1}-k_i$,
the strong mixing of the point process implies
\[\fint_{Q_\ell(z_1)\times\ldots\times Q_\ell(z_j)}f_j-\prod_{i=1}^l\Big( \fint_{Q_\ell(z_{k_i+1})\times\ldots\times Q_\ell(z_{k_{i+1}})}f_{j_i}\Big)\longrightarrow0,\]
as $\min_{i \ne i'} \dist(Z_{i},Z_{i'})\to\infty$, where we use the short-hand notation $Z_i:=\{z_{k_i+1},\ldots,z_{k_{i+1}}\}$.
In view of~\eqref{eq:high-intens/Re}, using stationarity, this proves the estimate
$\lambda_j(\Pc) \ge \prod_{i=1}^l \lambda_{j_i}(\Pc)$,
from which the claim~(ii) easily follows.

\medskip\noindent
Finally, item~(iii) is a direct consequence of definition~\eqref{eq:high-intens/Re} of multi-point intensities, further using that the $j$-point density satisfies $f_j(x_1,\ldots,x_j)=0$ whenever there are some $i\ne i'$ with $|x_i-x_{i'}|<\ell(\Pc)$.
\end{proof}

\subsubsection{Scaling of cluster expansions}
With the above definitions, we may now determine the scaling of the terms in the cluster expansion~\eqref{eq:clusterphiT} and show the relevance of multi-point intensities. For that purpose, by way of illustration, we place ourselves in the elementary setting of short-range interactions, which will serve as a guideline in the sequel. More precisely, consider a set function $\Phi:P(\N)\to\R$ of the form
\begin{equation}\label{eq:def-Phi-short}
\Phi^H\,:=\,\expecM{g\Big(\sum_{n\in H}h(x_n)\Big)},
\end{equation}
for some $h:\R^d\to\R$ and $g:\R\to\R$ such that
\begin{enumerate}[(a)]
\item $h$ is {\it short-range}, in the sense that $\int_{\R^d}(\sup_{B(z)}|h|)\,dz<\infty$;
\item $g$ is {\it smooth}, in the sense that $g\in C^\infty_b(\R)$.
\end{enumerate}
The cluster expansion of $\Phi^\N$, cf.~\eqref{eq:clusterphiT}, then takes the form
\begin{equation}\label{eq:Phi-cluster}
\Phi^\N \sim \sum_{j=0}^\infty \tfrac1{j!} \bar \Phi^j, \qquad\text{where}\quad \bar \Phi^j:= j! \sum_{\sharp F=j}\del^F\Phi^\varnothing.
\end{equation}
Although cluster coefficients $\{\bar\Phi^j\}_j$ are defined by infinite series, these series are always summable in this short-range setting and we show that they are naturally estimated by multi-point intensities.In particular, the second-order coefficient $\bar\Phi^2$ is bounded by $O(\lambda_2(\Pc))$, which contradicts in general the bound $O(\lambda(\Pc)^2)=O(\varphi^2)$ that one could have naively expected.
Our main goal in this memoir is precisely to establish corresponding expansions and estimates for the effective viscosity~\eqref{eq:Einstein} \& \eqref{def:eff-vis}.

\begin{lem}[Cluster expansions in the short-range setting]\label{lem:short}
Let $\Pc = \{x_n\}_{n}$ be an ergodic stationary point process on $\R^d$ with $\ell(\Pc)\lesssim1$,
let $\Phi$ be a set function of the form~\eqref{eq:def-Phi-short} satisfying the short-range and smoothness assumptions~\emph{(a) \& (b)} above, and let $\{\bar\Phi^j\}_j$ be the associated cluster coefficients~\eqref{eq:Phi-cluster}. Then we have for all~\mbox{$k\ge1$,}
\begin{equation}\label{e.scalings-local}
\Big|\Phi^\N -  \sum_{j=0}^k  \tfrac1{j!}\bar \Phi^j \Big|\,\lesssim_{k,g,h}\, \lambda_{k+1}(\Pc), \qquad |\bar \Phi^k|\,\lesssim_{k,g,h}\, \lambda_k(\Pc),
\end{equation}
in terms of multi-point intensities $\{\lambda_j(\Pc)\}_j$, cf.~\eqref{eq:high-intens}.
\end{lem}

\begingroup\allowdisplaybreaks
\begin{proof}
Given a  sequence $Y:=\{y_n\}_n \subset \R$, define a set function $\Psi_Y:P(\N)\to\R$ by
\[\Psi_Y^H\,:=\,g\Big(\sum_{n\in H}y_n\Big),\qquad H\subset\N.\]
By definition of difference operators, cf.~\eqref{eq:def-diff}, we find, in the spirit of Taylor's remainder formulas, 
\begin{eqnarray*}
\delta^{\{n_1,\ldots,n_k\}}\Psi_Y^\varnothing&=&\int_0^{y_{n_1}}\!\!\!\!\!\ldots\!\int_0^{y_{n_k}}g^{(k)}(t_1+\ldots+t_k)\,dt_1\ldots dt_k,\\
\Psi_Y^\N-\sum_{j=0}^k\sum_{\sharp F=j}\delta^F\Psi_Y^\varnothing
&=&
\sum_{n_1<\ldots<n_{k+1}}\int_0^{y_{n_1}}\!\!\!\!\!\ldots\!\int_0^{y_{n_{k+1}}}\\
&&\hspace{1.7cm}\times g^{(k+1)}\Big(t_1+\ldots+t_{k+1}+\sum_{n>n_{k+1}}y_n\Big)\,dt_1\ldots dt_{k+1}.
\end{eqnarray*}
These identities yield in particular
\begin{eqnarray*}
|\delta^{\{n_1,\ldots,n_k\}}\Psi_Y^\varnothing|&\le&\|g^{(k)}\|_{\Ld^\infty(\R^d)}\prod_{j=1}^k|y_{n_j}|,\\
\Big|\Psi_Y^\N-\sum_{j=0}^k\sum_{\sharp F=j}\delta^F\Psi_Y^\varnothing\Big|
&\le&\|g^{(k+1)}\|_{\Ld^\infty(\R^d)}\sum_{\sharp F=k+1}\prod_{n\in F}|y_{n}|.
\end{eqnarray*}
Setting $Y:=\{h(x_n)\}_n$, noting that definition~\eqref{eq:def-Phi-short} reads $\Phi^H=\expec{\Psi_Y^H}$, inserting the definition~\eqref{eq:Phi-cluster} of cluster coefficients, and recalling the definition~\eqref{eq:def-fj} of multi-point density functions, this yields
\begin{eqnarray*}
|\bar\Phi^k|&\le&\|g^{(k)}\|_{\Ld^\infty(\R^d)}\int_{(\R^d)^k}|h|^{\otimes k}f_k,\\
\Big|\Phi^\N-\sum_{j=0}^k\frac1{j!}\bar\Phi^j\Big|&\le&\tfrac1{(k+1)!}\|g^{(k+1)}\|_{\Ld^\infty(\R^d)}\int_{(\R^d)^{k+1}}|h|^{\otimes(k+1)}f_{k+1}.
\end{eqnarray*}
By definition~\eqref{eq:high-intens/Re} of multi-point intensities, the conclusion follows.
\end{proof}
\endgroup

\subsubsection{Effective viscosity: long-range issues and renormalization}\label{sec:need-renorm}
We apply the above cluster expansion formalism to the effective viscosity~\eqref{def:eff-vis}.
For a finite subset $H\subset\N$, recall the notation $\psi_E^H$ for the solution of the corrector problem~\eqref{eq:cor} where the full set $\Ic$ of particles is replaced by its subset $\Ic^H=\cup_{n\in H}I_n$ (this corrector problem is trivially well-posed when $H$ is finite).
We then define a symmetric linear map $\Bb^H$ on $\Md_0^\Sym$ by
\[E:\Bb^HE\,:=\,\expecm{|\!\D(\psi_{E}^H)(0)+E|^2},\qquad E\in\Md_0^\Sym.\]
In these terms, the formal cluster expansion of the effective viscosity~\eqref{def:eff-vis} takes the form
\begin{equation}\label{e.formal-expansion}
\Bb \,\sim\,  \sum_{j=0}^\infty  \tfrac1{j!}\Bb^j, \qquad \text{where}\quad\Bb^j:= j! \sum_{\sharp F=j}\del^F\Bb^\varnothing.
\end{equation}
Note that $\Bb^0= \Id$ is the plain fluid viscosity.
In contrast with the short-range setting of Lemma~\ref{lem:short} above, however, series defining cluster coefficients~$\{\Bb^j\}_{j\ge1}$ are not summable due to the long-range nature of hydrodynamic interactions.
Indeed, the first coefficient~$\Bb^1$ takes the form
\begin{eqnarray}
E:\Bb^1E&=&\sum_nE:\delta^{\{n\}}\Bb^\varnothing E\nonumber\\
&=&\sum_n\expec{|\!\D(\psi_E^{\{n\}})(0)|^2+2E:\D(\psi_E^{\{n\}})(0)}.\label{eq:cluster-B1}
\end{eqnarray}
As $\psi_E^{\{n\}}$ satisfies the single-particle problem~\eqref{eq:singlepart}, it decays like $|\!\D(\psi_E^{\{n\}})(x)|\lesssim\langle x-x_n\rangle^{-d}$, which entails that the above series is not absolutely convergent,
\[\sum_n\expec{|\!\D(\psi_E^{\{n\}})(0)|}=\infty.\]
The same borderline divergence is observed for all cluster coefficients $\{\Bb^j\}_{j\ge1}$ in~\eqref{e.formal-expansion}.
In order to make sense of cluster coefficients, suitable renormalization procedures are thus required and constitute the major difficulty of the problem.

To first order, the needed renormalization happens to be trivial: by definition of the intensity of the point process, identity~\eqref{eq:cluster-B1} can be equivalently rewritten as follows (say, in case of deterministic particle shapes),
\begin{equation*}
E:\Bb^1E\,=\,\lambda(\Pc)\int_{\R^d}\Big(|\!\D(\psi_E^{\circ})|^2+2E:\D(\psi_E^{\circ})\Big),
\end{equation*}
where $\psi_E^{\circ}$ stands for the solution of the single-particle problem~\eqref{eq:singlepart} with a particle centered at the origin. Here, we observe that in any finite-volume approximation the linear term $\int_{\R^d}E:\D(\psi_E^{\circ})$ would be given a vanishing value as the integral of a gradient.
Removing this linear term, we are left with the following summable integral,
\begin{equation}\label{eq:B1-renorm}
E:\Bb^1E\,=\,\lambda(\Pc)\int_{\R^d}|\!\D(\psi_E^{\circ})|^2,
\end{equation}
which happens to coincide with Einstein's formula~\eqref{e.Einstein-intro} in case of spherical particles.
In contrast, higher-order renormalizations are not obtained by such simple cancellations.
In the physics literature, the difficulty was recognized by Batchelor and Green~\cite{BG-72}, who managed to provide a heuristic renormalization for the second-order term $\Bb^2$. The systematic renormalization of higher orders is more involved and has remained an open problem so far even on the heuristic level in physics.
The present memoir is precisely devoted to the systematic treatment of this difficulty: we provide suitable renormalizations of cluster coefficients and in turn justify the expansion~\eqref{e.formal-expansion} to all orders. In the end, we prove essentially the same estimates on the cluster expansion as in the short-range setting~\eqref{e.scalings-local}, up to (sharp) logarithmic corrections that are persisting manifestations of the long-range nature of interactions, cf.~\eqref{e.intro-main-high-order} below.

\subsection{Main results}
This section is devoted to a brief, informal account of the main results of this memoir, with precise references to the relevant sections. We refer to the conclusion in Section~\ref{sec:gen-dil} for a detailed recap of all our results.
We start with the main assumptions on the ensemble of rigid particles.

\subsubsection{Main assumptions}
Given an underlying probability space $(\Omega,\Pm)$, let $\Pc=\{x_n\}_n$ be a random point process on $\R^d$, consider an associated collection of random shapes $\{I_n^\circ\}_n$, where each $I_n^\circ$ is a random connected open subset of the unit ball $B$, centered at the origin in the sense of $\int_{I_n^\circ} y\,dy=0$, and then define the corresponding inclusions
\[I_n:=x_n+I_n^\circ.\]
Note that random shapes are not required to be independent of the point process~$\Pc$.
We then consider the random set $\Ic:=\bigcup_nI_n$, which we assume to satisfy the following conditions.  Note that the disjointness and $\rho$-regularity conditions below entail that the point process~$\Pc$ is hardcore with $\ell(\Pc)\gtrsim\rho$, cf.~\eqref{eq:def-ell(P)}.

\begin{asn}{\textbf{(H$_\rho$)}}{General conditions with parameter $\rho>0$}\label{H0}$ $
\begin{enumerate}[\quad$\bullet$]
\item \emph{Stationarity and ergodicity:} The point process $\Pc=\{x_n\}_n$ and the associated random set $\Ic$ are stationary and ergodic.\footnote{More precisely, stationarity means that the laws of the translated point process $x+\Pc$ and of translated set $x+\Ic$ are independent of the shift $x\in\R^d$. Ergodicity then means that a measurable function of $\Pc$ or~$\Ic$ is almost surely unchanged for $\Pc$ or $\Ic$ replaced by $x+\Pc$ or $x+\Ic$ for any $x\in\R^d$ only if it is almost surely constant. Note that shifts $x\in\R^d$ can be replaced by discrete shifts $x\in\Z^d$, and periodic point sets can be considered as a particular case, for which the expectation is replaced by the average over a period.}
\smallskip\item \emph{Disjointness:} There holds $I_n\cap I_m=\varnothing$ almost surely for all $n\ne m$.
\smallskip\item \emph{$\rho$-Regularity:}
Random shapes $\{I_n^\circ\}_n$ almost surely satisfy interior and exterior ball conditions with radius $\rho$.
\qedhere
\end{enumerate}
\end{asn}

Next, we define the effective viscosity tensor $\Bb$ associated with the suspension $\Ic$ as the quadratic form on $\Md_0^\Sym$ given in~\eqref{def:eff-vis}.
We emphasize that the corrector problem \eqref{eq:cor} only makes sense provided that all particles $\{I_n\}_n$ are separated. If this separation is uniform, the pressure $\Sigma_{E}\mathds1_{\R^d\setminus \Ic}$ can also be uniquely constructed as a stationary field with finite second moment and vanishing expectation, cf.~\cite[Proposition~2.1]{DG-19}.
When particles are not well separated, the corrector problem should rather be considered via its variational formulation
and the effective viscosity is then defined as the minimum value
\begin{multline}\label{def:eff-vis-min}
E:\Bb E\,=\,\inf\Big\{\expecm{|\!\D(\psi)+E|^2}~:~\text{$\psi\in\Ld^2(\Omega;H^1_\loc(\R^d)^d)$, $\nabla\psi$ stationary,}\\
\text{$\Div(\psi)=0$, $(\D(\psi)+E)|_{\Ic}=0$, $\expec{\D(\psi)}=0$}\Big\}.
\end{multline}
At this stage, nothing prevents this infimum from being infinite: as explained after~\eqref{eq:cor} above, the problem originates from the possible existence of unbounded chains of touching particles. This will be excluded by means of further geometric assumptions, cf.~\ref{Gunif}, \ref{Gmom}, or~\ref{Gperc} below. Even if the infimum is finite, nothing ensures in general that~$\Bb$ defines the effective viscosity in the sense of homogenization theory: we view this as a separate question, which is extensively discussed in different settings in our previous work~\cite{DG-19,D-20a,DG-21a,DG-21b} and will not be further discussed here.
We are now in the position to describe our main results.

\subsubsection{First-order expansion: Einstein's formula}
Assuming that particles are uniformly separated by a positive distance, cf.~\ref{Gunif}, we prove in the general ergodic stationary setting,
\begin{equation}\label{e.intro-main-order1}
|\Bb-\Id-\Bb^1| \,\lesssim\,  \lambda_2(\Pc) \log  \big(2+\tfrac{\lambda(\Pc)}{\lambda_2(\Pc)}\big),\qquad |\Bb^1|\,\simeq \,\lambda(\Pc),
\end{equation}
where $\Bb^1$ is given by the renormalized cluster formula~\eqref{eq:B1-renorm} and takes the explicit form of Einstein's formula $\Bb^1=\frac{d+2}{2}\varphi$ in case of spherical particles. This error estimate is new and optimal, and the stochastic assumption of mere ergodicity is minimal. In particular, we find that Einstein's formula $\Bb\approx\Id+\Bb^1$ is accurate to leading order provided that $\lambda_2(\Pc)\ll\frac{\lambda(\Pc)}{|\!\log\lambda(\Pc)|}$, which amounts to a very weak local independence assumption.
Yet, the uniform separation assumption is not satisfactory from the physical point of view.
At the price of weakening the error estimate~\eqref{e.intro-main-order1}, we may relax this assumption as we did for the homogenization result in~\cite{D-20a,DG-21a}: either we assume moment bounds on the interparticle distance~\ref{Gmom} (see also~\cite{GV-Hofer-20}),
or we consider a subcritical percolation condition~\ref{Gperc} (in which case particles are allowed to touch provided they do not cluster --- which is new).
We refer to Theorem~\ref{theor:Einsteinx3} in Section~\ref{sec:Einsteinx3} for a detailed statement.

\subsubsection{Higher-order cluster corrections}
For the higher-order analysis, we assume for simplicity that particles are uniformly separated by a positive distance, cf.~\ref{Gunif}. Under a slight strengthening of ergodicity,
the formal cluster expansion is well-defined, up to suitable renormalization of cluster coefficients~\eqref{e.formal-expansion}, and it essentially\footnote{The true estimate is in general slightly more complicated than what is stated here; cf.~Theorem~\ref{th:mod-free1} in Section~\ref{sec:gen-dil}.} satisfies for all $k\ge1$,
\begin{equation}\label{e.intro-main-high-order}
\Big|\Bb-\sum_{j=0}^k\Bb^j\Big| \,\lesssim\,  \lambda_{k+1}(\Pc) |\!\log  \lambda(\Pc)|^{k}, \qquad  |\Bb^k|\lesssim  \lambda_k(\Pc)\,|\!\log \lambda(\Pc)|^{k-1}.
\end{equation}
These estimates coincide remarkably with the corresponding result~\eqref{e.scalings-local} in the short-range setting, to the exception of logarithmic corrections that are precisely the manifestation of the long-range nature of hydrodynamic interactions.
The result is new for any $k\ge 2$ and logarithmic corrections are expected to be optimal (optimality is proved for $k=2$, cf.~Theorem~\ref{lem:optimality}).
We also believe that the slightly strengthened ergodicity assumption is necessary for the result to hold.
We refer to Theorem~\ref{th:mod-free1} in Section~\ref{sec:gen-dil} for a detailed statement.
In particular, our analysis justifies the  Batchelor--Green formula for the second-order term~$\Bb^2$, cf.~Proposition~\ref{prop:B2-ren} (see also Corollary~\ref{cor:B2ell-re}), and we develop a systematic renormalization scheme for all higher-order cluster coefficients by means of diagrammatic expansions, cf.~Section~\ref{sec:explicit}.

We emphasize that the above result~\eqref{e.intro-main-high-order} holds without any structural assumption on the dilution process (which we call the model-free setting).
If we make the dilution more specific, considering for instance a random deletion procedure (as in~\cite{DG-16a}) or dilation, then the cluster expansion can be shown to define an absolutely converging series. We refer to Theorem~\ref{th:analytic} for a detailed analyticity statement. All previous results on the second-order expansion~\cite{GVH,GVM-20,GV-20} were, in fact, essentially restricted to such specific settings.

\subsection{Roadmap to the main results}
The rest of the memoir is divided into four sections. Section~\ref{sec:Einsteinx3} is dedicated to the proof of Einstein's formula. Section~\ref{sec:cluster} studies the cluster expansion of finite-volume approximations $\{\Bb_L\}_L$ of the effective viscosity $\Bb$. In Section~\ref{sec:intermezzo}, we deal with the issue of systematic renormalization of cluster coefficients, which leads us to justifying the cluster expansion of $\Bb$. Our different results are combined and summarized in Section~\ref{sec:gen-dil}.
We briefly describe below our approach for each step.

\subsubsection{Einstein's formula: first-order expansion --- Section~\ref{sec:Einsteinx3}}
We develop a new, purely variational approach to Einstein's formula~\eqref{e.intro-main-order1}; a short self-contained proof is given in Section~\ref{sec:Einsteinx3}. It amounts to constructing competitors for the variational problem \eqref{def:eff-vis-min} and to controlling their energy difference by means of elliptic regularity.
The variational nature of the argument allows us to avoid uniform particle separation assumptions and to cover in particular the case of colliding particles under a general non-clustering assumption. It also allows to avoid the need for fine pressure estimates, which is crucial as such estimates would be problematic in case of colliding particles.

\subsubsection{Cluster expansion of the effective viscosity --- Section~\ref{sec:cluster}}
While coefficients in the formal cluster expansion of the effective viscosity $\Bb$ are given by infinite series that are not summable due to the long-range nature of hydrodynamic interactions, cf.~Section~\ref{sec:need-renorm}, we start by considering finite-volume approximations $\{\Bb_L\}_{L\ge1}$ obtained by periodization of the variational problem~\eqref{def:eff-vis-min}.
Section~\ref{sec:cluster} provides a detailed analysis of the cluster expansion of $\Bb_L$ for fixed~$L$.
\begin{enumerate}[$\bullet$]
\item First, we give explicit formulas for the coefficients $\{\Bb_L^j\}_j$ of the cluster expansion, as well as an explicit estimate for the remainder $R^{k+1}_L:=\Bb_L-\sum_{j=0}^k \frac1{j!} \Bb_L^j$,
in terms of correctors associated with finite subsets of particles;
see Theorem~\ref{thm:expansion}. The argument is essentially combinatorial. Note that the proof of remainder estimates further makes key use of the rigidity of the particles.
\smallskip\item Second, we prove that the cluster coefficients $\{\Bb_L^j\}_j$ and the remainder $R_L^{k+1}$ are bounded uniformly in $L$. The idea of the proof is as follows: if infinite-volume cluster formulas are given by infinite series that are not summable, they can in fact be viewed as complicated (non-explicit) combinations of Calder\'on--Zygmund kernels. As the effective viscosity is an $\Ld^2$-based quantity, we may expect to estimate cluster formulas by means of suitable energy estimates, carefully avoiding to take absolute values of any Calder\'on--Zygmund kernel.
Taking inspiration from our previous work~\cite{DG-16a}, this is achieved by means of a hierarchy of so-called interpolating $\ell^1-\ell^2$ energy estimates (also crucially used in \cite{GGMN-21,DG-22}).
As a corollary, uniform estimates allow to define infinite-volume cluster coefficients in the limit \mbox{$\{\Bb^j\}_j:=\lim_{L\uparrow\infty}\{\Bb_L^j\}_j$}.
Yet, being based on energy arguments, these estimates do not display the desired dependence~\eqref{e.intro-main-high-order} on multi-point intensities $\{\lambda_j\}_j$.
\smallskip\item Third, we prove corresponding cluster estimates that have the same dependence on multi-point intensities as in the short-range setting, but display a logarithmic divergence in the large-volume limit.
This is obtained by proceeding as for the short-range setting of Lemma~\ref{lem:short}, and the logarithmic divergence follows from estimating hydrodynamic interactions too roughly.
\end{enumerate}
It remains to show that the dependence on multi-point intensities is actually kept in the large-volume limit (at the price of logarithmic corrections).

\subsubsection{Renormalization of cluster formulas --- Section~\ref{sec:intermezzo}}
In order to prove the relevant infinite-volume cluster estimates~\eqref{e.intro-main-high-order}, we need a better understanding of cluster formulas and of the underlying compensations that make them well-defined in the large-volume limit.
A first route proceeds by assuming an algebraic convergence rate for the finite-volume approximations $\{\Bb_L\}_L$ of the effective viscosity: this is known to hold under quantitative $\alpha$-mixing condition whose rate is then transmitted (suboptimally) to cluster coefficients~$\{\Bb_L^j\}_j$, which allows in turn to keep the desired dependence on multi-point intensities in the cluster estimates while removing the logarithmic divergence. This implicit renormalization argument is particularly robust (see also~\cite{DG-16a}), but it does not provide any understanding of underlying cancellations and leaves several questions open.

Next, further assuming that particle shapes are independent of particle positions, we show that an explicit renormalization of cluster formulas can be developed: taking advantage of several explicit cancellations, cluster formulas can be transformed into summable integral formulas. This renormalization is trivial for $\Bb^1$, cf.~\eqref{eq:B1-renorm}, and the required cancellations are already more involved for $\Bb^2$, as formally understood by Batchelor and Green~\cite{BG-72a}. At higher-order, renormalizations rely on a suitable diagrammatic decomposition of cluster formulas to make cancellations manifest. Next, the direct analysis of renormalized formula allows to recover the desired cluster estimates~\eqref{e.intro-main-high-order} and to show that logarithmic corrections in those bounds are actually optimal in general.

\vfill
\subsection*{Notation}
\begin{enumerate}[\quad$\bullet$]
\item For vector fields $u,u'$ and matrix fields $T,T'$, we set $(\nabla u)_{ij}=\nabla_ju_i$, $\Div (T)_i=\nabla_jT_{ij}$, $T:T'=T_{ij}T_{ij}'$, $(u\otimes u')_{ij}=u_iu'_j$, where we systematically use Einstein's summation convention on repeated indices. We also denote by $(\D(u))_{ij}=\frac12(\nabla_ju_i+\nabla_iu_j)$ the symmetrized gradient. For a velocity field $u$ and associated pressure field $P$, we define the associated Cauchy stress tensor, cf.~\eqref{eq:Cauchystress},
\begin{equation}\label{eq:Cauchystress-re}
\sigma(u,P)\,:=\,2\D(u)-P\Id.
\end{equation}
\item We denote by $\Md_0^\Sym\subset\R^{d\times d}$ the subset of symmetric trace-free matrices, and by $\Md^\Skew$ the subset of skew-symmetric matrices.
\smallskip\item We use the notation $\lesssim$ (resp.\@ $\gtrsim$) for $\le C \times$ (resp.\@ $\ge \tfrac1C\times$) with a constant $C$ that depends only on the dimension $d$ and on the parameters appearing in the different assumptions when applicable. Note that the value of the constant $C$ is allowed to change from one line to another. We add subscripts to $C$, $\lesssim$, or $\gtrsim$ to indicate the dependence on other parameters.
We write $a \simeq b$ when both $a \lesssim b$ and $a\gtrsim b$ hold.
In addition, we write $\ll$ (resp.\@ $\gg$) for $\le \frac1C\times$ (resp.\@ $\ge C\times$) for some sufficiently large constant $C$.
\smallskip\item The ball centered at $x$ of radius $r$ in $\R^d$ is denoted by $B_r(x)$, and we set $B(x)=B_1(x)$, $B_r=B_r(0)$, and $B=B_1(0)$. We denote by $Q_{r}(x)=x+[-\tfrac r2,\tfrac r2)^d$ the cube of sidelength~$r$ centered at $x$, and we set $Q(x)=Q_1(x)$, $Q_r=Q_r(0)$, and $Q=Q_1(0)$.
\smallskip\item  For $x\in\R^d$, we denote by $|x|$ its Euclidean norm and by $|x|_\infty$ its supremum norm.
We also set $\langle x\rangle =(1+|x|^2)^{1/2}$, and similarly $\langle\nabla\rangle=(1-\triangle)^{1/2}$.
\smallskip\item We use the short-hand notation $\sum_{n_1,\ldots,n_j}^{\ne}$ for sums over $j$-tuples $(n_1,\ldots,n_j)$ of distinct indices. We also use the notation $\sum_{\sharp F=j}$ for the sum over all subsets $F$ of $j$ distinct indices.
\end{enumerate}

\newpage
\section{Einstein's formula: first-order expansion}\label{sec:Einsteinx3}

\subsection{Main result}\label{sec:main-Einstein}
Assumption~\ref{H0} first needs to be complemented with suitable geometric assumptions on the ensemble of particles to ensure that the effective viscosity~\eqref{def:eff-vis-min} is finite.
This can either be performed by means of conditions on interparticle distances,
\begin{equation}\label{eq:dist-rhon}
\rho_n\,:=\,\frac12\min_{m:m\ne n}\dist(I_n,I_m),
\end{equation}
or in terms of conditions on the size of clusters of close particles.
This has been the subject of our recent series of articles~\cite{DG-19,D-20a,DG-21a}, where the finiteness of the effective viscosity and the validity of a homogenization result are obtained 
under any of the following three types of assumptions:
\begin{enumerate}[---]
\item interparticle distances are uniformly bounded below, cf.~\cite{DG-19};
\smallskip\item interparticle distances satisfy suitable reciprocal moment bounds, cf.~\cite{D-20a};
\smallskip\item diameters of clusters of close particles satisfy suitable moment bounds in a subcritical percolation perspective, cf.~\cite{DG-21a}.
\end{enumerate}
These are formulated more precisely in Assumptions~\ref{Gunif}, \ref{Gmom}, and~\ref{Gperc} below, respectively.

\begin{asn}{\textbf{(H$_\rho^\text{unif}$)}}{Uniform separation with parameter $\rho>0$}\label{Gunif}$ $\\
Particles are uniformly separated with minimal distance \mbox{$\rho_n>\rho$}, that is, we have almost surely $(I_n+\rho B)\cap (I_m+\rho B)=\varnothing$ for all $n\ne m$.
\end{asn}
\begin{asn}{\textbf{(H$_{\rho,\kappa}^\text{mom}$)}}{Moment condition with parameters $\rho>0$, $\kappa>1$}\label{Gmom}$ $
\begin{enumerate}[\quad$\bullet$]
\item \emph{$\rho$-Uniform non-degeneracy of contact points:}
Pairs of ``$\rho$-close'' particles can be ``$\rho$-locally'' included in pairs of disjoint spheres with ``$\rho$-uniformly'' bounded radius.
For ``$\rho$-close'' particles, instead of~\eqref{eq:dist-rhon}, we then define $\rho_n$
as (half of) the distance between locally covering spheres. For a more precise statement of this geometric condition, we refer to~\cite[Assumption~{\bf(H$_\delta'$)}]{D-20a}. Note that this condition is trivially satisfied in case of spherical particles.
\smallskip\item \emph{Reciprocal moment bound:}
There exists $\Kc_\kappa<\infty$ such that
\[
\lim_{R \uparrow\infty} \bigg(\frac1{\sharp\{n:I_n \subset Q_R\}}  \sum_{n:I_n \subset Q_R} \mu(\rho_n)^{\kappa}\bigg)^\frac1\kappa \,\le \, \Kc_\kappa,\]
in terms of
\begin{equation}\label{e.def-mu}
\qquad\mu(t)\,:=\,\left\{\begin{array}{lll}
t^{-\frac{1}2(5-d)}&:&d<5,\\
\log(2+\frac1t)&:&d=5,\\
1&:&d>5.
\end{array}\right.
\end{equation}
Note that this condition is trivially satisfied for any $\kappa>1$ in case $d>5$.\qedhere
\end{enumerate}
\end{asn}
\begin{asn}{\textbf{(H$_{\rho,\kappa}^\text{perc}$)}}{Cluster condition with parameters $\rho>0$, $\kappa>1$}\label{Gperc}$ $\\
Let $\{K_{q,\rho}\}_q$ be the family of connected components of the fattened set $\Ic+\rho B$, and consider the corresponding clusters
\[\qquad J_{q,\rho}\,:=\,\bigcup_{I_n\subset K_{q,\rho}}I_n.\]
Given $p_0 \gg 1$ large enough (related to the existence of correctors in~\cite[Proposition~2]{DG-21a}),
there exists~$\Kc_\kappa<\infty$ such that
\begin{equation}\label{e.bd.perco}
\lim_{R \uparrow\infty} \bigg(\frac{1}{\sharp\{q:J_{q,\rho}\cap Q_R\ne\varnothing\}}  \sum_{q:J_{q,\rho}\cap Q_R\ne\varnothing} \diam(J_{q,\rho})^{\kappa p_0+d}\bigg)^\frac1\kappa \le \Kc_\kappa.\qedhere
\end{equation}
\end{asn}

Assumptions~\ref{Gmom} and~\ref{Gperc} are always weaker than~\ref{Gunif}. While~\ref{Gmom} only allows for particle contacts in dimension $d>5$ and is in particular incompatible with the 3D steady-state behavior of the two-particle density as computed in~\cite{BG-72a},
Assumption~\ref{Gperc} is of a different nature and allows for particle contacts in any dimension, but implicitly requires some strong mixing condition to ensure the validity of moment bounds on cluster diameters, cf.~\cite{DG-21a}.
In~\cite{DG-19,D-20a,DG-21a}, we show that these assumptions ensure the finiteness of the effective viscosity~\eqref{def:eff-vis-min} and the well-posedness of the corrector problem. In case of~\ref{Gmom} or~\ref{Gperc}, the validity of the homogenization result requires further strengthened conditions.

The following theorem states the validity of Einstein's formula under each of those assumptions. The proof is split between Sections~\ref{sec:Einstein-var}, \ref{sec:prelim-Einstein}, \ref{sec:proof-lem:main}, and~\ref{sec:proof-lem:diff-DN} below.

\begin{theor1}[Einstein's formula]\label{theor:Einsteinx3}
Under Assumption~\ref{H0} and either Assumption~\ref{Gunif}, \ref{Gmom}, or~\ref{Gperc}, for some $\rho>0$ and $\kappa>1$,
we have
\begin{align}\label{e.def-bd-Einstein}
&|\Bb-(\Id+\Bb^1)|\,\lesssim_{\rho}\,\lambda_2(\Pc) \log\big(2+\tfrac{\lambda(\Pc)}{\lambda_2(\Pc)}\big)\\
&\hspace{4cm}+\left\{
\begin{array}{lll}
0&:&\text{in case of~\ref{Gunif}},\\
\Kc_\kappa\, \lambda_2(\Pc)^{1-\frac1\kappa}\lambda(\Pc)^{\frac1\kappa}&:& \text{in case of~\ref{Gmom} or~\ref{Gperc}},
\end{array}
\right.\nonumber
\end{align}
where $\Bb^1$ satisfies
\[|\Bb^1|\simeq \lambda(\Pc),\]
and is defined for all $E\in\Md_0^\Sym$ by
\begin{equation}\label{eq:def-B1}
E:\Bb^1E\,:=\,\sum_n\expecM{\frac{\mathds1_{0\in I_n}}{|I_n|}\int_{\R^d}|\!\D(\psi_E^{\{n\}})|^2},
\end{equation}
where $\psi_E^{\{n\}}$
is the unique decaying solution of the single-particle problem~\eqref{eq:singlepart}.
In particular, the estimate $|\Bb-(\Id+\Bb^1)|=o(\lambda(\Pc))$ holds provided the point process $\Pc$ satisfies the weak local independence condition $\lambda_2(\Pc)=o({\lambda(\Pc)}/{|\!\log\lambda(\Pc)|})$.
\end{theor1}

As outlined in Section~\ref{sec:Einstein-var}, our proof is variational and amounts to proving lower and upper bounds on $\Bb$ that match with $\Id+\Bb^1$ to the required accuracy.
This approach is particularly robust: it allows to obtain the first optimal error estimate and to cover the most general setting regarding particle separation assumptions.
We briefly emphasize these two points:
\begin{enumerate}[---]
\item {\it Optimality:} In case of~\ref{Gunif}, the error estimate~\eqref{e.def-bd-Einstein} for Einstein's formula is new and sharp. As will be seen in Theorem~\ref{lem:optimality}, it indeed coincides with the general scaling of the next term $\Bb^2$ in the cluster expansion: the logarithmic correction is related to the long-range nature of hydrodynamic interactions
and cannot be avoided in general, thus contrasting with the short-range setting~\eqref{e.scalings-local}.\footnote{In some special cases, however, for instance in the statistically isotropic setting, the logarithmic correction can be removed, cf.~Theorem~\ref{lem:optimality}(i).}
In case of~\ref{Gmom} or~\ref{Gperc}, the error estimate~\eqref{e.def-bd-Einstein} displays a further algebraic loss, which is also new and expected to be optimal in general.
If for some exponent $\gamma \ge 1$ the moment bounds in~\ref{Gmom} or~\ref{Gperc} hold with constant $\Kc_\kappa \lesssim \kappa^\gamma$ for all $\kappa\ge1$,\footnote{For~\ref{Gmom}, this would amount to having stretched exponential moment bounds. For~\ref{Gperc}, this holds for some point processes such as the random parking measure with $\gamma$ large enough, cf.~\cite{DG-21a}.} then the error estimate~\eqref{e.def-bd-Einstein} could be upgraded to $\lambda_2(\Pc)\log^{1\vee\gamma}(2+\tfrac{\lambda(\Pc)}{\lambda_2(\Pc)})$ after optimizing in $\kappa$.
\smallskip\item {\it Particle separation:} Most works on the topic~\cite{Niethammer-Schubert-19,Hillairet-Wu-19,GVH,GVM-20,GV-Hofer-20,GV-20} have focussed so far on the simplest setting of~\ref{Gunif} in case when diluteness further holds in the strong form of $\ell(\Pc)\gg1$.
The only exception is the recent independent work~\cite{GV-Hofer-20}, where this last condition is relaxed and where the case of~\ref{Gmom} is further covered.
More generally, our approach allows to further cover essentially any situation for which the effective viscosity~\eqref{def:eff-vis-min} can be proved to be finite. Applied to~\ref{Gperc}, it allows to treat for the first time a 3D setting where particles are allowed to touch.
\end{enumerate}

Next, we further simplify formula~\eqref{eq:def-B1} for the first-order cluster coefficient $\Bb^1$ in the case when particle shapes are independent: we recover the formula obtained in~\cite{Hillairet-Wu-19}, as well as Einstein's explicit formula~\eqref{e.Einstein-intro} in case of spherical particles. The proof is postponed to Section~\ref{sec:explicitEinstein}.

\begin{prop1}[First-order coefficient]\label{prop:B1}
On top of Assumption~\ref{H0},
further assume
\begin{enumerate}[\qquad~~(a)]
\renewcommand{\labelenumi}{\emph{\textbf{(Indep)}}}
\renewcommand{\theenumi}{{\textbf{(Indep)}}}
\item\label{B1} \emph{Independent shapes:} Random shapes $\{I_n^\circ\}_n$ are iid copies of a given random open subset $I^\circ$ in the unit ball $B$, independent of the point process $\Pc$.
\end{enumerate}
Then,
the first-order coefficient $\Bb^1$ defined in~\eqref{eq:def-B1} can be written as
\begin{equation}\label{e.Einstein0}
\Bb^1 \,=\,\lambda(\Pc)\Bh^1,\qquad E:\Bh^1 E \,=\,\expecM{\int_{\R^d}|\!\D(\psi^\circ_E)|^2},
\end{equation}
in terms of
the unique decaying solution of the single-particle problem
\begin{equation}\label{eq:psicirc}
\left\{\begin{array}{ll}
-\triangle \psi^\circ_E+\nabla \Sigma^\circ_E=0,&\text{in $\R^d\setminus I^\circ$},\\
\Div( \psi^\circ_E)=0,&\text{in $\R^d\setminus I^\circ$},\\
\D(\psi^\circ_E)+E=0,&\text{in $I^\circ$},\\
\int_{\partial I^\circ}\sigma^\circ_E\nu=0,&\\
\int_{\partial I^\circ}\Theta x\cdot\sigma^\circ_E\nu=0,&\forall \Theta\in\Md^\Skew.
\end{array}\right.
\end{equation}
In case of spherical particles, $I_n=B(x_n,r_n)$, with iid random radii $\{r_n\}_n$, this reduces to  Einstein's celebrated formula
\begin{equation}\label{e.Einstein}
\Bb^1\,=\,\tfrac{d+2}2\varphi\Id,
\end{equation}
where the volume fraction is in this case $\varphi=\lambda(\Pc)\expec{|I_0|}$.
\end{prop1}

\subsection{Variational approach}\label{sec:Einstein-var}
This section is devoted to setting up our variational approach to prove Theorem~\ref{theor:Einsteinx3},
which is partly inspired by the theory of optimal bounds in homogenization; see e.g.~\cite[Chapters~13 \& 23]{Milton-02}. The new main ingredients are the use of Voronoi tessellations and of elliptic regularity.
Let $E\in\Md_0^\Sym$ be fixed.
In the spirit of the heuristic approximation~\eqref{eq:approx-Einstein-Vor} for the corrector, we start by defining single-particle problems in the neighborhood of each particle.
For a random set $\Ic$ satisfying \ref{H0}, we define 
the associated Voronoi tessellation $\{V_n\}_{n}$ as follows,
$$
V_n \, := \,\big\{x \in \R^d\,:\, \dist(x,I_n)<\dist(x,I_m)~\forall m\ne n\big\}.
$$
By definition, these Voronoi cells pave the whole space $\R^d$ and each $V_n$ contains exactly one inclusion $I_n$. 
We then consider the single-particle problems in $V_n$, with either homogeneous Dirichlet or Neumann boundary conditions on $\partial V_n$,
\begingroup\allowdisplaybreaks
\begin{eqnarray}
\hspace{-0.5cm}\Ec_{n,D}\!\!&:=&\!\!\inf\bigg\{\int_{V_n}|\!\D(\psi)|^2\,:\,\psi\in H^1_0(V_n)^d,~
\Div(\psi)=0,~(\D(\psi)+E)|_{I_n}=0\bigg\},\label{e:def-Dir}\\
\hspace{-0.5cm}\Ec_{n,N}\!\!&:=&\!\!\inf\bigg\{\int_{V_n}|\!\D(\psi)|^2\,:\,\psi\in H^1(V_n)^d,~
\Div(\psi)=0,~(\D(\psi)+E)|_{I_n}=0\bigg\}.\label{e:def-Neu}
\end{eqnarray}
\endgroup
Provided $\Ec_{n,D}<\infty$, the Dirichlet problem~\eqref{e:def-Dir} is well-posed and we denote by $\psi_{n,D}$ its unique minimizer. The Neumann problem~\eqref{e:def-Neu}, on the other hand, is always well-posed and one has the deterministic uniform bound $\Ec_{n,N}\lesssim1$. We denote by $\psi_{n,N}$ the corresponding minimizer: as it is only defined up to a rigid motion,
it can be uniquely chosen such that
$(\psi_{n,N}+E(x-x_n))|_{I_n}=0$.
Next, we define the single-particle problem on the whole space via
\begin{eqnarray}\label{eq:def-infty}
\Ec_{n,\infty}~:=~ \inf\bigg\{\int_{\R^d}|\!\D(\psi)|^2\,:\,\psi\in H^1(\R^d)^d,~
\Div(\psi)=0,~(\D(\psi)+E)|_{I_n}=0\bigg\}.
\end{eqnarray}
Note that the unique minimizer $\psi_{n,\infty}$ of this variational problem coincides with the solution~$\psi_E^{\{n\}}$ of~\eqref{eq:singlepart}.
In case of~\ref{Gperc}, as we only control clusters of close particles,
we naturally merge Voronoi cells that intersect the same cluster:
more precisely, we consider the Voronoi cell associated with each cluster $J_{q,\rho}$,
\[W_q\,:=\,\bigcup_{n:I_n\subset J_{q,\rho}}V_n\]
and we then partition the whole space as
\[\R^d\,=\,\Big(\bigcup_{n\in\Sc}V_n\Big)\cup\Big(\bigcup_{q\in\Sc'}W_q\Big),\]
where $\Sc:=\{n:\rho_n \ge \rho\}$ is the set of indices for well-separated particles
and where $\Sc'$ is the set of indices $q$ such that the cluster $J_{q,\rho}$ is made of at least two particles.
For $n\in\Sc$ we shall consider the single-inclusion problems $\Ec_{n,D},\Ec_{n,N},\Ec_{n,\infty}$ as above, while
for $q\in \Sc'$ it will suffice to consider the single-cluster problem with Dirichlet conditions,
\begin{equation}\label{e:def-Dir-W}
\Fc_{q,D}~:=~ \inf\bigg\{\int_{W_q}|\!\D(\psi)|^2\,:\,\psi\in H^1_0(W_{q})^d,~
\Div(\psi)=0,~  (\D(\psi)+E)|_{J_{q,\rho}}=0\bigg\}.
\end{equation}
The upcoming lemma shows that the error in the first-order expansion $\Bb\sim\Id+\Bb^1$ can be controlled
using single-particle problems (as well as single-cluster problems in case of~\ref{Gperc}). This provides a drastic reduction of complexity since $\Bb$ itself involves the corrector $\psi_E$ with the full set of particles. The proof is postponed to Section~\ref{sec:proof-lem:main} below.

\begin{lem}\label{lem:main}
Under the assumptions of Theorem~\ref{theor:Einsteinx3}, using the above notation~\eqref{e:def-Dir}--\eqref{e:def-Dir-W},
we have
\begin{align}\label{e.var}
&\big|E:\big(\Bb -(\Id+ \Bb^1)\big)E\big|\\
&\qquad\,\lesssim\,\left\{
\begin{array}{lll}
\E\Big[{\sum_n \frac{\mathds 1_{0 \in I_n}}{|I_n|}( \Ec_{n,D}-\Ec_{n,N})}\Big]&:&\text{in case of~\ref{Gunif} or~\ref{Gmom}},
\\
\vspace{-0.3cm}&\\
\E\Big[{\sum_{n\in \Sc} \frac{\mathds 1_{0 \in I_n}}{|I_n|}( \Ec_{n,D}-\Ec_{n,N})}\Big]&&\\
\hspace{1cm}+\E\Big[{\sum_{q\in\Sc'} \frac{\mathds 1_{0 \in J_{q,\rho}}}{|J_{q,\rho}|}\Fc_{q,D}}\Big]&:&\text{in case of~\ref{Gperc}},
\end{array}
\right.\nonumber
\end{align}
where $\Bb^1$ is defined in~\eqref{eq:def-B1}.
\end{lem}

It remains to control the right-hand side in the error estimate~\eqref{e.var}, which amounts to comparing the single-particle problems with Dirichlet or Neumann boundary conditions on Voronoi cells. The proof is postponed to Section~\ref{sec:proof-lem:diff-DN} below.
\begin{lem}\label{lem:diff-DN}
For all $n$,
we have almost surely
\begin{equation}\label{e.estim-diff-DN}
0 \,\le\, \Ec_{n,D}-\Ec_{n,N} \,\lesssim\, \mu(\rho_n)\mathds{1}_{\rho_n<1}+\rho_n^{-d} \mathds{1}_{\rho_n\ge 1},
\end{equation}
where we recall that $\rho_n$ stands for (half of) the interparticle distance, cf.~\eqref{eq:dist-rhon}, and that the weight $\mu$ is defined in~\eqref{e.def-mu}.
In addition, there is $p_0<\infty$ such that for all $q$,
\begin{equation}\label{e.dir-perc-estim}
\Fc_{q,D} \,\lesssim\, \diam(J_{q,\rho})^{p_0}.\qedhere
\end{equation}
\end{lem}

With these two lemmas at hand, combining the estimates, we may now quickly conclude the proof of Theorem~\ref{theor:Einsteinx3}.

\begin{proof}[Proof of Theorem~\ref{theor:Einsteinx3}]
Combining Lemmas~\ref{lem:main}
and~\ref{lem:diff-DN}, we get
\begin{align}
&|\Bb - (\Id+  \Bb^1)|\label{e.Einstein-pr-1}\\
&\quad\,\lesssim\,
\left\{
\begin{array}{lll}
\E\Big[{\sum_n \frac{\mathds1_{0\in I_n}}{|I_n|}\big( \mu(\rho_n) \mathds{1}_{\rho_n<1}+\rho_n^{-d} \mathds1_{\rho_n\ge 1}\big)}\Big]&:&\text{in case of~\ref{Gunif} or~\ref{Gmom}},\\
\vspace{-0.3cm}&&\\
\E\Big[{\sum_{n} \frac{\mathds1_{0\in I_n}}{|I_n|} \rho_n^{-d}\mathds1_{\rho_n\ge\rho}}\Big]&&\\
\qquad+\E\Big[{\sum_{q\in\Sc'} \frac{\mathds1_{0\in J_{q,\rho}}}{|J_{q,\rho}|}\diam (J_{q,\rho})^{p_0} }\Big]&:&\text{in case of~\ref{Gperc},}
\end{array}
\right.\nonumber
\end{align}
and it remains to estimate these expectations. We split the proof into two steps.

\medskip 
\step1 Proof that, if $g\in\Ld^\infty(\R^+)$ is non-increasing with $g(r)\downarrow0$ as $r\uparrow\infty$, then
\begin{equation}\label{e.Einstein-pr-2}
\expecM{\sum_n\frac{\mathds1_{0\in I_n}}{|I_n|}g(\rho_n)}\,\lesssim\, \lambda_2(\Pc)\|g\|_{\Ld^\infty(\R^+)}+\int_{0}^\infty |g'(r)|\,\Big((\lambda_2(\Pc)\,\langle r\rangle^d)\wedge\lambda(\Pc)\Big)\,dr.
\end{equation}
To start with, we rewrite the left-hand side as 
\begin{equation}\label{eq:sum-g-rhon}
\expecM{\sum_{n}\frac{\mathds1_{0\in I_n}}{|I_n|}g(\rho_n)}\,=\,\int_0^\infty g(r)\,d\Lambda(r),
\end{equation}
where the positive measure $\Lambda$ on $\R^+$ is defined by its distribution function
\[\Lambda([0,r])\,:=\,\expecM{\sum_{n}{\frac{\mathds1_{0\in I_n}}{|I_n|}}\mathds1_{\rho_n\le r}}
\,=\,\expecM{\sum_{n}{\frac{\mathds1_{0\in I_n}}{|I_n|}}\,\mathds1_{\exists m\ne n:\,\frac12\dist(I_m,I_n)\le r}}.\]
As $I_n\subset B(x_n)$ for all $n$, we can estimate the latter as
\[\Lambda([0,r])\,\le\,\expecM{\sum_{n}\mathds1_{|x_n|\le1}\,\mathds1_{\exists m\ne n:\,|x_m-x_n|\le 2(r+1)}}.\]
Recalling that $\ell=\ell(\Pc)$ is the minimal distance~\eqref{eq:def-ell(P)}, we deduce that $\Lambda([0,r])=0$ for all $r\le\frac12\ell-1$.
Moreover, we can bound, on the one hand,
\[\Lambda([0,r])\,\le\, \expecM{\sum_n\mathds1_{|x_n|\le1}}\,=\,\lambda(\Pc)|B|,\]
and on the other hand, in terms of the two-point density and intensity, for $r\ge\frac12\ell-1$,
\begin{eqnarray*}
\Lambda([0,r])&\le&\expecM{\sum_{n\ne m}\mathds1_{|x_n|\le1}\,\mathds1_{|x_m-x_n|\le 2(r+1)}}\nonumber\\
&=&\iint_{B\times B_{2(r+1)}}f_2(x,x+y)\,dxdy\nonumber\\
&=&(2(r+1))^{-d}\iint_{B_{2(r+1)}\times B_{2(r+1)}}f_2(x,x+y)\,dxdy\nonumber\\
&\lesssim&\lambda_2(\Pc)\langle r\rangle^d.
\end{eqnarray*}
Combining these estimates yields
\begin{equation}\label{eq:bnd-Lambda}
\Lambda([0,r])\,\lesssim\,
(\lambda_2(\Pc)\langle r\rangle^d)\wedge\lambda(\Pc).
\end{equation}
Under our assumptions on $g$, an integration by parts yields
\[\int_0^\infty g(r)\,d\Lambda(r)\,=\,-g(0)\Lambda(\{0\})+\int_{0}^\infty |g'(r)|\,\Lambda([0,r])\,dr,\]
and the conclusion follows in combination with~\eqref{eq:sum-g-rhon} and~\eqref{eq:bnd-Lambda}.

\medskip
\step2 Conclusion.\\
In case of~\ref{Gunif}, as we have $\rho_n\ge\rho$ for all $n$,
the contributions of $\rho_n<\rho$ can be removed in~\eqref{e.Einstein-pr-1}. Applying~\eqref{e.Einstein-pr-2} with $g(r)=\langle r\rangle^{-d}$, we are then led to
\begin{eqnarray*}
|\Bb-(\Id+\Bb^1)|&\lesssim&\expecM{\sum_n{\frac{\mathds1_{0\in I_n}}{|I_n|}}\langle\rho_n\rangle^{-d}}\\
&\lesssim&\lambda_2(\Pc)+\int_{0}^\infty\langle r\rangle^{-d-1}\Big((\lambda_2(\Pc)\langle r\rangle^d)\wedge\lambda(\Pc)\Big)\,dr,
\end{eqnarray*}
and the conclusion~\eqref{e.def-bd-Einstein} follows after estimating this integral.

\medskip\noindent
Next, in case of~\ref{Gmom}, repeating the same computation as above for the contributions of $\rho_n\ge1$ in~\eqref{e.Einstein-pr-1}, and separating the contributions of $\rho_n\le1$, we find
\[|\Bb-(\Id+\Bb^1)|\,\lesssim\,\lambda_2(\Pc)\log\big(2+\tfrac{\lambda(\Pc)}{\lambda_2(\Pc)}\big)+\expecM{\sum_n{\frac{\mathds1_{0\in I_n}}{|I_n|}}\mu(\rho_n)\mathds1_{\rho_n\le1}},\]
and it remains to estimate the last term.
By H\"older's inequality, we can write for any~$\kappa\ge1$,
\begin{equation*}
\expecM{\sum_n{\frac{\mathds1_{0\in I_n}}{|I_n|}}\mu(\rho_n)\mathds1_{\rho_n\le1}}
\,\le\,\expecM{\sum_n{\frac{\mathds1_{0\in I_n}}{|I_n|}}\mathds1_{\rho_n\le 1}}^{1-\frac1\kappa}\expecM{\sum_n{\frac{\mathds1_{0\in I_n}}{|I_n|}}\mu(\rho_n)^{\kappa}}^{\frac1\kappa}.
\end{equation*}
On the one hand, \eqref{eq:bnd-Lambda} yields
\[\expecM{\sum_n{\frac{\mathds1_{0\in I_n}}{|I_n|}}\mathds1_{\rho_n\le 1}}\,=\,\Lambda([0,1])
\,\lesssim\,\lambda_2(\Pc).\]
On the other hand, by the ergodic theorem, using~\eqref{eq:conv-lambda} and the reciprocal moment condition in~\ref{Gmom}, we find
\begin{eqnarray*}
\lefteqn{\expecM{\sum_n{\frac{\mathds1_{0\in I_n}}{|I_n|}}\,\mu(\rho_n)^{\kappa}}~=~
\lim_{R\uparrow\infty} R^{-d} \sum_{n}{\frac{|I_n\cap Q_R|}{|I_n|}} \,\mu(\rho_n)^\kappa}
\\
&\le&\limsup_{R\uparrow\infty} \frac{\sharp\{n:I_n\cap Q_R\ne\varnothing\}}{R^{d}}\, \frac1{\sharp\{n:I_n\cap Q_R\ne\varnothing\}}\sum_{n:I_n\cap Q_R\ne\varnothing} \mu(\rho_n)^\kappa
\\
&\le& \lambda(\Pc) (\Kc_\kappa)^\kappa,
\end{eqnarray*}
and the conclusion~\eqref{e.def-bd-Einstein} follows.

\medskip
\noindent
Finally, in case of~\ref{Gperc}, repeating again the same computation for the contributions of~$\rho_n\ge\rho$ in~\eqref{e.Einstein-pr-1}, we find
\[|\Bb-(\Id+\Bb^1)|\,\lesssim\,\lambda_2(\Pc)\log\big(2+\tfrac{\lambda(\Pc)}{\lambda_2(\Pc)}\big)+\expecM{\sum_{q\in\Sc'}{\frac{\mathds1_{0\in J_{q,\rho}}}{|J_{q,\rho}|}}\diam(J_{q,\rho})^{p_0}},\]
and it remains to estimate the last term. By H\"older's inequality, we can write for any~$\kappa\ge1$,
\[\expecM{\sum_{q\in\Sc'}\frac{\mathds1_{0\in J_{q,\rho}}}{|J_{q,\rho}|}\diam(J_{q,\rho})^{p_0}}\,\le\,\expecM{\sum_{q\in\Sc'}{\frac{\mathds1_{0\in J_{q,\rho}}}{|J_{q,\rho}|}}}^{1-\frac1\kappa}\expecM{\sum_{q}{\frac{\mathds1_{0\in J_{q,\rho}}}{|J_{q,\rho}|}}\diam(J_{q,\rho})^{\kappa p_0}}^\frac1\kappa.\]
On the one hand, by definition of $\Sc'$, \eqref{eq:bnd-Lambda} yields
\[ \expecM{\sum_{q\in\Sc'}{\frac{\mathds1_{0\in J_{q,\rho}}}{|J_{q,\rho}|}}}\,=\,\expecM{\sum_{n}{\frac{\mathds1_{0\in I_n}}{|I_n|}}\mathds1_{\rho_n\le\rho}}\,=\,\Lambda([0,\rho])
\,\lesssim\,\lambda_2(\Pc).\]
On the other hand, by the ergodic theorem, using~\eqref{eq:conv-lambda}, the condition~\eqref{e.bd.perco} in~\ref{Gperc}, and the fact that there are less clusters than particles, we find
\begingroup\allowdisplaybreaks
\begin{eqnarray*}
\lefteqn{\expecM{\sum_{q}{\frac{\mathds1_{0\in J_{q,\rho}}}{|J_{q,\rho}|}}\diam(J_{q,\rho})^{\kappa p_0}}~=~\lim_{R\uparrow\infty}R^{-d}\sum_{q}{\frac{|J_{q,\rho}\cap Q_R|}{|J_{q,\rho}|}}\diam(J_{q,\rho})^{\kappa p_0}}\\
&\le&\limsup_{R\uparrow\infty}\frac{\sharp\{q:J_{q,\rho}\cap Q_R\ne\varnothing\}}{R^{d}}\frac{1}{\sharp\{q:J_{q,\rho}\cap Q_R\ne\varnothing\}}\sum_{q:J_{q,\rho}\cap Q_R\ne\varnothing} \diam(J_{q,\rho})^{\kappa p_0}\\
&\le&\limsup_{R\uparrow\infty}\frac{\sharp\{n:I_n\cap Q_R\ne\varnothing\}}{R^{d}}\frac1{\sharp\{q:J_{q,\rho}\cap Q_R\ne\varnothing\}}\sum_{q:J_{q,\rho}\cap Q_R\ne\varnothing}\diam(J_{q,\rho})^{\kappa p_0}\\
&\le&\lambda(\Pc)(\Kc_{\kappa})^{\kappa},
\end{eqnarray*}
\endgroup
and the conclusion~\eqref{e.def-bd-Einstein} follows.
\end{proof}

\subsection{Preliminary lemmas}\label{sec:prelim-Einstein}

Before turning to the proof of Lemmas~\ref{lem:main} and~\ref{lem:diff-DN}, which are key to Theorem~\ref{theor:Einsteinx3} as explained above, we start with a couple of preliminary PDE and probabilistic lemmas. We first prove the following trace estimates at particle boundaries.

\begin{samepage}\begin{lem}[Trace estimates]\label{lem:trace-0}$ $
\begin{enumerate}[(i)]
\item For any $\psi\in H^1(I_n)$, we have
\[\qquad\inf_{\kappa\in\R^d,\,\Theta\in\Md^\Skew}\int_{\partial I_n}|\psi-(\kappa+\Theta (x-x_n))|^2\,\lesssim\,\int_{I_n}|\!\D(\psi)|^2.\]
\item For any $\psi\in H^1(I_n+\rho B)$ satisfying the following relations, for some $E\in\Md^\Sym_0$,
\[\qquad\left\{\begin{array}{ll}
-\triangle\psi+\nabla\Sigma=0,&\text{in $(I_n+\rho B)\setminus I_n$},\\
\Div(\psi)=0,&\text{in $(I_n+\rho B)\setminus I_n$},\\
\D(\psi)+E=0,&\text{in $I_n$},
\end{array}\right.\]
we have
\[\qquad\inf_{c\in\R}\int_{\partial I_n}|\sigma(\psi,\Sigma)-c\Id\!|^2\,\lesssim\,\int_{I_n+\rho B}|\!\D(\psi)|^2,\]
where we recall that multiplicative constants may implicitly depend on~$\rho$.
\qedhere
\end{enumerate}
\end{lem}\end{samepage}

\begin{proof}
We split the proof into two steps.

\medskip
\step1 Proof of~(i).\\
We appeal to a trace estimate in form of
\[\int_{\partial I_n}|\psi-(\kappa+\Theta (x-x_n))|^2\,\lesssim\,\int_{I_n}|\langle\nabla\rangle^\frac12(\psi-(\kappa+\Theta (x-x_n)))|^2,\]
and the conclusion follows from Poincaré's and Korn's inequalities.

\medskip
\step2 Proof of~(ii).\\
By definition of the Cauchy stress tensor, a trace estimate yields
\begin{equation}\label{eq:est-bndary-0}
\int_{\partial I_{n}}|\sigma(\psi,\Sigma)-c\Id\!|^2
\,\lesssim\,\int_{(I_{n}+\frac12\rho B)\setminus I_{n}}|\langle\nabla\rangle^\frac12\nabla\psi|^2
+|\langle\nabla\rangle^\frac12(\Sigma-c)|^2.
\end{equation}
By the local regularity theory for the steady Stokes equation near a boundary, e.g.~\cite[Theorems~IV.5.1--5.3]{Galdi}, we have for all $m\ge0$, for all constants $\kappa\in\R^d$ and $c\in\R$,
\begin{multline*}
\|\nabla\psi\|_{H^{m+1}((I_{n}+\frac12\rho B)\setminus I_{n})}+\|\Sigma-c\Id\!\|_{H^{m+1}((I_{n}+\frac12\rho B)\setminus I_n)}\\
\,\lesssim\,\|\psi|_{I_{n}}-\kappa\|_{H^{m+\frac 32}(\partial I_{n})}
+\|\psi-\kappa\|_{H^1((I_{n}+\rho B)\setminus I_n)}
+\|\Sigma-c\Id\!\|_{\Ld^2((I_{n}+\rho B)\setminus I_n)}.
\end{multline*}
Choosing $c:=\fint_{(I_{n}+\rho B)\setminus I_{n}}\Sigma$ and using a local pressure estimate for the steady Stokes equation, e.g.~\cite[Lemma~3.3]{DG-21b}, we find
\begin{equation*}
\|\Sigma-c\Id\!\|_{\Ld^2((I_{n}+\rho B)\setminus I_{n})}
\,\lesssim\,\|\nabla\psi\|_{\Ld^2((I_{n}+\rho B)\setminus I_{n})},
\end{equation*}
so that the above reduces to
\begin{multline*}
\|\nabla\psi\|_{H^{m+1}((I_{n}+\frac12\rho B)\setminus I_{n})}+\|\Sigma-c\Id\!\|_{H^{m+1}((I_{n}+\frac12\rho B)\setminus I_{n})}\\
\,\lesssim\,\|\psi|_{I_n}-\kappa\|_{H^{m+\frac 32}(\partial I_{n})}+\|\psi-\kappa\|_{H^1((I_n+\rho B)\setminus I_n)}.
\end{multline*}
As $\psi$ is affine in $I_n$, we have
\[\|\psi|_{I_{n}}-\kappa\|_{H^{m+\frac 32}(\partial I_{n})}\,\lesssim\,\|\psi-\kappa\|_{H^{m+2}(I_{n})}\,=\,\|\psi-\kappa\|_{H^{1}(I_{n})},\]
and the above then becomes
\begin{equation*}
\|\nabla\psi\|_{H^{m+1}((I_{n}+\frac12\rho B)\setminus I_{n})}+\|\Sigma-c\Id\!\|_{H^{m+1}((I_{n}+\frac12\rho B)\setminus I_{n})}
\,\lesssim\,\|\psi-\kappa\|_{H^1(I_{n}+\rho B)}.
\end{equation*}
Further choosing $\kappa:=\fint_{I_{n}+\rho B}\psi$ and applying Poincaré's inequality, we deduce
\begin{equation*}
\|\nabla\psi\|_{H^{m+1}((I_{n}+\frac12\rho B)\setminus I_n)}+\|\Sigma-c\Id\!\|_{H^{m+1}((I_{n}+\frac12\rho B)\setminus I_n)}
\,\lesssim\,\|\nabla\psi\|_{\Ld^2(I_{n}+\rho B)}.
\end{equation*}
In particular, combined with~\eqref{eq:est-bndary-0}, this leads us to
\begin{equation*}
\inf_{c\in\R}\int_{\partial I_{n}}|\sigma(\psi,\Sigma)-c\Id\!|^2
\,\lesssim\,\int_{I_n+\rho B}|\nabla\psi|^2.
\end{equation*}
Noting that $\sigma(\psi,\Sigma)$ and the equations satisfied by $(\psi,\Sigma)$ are unchanged if a rigid motion is added to $\psi$, the conclusion now follows from Korn's inequality.
\end{proof}

Next, we recall the following standard elliptic regularity estimate for solutions of the free steady Stokes equation.

\begin{lem}[Mean-value property]\label{lem:regularity}
Given $r>0$, if $(\psi,\Sigma)$ is a weak solution of the free Stokes equation in $B_r$,
\[
-\triangle\psi+\nabla\Sigma=0,\qquad
\Div(\psi)=0,\qquad\text{in $B_r$},
\]
then it satisfies
\[|\!\D(\psi)(0)|^2\,\lesssim\,\fint_{B_r}|\!\D(\psi)|^2.\qedhere\]
\end{lem}

\begin{proof}
By scaling, it suffices to consider $r=1$.
For $m>\frac d2$, the Sobolev embedding yields
\begin{equation}\label{eq:Sob}
|\!\D(\psi)(0)|\,\lesssim\,\|\!\D(\psi)\|_{H^m(\frac12B)},
\end{equation}
and it remains to estimate this Sobolev norm.
By the local regularity theory for the steady Stokes equation, e.g.~\cite[Theorem~IV.4.1]{Galdi}, we find for all $\kappa\in\R^d$ and~$c\in\R$,
\[\|\nabla\psi\|_{H^{m}(\frac12B)}+\|\Sigma-c\|_{H^m(\frac12B)}\,\lesssim\,\|\psi-\kappa\|_{H^{1}(B)}+\|\Sigma-c\|_{\Ld^2(B)}.\]
Choosing $c=\fint_B\Sigma$ and using a local pressure estimate for the steady Stokes equation, e.g.~\cite[Lemma~3.3]{DG-21b}, we find
\[\|\Sigma-c\|_{\Ld^2(B)}\,\lesssim\,\|\nabla\psi\|_{\Ld^2(B)}.\]
Inserting this into the above and applying Poincaré's inequality for the choice $\kappa=\fint_B\psi$, we deduce
\[\|\nabla\psi\|_{H^{m}(\frac12B)}\,\lesssim\,\|\nabla\psi\|_{\Ld^2(B)}.\]
For any $\Theta\in\Md^\Skew$, this entails
\[\|\!\D(\psi)\|_{H^{m}(\frac12B)}\,\le\,\|\nabla(\psi-\Theta x)\|_{H^{m}(\frac12B)}\,\lesssim\,\|\nabla(\psi-\Theta x)\|_{\Ld^2(B)},\]
hence, by Korn's inequality,
\[\|\!\D(\psi)\|_{H^{m}(\frac12B)}\,\lesssim\,\|\!\D(\psi)\|_{\Ld^2(B)}.\]
Inserting this into~\eqref{eq:Sob}, the conclusion follows.
\end{proof}

Finally, the following lemma provides a useful property of Voronoi tessellations.
Although it could be obtained as a direct consequence of Palm theory, we include a more elementary proof by means of an approximation argument.

\begin{lem}[Property of Voronoi tessellations]
For all stationary random fields~$\zeta$ with
$\expec{|\zeta|}<\infty$, we have in case of~\ref{Gunif} or~\ref{Gmom},
\begin{equation}\label{e.prop-randtess}
\expec{\zeta}\,=\,\expecM{\sum_{n} \frac{\mathds 1_{0 \in I_n}}{|I_n|} \int_{V_n} \zeta},
\end{equation}
and in case of~\ref{Gperc},
\begin{equation}\label{e.prop-randtess-re}
\expec{\zeta}\,=\,\expecM{\sum_{n\in \Sc} \frac{\mathds 1_{0 \in I_n}}{|I_n|} \int_{V_n} \zeta+\sum_{q\in\Sc'} \frac{\mathds 1_{0 \in J_{q,\rho}}}{|J_{q,\rho}|}\int_{W_q}\zeta}.\qedhere
\end{equation}
\end{lem}

\begin{proof}
By the monotone convergence theorem, it is enough to prove the result for any bounded non-negative random field $0\le \zeta \le M$ with any fixed $M>0$. Let such a $\zeta$ be fixed.
We split the proof into two steps.

\medskip
\step1 Proof of~\eqref{e.prop-randtess} \& \eqref{e.prop-randtess-re} under the additional assumption that almost surely
\begin{equation}\label{e.hyp-Vor-bdd}
\sup_n\diam(V_n)\,<\,\infty.
\end{equation}
In that case, let $K\ge1$ be such that $\diam(V_n)\le K$ almost surely for all $n$.
We consider~\eqref{e.prop-randtess} and~\eqref{e.prop-randtess-re} separately, and split the proof into two further substeps.

\medskip
\substep{1.1} Proof of~\eqref{e.prop-randtess} under assumption~\eqref{e.hyp-Vor-bdd}.\\
By the ergodic theorem, we have almost surely
\[\expecM{\sum_{n} \frac{\mathds 1_{0 \in I_n}}{|I_n|} \int_{V_{n}} \zeta}\,=\,\lim_{R \uparrow\infty}
R^{-d}\sum_{n} \frac{|I_n\cap Q_R|}{|I_n|} \int_{V_{n}} \zeta.\]
As $\zeta \ge 0$ and as assumption~\eqref{e.hyp-Vor-bdd} entails $V_{n}\subset B_{K}(x_n)$ for all $n$, we easily get the two-sided estimate
\[\int_{Q_{R-CK}} \zeta\, \le\, \sum_{n} \frac{|I_n\cap Q_R|}{|I_n|} \int_{V_{n}} \zeta \, \le\, \int_{Q_{R+CK}} \zeta,\]
and the claim~\eqref{e.prop-randtess} then follows from the ergodic theorem.

\medskip
\substep{1.2} Proof of~\eqref{e.prop-randtess-re} under assumption~\eqref{e.hyp-Vor-bdd}.\\
By the ergodic theorem, we have almost surely
\begin{multline*}
\expecM{\sum_{n\in\Sc}\frac{\mathds1_{0\in I_n}}{|I_n|}\int_{V_n}\zeta+\sum_{q\in\Sc'}\frac{\mathds1_{0\in J_{q,\rho}}}{|J_{q,\rho}|}\int_{W_q}\zeta}\\
\,=\,\lim_{R\uparrow\infty}R^{-d}\bigg(\sum_{n\in\Sc}\frac{|I_n\cap Q_R|}{|I_n|}\int_{V_n}\zeta+\sum_{q\in\Sc'}\frac{|J_{q,\rho}\cap Q_R|}{|J_{q,\rho}|}\int_{W_q}\zeta\bigg).
\end{multline*}
As $0\le\zeta\le M$ and as assumption~\eqref{e.hyp-Vor-bdd} entails $V_n\subset B_{K}(x_n)$ for all $n$, we get the two-sided estimate
\begin{multline*}
\int_{Q_{R-CK}}\zeta
\,-\,M\sum_{q:J_{q,\rho}\cap (Q_{R+1}\setminus Q_R)\ne\varnothing}|W_q|\\
\,\le\,\sum_{n\in\Sc}\frac{|I_n\cap Q_R|}{|I_n|}\int_{V_n}\zeta+\sum_{q\in\Sc'}\frac{|J_{q,\rho}\cap Q_R|}{|J_{q,\rho}|}\int_{W_q}\zeta\\
\,\le\,\int_{Q_{R+CK}}\zeta
\,+\,M\sum_{q:J_{q,\rho}\cap (Q_{R+1}\setminus Q_R)\ne\varnothing}|W_q|.
\end{multline*}
By the ergodic theorem, in order to prove~\eqref{e.prop-randtess-re}, it thus remains to show almost surely
\begin{equation*}
\lim_{R\uparrow\infty}R^{-d}\sum_{q:J_{q,\rho}\cap (Q_{R+1}\setminus Q_R)\ne\varnothing}|W_q|\,=\,0,
\end{equation*}
which would follow provided that we show almost surely
\begin{equation}\label{eq:suff-Wq-bnd}
\limsup_{R\uparrow\infty}R^{-d}\sum_{q:J_{q,\rho}\cap Q_R\ne\varnothing}|W_q|\,<\,\infty.
\end{equation}
As $|W_q|\lesssim(\diam(J_{q,\rho})+K)^d$, we can estimate
\begin{multline*}
R^{-d}\sum_{q:J_{q,\rho}\cap Q_R\ne\varnothing}|W_q|\\
\,\lesssim\,K^d\Big(R^{-d}\,\sharp\{q:J_{q,\rho}\cap Q_R\ne\varnothing\}\Big)\bigg(\frac1{\sharp\{q:J_{q,\rho}\cap Q_R\ne\varnothing\}}\sum_{q:J_{q,\rho}\cap Q_R\ne\varnothing}\diam(J_{q,\rho})^d\bigg).
\end{multline*}
To bound the first factor, we simply note that
\begin{eqnarray*}
R^{-d}\,\sharp\{q:J_{q,\rho}\cap Q_R\ne\varnothing\}
&\le&R^{-d}\,\sharp\{n:I_n\cap Q_R\ne\varnothing\}\\
&\le&R^{-d}\,\sharp\{n:x_n\in Q_{R+1}\}\,\xrightarrow{R\uparrow\infty}\,\lambda(\Pc).
\end{eqnarray*}
Appealing to the condition~\eqref{e.bd.perco} in~\ref{Gperc} to estimate the second factor, the claim~\eqref{eq:suff-Wq-bnd} follows.

\medskip
\step2 Relaxing assumption~\eqref{e.hyp-Vor-bdd}.\\
It remains to consider the case when $\sup_n\diam (V_n)=\infty$,
and we proceed by approximation.
Consider a point process $\Pc'=\{x_n'\}_n$ independent of $\Pc,\Ic$ such that almost surely
\[\min_{n\ne m}|x_n'-x_m'|\ge\tfrac12,\qquad\min_{m:m\ne n}|x_n'-x_m'|\le1~~\text{for all $n$}.\]
For instance, $\Pc'$ can be chosen as the random parking process of parameter $\frac14$, cf.~\cite{Penrose-01}. Now, for any integer $\alpha\ge1$, we define the `enriched' point process $\Pc_\alpha$ as follows,
\[\Pc_\alpha\,:=\,\Pc\cup\big\{2^{2\alpha}x_n':\dist(2^{2\alpha}x_n',\Pc)\ge 2^{2\alpha+3}\big\},\]
as well as the corresponding random set
\[\Ic_\alpha\,:=\,\Ic\cup\bigcup_{n:\dist(2^{2\alpha}x_n',\Pc)\ge 2^{2\alpha+3}}B(2^{2\alpha}x_n').\]
Denote by $V_\alpha(x_n)$ the Voronoi cell associated with $I_n$ in $\Ic_\alpha$, and by $V_\alpha(x_n')$ the Voronoi cell associated with $B(2^{2\alpha}x_n')$. By construction, it can be checked that for all $n$,
\[V_n\cap B_{2^{2\alpha+2}}(x_n)\,\subset\,V_\alpha(x_n)\,\subset\,V_n\cap B_{2^{2\alpha+4}}(x_n),\]
which entails that $V_\alpha(x_n)\uparrow V_n$ increasingly as $\alpha\uparrow\infty$ (over integers).
In addition, $\Pc_\alpha,\Ic_\alpha$ satisfy~\ref{H0},
as well as~\eqref{e.hyp-Vor-bdd} with Voronoi diameters bounded by $O(2^{2\alpha})$. They further satisfy Assumption~\ref{Gunif}, \ref{Gmom}, or \ref{Gperc} provided that~$\Pc,\Ic$ satisfy the corresponding assumption.
We focus on the case of~\ref{Gunif} or~\ref{Gmom}, while the case of~\ref{Gperc} is analogous.
Applying the result~\eqref{e.prop-randtess} of Step~1, by definition of $\Pc_\alpha$, we get
\begin{equation*}
\expec\zeta\,=\,\expecM{\sum_n\frac{\mathds1_{0\in I_n}}{|I_n|}\int_{V_\alpha(x_n)}\zeta}+\expecM{\sum_n\mathds1_{\dist(2^{2\alpha}x_n',\Pc)\ge2^{2\alpha+3}}\frac{\mathds1_{0\in B(2^{2\alpha}x_n')}}{|B|}\int_{V_\alpha(x_n')}\zeta}.
\end{equation*}
As $0\le\zeta\le M$, as Voronoi diameters are bounded by $C2^{2\alpha}$ almost surely, and using stationarity and the independence of $\Pc,\Ic$ and $\Pc',\Ic'$, the second right-hand side term satisfies
\begin{eqnarray*}
0&\le&\expecM{\sum_n\mathds1_{\dist(2^{2\alpha}x_n',\Pc)\ge2^{2\alpha+3}}\frac{\mathds1_{0\in B(2^{2\alpha}x_n')}}{|B|}\int_{V_\alpha(x_n')}\zeta}\\
&\lesssim&M2^{2\alpha d}\,\expecM{\sum_n\mathds1_{\dist(2^{2\alpha}x_n',\Pc)\ge2^{2\alpha+3}}\frac{\mathds1_{0\in B(2^{2\alpha}x_n')}}{|B|}}\\
&=&M2^{2\alpha d}\,\expecM{\sum_n\frac{\mathds1_{0\in B(2^{2\alpha}x_n')}}{|B|}}\pr{\dist(0,\Pc)\ge2^{2\alpha+3}}\\
&=&M\lambda(\Pc')\,\pr{\dist(0,\Pc)\ge2^{2\alpha+3}}.
\end{eqnarray*}
Inserting this into the above, we deduce
\begin{multline*}
\expecM{\sum_n\frac{\mathds1_{0\in I_n}}{|I_n|}\int_{V_\alpha(x_n)}\zeta}\\
\,\le\,\expec\zeta\,\le\,\expecM{\sum_n\frac{\mathds1_{0\in I_n}}{|I_n|}\int_{V_\alpha(x_n)}\zeta}+CM\lambda(\Pc')\,\pr{\dist(0,\Pc)\ge 2^{2\alpha+3}}.
\end{multline*}
and the conclusion~\eqref{e.prop-randtess} follows from the monotone convergence theorem.
\end{proof}

\subsection{Proof of Lemma~\ref{lem:main}}\label{sec:proof-lem:main}
Without loss of generality, we can assume that $\Ec_{n,D}<\infty$ almost surely as otherwise the claimed estimate~\eqref{e.var} would be trivial.
The variational definition of the effective viscosity~\eqref{def:eff-vis-min} can be written as
\begin{multline}\label{def:eff-vis-min-bis}
E:\Bb E\,=\, |E|^2+\inf\Big\{\expec{|\!\D(\psi)|^2}~:~\psi\in\Ld^2(\Omega;H^1_\loc(\R^d)^d),~\nabla\psi~\text{stationary},\\
\Div(\psi)=0,~(\D(\psi)+E)|_{\Ic}=0,~\expec{\D(\psi)}=0\Big\},
\end{multline}
and the definition~\eqref{eq:def-B1} of~$\Bb^1$ as
\begin{equation}\label{eq:B1-redef}
E:\Bb^1E
\,=\,\expecM{\sum_n\frac{\mathds1_{0\in I_n}}{|I_n|}\Ec_{n,\infty}}.
\end{equation}
Note that an energy estimate for~\eqref{eq:def-infty} using Bogovskii's construction yields the uniform bound $\Ec_{n,\infty}\,\lesssim\,|E|^2$.
In order to prove~\eqref{e.var}, it remains to compare~\eqref{def:eff-vis-min-bis} to a superposition of the single-particle problems~$\{\Ec_{n,\infty}\}_n$ and to recognize~\eqref{eq:B1-redef}. We split the proof into three steps.

\medskip
 \step{1} Upper bound: proof that we have in case of~\ref{Gunif} or~\ref{Gmom},
\begin{equation}\label{e.above}
E : \Bb E\,\le\,|E|^2+\expecM{\sum_{n} \frac{\mathds 1_{0 \in I_n}}{|I_n|} \Ec_{n,D}},
\end{equation}
and in case of~\ref{Gperc},
\begin{equation}\label{e.above-re}
E : \Bb E\,\le\,|E|^2+\expecM{\sum_{n\in \Sc} \frac{\mathds 1_{0 \in I_n}}{|I_n|} \Ec_{n,D}+\sum_{q\in\Sc'} \frac{ \mathds 1_{0 \in J_{q,\rho}}}{|J_{q,\rho}|}\Fc_{q,D}}.
\end{equation}
We focus on~\eqref{e.above-re}, the proof of~\eqref{e.above} being identical.
We define almost surely
$$
\psi_D\,:=\,\sum_{n\in \Sc} \psi_{n,D}+\sum_{q\in\Sc'}\psi_{q,D} ~~\in H^1_\loc(\R^d)^d,
$$
where the summands $\psi_{n,D}\in H^1_0(V_n)^d$ and $\psi_{q,D} \in H^1_0(W_q)^d$ are implicitly extended by zero outside $V_n$ and $W_q$, respectively.
Properties of Dirichlet minimizers $\{\psi_{n,D}\}_n,\{\psi_{q,D}\}_q$ ensure that $\nabla\psi_D$ is stationary and satisfies $\Div(\psi_D)=0$ and $(\D(\psi_D)+E)|_{\Ic}=0$. Assume that $\D(\psi_D) \in \Ld^2(\Omega)^{d\times d}$ (for otherwise the claim is trivial by~\eqref{e.prop-randtess-re}). Then,
appealing to~\eqref{e.prop-randtess-re}, we find $\expec{\D(\psi_D)}=0$. 
We may then use $\psi_D$ as a test function in the variational problem~\eqref{def:eff-vis-min-bis}, to the effect of
\[E : \Bb E\,\le\,|E|^2+\expecm{|\!\D(\psi_D)|^2},\]
and the claim~\eqref{e.above-re} now follows from~\eqref{e.prop-randtess-re}.

\medskip
\step{2} Lower bound: proof that
\begin{eqnarray}\label{e.below}
E : \Bb E\,\ge\,|E|^2+\expecM{\sum_{n} \frac{\mathds 1_{0 \in I_n}}{|I_n|} \Ec_{n,N}}.
\end{eqnarray}
By~\eqref{e.prop-randtess}, we can write
\begin{eqnarray*}
E : \Bb E&=&|E|^2+\expec{|\!\D(\psi_E)|^2}\\
&=&|E|^2+\expecM{\sum_{n} \frac{\mathds 1_{0 \in I_n}}{|I_n|}\int_{V_n}|\!\D(\psi_E)|^2}.
\end{eqnarray*}
Using the corrector $\psi_E$ as a test function for the Neumann single-particle problem~\eqref{e:def-Neu}, we find $\Ec_{n,N}\le\int_{V_n}|\!\D(\psi_E)|^2$ and the claim~\eqref{e.below} follows.

\medskip
\step{3} Conclusion.\\
In view of~\eqref{e.above} and~\eqref{e.below}, it remains to compare $\Ec_{n,D}$ and $\Ec_{n,N}$ to $\Ec_{n,\infty}$.
On the one hand, since $\psi_{n,D}$ is an admissible test function for $\Ec_{n,\infty}$, we have
\[\Ec_{n,\infty} \,\le\, \int_{\R^d} |\!\D(\psi_{n,D})|^2 \,=\,\int_{V_n} |\!\D(\psi_{n,D})|^2\,=\, \Ec_{n,D}.\]
On the other hand, since the restriction $\psi_{n,\infty}|_{V_n}$ is an admissible test function for $\Ec_{n,N}$, we have
\[\Ec_{n,N} \,\le\, \int_{V_n} |\!\D(\psi_{n,\infty})|^2 \,\le\, \Ec_{n,\infty}.\]
This yields
\[\Ec_{n,N}  \,\le\, \Ec_{n,\infty} \,\le\, \Ec_{n,D},\]
or alternatively,
\[|\Ec_{n,N}-\Ec_{n,\infty}|+|\Ec_{n,D}-\Ec_{n,\infty}| \,=\, \Ec_{n,D}-\Ec_{n,N}.\]
Further note that the minimality of Neumann problems entails
\[\sum_{n:I_n\subset J_{q,\rho}}\Ec_{n,N} \,\le\, \Fc_{q,D},\]
and thus
\[\expecM{\sum_{n\notin\Sc} \frac{\mathds1_{0\in I_n}}{|I_n|}\Ec_{n,N}}\,=\,\expecM{\sum_{q\in\Sc'}\frac{\mathds1_{0\in J_{q,\rho}}}{|J_{q,\rho}|}\sum_{n:I_n\subset J_{q,\rho}}\Ec_{n,N}}\, \le\, \expecM{\sum_{q\in\Sc'}\frac{\mathds1_{0\in J_{q,\rho}}}{|J_{q,\rho}|}\Fc_{q,D}}.\]
Combining these observations with~\eqref{e.above}--\eqref{e.below},
the conclusion~\eqref{e.var}  follows.
\qed

\subsection{Proof of Lemma~\ref{lem:diff-DN}}\label{sec:proof-lem:diff-DN}
The bound~\eqref{e.dir-perc-estim}
on $\Fc_{q,D}$ follows from Bogovskii's construction in form of~\cite[Lemma~4.2]{DG-21a}.
We turn to the proof of~\eqref{e.estim-diff-DN}.
By~\cite[Section~4.1]{D-20a}, there exists $w_n \in W^{1,\infty}_0(V_n)^d$ that is an admissible test function for the Dirichlet problem $\Ec_{n,D}$ such that
\[\Ec_{n,D}\,\le\,\int_{V_n}|\!\D(w_n)|^2\,\lesssim\, \mu(\rho_n),\]
which entails
\[0\,\le\, \Ec_{n,D}-\Ec_{n,N}\,\le\, \Ec_{n,D}\, \lesssim\, \mu(\rho_n).\]
To prove~\eqref{e.estim-diff-DN}, it remains to show that in the case $\rho_n\ge1$ we have
\begin{equation}\label{eq:diff-est1}
\Ec_{n,D}-\Ec_{n,N}~\le~ \rho_n^{-d}.
\end{equation}
This amounts to investigating the role of the different boundary conditions on $\partial V_n$.
The proof will require us to establish in passing the following two fine estimates,
\begin{eqnarray}
\int_{V_n}|\!\D(\psi_{n,D}-\psi_{n,N})|^2&\lesssim&\rho_n^{-d},\label{eq:diff-est2}\\
\int_{I_n+\frac12B}|\!\D(\psi_{n,D}-\psi_{n,N})|^2&\lesssim&\rho_n^{-2d}.\label{eq:diff-est3}
\end{eqnarray}
We assume from now on that $\rho_n\ge1$
and, without loss of generality, $x_n=0$. We drop the index $n$ to simplify notation and we set $r=\rho_n$ (to avoid confusion with the constant~$\rho$ in Assumption~\ref{H0} and elsewhere). 
We split the argument into three steps.

\medskip
\step1 Proof that
\begin{equation}\label{diff-estA}
\Ec_{D}-\Ec_{\infty}\,\lesssim\, r^{-d}.
\end{equation}
First recall the following standard bounds on the whole-space single-particle solution $\psi_\infty$ and on the associated pressure field $\Sigma_\infty$,
\begin{equation}\label{e:decay-single-inc}
|\psi_\infty (x)|\,\lesssim\, \langle x\rangle^{1-d},\qquad |\nabla \psi_\infty(x)| +|\Sigma_\infty(x)|\,\lesssim\, \langle x\rangle^{-d}.
\end{equation}
Define the neighborhood $I^+_r:=I+\frac r2B$, which satisfies $I \subset I^+_r \subset V$ and $\dist(I^+_r,\partial V)=\frac r2$,
and let $\chi\in C^\infty_c(V)$ be a smooth cut-off function with
\[\chi|_{I^+_r}\equiv1,\quad\supp\chi\subset I+rB\subset V,\quad 0\le\chi\le1,\quad|\nabla \chi|\lesssim r^{-1}.\]
Now consider the map $\chi \psi_{\infty}\in H^1_0(V)^d$, which coincides with $\psi_\infty$ on $I^+_r$. To make it an admissible test function for the Dirichlet problem~$\Ec_{D}$, we need to make it divergence-free in $V\setminus I^+_r$.
As the properties of the cut-off function $\chi$ and the incompressibility of~$\psi_\infty$ yield
\[\int_{V \setminus I^+_r} \Div (\chi \psi_\infty) \,=\, \int_{\partial V}  \chi \psi_\infty\cdot\nu  - \int_{I^+_r} \Div(\chi\psi_\infty)\,=\,0,\]
Bogovskii's construction entails that there exists $w \in H^1_0(V \setminus I^+)^d$ such that
\begin{gather}
\Div(w) \,=\,  \Div (\chi \psi_\infty)\,=\, \nabla \chi\cdot  \psi_\infty\qquad \text{in $V\setminus I^+_r$},\nonumber\\
\int_{V \setminus I^+_r} |\nabla w|^2 \,\lesssim\, \int_{V \setminus I^+_r} |\nabla \chi\cdot  \psi_\infty|^2.
\label{e:decay-bog-inc}
\end{gather}
By construction, the map $\chi \psi_{\infty}-w$ is now an admissible test function for $\Ec_{D}$ and we deduce 
\begin{eqnarray*}
 \Ec_{D} &\le& \int_{V} |\!\D(\chi \psi_{\infty}-w)|^2\\
 &=& \int_{V} \chi^2 |\!\D(\psi_{\infty})|^2
+2\int_{V} \chi \D(\psi_\infty) : (\nabla \chi\otimes\psi_\infty)\\
&&+\frac14 \int_{V} |\psi_\infty \otimes \nabla \chi+\nabla \chi \otimes \psi_\infty|^2
+\int_{V} |\!\D(w)|^2-2 \int_{V} \D(w) : \D(\chi \psi_\infty).
\end{eqnarray*}
Using~\eqref{e:decay-single-inc}, \eqref{e:decay-bog-inc}, and properties of the cut-off function $\chi$, the claim~\eqref{diff-estA} follows.

\medskip
\step2 Proof that
\begin{equation}\label{diff-estB}
\Ec_{\infty}-\Ec_{N}\,\lesssim\,r^{-d}.
\end{equation}
More precisely, we shall note that
\begin{equation}\label{diff-estC}
\Ec_{\infty}-\Ec_{N}\,=\,\int_{V}|\!\D(\psi_{N}-\psi_{\infty})|^2,
\end{equation}
and we shall establish the following more precise estimates,
\begin{eqnarray}
\int_{V}|\!\D(\psi_{D}-\psi_{\infty})|^2+\int_{V}|\!\D(\psi_{\infty}-\psi_{N})|^2&\lesssim&r^{-d},\label{diff-estD}\\
\int_{I^+}|\!\D(\psi_{D}-\psi_{\infty})|^2+\int_{I^+}|\!\D(\psi_{\infty}-\psi_{N})|^2&\lesssim&r^{-2d},\label{diff-estE}
\end{eqnarray}
where we use the short-hand notation $I^+:=I+\frac12B$.
We split the proof into five further substeps.

\medskip
\substep{2.1} Proof of~\eqref{diff-estC}.\\
By the Euler--Lagrange equation for $\psi_{N}$ in form of
\[\int_{V} \D(\psi_{N}): \D(\psi_\infty-\psi_{N})\,=\,0,\]
we find
\begin{eqnarray*} 
\int_{V}|\!\D(\psi_\infty)|^2-\int_{V}|\!\D(\psi_{N})|^2&=& \int_{V} \D(\psi_\infty+\psi_{N}):\D(\psi_\infty-\psi_{N})\\
&=&\int_{V} |\!\D(\psi_\infty-\psi_{N})|^2,
\end{eqnarray*}
that is,~\eqref{diff-estC}.

\medskip
\substep{2.2} Proof that
\begin{equation}\label{diff-estD-prep}
\int_{V}|\!\D(\psi_{N}-\psi_{\infty})|^2\lesssim r^{-d}+\Big(\int_{I^+}|\!\D(\psi_{N}-\psi_{\infty})|^2\Big)^\frac12.
\end{equation}
As in~\eqref{eq:singlepart}, the Euler--Lagrange equation for $\psi_\infty$ takes the following form, in terms of the associated pressure field $\Sigma_\infty$ and Cauchy stress tensor \mbox{$\sigma_\infty:=\sigma(\psi_\infty+Ex,\Sigma_\infty)$},
\begin{equation}\label{eq:psi-infty}
\left\{\begin{array}{ll}
-\triangle \psi_\infty+\nabla \Sigma_\infty=0,&\text{in $\R^d\setminus I$},\\
\Div( \psi_\infty)=0,&\text{in $\R^d\setminus I$},\\
\D(\psi_\infty+Ex)=0,&\text{in $I$},\\
\int_{\partial I}\sigma_\infty\nu=0,&\\
\int_{\partial I}\Theta x\cdot\sigma_\infty\nu=0,&\forall \Theta\in\Md^\Skew,
\end{array}\right.
\end{equation}
and similarly the equation for $\psi_N$ is as follows, in terms of the associated pressure $\Sigma_N$ and Cauchy stress tensor $\sigma_N:=\sigma(\psi_N+Ex,\Sigma_N)$,
\begin{equation*}
\left\{\begin{array}{ll}
-\triangle \psi_N+\nabla \Sigma_N=0,&\text{in $V\setminus I$},\\
\Div( \psi_N)=0,&\text{in $V\setminus I$},\\
\sigma_N\nu=0,&\text{on $\partial V$},\\
\D(\psi_N+Ex)=0,&\text{in $I$},\\
\int_{\partial I}\sigma_N\nu=0,&\\
\int_{\partial I}\Theta x\cdot\sigma_N\nu=0,&\forall \Theta\in\Md^\Skew.
\end{array}\right.
\end{equation*}
Testing both equations with $\psi_N-\psi_\infty$, we obtain
\begin{equation}\label{eq:pre-N-infty-L2}
2\int_{V}|\!\D(\psi_N-\psi_\infty)|^2\,=\,
-\int_{\partial V}(\psi_N-\psi_\infty)\cdot\sigma_\infty\nu.
\end{equation}
In order to estimate the right-hand side, we note that, testing the equation for $\psi_N$ with~$\psi_\infty$ and the equation for $\psi_\infty$ with $\psi_N$ in $V$ yields
\begin{eqnarray*}
2\int_V\D(\psi_\infty):\D(\psi_N)&=&-\int_{\partial I}\psi_\infty\cdot\sigma_N\nu,\\
2\int_V\D(\psi_N):\D(\psi_\infty)&=&\int_{\partial V}\psi_N\cdot\sigma_\infty\nu-\int_{\partial I}\psi_N\cdot\sigma_\infty\nu,
\end{eqnarray*}
from which we deduce
\[\int_{\partial V}\psi_N\cdot\sigma_\infty\nu\,=\,\int_{\partial I}\psi_N\cdot\sigma_\infty\nu-\int_{\partial I}\psi_\infty\cdot\sigma_N\nu.\]
Inserting this into~\eqref{eq:pre-N-infty-L2}, recalling $(\psi_N+Ex)|_I=(\psi_\infty+Ex)|_I=0$, and using the boundary conditions for $\psi_N,\psi_\infty$, we get
\begin{equation*}
2\int_{V}|\!\D(\psi_N-\psi_\infty)|^2\,=\,
\int_{\partial V}\psi_\infty\cdot\sigma_\infty\nu-\int_{\partial I}Ex\cdot(\sigma_N-\sigma_\infty)\nu.
\end{equation*}
In view of~\eqref{e:decay-single-inc}, this entails
\begin{equation*}
\int_{V}|\!\D(\psi_N-\psi_\infty)|^2\,\lesssim\,
r^{-d}+\int_{\partial I}|\sigma_N-\sigma_\infty|,
\end{equation*}
and thus, by the trace estimate in Lemma~\ref{lem:trace-0}(ii), the claim~\eqref{diff-estD-prep} follows.

\medskip
\substep{2.3} Proxy for $\psi_N-\psi_\infty$.\\
If the difference $\psi:=\psi_N-\psi_\infty$ satisfied the free steady Stokes equation in the whole domain~$V$, then Lemma~\ref{lem:regularity} would yield $\int_{I^+}|\!\D(\psi)|^2\lesssim r^{-d}\int_{V}|\!\D(\psi)|^2$, and the conclusion~\eqref{diff-estD} for $\psi$ would already follow from~\eqref{diff-estD-prep} together with Young's inequality. However, $\psi$ is rigid in $I$ and does not satisfy the free steady Stokes equation in the whole domain.\footnote{As showed in Lemma~\ref{*lem:mean-value-inc} in Appendix~\ref{append:ell}, the mean-value property can actually be extended in presence of rigid particles. Rather than appealing to this general result here, we provide a self-contained and more elementary approach in the present single-particle setting.} To overcome this issue, we shall compare $\psi$ to a suitable proxy: we consider the solution $(\tilde\psi,\tilde\Sigma)$ of the following auxiliary Neumann problem in $V$,
\begin{equation}\label{eq:psitilde}
\left\{\begin{array}{ll}
-\triangle\tilde\psi+\nabla\tilde \Sigma=0,&\text{in $V$},\\
\Div(\tilde\psi)=0,&\text{in $V$},\\
\sigma(\tilde\psi,\tilde\Sigma)\nu=-\sigma_\infty\nu,&\text{on $\partial V$}.
\end{array}\right.
\end{equation}
Note that the solution is only defined up to a rigid motion, which can be fixed for instance by choosing $\int_V\tilde\psi=0$ and $\int_V\nabla\tilde\psi\in\Md^\Sym_0$.

\medskip\noindent
The rest of this step is devoted to the proof of the following estimates for $\tilde\psi$,
\begin{equation}\label{eq:energy-tilde}
\int_V|\!\D(\tilde\psi)|^2\,\lesssim\,r^{-d},\qquad \int_{I^+}|\!\D(\tilde\psi)|^2\,\lesssim\,r^{-2d}.
\end{equation}
For that purpose, we start by testing equation~\eqref{eq:psitilde} with $\tilde\psi$ itself, to the effect of
\[2\int_V|\!\D(\tilde\psi)|^2\,=\,-\int_{\partial V}\tilde\psi\cdot\sigma_\infty\nu.\]
In order to estimate the right-hand side, we note that, testing the equation for $\tilde\psi$ with $\psi_\infty$ and the equation for $\psi_\infty$ with $\tilde\psi$,
\begin{eqnarray*}
2\int_V\D(\psi_\infty):\D(\tilde\psi)&=&-\int_{\partial V}\psi_\infty\cdot\sigma_\infty\nu,\\
2\int_V\D(\tilde\psi):\D(\psi_\infty)&=&\int_{\partial V}\tilde\psi\cdot\sigma_\infty\nu-\int_{\partial I}\tilde\psi\cdot\sigma_\infty\nu,
\end{eqnarray*}
from which we deduce
\[\int_{\partial V}\tilde\psi\cdot\sigma_\infty\nu\,=\,-\int_{\partial V}\psi_\infty\cdot\sigma_\infty\nu+\int_{\partial I}\tilde\psi\cdot\sigma_\infty\nu.\]
Inserting this into the above yields
\[2\int_V|\!\D(\tilde\psi)|^2\,=\,\int_{\partial V}\psi_\infty\cdot\sigma_\infty\nu-\int_{\partial I}\tilde\psi\cdot\sigma_\infty\nu,\]
and thus, using~\eqref{e:decay-single-inc}, noting that
any rigid motion can be added to $\tilde\psi$,
and appealing to the trace estimate in Lemma~\ref{lem:trace-0}(i),
\[\int_V|\!\D(\tilde\psi)|^2\,\lesssim\,r^{-d}+\Big(\int_{I}|\!\D(\tilde\psi)|^2\Big)^\frac12.\]
As $\tilde\psi$ satisfies the free steady Stokes equation in $V$ and as $|I^+|\lesssim1$ and $\dist(I^+,\partial V)\ge\frac r2$, we may now appeal to Lemma~\ref{lem:regularity}, to the effect of
\begin{equation}\label{eq:regularity}
\int_{I^+}|\!\D(\tilde\psi)|^2\,\lesssim\, r^{-d}\int_{V}|\!\D(\tilde\psi)|^2,
\end{equation}
and the claim~\eqref{eq:energy-tilde} follows.

\medskip
\substep{2.4} Proof of~\eqref{diff-estD} \& \eqref{diff-estE} for $\psi=\psi_N-\psi_\infty$.\\
Decomposing $\psi=(\psi-\tilde\psi)+\tilde\psi$ and using~\eqref{eq:energy-tilde}, we find
\begin{equation}\label{eq:pre-dpsi-V+}
\int_{I^+}|\!\D(\psi)|^2
\,\lesssim\,r^{-2d}+\int_{V}|\!\D(\psi-\tilde\psi)|^2,
\end{equation}
and thus, inserting this into~\eqref{diff-estD-prep},
\begin{equation}\label{eq:pre-dpsi-V}
\int_{V}|\!\D(\psi)|^2
\,\lesssim\,r^{-d}+\Big(\int_{V}|\!\D(\psi-\tilde\psi)|^2\Big)^\frac12.
\end{equation}
Next, noting that $\psi$ satisfies
\begin{equation}\label{eq:psi}
\left\{\begin{array}{ll}
-\triangle \psi+\nabla \Sigma=0,&\text{in $V\setminus I$},\\
\Div( \psi)=0,&\text{in $V\setminus I$},\\
\sigma(\psi,\Sigma)\nu=-\sigma_\infty\nu,&\text{on $\partial V$},\\
\D(\psi)=0,
&\text{in $I$},\\
\int_{\partial I}\sigma (\psi,\Sigma)\nu=0,&\\
\int_{\partial I}\Theta x\cdot\sigma(\psi,\Sigma)\nu=0,&\forall \Theta\in\Md^\Skew,
\end{array}\right.
\end{equation}
and testing this equation as well as~\eqref{eq:psitilde} with $\psi-\tilde\psi$, we find
\[2\int_V|\!\D(\psi-\tilde\psi)|^2\,=\,\int_{\partial I}\tilde\psi\cdot\sigma(\psi,\Sigma)\nu.\]
As we may add any rigid motion to $\tilde\psi$, the trace estimates in Lemma~\ref{lem:trace-0} lead us to
\begin{equation*}
\int_V|\!\D(\psi-\tilde\psi)|^2\,\lesssim\,\Big(\int_{I}|\!\D(\tilde\psi)|^2\Big)^\frac12\Big(\int_{I^+}|\!\D(\psi)|^2\Big)^\frac12.
\end{equation*}
Decomposing again $\psi=(\psi-\tilde\psi)+\tilde\psi$ and using Young's inequality, we get
\begin{equation*}
\int_V|\!\D(\psi-\tilde\psi)|^2\,\lesssim\,\int_{I^+}|\!\D(\tilde\psi)|^2.
\end{equation*}
Inserting this into~\eqref{eq:pre-dpsi-V} and appealing again to~\eqref{eq:energy-tilde}, the conclusion~\eqref{diff-estD} \& \eqref{diff-estE} follows for $\psi=\psi_N-\psi_\infty$.

\medskip
\substep{2.5} Proof of~\eqref{diff-estD} \& \eqref{diff-estE} for $\psi_D-\psi_\infty$.\\
Although this part is in fact not needed for the proof of~\eqref{diff-estB}, it is included for future reference.
We start by decomposing
\[\int_{V}|\!\D(\psi_{D}-\psi_{\infty})|^2\,=\,\int_{V}|\!\D(\psi_D)|^2-\int_{V}|\!\D(\psi_\infty)|^2-2\int_{V}\D(\psi_{D}-\psi_{\infty}):\D(\psi_{\infty}).\]
Testing the equation for $\psi_\infty$ with $\psi_D-\psi_\infty$, we find
\begin{equation*}
2\int_{V}\D(\psi_{D}-\psi_{\infty}):\D(\psi_{\infty})
\,=\,-\int_{\partial V}\psi_{\infty}\cdot\sigma_{\infty}\nu,
\end{equation*}
so that the above becomes
\[\int_V|\!\D(\psi_D-\psi_\infty)|^2\,=\,\Ec_D-\Ec_\infty+\int_{\partial V}\psi_\infty\cdot\sigma_\infty\nu.\]
Using~\eqref{diff-estA} and~\eqref{e:decay-single-inc}, this entails
\begin{equation}\label{diff-estF}
\int_{V}|\!\D(\psi_{D}-\psi_{\infty})|^2\,\lesssim\, r^{-d},
\end{equation}
that is, \eqref{diff-estD} for $\psi_D-\psi_\infty$.

\medskip\noindent
It remains to prove~\eqref{diff-estE}. For that purpose, as above, we proceed by defining a suitable proxy for $\psi_D-\psi_\infty$ satisfying the free steady Stokes equation in the whole domain $V$: we consider the solution $(\hat\psi,\hat\Sigma)$ of the following auxiliary Dirichlet problem in $V$,
\begin{equation*}
\left\{\begin{array}{ll}
-\triangle\hat\psi+\nabla\hat \Sigma=0,&\text{in $V$},\\
\Div(\hat\psi)=0,&\text{in $V$},\\
\hat\psi=-\psi_\infty,&\text{on $\partial V$}.
\end{array}\right.
\end{equation*}
A straightforward adaptation of the proof of~\eqref{eq:energy-tilde} yields
\begin{equation*}
\int_V|\!\D(\hat\psi)|^2\,\lesssim\,r^{-d},\qquad \int_{I^+}|\!\D(\hat\psi)|^2\,\lesssim\,r^{-2d}.
\end{equation*}
From this together with~\eqref{diff-estF}, the conclusion~\eqref{diff-estE} for $\psi_D-\psi_\infty$ easily follows by similar arguments as in Substep~2.4.
\qed

\subsection{Explicit form of Einstein's formula}\label{sec:explicitEinstein}
This last section is devoted to the proof of Proposition~\ref{prop:B1}.
Under Assumption~\ref{B1}, the definition~\eqref{eq:def-B1} of $\Bb^1$ becomes
\begin{equation*}
E:\Bb^1E
\,=\,\lambda(\Pc)\,\expecM{\int_{\R^d}|\!\D(\psi_E^\circ)|^2},
\end{equation*}
that is,~\eqref{e.Einstein0} in terms of the unique decaying solution $\psi_E^\circ$ of the single-particle problem~\eqref{eq:psicirc}.
It remains to prove Einstein's formula~\eqref{e.Einstein} for spherical particles, $I_n=B(x_,r_n)$, with iid random radii $\{r_n\}_n$. By scaling, the above becomes
\begin{equation*}
E:\Bb^1E
\,=\,\lambda(\Pc)\,\expecm{(r_n)^d}\int_{\R^d}|\!\D(\tilde\psi_E^\circ)|^2,
\end{equation*}
in terms of the unique decaying solution $\tilde \psi_E^\circ$ of the rescaled elementary problem
\[\left\{\begin{array}{ll}
-\triangle\tilde \psi_E^\circ+\nabla\tilde \Sigma_E^\circ=0,&\text{in $\R^d\setminus B$},\\
\Div(\tilde \psi_E^\circ)=0,&\text{in $\R^d\setminus B$},\\
\D(\tilde \psi_E^\circ+Ex)=0,&\text{in $B$},\\
\int_{\partial B}\sigma(\tilde\psi_E^\circ+Ex,\tilde\Sigma_E^\circ)\nu=0,&\\
\int_{\partial B}\Theta x\cdot\sigma(\tilde\psi_E^\circ+Ex,\tilde\Sigma_E^\circ)\nu=0,&\forall \Theta\in\Md^\Skew.
\end{array}\right.\]
Alternatively, using the energy identity for this equation,
\begin{equation}\label{eq:pre-Einstein}
E:\Bb^1E
\,=\,\tfrac12\lambda(\Pc)\,\expecm{(r_n)^d}\int_{\partial B}Ex\cdot\sigma(\tilde\psi_E^\circ+Ex,\tilde\Sigma)\nu.
\end{equation}
As is well-known, e.g.~\cite[Section~2.1.3]{GM-11}, $\tilde \psi_E^\circ$ coincides with the unique solution of
\[\left\{\begin{array}{ll}
-\triangle \tilde \psi_E^\circ+\nabla \tilde \Sigma_E^\circ=0,&\text{in $\R^d\setminus B$},\\
\Div( \tilde \psi_E^\circ)=0,&\text{in $\R^d\setminus B$},\\
\tilde \psi_E^\circ=-Ex,&\text{on $\partial B$},
\end{array}\right.\]
and is explicitly given by the following formulas for $|x|\ge1$,
\begin{eqnarray*}
\tilde \psi_E^\circ(x)&:=&-\frac{d+2}2\frac{(x\cdot Ex)x}{|x|^{d+2}}-\frac12\bigg(2\frac{Ex}{|x|^{d+2}}-(d+2)\frac{(x\cdot Ex)x}{|x|^{d+4}}\bigg),\\
\tilde \Sigma_E^\circ(x)&:=&-(d+2)\frac{x\cdot Ex}{|x|^{d+2}}.
\end{eqnarray*}
Inserting this into~\eqref{eq:pre-Einstein}, a direct computation yields
\[E:\Bb^1E\,=\,\tfrac{d+2}2\,\lambda(\Pc)\,\expecm{(r_n)^d}|B||E|^2,\]
that is, Einstein's formula~\eqref{e.Einstein}.  
\qed

\newpage
\section{Cluster expansion of the effective viscosity}\label{sec:cluster}

This section is devoted to the higher-order cluster expansion of the effective viscosity~$\Bb$: starting from finite-volume approximations, we establish cluster formulas and prove uniform estimates in the large-volume limit.
These results are mainly inspired by our previous work~\cite{DG-16a} on the Clausius--Mossotti conductivity formula, where we introduced the triad consisting of: (1)~finite-volume approximation; (2)~cluster expansion; (3)~uniform $\ell^1-\ell^2$ energy estimates.
We further refine the analysis of~\cite{DG-16a}, in particular improving on error estimates, and properly estimating cluster coefficients in case of large uniform particle separation $\ell(\Pc)\gg1$; there are also some new twists due to the rigidity of the inclusions.
Henceforth, in the rest of this memoir, we assume that particles are uniformly separated in the sense of~\ref{H0} and~\ref{Gunif}.

\subsection{Finite-volume approximations}\label{sec:finitevol-approx}
In order to make sense of cluster expansions and avoid diverging series, we start by defining finite-volume approximations of the effective viscosity, obtained by a periodization procedure, which will in turn provide an implicit renormalization of cluster coefficients in the large-volume limit.
More precisely, we define a restriction $\Pc_{L}$ on $Q_L$ of the point process~$\Pc$ via
\[\Pc_{L}\,:=\,\{x_{n}:n\in \Nb_{L}\},\qquad \Nb_{L}\,:=\,\{n:x_{n}\in \Qd\},\qquad \Qd:=Q_{L-2(\ell(\Pc)\vee(1+\rho))}\]
and we consider the corresponding random set
\begin{equation}\label{e.set-perio}
\Ic_{L}\,:=\,\textstyle\bigcup_{n\in \Nb_{L}}I_{n},\qquad I_{n}\,=\,x_{n}+I_n^\circ.
\end{equation}
For convenience, we choose an enumeration $\Pc_{L}=\{x_{n,L}\}_n$ and set $I_{n,L}=x_{n,L}+I_{n,L}^\circ$.
Under Assumptions~\ref{H0} and~\ref{Gunif}, the periodic extension of $\Ic_L$  satisfies
\begin{enumerate}[\quad$\bullet$]
\item\emph{Regularity and separation:}
For all $L$, the periodized random set $\Ic_{L}+L\Z^d$ satisfies the $\rho$-regularity and uniform separation assumptions in~\ref{H0} and~\ref{Gunif}. Moreover, the periodized point process $\Pc_L+L\Z^d$ satisfies $\ell(\Pc_L+L\Z^d)\ge\ell(\Pc)\gtrsim1$.
\smallskip\item\emph{Stabilization:}
For all $L$, there holds $\Pc_L|_{\Qd}=\Pc|_{\Qd}$.
\end{enumerate}
We then define the following finite-volume approximation of the effective viscosity $\Bb$,
\begin{equation}\label{eq:def-per-effvis}
E:\Bb_{L}E\,:=\,\expecM{\fint_{ Q_L}|\!\D(\psi_{E;L})+E|^2},
\end{equation}
where $\psi_{E;L}\in \Ld^2(\Omega;H^1_\per(Q_L)^d)$ is almost surely the unique solution in $H^1_\per(Q_L)^d$, with vanishing average \mbox{$\int_{Q_L}\psi_{E;L}=0$}, of the periodized version of the corrector problem~\eqref{eq:cor},
\begin{equation}\label{eq:cor-per}
\left\{\begin{array}{ll}
-\triangle\psi_{E;L}+\nabla\Sigma_{E;L}=0,&\text{in $Q_L\setminus\Ic$},\\
\Div( \psi_{E;L})=0,&\text{in $Q_L\setminus\Ic$},\\
\D(\psi_{E;L}+Ex)=0,&\text{in $\Ic_L$},\\
\int_{\partial I_{n;L}}\sigma_{E;L}\nu=0,&\forall n,\\
\int_{\partial I_{n;L}}\Theta(x-x_{n;L})\cdot\sigma_{E;L}\nu=0,&\forall n,~\forall \Theta\in\Md^\Skew,
\end{array}\right.
\end{equation}
where we use the short-hand notation $\sigma_{E;L}:=\sigma(\psi_{E;L}+Ex,\Sigma_{E;L})$ for the Cauchy stress tensor.
As a corollary of~\cite[Theorem~1]{DG-19},\footnote{This requires to replace Dirichlet boundary conditions in~\cite{DG-19} by periodic conditions, as is standard in homogenization theory.} in view of the above stabilization property, this finite-volume approximation~\eqref{eq:def-per-effvis} is consistent in the sense of
\begin{equation}\label{eq:conv-BL-B-hom}
\lim_{L \uparrow \infty} \Bb_{L}\,=\,\Bb.
\end{equation}
As opposed to $\Bb$, we emphasize that the approximation $\Bb_L$ depends only on the finite number of inclusions $\{I_{n,L}\}_n$. Indeed, by~\ref{H0}, the number of inclusions in $Q_L$ has almost surely a deterministic upper bound $CL^d$.
The associated cluster expansion is therefore well-defined.

\subsection{Main results}
We start with the cluster expansion of the finite-volume approximation $\Bb_L$, establishing suitable formulas for cluster coefficients and for the remainder.
This is analogous to formulas obtained in our previous work on the conductivity problem~\cite{DG-16a}.
While the formula~\eqref{eq:form-clust-rem} for the remainder naturally involves the original corrector with the whole set~$\Pc_{L}$ of particles, we emphasize that the bound~\eqref{e.DGV-control} only involves correctors associated with finite numbers of inclusions (uniformly in $L$): this is key to the optimal estimates obtained in the sequel and constitutes the first twist wrt \cite{DG-16a}. Indeed, this control is based on the rigidity of the particles and is therefore not available in the generality considered for the conductivity problem in~\cite{DG-16a}; it was first observed at second order by G\'erard-Varet in~\cite{GV-20}. The proof is displayed in Section~\ref{sec:cluster-pr}.

\begin{theor1}[Finite-volume cluster expansion]\label{thm:expansion}
Under Assumptions~\ref{H0} and~\ref{Gunif}, finite-volume approximations of the effective viscosity can be expanded  for all~$L$ and $k\ge1$,
\begin{equation}\label{eq:preconcl-exp}
\Bb_{L}\,=\,\Id+\sum_{j=1}^k\tfrac{1}{j!}\Bb_{L}^{j}+R_{L}^{k+1},
\end{equation}
where the coefficients and remainders are defined as follows:

\smallskip\noindent
$\bullet$ The coefficients $\{\Bb_{L}^{j}\}_j$ are given by cluster formulas, cf.~\eqref{e.formal-expansion},
\begin{equation}\label{eq:form-clust}
E:\Bb^{j}_{L}E\,=\,j!\sum_{\sharp F=j}\expecM{\fint_{ Q_L}\del^F\big(|\!\D(\psi_{E;L}^\varnothing)+E|^2\big)},
\end{equation}
which can be alternatively expressed as
\begin{eqnarray}
\hspace{-0.9cm}E:\Bb^{j}_{L}E\!\!\!&=&\!\!\!
\tfrac12j!\,L^{-d}\sum_{\sharp F=j}\sum_{n\in F}\expecM{\int_{\partial I_{n,L}}E(x-x_{n,L})\cdot\del^{F\setminus\{n\}}\sigma_{E;L}^{\{n\}}\nu}\label{eq:form-clust1}\\
\!\!\!&=&\!\!\!\tfrac12j!\,L^{-d}\sum_{\sharp F=j}\sum_{n\in F}\expecM{\int_{\partial I_{n,L}}\del^{F\setminus\{n\}}\big(\psi_{E;L}^\varnothing+E(x-x_{n,L})\big)\cdot\sigma_{E;L}^{F}\nu},\label{eq:form-clust2}
\end{eqnarray}
where we use the short-hand notation $\sigma_{E;L}^F:=\sigma(\psi_{E;L}^F+Ex,\Sigma^F_{E;L})$ for the Cauchy stress tensor.

\smallskip\noindent
$\bullet$ The remainder $R_{L}^{k+1}$ can be represented as
\begin{equation}\label{eq:form-clust-rem}
E:R_{L}^{k+1}E~=~\tfrac12L^{-d}\sum_{\sharp F=k+1}\sum_{n\in F}\expecM{\int_{\partial I_{n,L}}\del^{F\setminus\{n\}}\big(\psi_{E;L}^\varnothing+E(x-x_{n,L})\big)\cdot \sigma_{E;L}\nu},
\end{equation}
and is estimated as follows,
\begingroup\allowdisplaybreaks
\begin{align}
&|E: R_{L}^{k+1}E|~\lesssim~\expecM{L^{-d}\sum_{n}\int_{I_{n,L}}\Big|\sum_{\sharp F=k\atop n\notin F}\D(\del^{F}\psi_{E;L}^\varnothing)\Big|^2}\label{e.DGV-control}\\
&\hspace{1cm}+\sum_{j=1}^k\,\E\bigg[L^{-d}\sum_{n}\bigg(\int_{I_{n,L}}\Big|\sum_{\sharp F=k\atop n\notin F}\D(\del^{F}\psi_{E;L}^\varnothing)\Big|^2\bigg)^\frac12\nonumber\\
&\hspace{4cm}\times\bigg(\int_{I_{n,L}+\rho B}\Big|\sum_{\sharp F=j-1\atop n\notin F}\D\big(\del^{F}(\psi^{\{n\}}_{E;L}+E(x-x_{n,L}))\big)\Big|^2\bigg)^\frac12\bigg].\nonumber\qedhere
\end{align}
\endgroup
\end{theor1}

In view of the short-range setting~\eqref{e.scalings-local}, we expect $\Bb_{L}^j=O(\lambda_j(\Pc))$ and we aim to prove uniform-in-$L$ estimates that would allow to pass to the large-volume limit and recover a dilute expansion for the original effective viscosity $\Bb$.
This is partially achieved in the upcoming theorem, which states fine estimates on cluster coefficients and on the remainder.
However note that we cannot directly obtain uniform-in-$L$ estimates with the desired scalings $O(\lambda_j(\Pc))$. Instead, the result is twofold:
\begin{enumerate}[---]
\item {\it Uniform estimates:} In item~(i), we state uniform-in-$L$ estimates, which further display the optimal scaling in the order $j$ and in the minimal distance $\ell=\ell(\Pc)$, but fail to capture the general expected dependence on multi-point intensities $\{\lambda_j(\Pc)\}_j$.
\smallskip\item {\it Non-uniform estimates:} In item~(ii), we state non-uniform estimates, which display a logarithmic divergence in the large-volume limit~$L\uparrow\infty$, but have the merit of capturing the correct dependence on multi-point intensities.
\end{enumerate}
Uniform estimates in~(i) allow to deduce the convergence of cluster coefficients~$\{\Bb_{L}^j\}_j$ in the large-volume limit~$L\uparrow\infty$, cf.~\eqref{eq:form-Bj} below: this actually {\it defines} infinite-volume cluster coefficients in a meaningful way, providing an implicit renormalization of diverging series and answering the question raised in Section~\ref{sec:need-renorm}.  As they display the optimal dependence on the minimal distance $\ell=\ell(\Pc)$,
these estimates already yield the desired infinite-volume cluster expansion in the large-separation regime $\ell\gg1$ with $\lambda_j(\Pc)$ replaced by $(\ell^{-d})^j$, which is optimal in some cases (see dilation setting in Theorem~\ref{th:analytic}).
To treat the general model-free dilute setting, however, uniform estimates need to be further derived with the correct dependence on multi-point intensities: this requires to overcome logarithmic divergences in non-uniform estimates in~(ii), which is the subject of Section~\ref{sec:intermezzo}. The proof of the present result is split between Sections~\ref{sec:unif-ell12-pr}, \ref{sec:proof-thm:bounds}, \ref{sec:proof-eq:form-Bj}, and~\ref{sec:non-unif}.

\begin{theor1}[Cluster estimates and large-volume limit]\label{thm:bounds}
Under Assumptions~\ref{H0} and \ref{Gunif}, the coefficients and the remainder of the finite-volume cluster expansion in Theorem~\ref{thm:expansion} satisfy the following two classes of estimates.
\begin{enumerate}[(i)]
\item \emph{Uniform estimates:}
For all $L$ and $k, j\ge1$,
\begin{eqnarray}\label{e.unif-bd-Bb-R}
|\Bb_{L}^{j}|&\le& j!(C\ell^{-d})^{j},\nonumber\\
|R_{L}^{k+1}|&\le& (C\ell^{-d})^{k+1}.
\end{eqnarray}
\item \emph{Non-uniform estimates:}
For all $L$ and~$k, j\ge 1$,
\begin{eqnarray}\label{e.non-unif-bd-Bb-R}
|\Bb_{L}^j|&\lesssim_{j}&\lambda_j(\Pc)\,(\log L)^{j-1},\nonumber\\
|R_{L}^{k+1}|&\lesssim_{k}&\sum_{l=k}^{2k}\lambda_{l+1}(\Pc) (\log L)^{l}.
\end{eqnarray}
\end{enumerate}
In particular, as a consequence of~(i), for all $k,j\ge1$, the following large-volume limits are well-defined,
\begin{equation}\label{eq:form-Bj}
\Bb^j:=\lim_{L\uparrow\infty}\Bb^j_{L},\qquad R^{k+1}:=\lim_{L\uparrow\infty}R^{k+1}_{L},
\end{equation}
so that the cluster expansion~\eqref{eq:preconcl-exp} becomes, for all~$k\ge1$,
\begin{equation}\label{eq:preconcl-exp-error}
\bigg|\Bb-\Big(\Id+\sum_{j=1}^k\tfrac{1}{j!}\Bb^{j}\Big)\bigg|\,\le\,|R^{k+1}|\,\le\,(C\ell^{-d})^{k+1}.\qedhere
\end{equation}
\end{theor1}

\subsection{Preliminary lemmas}\label{sec:preliminary}$ $\hspace{-0.2cm}
Henceforth, we fix $E$ with $|E|=1$ and we skip the associated subscript for notational convenience.
Before turning to the proof of Theorems~\ref{thm:expansion} and~\ref{thm:bounds}, we state a series of preliminary lemmas.
We start with the following useful reformulation of the corrector equation~\eqref{eq:cor}, where the rigidity constraint is viewed as generating a source term concentrated at particle boundaries in steady Stokes equations.

\begin{lem}[Reformulation of the corrector equation]\label{lem:eqn-corH}
For all $H\subset\N$ we have in $ Q_L$,
\begin{equation}\label{eq:reform}
-\triangle \psi_{L}^H+\nabla(\Sigma_{L}^H\mathds1_{ Q_L\setminus\Ic_{L}^H})\,=\,-\sum_{n\in H}\delta_{\partial I_{n,L}}\sigma_{L}^H\nu,
\end{equation}
where $\delta_{\partial I_{n,L}}$ stands for the Hausdorff measure on the boundary of $I_{n,L}$.\footnote{More precisely, we define $\int_{Q_L}\phi_L\,\delta_{\partial I_{n,L}}:=\int_{\partial I_{n,L}}\phi_L$ for any test function~$\phi_L\in C^\infty_\per(Q_L)$.}
\end{lem}

\begin{proof}
For any test function $\phi_L\in C^\infty_\per(Q_L)^d$, recalling that $\psi_L^H$ is divergence-free
and that it satisfies $\D(\psi_L^H)+E=0$ in $\Ic_L^H$,
we find
\begin{eqnarray*}
\lefteqn{\int_{Q_L}\nabla\phi_L:\nabla\psi_L^H-\int_{Q_L\setminus\Ic_L^H}\Sigma_L^H\Div(\phi_L)}\\
&=&2\int_{Q_L}\nabla\phi_L:\D(\psi_L^H)-\int_{Q_L\setminus\Ic_L^H}\Sigma_L^H\Div(\phi_L)\\
&=&\int_{Q_L\setminus\Ic_L^H}\nabla\phi_L:\sigma(\psi_L^H+Ex,\Sigma_L^H).
\end{eqnarray*}
Since the steady Stokes equation for $\psi_L^H$ writes $\Div(\sigma(\psi_L^H+Ex,\Sigma_L^H))=0$ in $Q_L\setminus\Ic_L^H$,
we deduce, after integration by parts,
\begin{equation*}
\int_{Q_L}\nabla\phi_L:\nabla\psi_L^H-\int_{Q_L\setminus\Ic_L^H}\Sigma_L^H\Div(\phi_L)
\,=\,-\sum_{n\in H}\int_{\partial I_{n,L}}\phi_L\cdot\sigma(\psi_L^H+Ex,\Sigma_L^H)\nu.
\end{equation*}
By the arbitrariness of~$\phi_L$, this proves~\eqref{eq:reform}.
\end{proof}

Next, the following result provides corresponding Stokes equations for corrector differences, which will be used abundantly in the sequel.

\begin{lem}[Equations for corrector differences]\label{eq:cordiff-form}
For all disjoint subsets $F,H\subset\N$ with~$F$ finite, we have
in $ Q_L$,
\begin{align}
&-\triangle\del^{F}\psi_{L}^H+\nabla\del^{F}\big(\Sigma_{L}^H\mathds1_{ Q_L\setminus\Ic_{L}^{H}}\big)\label{eq:diff-st}\\
&\hspace{3cm}\,=\,-\sum_{n\in H}\delta_{\partial I_{n,L}}\del^F\sigma^{H}_{L}\nu
-\sum_{n\in F}\delta_{\partial I_{n,L}}\del^{F\setminus\{n\}}\sigma^{H\cup\{n\}}_{L}\nu.\nonumber
\qedhere
\end{align}
\end{lem}
\begin{proof}
The starting point is the equation~\eqref{eq:reform} satisfied by $\psi_{L}^{S\cup H}$,
\[-\triangle\psi_{L}^{S\cup H}+\nabla\big(\Sigma_{L}^{S\cup H}\mathds1_{ Q_L\setminus\Ic_{L}^{S\cup H}}\big)=-\sum_{n\in S\cup H}\delta_{\partial I_{n,L}}\sigma^{S\cup H}_{L}\nu.\]
Using the definition~\eqref{eq:def-diff} of the difference operator, we deduce
\begin{equation*}
-\triangle\del^{F}\psi_{L}^H+\nabla\del^{F}\big(\Sigma_{L}^H\mathds1_{ Q_L\setminus\Ic_{L}^{H}}\big)=-\sum_{S\subset F}(-1)^{|F\setminus S|}\sum_{n\in S\cup H}\delta_{\partial I_{n,L}}\sigma^{S\cup H}_{L}\nu,
\end{equation*}
and it remains to reformulate the right-hand side.
For that purpose, we decompose
\begin{multline*}
-\triangle\del^{F}\psi_{L}^H+\nabla\del^{F}\big(\Sigma_{L}^H\mathds1_{ Q_L\setminus\Ic_{L}^{H}}\big)\\
\,=\,-\sum_{n\in H}\delta_{\partial I_{n,L}}\sum_{S\subset F}(-1)^{|F\setminus S|}\sigma^{S\cup H}_{L}\nu
-\sum_{n\in F}\delta_{\partial I_{n,L}}\sum_{S\subset F}\mathds1_{n\in S}(-1)^{|F\setminus S|}\sigma^{S\cup H}_{L}\nu.
\end{multline*}
Changing summation variables and recognizing the definition~\eqref{eq:def-diff} of the difference operator, the conclusion follows.
\end{proof}

We now state and prove trace estimates, which constitute an upgraded version of Lemma~\ref{lem:trace-0}.
We shall repeatedly appeal to these estimates to control force terms at particle boundaries, which appear in our formulation~\eqref{eq:diff-st} of equations for corrector differences.

\begin{lem}[Trace estimates]\label{lem:trace}
Under Assumptions~\ref{H0} and~\ref{Gunif}, for all families $\Fc$ of finite subsets of $\N$, for all $H\subset \N$ and $n\in\N$ with $n\notin \bigcup_{F\in\Fc}F$, we have
\[\inf_{\kappa\in\R^d,\Theta\in\Md^\Skew}\int_{\partial I_{n,L}}\Big|\sum_{F\in\Fc}\del^F\psi_{L}^H-(\kappa+\Theta (x-x_{n,L}))\Big|^2~\,\lesssim\,~\int_{I_{n,L}}\Big|\sum_{F\in\Fc}\D(\del^F\psi_{L}^H)\Big|^2,\]
and
\[\inf_{c\in\R}\int_{\partial I_{n,L}}\Big|\sum_{F\in\Fc}\del^{F}\sigma_{L}^H-c\Id\Big|^2~\,\lesssim\,~\int_{I_{n,L}+\rho B}\Big|\sum_{F\in\Fc}\D(\del^F(\psi_{L}^H+Ex))\Big|^2.\qedhere\]
\end{lem}

\begin{proof}
We split the proof into three steps.
We set for abbreviation
\[\psi^{\Fc,H}_{L}\,:=\,\sum_{F\in\Fc}\del^{F}\psi_{L}^H,\qquad\Sigma^{\Fc,H}_{L}\,:=\,\sum_{F\in\Fc}\del^{F}\Sigma_{L}^H,\qquad\sigma_{L}^{\Fc,H}\,:=\,\sum_{F\in\Fc}\del^{F}\sigma_{L}^H.\]
We also use the short-hand notation $\bar\psi_L^H:=\psi_L^H+Ex$ and $\bar\psi^{\Fc,H}_{L}:=\sum_{F\in\Fc}\del^{F}\bar\psi_{L}^H$. This expression is equal to $\sum_{F\in\Fc}\delta^F\psi_L^H+Ex$ if $\varnothing\in\Fc$, and to $\sum_{F\in\Fc}\delta^F\psi_L^H$ otherwise.

\medskip
\step1 Proof of the first estimate on $\psi^{\Fc,H}_{L}$.\\
By the trace estimate 
\[\int_{\partial I_{n,L}}|\psi_{L}^{\Fc,H}-(\kappa+\Theta(x-x_{n,L}))|^2\,\lesssim\,\int_{I_{n,L}}\big|\langle\nabla\rangle^\frac12\big(\psi_{L}^{\Fc,H}-(\kappa+\Theta(x-x_{n,L}))\big)\big|^2,\]
combined with Poincaré's inequality, the conclusion follows.

\medskip
\step2 Proof of the second estimate on $\sigma^{\Fc,H}_{L}$ in the case $n\notin H$.\\
As $\sigma_L^{\Fc,H}=\sigma(\bar\psi_L^{\Fc,H},\Sigma_L^{\Fc,H})$, a trace estimate yields
\begin{equation}\label{eq:est-bndary}
\int_{\partial I_{n,L}}|\sigma_{L}^{\Fc,H}-c\Id|^2
\,\lesssim\,\int_{(I_{n,L}+\frac12\rho B)\setminus I_{n,L}}|\langle\nabla\rangle^\frac12\nabla\bar\psi^{\Fc,H}_{L}|^2
+|\langle\nabla\rangle^\frac12(\Sigma^{\Fc,H}_{L}-c)|^2.
\end{equation}
Given $n\notin H$, as the uniform separation assumption in~\ref{Gunif} ensures that no other particle intersects~\mbox{$I_{n,L}+\rho B$}, we note that $(\bar\psi^{\Fc,H}_{L},\Sigma^{\Fc,H}_{L})$ satisfies
\begin{equation}\label{eq:stokes-prereg}
-\triangle\bar\psi_{L}^{\Fc,H}+\nabla\Sigma^{\Fc,H}_{L}=0,\qquad\text{in $I_{n,L}+\rho B$}.
\end{equation}
By the local regularity theory for steady Stokes equations, e.g.~\cite[Theorem~IV.4.1]{Galdi}, we deduce for all $m\ge0$, for all constants $\kappa\in\R^d$ and $c\in\R$,
\begin{multline*}
\|\nabla\bar\psi_{L}^{\Fc,H}\|_{H^{m+1}(I_{n,L}+\frac12\rho B)}+\|\Sigma_{L}^{\Fc,H}-c\|_{H^{m+1}(I_{n,L}+\frac12\rho B)}\\
\,\lesssim\,\|\bar\psi_{L}^{\Fc,H}-\kappa\|_{H^{1}(I_{n,L}+\rho B)}+\|\Sigma_{L}^{\Fc,H}-c\|_{\Ld^2(I_{n,L}+\rho B)}.
\end{multline*}
Choosing $c:=\fint_{I_{n,L}+\rho B}\Sigma_{L}^{\Fc,H}$ and using a local pressure estimate for the steady Stokes equation, e.g.~\cite[Lemma~3.3]{DG-21b}, we find
\begin{equation*}
\|\Sigma_{L}^{\Fc,H}-c\|_{\Ld^2(I_{n,L}+\rho B)}
\,\lesssim\,\|\nabla\bar\psi_{L}^{\Fc,H}\|_{\Ld^2(I_{n,L}+\rho B)},
\end{equation*}
so that the above reduces to
\begin{equation*}
\|\nabla\bar\psi_{L}^{\Fc,H}\|_{H^{m+1}(I_{n,L}+\frac12\rho B)}+\|\Sigma_{L}^{\Fc,H}-c\|_{H^{m+1}(I_{n,L}+\frac12\rho B)}
\,\lesssim\,\|\bar\psi_{L}^{\Fc,H}-\kappa\|_{H^{1}(I_{n,L}+\rho B)}.
\end{equation*}
Further choosing $\kappa:=\fint_{I_{n,L}+\rho B}\bar\psi_{L}^{\Fc,H}$ and applying Poincaré's inequality, we conclude
\begin{equation}\label{eq:reg-inner}
\|\nabla\bar\psi_{L}^{\Fc,H}\|_{H^{m+1}(I_{n,L}+\frac12\rho B)}+\|\Sigma_{L}^{\Fc,H}-c\|_{H^{m+1}(I_{n,L}+\frac12\rho B)}
\,\lesssim\,\|\nabla\bar\psi_{L}^{\Fc,H}\|_{\Ld^2(I_{n,L}+\rho B)}.
\end{equation}
In particular, combining this with~\eqref{eq:est-bndary} and noting that the Cauchy stress tensor $\sigma_L^{\Fc,H}$ is unchanged if we add a rigid motion to $\bar\psi_L^{\Fc,H}$, the conclusion follows from Korn's inequality.

\medskip
\step3 Proof of the second estimate on $\sigma^{\Fc,H}_{L}$ in the case $n\in H$.\\
The starting point is again~\eqref{eq:est-bndary}.
Now, given $n\in H$, we note that $(\bar\psi^{\Fc,H}_{L},\Sigma^{\Fc,H}_{L})$ satisfies, instead of~\eqref{eq:stokes-prereg},
\begin{equation}\label{eq:stokes-prereg-re}
-\triangle\bar\psi_{L}^{\Fc,H}+\nabla\Sigma^{\Fc,H}_{L}=0,\qquad\text{in $(I_{n,L}+\rho B)\setminus I_{n,L}$},
\end{equation}
and $\bar\psi^{\Fc,H}_{L}$ is affine in $I_{n,L}$.
By the local regularity theory for the steady Stokes equation near a boundary, e.g.~\cite[Theorem~IV.5.1--5.3]{Galdi}, we obtain for all $m\ge0$, for all constants~$\kappa\in\R^d$ and~$c\in\R$,
\begin{multline*}
\|\nabla\bar\psi_{L}^{\Fc,H}\|_{H^{m+1}((I_{n,L}+\frac12\rho B)\setminus I_{n,L})}+\|\Sigma_{L}^{\Fc,H}-c\|_{H^{m+1}((I_{n,L}+\frac12\rho B)\setminus I_{n,L})}\\
\,\lesssim\,\|\bar\psi_{L}^{\Fc,H}|_{I_{n,L}}-\kappa\|_{H^{m+\frac 32}(\partial I_{n,L})}
+\|\bar\psi_{L}^{\Fc,H}-\kappa\|_{H^1((I_{n,L}+\rho B)\setminus I_{n,L})}\\
+\|\Sigma_{L}^{\Fc,H}-c\|_{\Ld^2((I_{n,L}+\rho B)\setminus I_{n,L})}.
\end{multline*}
Choosing $c:=\fint_{(I_{n,L}+\rho B)\setminus I_{n,L}}\Sigma_{L}^{\Fc,H}$ and using a local pressure estimate for the steady Stokes equation, e.g.~\cite[Lemma~3.3]{DG-21b}, we find
\begin{equation*}
\|\Sigma_{L}^{\Fc,H}-c\|_{\Ld^2((I_{n,L}+\rho B)\setminus I_{n,L})}
\,\lesssim\,\|\nabla\bar\psi_{L}^{\Fc,H}\|_{\Ld^2((I_{n,L}+\rho B)\setminus I_{n,L})},
\end{equation*}
so that the above reduces to
\begin{multline*}
\|\nabla\bar\psi_{L}^{\Fc,H}\|_{H^{m+1}((I_{n,L}+\frac12\rho B)\setminus I_{n,L})}+\|\Sigma_{L}^{\Fc,H}-c\|_{H^{m+1}((I_{n,L}+\frac12\rho B)\setminus I_{n,L})}\\
\,\lesssim\,\|\bar\psi_{L}^{\Fc,H}|_{I_{n,L}}-\kappa\|_{H^{m+\frac 32}(\partial I_{n,L})}+\|\bar\psi_{L}^{\Fc,H}-\kappa\|_{H^1((I_{n,L}+\rho B)\setminus I_{n,L})}.
\end{multline*}
As $\bar\psi_{L}^{\Fc,H}$ is affine in $I_{n,L}$, we have
\[\|\bar\psi_{L}^{\Fc,H}|_{I_{n,L}}-\kappa\|_{H^{m+\frac 32}(\partial I_{n,L})}\,\lesssim\,\|\bar\psi_{L}^{\Fc,H}-\kappa\|_{H^{m+2}(I_{n,L})}\,=\,\|\bar\psi_{L}^{\Fc,H}-\kappa\|_{H^{1}(I_{n,L})},\]
and the above then becomes
\begin{multline*}
\|\nabla\bar\psi_{L}^{\Fc,H}\|_{H^{m+1}((I_{n,L}+\frac12\rho B)\setminus I_{n,L})}+\|\Sigma_{L}^{\Fc,H}-c\|_{H^{m+1}((I_{n,L}+\frac12\rho B)\setminus I_{n,L})}\\
\,\lesssim\,\|\bar\psi_{L}^{\Fc,H}-\kappa\|_{H^1(I_{n,L}+\rho B)}.
\end{multline*}
Further choosing $\kappa:=\fint_{I_{n,L}+\rho B}\bar\psi_{L}^{\Fc,H}$ and applying Poincaré's inequality, we deduce
\begin{equation*}
\|\nabla\bar\psi_{L}^{\Fc,H}\|_{H^{m+1}((I_{n,L}+\frac12\rho B)\setminus I_{n,L})}+\|\Sigma_{L}^{\Fc,H}-c\|_{H^{m+1}((I_{n,L}+\frac12\rho B)\setminus I_{n,L})}
\,\lesssim\,\|\nabla\bar\psi_{L}^{\Fc,H}\|_{\Ld^2(I_{n,L}+\rho B)}.
\end{equation*}
In particular, combined with~\eqref{eq:est-bndary}, this yields the conclusion as in Step~2.
\end{proof}

\subsection{Cluster formulas}\label{sec:cluster-pr}
This section is devoted to the proof of Theorem~\ref{thm:expansion}.
We start by establishing the validity of expansion~\eqref{eq:preconcl-exp} with coefficients given by formula~\eqref{eq:form-clust2} and with the explicit remainder~\eqref{eq:form-clust-rem}.
The proof is similar to its counterpart for the conductivity problem in our previous work~\cite{DG-16a}.

\begin{lem}[Finite-volume cluster expansion]\label{lem:proof-prop1-1}
Under Assumptions~\ref{H0} and~\ref{Gunif}, finite-volume approximations of the effective viscosity can be expanded for all~$L$ and $k\ge1$ as
\begin{equation}\label{eq:preconcl-exp-RE}
\Bb_{L}\,=\,\Id+\sum_{j=1}^k\tfrac1{j!}\Bb_{L}^{j}+R_{L}^{k+1},
\end{equation}
where the coefficients $\{\Bb_{L}^{j}\}_j$ and the remainder $R_{L}^{k+1}$ are given by formulas~\eqref{eq:form-clust2} and~\eqref{eq:form-clust-rem}, respectively.
\end{lem}

\begin{proof}
Given $E\in\Md_0^\Sym$ with $|E|=1$, we recall that we drop the corresponding subscripts in the notation.
We split the proof into three steps.

\medskip
\step1 General strategy.\\
The starting point is formula~\eqref{eq:def-per-effvis} for the finite-volume approximation of the effective viscosity,
\[E:\Bb_{L}E\,=\,1+\expecM{\fint_{ Q_L}|\!\D(\psi_{L})|^2}.\]
The energy identity for the corrector equation~\eqref{eq:cor-per} takes the form
\begin{equation}\label{eq:indep0}
2\int_{ Q_L}|\!\D(\psi_{L})|^2\,=\,
\sum_{n}\int_{\partial I_{n,L}}E(x-x_{n,L}) \cdot\sigma_{L}\nu,
\end{equation}
and thus, further decomposing
$\sigma_L=\sigma_L^{\{n\}}+(\sigma_L-\sigma_L^{\{n\}})$, we obtain
\begin{multline}\label{eq:1step}
\expecM{2\int_{ Q_L}|\!\D(\psi_{L})|^2}
\,=\,\sum_{n}\expecM{\int_{\partial I_{n,L}}E(x-x_{n,L})\cdot\sigma_{L}^{\{n\}}\nu}\\
+\sum_{n}\expecM{\int_{\partial I_{n,L}}E(x-x_{n,L})\cdot(\sigma_{L}-\sigma_{L}^{\{n\}})\nu}.
\end{multline}
In addition, we shall prove below that for all $k\ge1$,
\begin{multline}\label{eq:magic}
\sum_{\sharp F=k}\sum_{n\in F}\expecM{\int_{\partial I_{n,L}}\del^{F\setminus\{n\}}\big(\psi_{L}^\varnothing+E(x-x_{n,L})\big)\cdot (\sigma_{L}-\sigma_{L}^{ F})\nu}\\
\,=\,\sum_{\sharp F=k+1}\sum_{n\in F}\,\expecM{\int_{\partial I_{n,L}}\del^{F\setminus\{n\}}\psi_{L}^\varnothing\cdot\sigma_{L}\nu}.
\end{multline}
We note that~\eqref{eq:indep0} already proves the claim~\eqref{eq:preconcl-exp-RE} for $k=0$.
Next, we proceed by induction:
if~\eqref{eq:preconcl-exp-RE} holds for some~$k\ge0$,
formulas~\eqref{eq:form-clust2} and~\eqref{eq:form-clust-rem} for $R_L^{k+1},\Bb_L^{k+1}$ allow to
decompose
\begin{multline*}
E:R_{L}^{k+1}E\,=\,\tfrac1{(k+1)!}E:\Bb_{L}^{k+1}E\\
+\tfrac12L^{-d}\sum_{\sharp F=k+1}\sum_{n\in F}\,\expecM{\int_{\partial I_{n,L}}\del^{F\setminus\{n\}}\big(\psi_{L}^\varnothing+E(x-x_{n,L})\big)\cdot(\sigma_{L}-\sigma_{L}^{ F})\nu}.
\end{multline*}
Inserting identity~\eqref{eq:magic}, noting that for $\sharp F=k+2$ there holds
\[\del^{F\setminus\{n\}}\psi_{L}^\varnothing=\del^{F\setminus\{n\}}\big(\psi_{L}^\varnothing+E(x-x_{n,L})\big),\]
and recognizing formula~\eqref{eq:form-clust-rem} for $R_L^{k+2}$, we deduce
\[R_{L}^{k+1}\,=\,\tfrac1{(k+1)!}\Bb_{L}^{k+1}+R_{L}^{k+2},\]
hence the claim~\eqref{eq:preconcl-exp-RE} follows with~$k$ replaced by $k+1$.
It remains to prove~\eqref{eq:magic}, which we do in the next two steps.

\medskip
\step2 Proof that for all $\sharp F=k\ge1$ and $G\subset F$,
\begin{equation}\label{eq:magic-1}
\sum_{n\in F\setminus G}\int_{\partial I_{n,L}}\big(\psi_{L}^{G}+E(x-x_{n,L})\big)\cdot(\sigma_L-\sigma_L^F)\nu
\,=\,\sum_{n\notin F}\int_{\partial I_{n,L}}(\psi_{L}^{ F}-\psi_{L}^{ G})\cdot\sigma_{L}\nu.
\end{equation}
On the one hand, testing the equation~\eqref{eq:reform} for $\psi_{L}^{ G}$ with the difference $\psi_{L}-\psi_{L}^{ F}$, and using the boundary conditions for $\psi_{L},\psi_{L}^{ F},\psi_{L}^{ G}$ on~$\partial I_{n,L}$ with $n\in  G\subset F$, we find
\begin{equation}\label{eq:pre-ident}
\int_{ Q_L}\nabla(\psi_{L}-\psi_{L}^{ F}):\nabla\psi_{L}^{ G}
\,=\,-\sum_{n\in  G}\int_{\partial I_{n,L}}(\psi_{L}-\psi_{L}^{ F})\cdot\sigma_{L}^{ G}\nu\,=\,0.
\end{equation}
On the other hand, equations~\eqref{eq:reform} for $\psi_{L}$ and $\psi_{L}^{ F}$ entail
\begin{equation*}
-\triangle(\psi_{L}-\psi_{L}^{ F})+\nabla\big(\Sigma_{L}\mathds1_{ Q_L\setminus\Ic_{L}}-\Sigma_{L}^{ F}\mathds1_{ Q_L\setminus\Ic_{L}^{F}}\big)\\
\,=\,-\sum_{n\notin F}\delta_{\partial I_{n,L}}\sigma_{L}\nu-\sum_{n\in  F}\delta_{\partial I_{n,L}}(\sigma_{L}-\sigma_{L}^{ F})\nu,
\end{equation*}
and thus, testing with $\psi_{L}^{ G}$ and using the boundary conditions for $\psi_{L},\psi_{L}^{ F},\psi_{L}^{ G}$ on $\partial I_{n,L}$ with $n\in G\subset F$,
\begin{eqnarray*}
\lefteqn{\int_{ Q_L}\nabla\psi_{L}^{ G}:\nabla(\psi_{L}-\psi_{L}^{ F})}\\
&=&-\sum_{n\notin F}\int_{\partial I_{n,L}}\psi_{L}^{ G}\cdot\sigma_{L}\nu-\sum_{n\in  F}\int_{\partial I_{n,L}}\psi_{L}^{ G}\cdot(\sigma_{L}-\sigma_{L}^{ F})\nu\\
&=&-\sum_{n\notin F}\int_{\partial I_{n,L}}\psi_{L}^{ G}\cdot\sigma_{L}\nu
-\sum_{n\in F\setminus G}\int_{\partial I_{n,L}}\psi_{L}^{ G}\cdot(\sigma_{L}-\sigma_{L}^{ F})\nu\\
&&\hspace{4cm}+\sum_{n\in  G}\int_{\partial I_{n,L}}E(x-x_{n,L})\cdot(\sigma_{L}-\sigma_{L}^{ F})\nu.
\end{eqnarray*}
Combined with~\eqref{eq:pre-ident}, this entails
\begin{equation*}
\sum_{n\in F\setminus G}\int_{\partial I_{n,L}}\psi_{L}^{ G}\cdot(\sigma_{L}-\sigma_{L}^{ F})\nu
\,=\,\sum_{n\in  G}\int_{\partial I_{n,L}}E(x-x_{n,L})\cdot(\sigma_{L}-\sigma_{L}^{ F})\nu-\sum_{n\notin F}\int_{\partial I_{n,L}}\psi_{L}^{ G}\cdot\sigma_L\nu,
\end{equation*}
or alternatively,
\begin{multline}\label{eq:magic-0}
\sum_{n\in F\setminus G}\int_{\partial I_{n,L}}\big(\psi_{L}^{ G}+E(x-x_{n,L})\big)\cdot(\sigma_{L}-\sigma_{L}^{ F})\nu\\
\,=\,\sum_{n\in  F}\int_{\partial I_{n,L}}E(x-x_{n,L})\cdot(\sigma_{L}-\sigma_{L}^{ F})\nu
-\sum_{n\notin F}\int_{\partial I_{n,L}}\psi_{L}^{ G}\cdot\sigma_{L}\nu.
\end{multline}
For $G=F$, the left-hand side vanishes, hence
\begin{equation*}
\sum_{n\in  F}\int_{\partial I_{n,L}}E(x-x_{n,L})\cdot(\sigma_{L}-\sigma_{L}^{ F})\nu
\,=\,\sum_{n\notin F}\int_{\partial I_{n,L}}\psi_{L}^{ F}\cdot\sigma_{L}\nu,
\end{equation*}
which allows us to reformulate~\eqref{eq:magic-0} into~\eqref{eq:magic-1}.

\medskip
\step4 Proof of~\eqref{eq:magic}.\\
Denote by  $T_{k,L}$ the left-hand side of~\eqref{eq:magic}.
By the definition~\eqref{eq:def-diff} of the difference operator, we have
\begin{equation*}
T_{k,L}\,=\,-\sum_{\sharp F=k}\sum_{n\in F}\sum_{G\subset F\setminus\{n\}}(-1)^{|F\setminus G|}\,
\expecM{\int_{\partial I_{n,L}}\big(\psi_{L}^{ G}+E(x-x_{n,L})\big)\cdot (\sigma_{L}-\sigma_{L}^{ F})\nu},
\end{equation*}
or alternatively, after changing summation variables,
\begin{equation*}
T_{k,L}\,=\,-\sum_{\sharp F=k}\sum_{G\subset F}(-1)^{|F\setminus G|}
\,\expecM{\sum_{n\in F\setminus G}\int_{\partial I_{n,L}}\big(\psi_{L}^{ G}+E(x-x_{n,L})\big)\cdot(\sigma_{L}-\sigma_{L}^{ F})\nu}.
\end{equation*}
We now appeal to~\eqref{eq:magic-1}, to the effect of
\begin{equation*}
T_{k,L}\,=\,-\sum_{\sharp F=k}\sum_{G\subset F}(-1)^{|F\setminus G|}\,\expecM{\sum_{n\notin F}\int_{\partial I_{n,L}}(\psi_{L}^{ F}-\psi_{L}^{ G})\cdot\sigma_{L}\nu}.
\end{equation*}
Using that $\sum_{G\subset F}(-1)^{|F\setminus G|}=0$ for $F\ne\varnothing$ and recalling the definition~\eqref{eq:def-diff} of the difference operator, this implies
\begin{eqnarray*}
T_{k,L}&=&\sum_{\sharp F=k}\sum_{G\subset F}(-1)^{|F\setminus G|} \,\expecM{\sum_{n\notin F}\int_{\partial I_{n,L}} \psi_{L}^{ G}\cdot\sigma_{L}\nu}\\
&=&\sum_{\sharp F=k}\,\expecM{\sum_{n\notin F}\int_{\partial I_{n,L}}\del^F\psi_{L}^\varnothing\cdot\sigma_{L}\nu},
\end{eqnarray*}
and the claim~\eqref{eq:magic} follows after changing summation variables.
\end{proof}

In the above result, we have naturally come up with the definition~\eqref{eq:form-clust2} of cluster coefficients~$\{\Bb_{L}^j\}_j$.
We now further establish the alternative formulas~\eqref{eq:form-clust} and~\eqref{eq:form-clust1}.
Note that~\eqref{eq:form-clust} coincides with the periodized version of the expected cluster formula~\eqref{e.formal-expansion}.

\begin{lem}[Equivalent cluster formulas]\label{lem:alter}
Under Assumptions~\ref{H0} and~\ref{Gunif}, for all~$L$ and $j\ge1$, the finite-volume cluster coefficient~$\Bb_{L}^j$ defined by formula~\eqref{eq:form-clust2} is equivalently given by~\eqref{eq:form-clust} and~\eqref{eq:form-clust1}.
\end{lem}

\begin{proof}
We split the proof into two steps.

\medskip
\step1 Equivalence of~\eqref{eq:form-clust1} and~\eqref{eq:form-clust2}.\\
It suffices to prove that for all finite $F\subset \N$,
\begin{equation}\label{eq:claim-ident-form12}
\sum_{n\in F}\int_{\partial I_{n,L}}\del^{F\setminus\{n\}}\big(\psi_L^\varnothing+E(x-x_{n,L})\big)\cdot\sigma_L^F\nu\,=\,\sum_{n\in F}\int_{\partial I_{n,L}}E(x-x_{n,L})\cdot\del^{F\setminus\{n\}}\sigma_L^{\{n\}}\nu.
\end{equation}
Decomposing $\del^{F\setminus\{n\}}\psi_L^\varnothing=\del^{F\setminus\{n\}}\psi_L^{\{n\}}-\del^{F}\psi_L^\varnothing$ for $n\in F$ and using the boundary conditions, we find
\[\sum_{n\in F}\int_{\partial I_{n,L}}\del^{F\setminus\{n\}}\big(\psi_L^\varnothing+E(x-x_n)\big)\cdot\sigma_L^{F}\nu\,=\,-\sum_{n\in F}\int_{\partial I_{n,L}}\del^{F}\psi_L^\varnothing\cdot\sigma_L^{F}\nu.\]
Testing the equation~\eqref{eq:reform} for $\psi_{L}^{F}$ with $\del^F\psi_L^\varnothing$, this becomes
\[\sum_{n\in F}\int_{\partial I_{n,L}}\del^{F\setminus\{n\}}\big(\psi_L^\varnothing+E(x-x_n)\big)\cdot\sigma_L^{F}\nu\,=\,\int_{ Q_L}\nabla\del^F\psi_L^\varnothing:\nabla\psi_L^{F}.\]
Now testing the equation~\eqref{eq:diff-st} for $\del^F\psi_L^\varnothing$ with $\psi_L^{F}$, and using the boundary conditions, we deduce
\begin{eqnarray*}
\sum_{n\in F}\int_{\partial I_{n,L}}\del^{F\setminus\{n\}}\big(\psi_L^\varnothing+E(x-x_n)\big)\cdot\sigma_L^{F}\nu&=&-\sum_{n\in F}\int_{\partial I_{n,L}}\psi_L^{F}\cdot\del^{F\setminus\{n\}}\sigma_L^{\{n\}}\nu\\
&=&\sum_{n\in F}\int_{\partial I_{n,L}}E(x-x_{n,L})\cdot\del^{F\setminus\{n\}}\sigma_L^{\{n\}}\nu,
\end{eqnarray*}
that is, \eqref{eq:claim-ident-form12}.

\medskip
\step2 Equivalence of~\eqref{eq:form-clust} and~\eqref{eq:form-clust1}.\\
It suffices to prove for all finite $F\subset\N$,
\begin{equation}\label{eq:form-equiv10}
\fint_{ Q_L}\del^F|\!\D(\psi_{L}^\varnothing)|^2
\,=\,\tfrac12L^{-d}\sum_{n\in F}\int_{\partial I_{n,L}}E(x-x_{n,L})\cdot\del^{F\setminus\{n\}}\sigma_{L}^{\{n\}}\nu.
\end{equation}
Recalling the definition~\eqref{eq:def-diff} of the difference operator, we can write
\begin{equation*}
\fint_{ Q_L}\del^F|\!\D(\psi_{L}^\varnothing)|^2=\sum_{G\subset F}(-1)^{|F\setminus G|}\fint_{ Q_L}|\!\D(\psi_{L}^G)|^2,
\end{equation*}
which entails, in view of the energy identity for $\psi_L^G$, cf.~\eqref{eq:indep0},
\begin{equation*}
\fint_{ Q_L}\del^F|\!\D(\psi_{L}^\varnothing)|^2
\,=\,\tfrac12L^{-d}\sum_{G\subset F}\sum_{n\in G}(-1)^{|F\setminus G|}\int_{\partial I_{n,L}}E(x-x_{n,L})\cdot\sigma_{L}^G\nu.
\end{equation*}
After changing summation variables and using again the definition~\eqref{eq:def-diff} of the difference operator, this yields the claim~\eqref{eq:form-equiv10}.
\end{proof}

To conclude the proof of Theorem~\ref{thm:expansion}, it remains to establish
the control~\eqref{e.DGV-control} of the remainder,
which is inspired by a recent work of G\'erard-Varet~\cite{GV-20}
and which we prove in the slightly refined form of~\eqref{e.DGV-control-bis} below.
This extends the argument of~\cite{GV-20} to all $k>2$.

\begin{lem}[Control of the remainder]\label{lem:DVG}
Under Assumptions~\ref{H0} and~\ref{Gunif}, for all $L$ and $j\ge1$, the remainder term defined in~\eqref{eq:form-clust-rem} can be estimated by
\begin{multline}\label{e.DGV-control-bis}
|R_{L}^{k+1}|\,\le\,\expecM{L^{-d}\sum_{n}\int_{I_{n,L}}\Big|\sum_{\sharp F=k\atop n\notin F}\D(\del^{F}\psi_{L}^\varnothing)\Big|^2}\\
+\sum_{j=1}^k\,\bigg|\,\expecM{L^{-d}\sum_{n}\int_{I_{n,L}}\Big(\sum_{\sharp F=k\atop n\notin F}\D(\del^{F}\psi_{L}^\varnothing)\Big):\Big(\sum_{\sharp F=j-1\atop n\notin F}\D(\del^{F}\hat\psi^{\{n\}}_{n,L})\Big)}\bigg|,
\end{multline}
where in view of~\eqref{eq:def-diff} we have defined, with a slight abuse of notation,
\begin{equation}\label{eq:def-delhatpsin}
\del^F\hat\psi^{\{n\}}_{n,L}\,:=\,\sum_{G\subset F}(-1)^{|F\setminus G|}\hat\psi_{n,L}^{G\cup\{n\}},
\end{equation}
where for all $H\subset \N$ and $n\in H$ we denote by $\hat \psi^H_{n,L}$ the solution of the following Neumann boundary value problem in the inclusion $I_{n,L}$ (unique up to a rigid motion),
\begin{equation}\label{e.DG+6}
\left\{\begin{array}{ll}
-\triangle\hat\psi_{n,L}^H+\nabla\hat\Sigma_{n,L}^H=0,&\text{in $I_{n,L}$},\\
\Div(\hat\psi_{n,L}^H)=0,&\text{in $I_{n,L}$},\\
\sigma(\hat\psi_{n,L}^H,\hat\Sigma_{n,L}^H)\nu=\sigma_{L}^H\nu,&\text{on $\partial I_{n,L}$}.
\end{array}\right.
\end{equation}
In particular, this yields the bound~\eqref{e.DGV-control}.
\end{lem}

\begin{proof}
We split the proof into two steps, first showing that~\eqref{e.DG+6} is well-posed, and then proving the bound~\eqref{e.DGV-control-bis}.

\medskip
\step1 Proof that the Neumann problem~\eqref{e.DG+6} is well-posed for all $H\subset\N$ and $n\in H$, and that the solution satisfies
\begin{equation}\label{eq:apest-tildepsin}
\int_{I_{n,L}}|\!\D(\hat\psi_{n,L}^H)|^2\,\lesssim\,\int_{I_{n,L}+\rho B}|\!\D(\psi_L^H)+E|^2.
\end{equation}
In addition, the proof yields similarly
\[\int_{I_{n,L}}\Big|\sum_{\sharp F=j-1\atop n\notin F}\D(\delta^F\hat\psi_{n,L}^{\{n\}})\Big|^2\,\lesssim\,\int_{I_{n,L}+\rho B}\Big|\sum_{\sharp F=j-1\atop n\notin F}\D\big(\delta^F(\psi_L^{\{n\}}+Ex)\big)\Big|^2.\]
This last estimate entails that the bound~\eqref{e.DGV-control} follows from~\eqref{e.DGV-control-bis}.

\medskip\noindent
We turn to the proof of~\eqref{eq:apest-tildepsin}.
The weak formulation of equation~\eqref{e.DG+6} yields for all~$\phi\in H^1(I_{n,L})^d$ with $\Div(\phi)=0$
\begin{equation}\label{eq:weak-form-tilde}
2\int_{I_{n,L}}\D(\phi):\D(\hat\psi_{n,L}^H)\,=\,\int_{\partial I_{n,L}}\phi\cdot\sigma_{L}^H\nu.
\end{equation}
Let us analyze the linear functional defining the right-hand side. Using the incompressibility of~$\phi$ in form of $\int_{\partial I_{n,L}}\phi\cdot\nu=0$, we can add any multiple of the identity matrix to~$\sigma_L^H$.
Further noting that the boundary conditions for $\psi_L^{H}$ on $\partial I_{n,L}$ with $n\in H$ allow to subtract a rigid motion from the test function $\phi$,
we are led to
\begin{multline*}
\Big|\int_{\partial I_{n,L}}\phi\cdot\sigma^{H}_L\nu\Big|\,\le\,\Big(\inf_{\kappa\in\R^d,\,\Theta\in\Md^\Skew}\int_{\partial I_{n,L}}|\phi-(\kappa+\Theta(x-x_{n,L}))|^2\Big)^\frac12\\
\times\Big(\inf_{c\in\R}\int_{\partial I_{n,L}}|\sigma^{H}_L-c\Id|^2\Big)^\frac12.
\end{multline*}
Appealing to the trace estimates of Lemma~\ref{lem:trace}, this becomes
\begin{equation}\label{eq:subtr-press-cond}
\Big|\int_{\partial I_{n,L}}\phi\cdot\sigma^H_L\nu\Big|\,\lesssim\,\Big(\int_{I_{n,L}}|\!\D(\phi)|^2\Big)^\frac12\Big(\int_{I_{n,L}+\rho B}|\!\D(\psi^H_L)+E|^2\Big)^\frac12.
\end{equation}
This proves that the right-hand side in the weak formulation~\eqref{eq:weak-form-tilde} is a continuous linear functional with respect to $\D(\phi)\in\Ld^2(I_{n,L})^{d\times d}$.
The Lax--Milgram theorem then ensures that equation~\eqref{e.DG+6} is well-posed in the sense that it admits a unique solution \mbox{$\D(\hat\psi_{n,L}^H)\in \Ld^2(I_{n,L})^{d\times d}$}, and the a priori bound~\eqref{eq:apest-tildepsin} follows.

\medskip
\step2 Proof of~\eqref{e.DGV-control-bis}.\\
Inserting the energy identity~\eqref{eq:indep0} and the formula~\eqref{eq:form-clust1} for the coefficients, the cluster expansion~\eqref{eq:preconcl-exp} yields the following formula for the remainder,
\begin{eqnarray*}
E:R_L^{k+1}E&=&E:\Bb_{L}E-1-\sum_{j=1}^k\tfrac{1}{j!}E:\Bb_{L}^{j}E\\
&=&\tfrac12L^{-d}\,\expecM{\sum_n\int_{\partial I_{n,L}}E(x-x_{n,L}) \cdot\sigma_{L}\nu}\\
&&-\sum_{j=1}^k\tfrac12L^{-d}\sum_{\sharp F=j}\sum_{n\in F}\expecM{\int_{\partial I_{n,L}}E(x-x_{n,L})\cdot\del^{F\setminus\{n\}}\sigma_{L}^{\{n\}}\nu},
\end{eqnarray*}
or equivalently, changing summation variables,
\begin{equation}\label{eq:rem-reform}
E:R_L^{k+1}E\,=\,\tfrac12L^{-d}\,\expecM{\sum_{n}\int_{\partial I_{n,L}}E(x-x_{n,L})
\cdot\Big(\sigma_{L}-\sum_{j=1}^k\sum_{\sharp F=j-1\atop n\notin F}\del^{F}\sigma_{L}^{\{n\}}\Big)\nu}.
\end{equation}
Consider the cluster expansion error
\begin{eqnarray}
\Psi^{k}_L&:=&\psi_L-\sum_{j=1}^k\sum_{\sharp F=j}\del^{F}\psi_L^\varnothing,\label{eq:def-PsiLp-diff}\\
\Xi^{k}_L&:=&\Sigma_L\mathds1_{ Q_L\setminus\Ic_L}-\sum_{j=1}^k\sum_{\sharp F=j}\del^{F}\big(\Sigma_L^\varnothing\mathds1_{ Q_L\setminus\Ic_L^\varnothing}\big),\nonumber
\end{eqnarray}
and note that in view of~\eqref{eq:diff-st} it satisfies the following equation in $ Q_L$,
\begin{equation}\label{eq:eqn-Psipk}
-\triangle\Psi^{k}_L+\nabla\Xi_L^{k}\,=\,-\sum_n\delta_{\partial I_{n,L}}\Big(\sigma_L-\sum_{j=1}^k\sum_{\sharp F=j-1\atop n\notin F}\del^{F}\sigma_L^{\{n\}}\Big)\nu.
\end{equation}
Testing this equation with $\psi_L$ and using the boundary conditions, the identity~\eqref{eq:rem-reform} for the remainder becomes
\begin{equation*}
E:R_L^{k+1}E\,=\,\expecM{\fint_{ Q_L}\D(\psi_L):\D(\Psi_L^{k})}.
\end{equation*}
Adding and subtracting $\sum_{j=1}^k\sum_{\sharp F=j}\del^F\psi_L^\varnothing$ to $\psi_L$, we deduce
by~\eqref{eq:def-PsiLp-diff},
\begin{equation*}
|E:R_L^{k+1}E|\,\le\,\expecM{\fint_{ Q_L}|\!\D(\Psi_L^{k})|^2}
+\sum_{j=1}^k\,\bigg|\,\expecM{\fint_{ Q_L}\D(\Psi_L^{k}):\sum_{\sharp F=j}\D(\del^F\psi_L^\varnothing)}\bigg|.
\end{equation*}
The conclusion~\eqref{e.DGV-control-bis} then follows from the  estimate
\begin{equation}\label{e.GV+3}
{\int_{ Q_L}|\!\D(\Psi_L^{k})|^2} \,\lesssim\, {\sum_{n}\int_{I_{n,L}}\Big|\sum_{\sharp F=k\atop n\notin F}\D(\del^{F}\psi_L^\varnothing)\Big|^2},
\end{equation}
and from the identity for all $1\le j\le k$
\begin{equation}\label{e.GV+5}
{\int_{ Q_L}\D(\Psi_L^{k}):\sum_{\sharp F=j}\D(\del^F\psi_L^\varnothing)}
\,=\,{\sum_{n}\int_{I_{n,L}}\Big(\sum_{\sharp F=k\atop n\notin F}\D(\del^{F}\psi_L^\varnothing)\Big):\Big(\sum_{\sharp F=j-1\atop n\notin F}\D(\del^{F}\hat\psi_{n,L}^{\{n\}})\bigg)},
\end{equation}
which we prove in the next two substeps, respectively.

\medskip
\substep{2.1} Proof of  \eqref{e.GV+3}.\\
In view of~\eqref{eq:eqn-Psipk}, the cluster expansion error $\Psi_L^{k}$ satisfies
\[-\triangle\Psi_L^k+\nabla\Xi_L^k=0,\qquad\Div(\Psi_L^k)=0,\qquad\text{in $ Q_L\setminus\Ic_L$},\]
which entails
\begin{eqnarray*}
\int_{ Q_L}|\!\D(\Psi_L^k)|^2&=&\sum_n\int_{I_{n,L}}|\!\D(\Psi_L^k)|^2+\int_{Q_L\setminus\Ic_L}|\!\D(\Psi_L^k)|^2\\
&=&\sum_n\int_{I_{n,L}}|\!\D(\Psi_L^k)|^2-\tfrac12\sum_n\int_{\partial I_{n,L}}\Psi_L^k\cdot\sigma(\Psi_L^k,\Xi_L^k)\nu.
\end{eqnarray*}
Hence, using the boundary conditions and the incompressibility constraint
to smuggle in arbitrary constants in the different factors, as in the proof of~\eqref{eq:subtr-press-cond},
and appealing to the trace estimates of Lemma~\ref{lem:trace-0}, we find
\begin{equation*}
\int_{ Q_L}|\!\D(\Psi_L^k)|^2\,\lesssim\,\sum_n\int_{I_{n,L}}|\!\D(\Psi_L^k)|^2+\sum_n\Big(\int_{I_{n,L}}|\!\D(\Psi_L^k)|^2\Big)^\frac12\Big(\int_{I_{n,L}^+}|\!\D(\Psi_L^k)|^2\Big)^\frac12,
\end{equation*}
from which we deduce by Young's inequality,\footnote{As argued in~\cite{GV-20}, this estimate~\eqref{eq:bnd-psi-L2-re} can alternatively be deduced from minimizing properties of Stokes equations for~$\Psi_L^k$ in~$Q_L\setminus\Ic_L$ with prescribed symmetric gradient in $\Ic_L$. We rather give a PDE argument that is more in line with the other arguments of this memoir.}
\begin{equation}\label{eq:bnd-psi-L2-re}
\int_{ Q_L}|\!\D(\Psi_L^{k})|^2\,\lesssim\, {\sum_{n}\int_{I_{n,L}}|\!\D(\Psi_L^{k})|^2}.
\end{equation}
Next, the definition of $\Psi_L^{k}$ and the rigidity constraint for $\psi_L$ in $I_{n,L}$ yield
\begin{equation}\label{eq:pre-telescop}
\D(\Psi_L^{k})\,=\,-E-\sum_{j=1}^k\sum_{\sharp F=j}\D(\del^{F}\psi_L^\varnothing)\qquad\text{in $I_{n,L}$}.
\end{equation}
Distinguishing between the cases $n\in F$ and $n\notin F$, and noting that for $n\in F$ we can decompose $\del^F\psi_L^\varnothing=\del^{F\setminus\{n\}}\psi_L^{\{n\}}-\del^{F\setminus\{n\}}\psi_L^\varnothing$, we find
\begin{equation*}
\sum_{\sharp F=j}\D(\del^{F}\psi_L^\varnothing)\,=\,\sum_{\sharp F=j\atop n\notin F}\D(\del^{F}\psi_L^\varnothing)+\sum_{\sharp F=j-1\atop n\notin F}\D(\del^{F}\psi_L^{\{n\}})-\sum_{\sharp F=j-1\atop n\notin F}\D(\del^{F}\psi_L^\varnothing),
\end{equation*}
and thus, in view of the rigidity constraint for $\del^F\psi_L^{\{n\}}$ in $I_{n,L}$,
\begin{equation*}
\sum_{\sharp F=j}\D(\del^{F}\psi_L^\varnothing)\,=\,-E\mathds1_{j=1}+\sum_{\sharp F=j\atop n\notin F}\D(\del^{F}\psi_L^\varnothing)-\sum_{\sharp F=j-1\atop n\notin F}\D(\del^{F}\psi_L^\varnothing)\qquad\text{in $I_{n,L}$}.
\end{equation*}
Inserting this into~\eqref{eq:pre-telescop} and recognizing a telescoping sum, we deduce for all $n$,
\begin{equation}\label{e.GV+4}
\D(\Psi_L^{k})\,=\,-\sum_{\sharp F=k\atop n\notin F}\D(\del^{F}\psi_L^\varnothing)\qquad\text{in $I_{n,L}$}.
\end{equation}
Combined with~\eqref{eq:bnd-psi-L2-re}, this yields the claim~\eqref{e.GV+3}.

\medskip
\substep{2.2} Proof of~\eqref{e.GV+5}.\\
Testing the equation~\eqref{eq:diff-st} for $\del^F\psi_L^\varnothing$ with $\Psi_L^{k}$, and changing summation variables, we find
\begin{eqnarray*}
2\int_{ Q_L}\D(\Psi_L^{k}):\sum_{\sharp F=j}\D(\del^F\psi_L^\varnothing)
&=&-\sum_{\sharp F=j}\sum_{n\in F}\int_{\partial I_{n,L}}\Psi_L^{k}\cdot\del^{F\setminus\{n\}}\sigma_L^{\{n\}}\nu\\
&=&-\sum_{n}\int_{\partial I_{n,L}}\Psi_L^{k}\cdot\sum_{\sharp F=j-1\atop n\notin F}\del^{F}\sigma_L^{\{n\}}\nu.
\end{eqnarray*}
In view of the equation~\eqref{eq:weak-form-tilde} for $\D(\del^F\hat\psi_{n,L}^{\{n\}})$, this can be rewritten as
\begin{equation*}
\int_{ Q_L}\D(\Psi_L^{k}):\sum_{\sharp F=j}\D(\del^F\psi_L^\varnothing)
\,=\,-\sum_{n}\int_{I_{n,L}}\D(\Psi_L^{k}):\sum_{\sharp F=j-1\atop n\notin F}\D(\del^{F}\hat\psi_{n,L}^{\{n\}}).
\end{equation*}
Combined with~\eqref{e.GV+4}, this yields the claim~\eqref{e.GV+5}.
\end{proof}

\subsection{Uniform $\ell^1-\ell^2$ energy estimates}\label{sec:unif-ell12-pr}
In order to prove uniform cluster estimates, cf.~Theorem~\ref{thm:bounds}(i), our main analytical achievement is the following hierarchy of interpolating $\ell^1-\ell^2$ energy estimates for corrector differences, inspired by our previous work~\cite{DG-16a} on the conductivity problem (which also considers `overlapping particles'; see~\cite{GGMN-21,DG-22} for refinements in that direction). More precisely, we consider the following quantities, for all~$H\subset\N$, all $L$, and~$k,j\ge0$,
\begin{eqnarray*}
S_{L}^{H}(k,j)&:=&\sum_{\sharp G=k}\fint_{ Q_L}\Big|\sum_{\sharp F=j\atop F\cap G=\varnothing}\D(\del^{F\cup G} \psi_{L}^{H})\Big|^2,\\
T_{L}^{H}(k,j)&:=& L^{-d}\sum_{\sharp G=k}\sum_{n\notin G\cup H}\int_{I_{n,L}+\rho B}\Big|\sum_{\sharp F=j\atop F\cap (G\cup \{n\})=\varnothing}\D(\del^{F\cup G} \psi_{L}^{H})\Big|^2,
\end{eqnarray*}
and we prove the following result. The novelty with respect to~\cite{DG-16a} is that we further identify the optimal dependence on the minimal distance $\ell=\ell(\Pc)\gtrsim1$, which appears to be surprisingly challenging and relies on a fine use of elliptic regularity via a duality argument.

\begin{theor}[Uniform $\ell^1-\ell^2$ energy estimates]\label{prop:apest-re}
Under Assumptions~\ref{H0} and~\ref{Gunif}, we have for all $H\subset\N$, all $L$, and~$k,j\ge0$,
\begin{align*}
S_{L}^{H}(k,j)&~\lesssim~\left\{\begin{array}{lll}
\ell^{-d}&:&k=j=0;\\
(C\ell^{-d})^{2(k+j)-1}&:&k,j\ge0,~k+j\ge1;
\end{array}\right.\\
T_{L}^{H}(k,j)&~\lesssim~\left\{\begin{array}{lll}
\ell^{-2d}&:&k=j=0;\\
(C\ell^{-d})^{2(k+j)+1}&:&k,j\ge0,~k+j\ge1.
\end{array}\right.\qedhere
\end{align*}
\end{theor}

The proof is split into two parts in the following two subsections:
to simplify the presentation, we first give a short proof in the spirit of~\cite{DG-16a} without keeping track of the $\ell$-dependence, and we then establish the estimates in their stated optimal form.

\subsubsection{Proof of Theorem~\ref{prop:apest-re} without $\ell$-dependence}\label{sec:proof-prop:apest-re-1}
This section is devoted to the proof that for all $H\subset\N$, all~$L$, and $k,j\ge0$,
\begin{equation}\label{eq:bnd-ap-S0}
S_{L}^{H}(k,j)+T_{L}^{H}(k,j)\,\lesssim\,C^{k+j}.
\end{equation}
For notational convenience, we set $S_{L}^H(k,j)=T_{L}^H(k,j)=0$ for $j<0$ or $k<0$. We split the proof into three steps.

\medskip
\step1 Reduction to $S_{L}^H$: for all $H\subset\N$ and $L,k,j$,
\begin{equation}\label{eq:T<S}
T_{L}^{H}(k,j)\,\lesssim\,S_{L}^H(k,j)+S_{L}^H(k,j-1),
\end{equation}
which entails in particular that it suffices to prove the bound~\eqref{eq:bnd-ap-S0} for $S_{L}^H$.

\medskip\noindent
First note that for all maps $f$ and all~$n\notin G$ we have
\begin{equation}\label{eq:decomp-FGn-not}
\sum_{\sharp F=j\atop F\cap G=\varnothing}f(F\cup G)=\sum_{\sharp F=j\atop F\cap(G\cup\{n\})=\varnothing}f(F\cup G)+\sum_{\sharp F=j-1 \atop F\cap (G\cup\{n\})=\varnothing}f(F\cup G\cup\{n\}).
\end{equation}
Using this identity to decompose $T_{L}^H(j,k)$ and changing summation variables, we find
\begin{multline*}
T_{L}^{H}(k,j)\,\lesssim\,L^{-d}\sum_{\sharp G=k}\sum_{n\notin G\cup H}\int_{I_{n,L}+\rho B}\Big|\sum_{\sharp F=j\atop F\cap G=\varnothing}\D(\del^{F\cup G} \psi_{L}^{H})\Big|^2\\
+L^{-d}\sum_{\sharp G=k+1}\sum_{n\in G\setminus H}\int_{I_{n,L}+\rho B}\Big|\sum_{\sharp F=j-1\atop F\cap G=\varnothing}\D(\del^{F\cup G} \psi_{L}^{H})\Big|^2,
\end{multline*}
and thus, using the disjointness of the fattened inclusions $\{I_{n,L}+\rho B\}_n$ and recognizing the definition of $S_{L}^H$, the claim~\eqref{eq:T<S} follows.

\medskip
\step2 Energy estimate for correctors: for all $H\subset\N$,
\begin{equation}\label{eq:S00}
S_{L}^H(0,0)\,\lesssim\,1.
\end{equation}
As in~\eqref{eq:indep0}, the energy identity for the corrector equation~\eqref{eq:reform} for $\psi_{L}^H$ takes the form
\begin{equation}\label{eq:energy-ident}
2\int_{ Q_L}|\!\D(\psi_{L}^H)|^2\,=\,\sum_{n\in H}\int_{\partial I_{n,L}}E(x-x_{n,L})\cdot\sigma_{L}^H\nu.
\end{equation}
Using the incompressibility constraint $\Tr (E)=0$ to add an arbitrary constant to the pressure in $\sigma_{L}^H$, as in the proof of~\eqref{eq:subtr-press-cond}, and then appealing to the trace estimates of Lemma~\ref{lem:trace-0}(ii), we obtain
\begin{equation*}
\int_{ Q_L}|\!\D(\psi_{L}^H)|^2\,\lesssim\,\sum_{n\in H}\Big(\int_{I_{n,L}+\rho B}|\!\D(\psi_{L}^H)+E|^2\Big)^\frac12.
\end{equation*}
Since the fattened inclusions $\{I_{n,L}+\rho B\}_n$ are disjoint, the Cauchy--Schwarz inequality then yields, recalling the choice of the periodization~\eqref{e.set-perio},
\begin{equation}\label{eq:det-bnd-phiFk00}
\int_{ Q_L}|\!\D(\psi_{L}^H)|^2\,\lesssim\,\sharp\{n\in H:x_{n}\in Q_L\}.
\end{equation}
As the right-hand side is bounded by $CL^d$, the claim~\eqref{eq:S00} follows.
For future reference, we also note that this bound entails, when taking the expectation,
\begin{equation}\label{eq:bnd-psi-L2}
\expecM{\fint_{ Q_L}|\!\D(\psi_{L})|^2}\,\lesssim\,\lambda(\Pc).
\end{equation}

\medskip
\step3 Key recurrence relation: for all $H\subset\N$ and $k,j\ge0$,
\begin{multline}\label{eq:recurr-S}
S_{L}^H(k,j)\,\lesssim\,\mathds1_{k+j\le1}
+S_{L}^H(k+1,j-1)\\
+S_{L}^H(k,j-1)+S_{L}^H(k-1,j)
+S_{L}^H(k,j-2)+S_{L}^H(k-1,j-1),
\end{multline}
which then leads to the conclusion~\eqref{eq:bnd-ap-S0} by a direct double induction argument.

\medskip\noindent
Let a finite subset $G\subset\N$ be momentarily fixed.
In view of~\eqref{eq:diff-st},
the following equation holds in $ Q_L$, for any $F\subset\N$ with $F\cap G=\varnothing$,
\begin{multline*}
-\triangle\del^{F\cup G}\psi_{L}^H+\nabla\del^{F\cup G}\big(\Sigma_{L}^H\mathds1_{ Q_L\setminus\Ic_{L}^{H}}\big)
\,=\,-\sum_{n\in H}\delta_{\partial I_{n,L}}\del^{F\cup G}\sigma^{H}_{L}\nu\\
-\sum_{n\in F \setminus H}\delta_{\partial I_{n,L}}\del^{(F\setminus\{n\})\cup G}\sigma^{H\cup\{n\}}_{L}\nu
-\sum_{n\in G \setminus H}\delta_{\partial I_{n,L}}\del^{F\cup(G\setminus\{n\})}\sigma^{H\cup\{n\}}_{L}\nu.
\end{multline*}
Hence, after summing over $F$ and changing summation variables,
\begin{multline*}
-\triangle\bigg(\sum_{\sharp F=j\atop F\cap G=\varnothing}\del^{F\cup G}\psi_{L}^H\bigg)+\nabla\bigg(\sum_{\sharp F=j\atop F\cap G=\varnothing}\del^{F\cup G}(\Sigma_{L}^H\mathds1_{ Q_L\setminus\Ic_{L}^{H}})\bigg)\\
\,=\,-\sum_{n\in H}\delta_{\partial I_{n,L}}\bigg(\sum_{\sharp F=j\atop F\cap G=\varnothing}\del^{F\cup G}\sigma^{H}_{L}\nu\bigg)
-\sum_{n\notin G\cup H}\delta_{\partial I_{n,L}}\bigg(\sum_{\sharp F=j-1\atop F\cap (G\cup\{n\})=\varnothing}\del^{F\cup G}\sigma^{H\cup\{n\}}_{L}\nu\bigg)\\
-\sum_{n\in G \setminus H}\delta_{\partial I_{n,L}}\bigg(\sum_{\sharp F=j\atop F\cap G=\varnothing}\del^{F\cup(G\setminus\{n\})}\sigma^{H\cup\{n\}}_{L}\nu\bigg).
\end{multline*}
Testing this equation with the solution $\sum_{\sharp F=j: F\cap G=\varnothing}\del^{F\cup G}\psi_{L}^H$ itself, we obtain the energy identity
\begin{equation}\label{eq:ener-ident00}
2\int_{ Q_L}\Big|\sum_{\sharp F=j\atop F\cap G=\varnothing}\D(\del^{F\cup G}\psi_{L}^H)\Big|^2
\,=\,A_{L}^{1}(G,j)+A_{L}^{2}(G,j)+A_{L}^{3}(G,j),
\end{equation}
in terms of
\begin{eqnarray}\label{eq:def-A123}
A_{L}^{1}(G,j)&:=&
-\sum_{n\in H}\int_{\partial I_{n,L}}\bigg(\sum_{\sharp F=j\atop F\cap G=\varnothing}\del^{F\cup G}\psi_{L}^H\bigg)\cdot\bigg(\sum_{\sharp F=j\atop F\cap G=\varnothing}\del^{F\cup G}\sigma^{H}_{L}\nu\bigg),\\
A_{L}^{2}(G,j)&:=&
-\sum_{n \notin G\cup H}\int_{\partial I_{n,L}}\bigg(\sum_{\sharp F=j\atop F\cap G=\varnothing}\del^{F\cup G}\psi_{L}^H\bigg)\cdot\bigg(\sum_{\sharp F=j-1\atop F\cap (G\cup\{n\})=\varnothing}\del^{F\cup G}\sigma^{H\cup\{n\}}_{L}\nu\bigg),\nonumber\\
A_{L}^{3}(G,j)&:=&
-\sum_{n\in G \setminus H}\int_{\partial I_{n,L}}\bigg(\sum_{\sharp F=j\atop F\cap G=\varnothing}\del^{F\cup G}\psi_{L}^H\bigg)\cdot\bigg(\sum_{\sharp F=j\atop F\cap G=\varnothing}\del^{F\cup(G\setminus\{n\})}\sigma^{H\cup\{n\}}_{L}\nu\bigg).\nonumber
\end{eqnarray}
We analyze these three contributions separately and we start with the first one.
In view of the boundary conditions for $\del^{F\cup G}\psi_{L}^H$ on $\partial I_{n,L}$ with $n\in H$, we can rewrite
\begin{eqnarray*}
A_{L}^{1}(G,j)
&=&\sum_{n\in H}\int_{\partial I_{n,L}}\bigg(\sum_{\sharp F=j\atop F\cap G=\varnothing}\del^{F\cup G}(E(x-x_{n,L}))\bigg)\cdot\bigg(\sum_{\sharp F=j\atop F\cap G=\varnothing}\del^{F\cup G}\sigma^{H}_{L}\nu\bigg)\\
&=&\mathds1_{G=\varnothing,j=0}\sum_{n\in H}\int_{\partial I_{n,L}}E(x-x_{n,L})\cdot\sigma^{H}_{L}\nu.
\end{eqnarray*}
Summing over $G\subset\N$ with $\sharp G=k$, and using the energy identity~\eqref{eq:energy-ident}, we deduce
\begin{equation}\label{eq:bnd-T10}
L^{-d}\sum_{\sharp G=k}A_{L}^{1}(G,j)
\,=\,\mathds1_{k=j=0}\,S_{L}^H(0,0).
\end{equation}
We turn to the second term $A_{L}^2$ in~\eqref{eq:ener-ident00}.
Using the boundary conditions and the incompressibility constraints to smuggle in arbitrary constants in the different factors, as in the proof of~\eqref{eq:subtr-press-cond},
and then appealing to the trace estimates of Lemma~\ref{lem:trace}, we find
\begin{multline}\label{eq:1stbnd-A2}
|A_{L}^{2}(G,j)|
\,\lesssim\,\sum_{n\notin G\cup H}\bigg(\int_{I_{n,L}}\Big|\sum_{\sharp F=j\atop F\cap G=\varnothing}\D(\del^{F\cup G}\psi_{L}^H)\Big|^2\bigg)^\frac12\\
\times\bigg(\int_{I_{n,L}+\rho B}
\Big|\sum_{\sharp F=j-1\atop F\cap (G\cup\{n\})=\varnothing}\!\!\!\!\!\D(\del^{F\cup G}(\psi^{H\cup\{n\}}_{L}+Ex))\Big|^2
\bigg)^\frac12.
\end{multline}
Decomposing the second factor via the following identity, for all $n\notin F\cup G\cup H$ and~$F\cap G=\varnothing$,
\[\delta^{F\cup G}(\psi_{L}^{H\cup\{n\}}+Ex)\,=\,\mathds1_{G=F=\varnothing}Ex+\delta^{F\cup G}\psi_{L}^{H}+\delta^{F\cup G\cup\{n\}}\psi_{L}^{H},\]
summing over \mbox{$G\subset\N$} with~\mbox{$\sharp G=k$}, using the Cauchy--Schwarz inequality, and using the disjointness of the fattened inclusions~\mbox{$\{I_{n,L}+\rho B\}_n$}, we get
\begin{equation}\label{eq:bnd-T20}
L^{-d}\sum_{\sharp G=k}|A_{L}^{2}(G,j)|\,\lesssim\,\big(S_{L}^H(k,j)\big)^\frac12\Big(\mathds1_{k=0,j=1}+T_{L}^H(k,j-1)+S_{L}^H(k+1,j-1)\Big)^\frac12.
\end{equation}
We turn to the third contribution $A_{L}^3$ in~\eqref{eq:ener-ident00}.
Decomposing for~$n\in G\setminus H$ and $F\cap G=\varnothing$,
\[\del^{F\cup G}\psi_{L}^H=\del^{F\cup(G\setminus\{n\})}\psi_{L}^{H\cup\{n\}}-\del^{F\cup(G\setminus\{n\})}\psi_{L}^H,\]
and using the boundary conditions, we can rewrite
\begin{multline*}
A_{L,\ell}^{3}(G,j)
\,=\,\mathds1_{\sharp G=1,j=0}\sum_{n\in G \setminus H}\int_{\partial I_{n,L}}E(x-x_{n,L})\cdot \sigma^{H\cup\{n\}}_{L}\nu\\
+\sum_{n\in G\setminus H}\int_{\partial I_{n,L}}\bigg(\sum_{\sharp F=j\atop F\cap G=\varnothing}\del^{F\cup (G\setminus\{n\})}\psi_{L}^{H}\bigg)
\cdot\bigg(\sum_{\sharp F=j\atop F\cap G=\varnothing}\del^{F\cup(G\setminus\{n\})}\sigma^{H\cup\{n\}}_{L}\nu\bigg).
\end{multline*}
Using the boundary conditions and the incompressibility constraints to smuggle in arbitrary constants in the different factors, as in the proof of~\eqref{eq:subtr-press-cond},
and then appealing to the trace estimates of Lemma~\ref{lem:trace},
we find
\begin{multline}\label{eq:start-A3est}
|A_{L}^{3}(G,j)|
\,\lesssim\,\mathds1_{\sharp G=1,j=0}\sum_{n\in G \setminus H}\bigg(\int_{I_{n,L}+\rho B}|\!\D(\psi_{L}^{H\cup\{n\}})+E|^2\bigg)^\frac12\\
+\sum_{n\in G\setminus H}\bigg(\int_{I_{n,L}}\Big|\sum_{\sharp F=j\atop F\cap G=\varnothing}\D(\del^{F\cup (G\setminus\{n\})}\psi_{L}^{H})\Big|^2\bigg)^\frac12\\
\times\bigg(\int_{I_{n,L}+\rho B}\Big|\sum_{\sharp F=j\atop F\cap G=\varnothing}\D(\del^{F\cup(G\setminus\{n\})}(\psi^{H\cup\{n\}}_{L}+Ex))\Big|^2\bigg)^\frac12.
\end{multline}
Decomposing the first right-hand side term and the last factor of the second term via the following identities, for all $n\in G\setminus H$ and $F\cap G=\varnothing$,
\begin{gather}
\psi_{L}^{H\cup\{n\}}=\psi_{L}^{H}+\del^{\{n\}}\psi_{L}^{H},\label{eq:decomp-psimult}\\
\del^{F\cup(G\setminus\{n\})}(\psi_{L}^{H\cup\{n\}}+Ex)\,=\,\mathds1_{\sharp G=1,\sharp F=0}Ex+\del^{F\cup(G\setminus\{n\})}\psi_{L}^{H}+\del^{F\cup(G\setminus\{n\})}\del^{\{n\}}\psi_{L}^{H},\nonumber
\end{gather}
summing over $G\subset\N$ with $\sharp G=k$, and using the Cauchy--Schwarz inequality and the disjointness of the fattened inclusions $\{I_{n,L}+\rho B\}_n$, this becomes
\begin{multline}\label{eq:bnd-T30}
L^{-d}\sum_{\sharp G=k}|A_{L}^{3}(G,j)|
\,\lesssim\,\mathds1_{k=1,j=0}\Big(1+S_{L}^H(0,0)+S_{L}^H(1,0)\Big)^\frac12\\
+\big(T_{L}^H(k-1,j)\big)^\frac12
\Big(\mathds1_{k=1,j=0}+T_{L}^H(k-1,j)+S_{L}^{H}(k,j)\Big)^\frac12.
\end{multline}
Inserting this into~\eqref{eq:ener-ident00}, together with~\eqref{eq:bnd-T10} and~\eqref{eq:bnd-T20}, we obtain
\begin{multline*}
S_{L}^H(k,j)\,\lesssim\,\mathds1_{k=0,j=0}\,S_{L}^H(0,0)+\mathds1_{k=1,j=0}\big(1+S_{L}^H(0,0)+S_{L}^H(1,0)\big)^\frac12\\
+\big(S_{L}^H(k,j)\big)^\frac12\Big(\mathds1_{k=0,j=1}+T_{L}^H(k,j-1)+S_{L}^H(k+1,j-1)\Big)^\frac12\\
+\big(T_{L}^H(k-1,j)\big)^\frac12
\Big(\mathds1_{k=1,j=0}+T_{L}^H(k-1,j)+S_{L}^{H}(k,j)\Big)^\frac12.
\end{multline*}
Using Young's inequality to absorb the occurrences of $S_{L}^H(k,j)$ in the right-hand side into the left-hand side, we are led to
\begin{multline*}
S_{L}^H(k,j)\,\lesssim\,\mathds1_{k=0,j=0}\,S_{L}^H(0,0)+\mathds1_{k=0,j=1}+\mathds1_{k=1,j=0}\big(1+S_{L}^H(0,0)\big)^\frac12\\
+S_{L}^H(k+1,j-1)
+T_{L}^H(k,j-1)
+T_{L}^H(k-1,j),
\end{multline*}
and the claim~\eqref{eq:recurr-S} now follows in combination with~\eqref{eq:T<S} and~\eqref{eq:S00}.\qed

\subsubsection{Proof of Theorem~\ref{prop:apest-re} with optimal $\ell$-dependence}\label{sec:proof-prop:apest-re-2}
It remains to refine the proof of the previous section to capture the optimal dependence on the minimal distance~$\ell=\ell(\Pc)\gtrsim1$. The proof involves a new intricate induction argument that combines both $S_{L}^H$ and $T_{L}^H$,
and the optimal scaling is then captured by a suitable application of elliptic regularity via a duality argument.
By the result of the previous section, we may assume $\ell\gg1$, in which case the uniform separation assumption in~\ref{Gunif} holds in the stronger form of
\begin{equation}\label{eq:ell-separation+}
\tfrac12\inf_{n\ne m}\dist(I_{n,L},I_{m,L})\,\ge\,\tfrac12\ell-1\,\ge\,\tfrac14\ell\,\ge\,\rho,
\end{equation}
and the definition~\eqref{e.set-perio} of the periodization further ensures
\[\inf_n\dist(I_{n,L},\partial Q_L)\,\ge\,\ell-1\,\ge\,\tfrac12\ell\,\ge\,\rho.\]
We split the proof into four steps.

\medskip
\step1 Energy estimate for correctors: for all $H\subset\N$,
\begin{gather}
S_{L}^H(0,0)\,=\,\fint_{ Q_L}|\!\D(\psi_{L}^H)|^2\,\lesssim\,\ell^{-d},\label{eq:st1-nabphiH}\\
T_{L}^H(0,0)\,=\,L^{-d}\sum_{n\notin H}\int_{I_{n,L}+\rho B}|\!\D(\psi_{L}^H)|^2\,\lesssim\,\ell^{-2d}.\label{eq:st1-nabphiH+}
\end{gather}
By the $\ell$-separation property~\eqref{eq:ell-separation+}, the number of points of the process $\Pc_L$ in $ Q_L$ is bounded by $C(L/\ell)^d$, so that the first estimate~\eqref{eq:st1-nabphiH} follows from~\eqref{eq:det-bnd-phiFk00}. It remains to prove~\eqref{eq:st1-nabphiH+}.
For that purpose, first note that for $n\notin H$ the $\ell$-separation property~\eqref{eq:ell-separation+} entails that the following free steady Stokes equations hold in $I_{n,L}+\frac14\ell B\subset Q_L\setminus\Ic_L^H$,
\begin{equation}\label{eq:psiLH-full}
-\triangle \psi_{L}^H+\nabla\Sigma_{L}^H=0,\qquad\Div(\psi_{L}^H)=0,\qquad\text{in $I_{n,L}+\tfrac14\ell B$}.
\end{equation}
Elliptic regularity in form of Lemma~\ref{lem:regularity} then yields
\begin{equation}\label{eq:stand-reg-ut}
\int_{I_{n,L}+\rho B}|\!\D(\psi_{L}^H)|^2\,\lesssim\,\ell^{-d}\int_{I_{n,L}+\frac14\ell B}|\!\D(\psi_{L}^H)|^2.
\end{equation}
Summing this over $n\notin H$ and using the $\ell$-separation property~\eqref{eq:ell-separation+} in form of the disjointness of the fattened inclusions $\{I_{n,L}+\frac14\ell B\}_n$, we deduce
\begin{equation*}
\sum_{n\notin H}\int_{I_{n,L}+\rho B}|\!\D(\psi_{L}^H)|^2\,\lesssim\,\ell^{-d}\int_{ Q_L}|\!\D(\psi_{L}^H)|^2,
\end{equation*}
and the claim~\eqref{eq:st1-nabphiH+} now follows from~\eqref{eq:st1-nabphiH}.

\medskip

\step2 Recurrence relation for $S_{L}^H$: for all $H\subset\N$ and $k,j\ge0$,
\begin{multline}\label{induct-details}
S_{L}^H(k,j)\,\lesssim\,\mathds1_{k+j\le1}\ell^{-d}
+S_{L}^H(k+1,j-1)\\
+T_{L}^H(k,j)
+T_{L}^H(k,j-1)
+T_{L}^H(k-1,j).
\end{multline}
This provides a refined version of the recurrence relation~\eqref{eq:recurr-S}, which can indeed be recovered by appealing to~\eqref{eq:T<S} to bound $T_{L}^H$ in terms of $S_{L}^H$.
The present refined version will be combined with a recurrence relation for~$T_{L}^H$ in the next step.

\medskip\noindent
Let $G\subset\N$ be momentarily fixed.
As in the proof of~\eqref{eq:recurr-S}, the starting point is identity~\eqref{eq:ener-ident00}, that is,
\begin{equation}\label{eq:ener-ident00+}
2\int_{ Q_L}\Big|\sum_{\sharp F=j\atop F\cap G=\varnothing}\D(\del^{F\cup G}\psi_{L}^H)\Big|^2
\,=\,A_{L}^{1}(G,j)+A_{L}^{2}(G,j)+A_{L}^{3}(G,j),
\end{equation}
where we recall that $A_{L}^1,A_{L}^2,A_{L}^3$ are defined in~\eqref{eq:def-A123}.
We analyze these contributions separately. The first one satisfies~\eqref{eq:bnd-T10}, and thus, combined with the energy estimate~\eqref{eq:st1-nabphiH},
\begin{equation}\label{eq:bnd-T10+}
L^{-d}\sum_{\sharp G=k}A_{L}^{1}(G,j)
\,=\,\mathds1_{k=j=0}\,S_{L}^H(0,0)\,\lesssim\,\mathds1_{k=j=0}\ell^{-d}.
\end{equation}
It remains to prove refined versions of~\eqref{eq:bnd-T20} and~\eqref{eq:bnd-T30} for $A^2_{L}$ and $A^3_{L}$,
and we start with the contribution of $A_{L}^2$.
The starting point is the trace estimate~\eqref{eq:1stbnd-A2} used in the proof of~\eqref{eq:bnd-T20}, that is,
\begin{multline*}
|A_{L}^{2}(G,j)|
\,\lesssim\,\sum_{n\notin G\cup H}\bigg(\int_{I_{n,L}}\Big|\sum_{\sharp F=j\atop F\cap G=\varnothing}\D(\del^{F\cup G}\psi_{L}^H)\Big|^2\bigg)^\frac12\\
\times\bigg(\int_{I_{n,L}+\rho B}
\Big|\sum_{\sharp F=j-1\atop F\cap (G\cup\{n\})=\varnothing}\!\!\!\!\!\D(\del^{F\cup G}(\psi^{H\cup\{n\}}_{L}+Ex))\Big|^2
\bigg)^\frac12,
\end{multline*}
which we shall now analyze more carefully. Using identity~\eqref{eq:decomp-FGn-not} to decompose the first factor,
and decomposing the second factor via the following identity, for all $n\notin F\cup G\cup H$ and $F\cap G=\varnothing$,
\[\delta^{F\cup G}(\psi_{L}^{H\cup\{n\}}+Ex)\,=\,\mathds1_{G=F=\varnothing}Ex+\delta^{F\cup G}\psi_{L}^{H}+\delta^{F\cup G\cup\{n\}}\psi_{L}^{H},\]
we find
\begin{multline*}
|A_{L}^{2}(G,j)|
\,\lesssim\,\sum_{n\notin G\cup H}\bigg(\int_{I_{n,L}}\underbrace{\Big|\!\!\sum_{\sharp F=j\atop F\cap (G\cup\{n\})=\varnothing}\!\!\!\!\!\D(\del^{F\cup G}\psi_{L}^H)\Big|^2}_{\displaystyle \clubsuit}
+\underbrace{\Big|\!\sum_{\sharp F=j-1\atop F\cap (G\cup\{n\})=\varnothing}\!\!\!\!\!\D(\del^{F\cup G\cup\{n\}}\psi_{L}^H)\Big|^2}_{\displaystyle \diamondsuit}\bigg)^\frac12\\
\times\bigg(\mathds1_{\sharp G=0,j=1}+\int_{I_{n,L}+\rho B}\underbrace{\Big|\!\sum_{\sharp F=j-1\atop F\cap (G\cup\{n\})=\varnothing}\!\!\!\!\!\D(\del^{F\cup G}\psi^{H}_{L})\Big|^2}_{\displaystyle \spadesuit}
+\underbrace{\Big|\!\sum_{\sharp F=j-1\atop F\cap (G\cup\{n\})=\varnothing}\!\!\!\!\!\D(\del^{F\cup G\cup\{n\}}\psi^{H}_{L})\Big|^2}_{\displaystyle \diamondsuit}
\bigg)^\frac12.
\end{multline*}
Summing over $G\subset\N$ with $\sharp G=k$, using Young's inequality, using the separation property in form of the disjointness of the fattened inclusions~{$\{I_{n,L}+\rho B\}_n$}, using that the number of points of the process~$\Pc_L$ in $Q_L$ is bounded by $C(L/\ell)^d$,
and reorganizing the terms, we conclude
\begin{equation}\label{eq:bnd-T20+}
L^{-d}\sum_{\sharp G=k}|A_{L}^{2}(G,j)|
\,\lesssim\,
\ell^{-d}\mathds1_{k=0,j=1}
+S_{L}^H(k+1,j-1)
+T_{L}^H(k,j)
+T_{L}^H(k,j-1),
\end{equation}
where the last three right-hand side terms come from $\diamondsuit$, $\clubsuit$, $\spadesuit$, respectively.

\medskip\noindent
We turn to the contribution of $A_{L}^3$. The starting point is the trace estimate~\eqref{eq:start-A3est} used in the proof of~\eqref{eq:bnd-T30}.
Further using the decomposition~\eqref{eq:decomp-psimult},
this estimate becomes
\begin{multline}\label{eq:bnd-T3-pre0}
|A_{L}^{3}(G,j)|
\,\lesssim\,\mathds1_{\sharp G=1,j=0}\sum_{n\in G \setminus H}\bigg(1+\int_{I_{n,L}+\rho B}|\!\D(\psi_{L}^{H})|^2+|\!\D(\del^{\{n\}}\psi_{L}^{H})|^2\bigg)^\frac12\\
+\sum_{n\in G\setminus H}\bigg(\int_{I_{n,L}}\Big|\sum_{\sharp F=j\atop F\cap G=\varnothing}\D(\del^{F\cup (G\setminus\{n\})}\psi_{L}^{H})\Big|^2\bigg)^\frac12\\
\times\bigg(\int_{I_{n,L}+\rho B}\Big|\sum_{\sharp F=j\atop F\cap G=\varnothing}\D(\del^{F\cup(G\setminus\{n\})}\psi^{H}_{L})\Big|^2
+\Big|\sum_{\sharp F=j\atop F\cap G=\varnothing}\D(\del^{F\cup G}\psi^{H}_{L})\Big|^2\bigg)^\frac12.
\end{multline}
Summing the first right-hand side term over $G\subset\N$ with $\sharp G=1$, using the Cauchy--Schwarz inequality, recalling that the number of points of the process~$\Pc_L$ in~$ Q_L$ is bounded by $C(L/\ell)^d$, and appealing to the energy estimate~\eqref{eq:st1-nabphiH+}, we find
\begin{eqnarray*}
\lefteqn{\sum_{n\notin H}\bigg(1+\int_{I_{n,L}+\rho B}|\!\D(\psi_{L}^{H})|^2+|\!\D(\del^{\{n\}}\psi_{L}^{H})|^2\bigg)^\frac12}\\
&\lesssim&L^\frac d2\ell^{-\frac d2}\bigg(L^d\ell^{-d}+\sum_{n\notin H}\int_{I_{n,L}+\rho B}|\!\D(\psi_{L}^{H})|^2+\sum_{n\notin H}\int_{ Q_L}|\!\D(\del^{\{n\}}\psi_{L}^{H})|^2\bigg)^\frac12\\
&\lesssim&L^d\Big(\ell^{-2d}+\ell^{-d}S_{L}^H(1,0)\Big)^\frac12.
\end{eqnarray*}
Now
summing~\eqref{eq:bnd-T3-pre0} over $G\subset\N$ with $\sharp G=k$, inserting the above estimate for the first right-hand side term, and using the Cauchy--Schwarz inequality,
we find
\begin{multline}\label{eq:bnd-T30+}
L^{-d}\sum_{\sharp G=k}|A_{L}^{3}(G,j)|
\,\lesssim\,\mathds1_{k=1,j=0}\Big(\ell^{-2d}+\ell^{-d}S_{L}^H(1,0)\Big)^\frac12\\
+\Big(T_{L}^H(k-1,j)\Big)^\frac12
\Big(T_{L}^H(k-1,j)
+S_{L}^H(k,j)\Big)^\frac12.
\end{multline}
Inserting this into~\eqref{eq:ener-ident00+}, together with~\eqref{eq:bnd-T10+} and~\eqref{eq:bnd-T20+}, we conclude
\begin{multline*}
S_{L}^H(k,j)\,\lesssim\,
\mathds1_{k=0,j\le 1}\ell^{-d}
+\mathds1_{k=1,j=0}\Big(\ell^{-2d}+\ell^{-d}S_{L}^H(1,0)\Big)^\frac12
+S_{L}^H(k+1,j-1)\\
+T_{L}^H(k,j)
+T_{L}^H(k,j-1)
+\Big(T_{L}^H(k-1,j)\Big)^\frac12
\Big(T_{L}^H(k-1,j)
+S_{L}^H(k,j)\Big)^\frac12.
\end{multline*}
Using Young's inequality to absorb the occurrence of~$S_{L}^H(k,j)$ in the right-hand side into the left-hand side, the claim~\eqref{induct-details} follows.

\medskip
\step3 Recurrence relation for $T_{L}^H$: for all $H\subset\N$ and $k,j\ge0$,
\begin{multline}\label{eq:recurr-T}
T_{L}^H(k,j)\,\lesssim\,\mathds1_{k=j=0}\ell^{-2d}+\mathds1_{k+j=1}\ell^{-3d}\\
+\ell^{-2d}\Big(T_L^H(k-1,j)+T_L^H(k,j-1)+T_L^H(k+1,j-2)\\
+S_L^H(k,j)+S_L^H(k+1,j-1)+S_L^H(k+2,j-2)\Big).
\end{multline}
Let $k,j\ge0$ be fixed with $k+j\ge1$ (the case $k=j=0$ already follows from~\eqref{eq:st1-nabphiH+}).
For~$G\subset\N$ and $n\notin G$, the $\ell$-separation property~\eqref{eq:ell-separation+} implies that the following free steady Stokes equations hold in $I_{n,L}+\frac14\ell B$,
\begin{gather*}
-\triangle\bigg(\sum_{\sharp F=j\atop F\cap (G\cup\{n\})=\varnothing}\!\!\!\!\del^{F\cup G}\psi_{L}^{H}\bigg)+\nabla\bigg(\sum_{\sharp F=j\atop F\cap (G\cup\{n\})=\varnothing}\!\!\!\!\del^{F\cup G}(\Sigma_{L}^{H}\mathds1_{ Q_L\setminus\Ic_L^{H}})\bigg)\,=\,0,\\
\Div\bigg(\!\!\!\sum_{\sharp F=j\atop F\cap (G\cup\{n\})=\varnothing}\!\!\!\!\del^{F\cup G}\psi_{L}^{H}\bigg)=0,\quad\qquad\text{in $I_{n,L}+\tfrac14\ell B$},
\end{gather*}
so that elliptic regularity in form of Lemma~\ref{lem:regularity} yields
\begin{equation}\label{eq:stand-reg-ut+}
T_{L}^H(k,j)\,\lesssim\,L^{-d}\ell^{-d}\sum_{\sharp G=k}\sum_{n\notin G\cup H}\int_{I_{n,L}+\frac14\ell B}\Big|\!\!\sum_{\sharp F=j\atop F\cap (G\cup\{n\})=\varnothing}\!\!\!\!\D(\del^{F\cup G}\psi_{L}^{H})\Big|^2.
\end{equation}
In order to analyze the right-hand side, we shall appeal to elliptic regularity a second time, now via a duality argument.
For that purpose, we use the following dual representation
\begin{multline}\label{eq:bnd-RLp-4-1}
\sum_{\sharp G=k}\sum_{n\notin G\cup H}\int_{I_{n,L}+\frac14\ell B}\Big|\sum_{\sharp F=j\atop F\cap(G\cup\{n\})=\varnothing}\D(\del^{F\cup G}\psi_{L}^{H})\Big|^2\\
\hspace{-2cm}\,=\,\sup_{\alpha,h}\bigg\{I(\alpha,h)^2\,:\,\sum_{\sharp G=k}\sum_{n\notin G\cup H}|\alpha_{n,G}|^2=1,\\
\int_{ Q_L}|h_{n,G}|^2=1,~~\supp h_{n,G}\subset I_{n,L}+\tfrac14\ell B,~~\forall n,G\bigg\},
\end{multline}
where for any $\alpha=\{\alpha_{n,G}\}_{n,G}\subset\R$ and $h=\{h_{n,G}\}_{n,G}\subset\Ld^2( Q_L)^{d\times d}_\Sym$ we have set for abbreviation
\begin{equation}\label{eq:def-Ialphh}
I(\alpha,h)\,:=\,\sum_{\sharp G=k}\sum_{n\notin G\cup H}\alpha_{n,G}\int_{ Q_L}h_{n,G}:\bigg(\sum_{\sharp F=j\atop F\cap(G\cup\{n\})=\varnothing}\D(\del^{F\cup G}\psi_{L}^{H})\bigg).
\end{equation}
Let $\alpha=\{\alpha_{n,G}\}_{n,G}\subset\R$ and $h=\{h_{n,G}\}_{n,G}\subset\Ld^2( Q_L)^{d\times d}_\Sym$ be momentarily fixed, satisfying the constraints in~\eqref{eq:bnd-RLp-4-1},
\begin{equation}\label{eq:constr-alph-h}
\sum_{\sharp G=k}\sum_{n\notin G\cup H}|\alpha_{n,G}|^2=1,\quad
\int_{ Q_L}|h_{n,G}|^2=1,\quad\supp h_{n,G}\subset I_{n,L}+\tfrac14\ell B,\quad\forall n,G.
\end{equation}
For $n\notin G\cup H$, consider the periodic solution $w_{h,n,G}$ of the following auxiliary steady Stokes problem,
\begin{equation}\label{eq:aux-wnh}
\left\{\begin{array}{ll}
-\triangle w_{h,n,G}+\nabla P_{h,n,G}=\Div( h_{n,G}),&\text{in $ Q_L\setminus\Ic_{L}^H$},\\
\Div( w_{h,n,G})=0,&\text{in $ Q_L\setminus\Ic_{L}^H$},\\
\D(w_{h,n,G})=0,&\text{in $\Ic_{L}^H$},\\
\int_{\partial I_{m,L}}\sigma(w_{h,n,G},P_{h,n,G})\nu=0,&\forall m\in H,\\
\int_{\partial I_{m,L}}\Theta(x-x_{m,L})
\cdot\sigma(w_{h,n,G},P_{h,n,G})\nu=0,&\forall \Theta\in\Md^\Skew,~\forall m\in H.
\end{array}\right.
\end{equation}
Note that this problem is well-posed since $h_{n,G}$ is supported in $I_{n,L}+\frac14\ell B\subset Q_L\setminus\Ic_{L}^H$.
The same argument as for~\eqref{eq:reform} shows that $w_{h,n,G}$ satisfies in $ Q_L$,
\begin{equation*}
-\triangle w_{h,n,G}+\nabla\big(P_{h,n,G}\mathds1_{ Q_L\setminus\Ic_{L}^H}\big)\,=\,\Div( h_{n,G})-\sum_{m\in H}\delta_{\partial I_{m,L}}\sigma(w_{h,n,G},P_{h,n,G})\nu,
\end{equation*}
and, appealing to~\eqref{eq:diff-st} and changing summation variables, we also find in $ Q_L$,
\begingroup\allowdisplaybreaks
\begin{multline*}
-\triangle\bigg(\sum_{\sharp F=j\atop F\cap(G\cup\{n\})=\varnothing}\del^{F\cup G}\psi_{L}^{H}\bigg)+\nabla\bigg(\sum_{\sharp F=j\atop F\cap(G\cup\{n\})=\varnothing}\del^{F\cup G}(\Sigma_{L}^{H}\mathds1_{ Q_L\setminus\Ic_L^{H}})\bigg)\\
\,=\,
-\sum_{m\in H}\delta_{\partial I_{m,L}}\bigg(\sum_{\sharp F=j\atop F\cap(G\cup\{n\})=\varnothing}\del^{F\cup G}\sigma^{H}_{L}\nu\bigg)\\
-\sum_{m\in G\setminus H}\delta_{\partial I_{m,L}}\bigg(\sum_{\sharp F=j\atop F\cap(G\cup\{n\})=\varnothing}\del^{F\cup (G\setminus\{m\})}\sigma^{H\cup\{m\}}_{L}\nu\bigg)\\
-\sum_{m\notin G\cup H\cup\{n\}}\delta_{\partial I_{m,L}}\bigg(\sum_{\sharp F=j-1\atop F\cap(G\cup\{n,m\})=\varnothing}\del^{F\cup G}\sigma^{H\cup\{m\}}_{L}\nu\bigg).
\end{multline*}
\endgroup
Testing the second of these two equations with the solution of the first one, and vice versa, and using the boundary conditions, we can reformulate $I(\alpha,h)$ in~\eqref{eq:def-Ialphh} as follows, provided $k+j\ge1$,
\begin{eqnarray}
I(\alpha,h)&=&-\sum_{\sharp G=k}\sum_{n\notin G\cup H}2\alpha_{n,G}\int_{ Q_L}\D( w_{h,n,G}):\bigg(\sum_{\sharp F=j\atop F\cap(G\cup\{n\})=\varnothing}\D(\del^{F\cup G}\psi_{L}^{H})\bigg)\nonumber\\
&=&I_1(\alpha,h)+I_2(\alpha,h),\label{eq:bnd-RLp-4-2}
\end{eqnarray}
where we have set
\begin{eqnarray*}
I_1(\alpha,h)\!\!&:=&\!\!\sum_{\sharp G=k}\sum_{n\notin G\cup H}\alpha_{n,G}\sum_{m\in G\setminus H}\int_{\partial I_{m,L}}w_{h,n,G}\cdot\bigg(\sum_{\sharp F=j\atop F\cap(G\cup\{n\})=\varnothing}\del^{F\cup (G\setminus\{m\})}\sigma^{H\cup\{m\}}_{L}\nu\bigg),
\\
I_2(\alpha,h)\!\!&:=&\!\!\sum_{\sharp G=k}\sum_{n\notin G\cup H}\alpha_{n,G}\sum_{m\notin G\cup H\cup\{n\}}\int_{\partial I_{m,L}}w_{h,n,G}\cdot\bigg(\sum_{\sharp F=j-1\atop F\cap(G\cup\{n,m\})=\varnothing}\del^{F\cup G}\sigma^{H\cup\{m\}}_{L}\nu\bigg).
\end{eqnarray*}
We only treat  $I_1(\alpha,h)$ in detail since the argument for $I_2(\alpha,h)$ is similar.  Appealing to identity~\eqref{eq:decomp-FGn-not}, we can rewrite
\begin{multline*}
I_1(\alpha,h)\,=\,\sum_{\sharp G=k}\sum_{m\in G\setminus H}\int_{\partial I_{m,L}}\bigg(\sum_{n\notin G\cup H}\alpha_{n,G}w_{h,n,G}\bigg)\cdot\bigg(\sum_{\sharp F=j\atop F\cap G=\varnothing}\del^{F\cup (G\setminus\{m\})}\sigma^{H\cup\{m\}}_{L}\nu\bigg)\\
-\sum_{\sharp G=k}\sum_{n\notin G\cup H}\alpha_{n,G}\sum_{m\in G\setminus H}\int_{\partial I_{m,L}}w_{h,n,G}\cdot\bigg(\sum_{\sharp F=j-1\atop F\cap(G\cup\{n\})=\varnothing}\del^{F\cup (G\setminus\{m\})\cup\{n\}}\sigma^{H\cup\{m\}}_{L}\nu\bigg),
\end{multline*}
or equivalently, after further changing summation variables in the second term,
\begin{multline*}
I_1(\alpha,h)\,=\,\sum_{\sharp G=k}\sum_{m\in G\setminus H}\int_{\partial I_{m,L}}\bigg(\sum_{n\notin G\cup H}\alpha_{n,G}w_{h,n,G}\bigg)\cdot\bigg(\sum_{\sharp F=j\atop F\cap G=\varnothing}\del^{F\cup (G\setminus\{m\})}\sigma^{H\cup\{m\}}_{L}\nu\bigg)\\
-\sum_{\sharp G=k+1}\sum_{m\in G\setminus H}\int_{\partial I_{m,L}}\!\!\bigg(\sum_{n\in G\setminus(H\cup\{m\})}\alpha_{n,G\setminus\{n\}}w_{h,n,G\setminus\{n\}}\bigg)\\
\cdot\bigg(\!\sum_{\sharp F=j-1\atop F\cap G=\varnothing}\del^{F\cup (G\setminus\{m\})}\sigma^{H\cup\{m\}}_{L}\nu\!\bigg).
\end{multline*}
Now using the boundary conditions and the incompressibility constraints to add arbitrary constants to the different factors, as in the proof of~\eqref{eq:subtr-press-cond}, and appealing to the trace estimates of Lemma~\ref{lem:trace}, we are led to
\begin{equation}\label{eq:I1-I11I12}
|I_1(\alpha,h)|\,\lesssim\,I_{1,1}(\alpha,h)+I_{1,2}(\alpha,h),
\end{equation}
where we have set
\begin{eqnarray*}
I_{1,1}(\alpha,h)&:=&\sum_{\sharp G=k}\sum_{m\in G\setminus H}\bigg(\int_{I_{m,L}}\Big|\sum_{n\notin G\cup H}\alpha_{n,G}\D( w_{h,n,G})\Big|^2\bigg)^\frac12\\
&&\hspace{2cm}\times\bigg(\int_{I_{m,L}+\rho B}\Big|\sum_{\sharp F=j\atop F\cap G=\varnothing}\D(\del^{F\cup (G\setminus\{m\})}(\psi^{H\cup \{m\}}_{L}+Ex))\Big|^2\bigg)^\frac12,\\
I_{1,2}(\alpha,h)&:=&\sum_{\sharp G=k+1}\sum_{m\in G\setminus H}\bigg(\int_{I_{m,L}}\Big|\sum_{n\in G\setminus(H\cup\{m\})}\alpha_{n,G\setminus\{n\}}\D( w_{h,n,G\setminus\{n\}})\Big|^2\bigg)^\frac12\\
&&\hspace{2cm}\times\bigg(\int_{I_{m,L}+\rho B}\Big|\sum_{\sharp F=j-1\atop F\cap G=\varnothing}\D(\del^{F\cup (G\setminus\{m\})}(\psi^{H\cup \{m\}}_{L}+Ex))\Big|^2\bigg)^\frac12.
\end{eqnarray*}
We start by estimating $I_{1,1}(\alpha,h)$. Decomposing the second factor via the following identity, for all $m\in G\setminus H$ and $F\cap G=\varnothing$,
\[\del^{F\cup (G\setminus\{m\})}(\psi^{H\cup \{m\}}_{L}+Ex)\,=\,\mathds1_{\sharp G=1,\sharp F=0}Ex+\del^{F\cup (G\setminus\{m\})}\psi^{H}_{L}+\del^{F\cup G}\psi^{H}_{L},\]
noting that the $\ell$-separation property~\eqref{eq:ell-separation+} entails that $\sum_{n\notin G\cup H} \alpha_{n,G}w_{h,n,G}$ satisfies the free steady Stokes equations in $I_{m,L}+\frac14\ell B$ for all $m\notin H$,
and appealing to  elliptic regularity in form of Lemma~\ref{lem:regularity},
we find
\begin{multline}\label{eq:pre-est-I1,1}
I_{1,1}(\alpha,h)\,\lesssim\,\sum_{\sharp G=k}\sum_{m\in G\setminus H}\bigg(\ell^{-d}\int_{I_{m,L}+\frac14\ell B}\Big|\sum_{n\notin G\cup H}\alpha_{n,G}\D( w_{h,n,G})\Big|^2\bigg)^\frac12\\
\times\bigg(\mathds1_{k=1,j=0}+\int_{I_{m,L}+\rho B}\Big|\sum_{\sharp F=j\atop F\cap G=\varnothing}\D(\del^{F\cup (G\setminus\{m\})}\psi^{H}_{L})\Big|^2\\
+\int_{I_{m,L}+\rho B}\Big|\sum_{\sharp F=j\atop F\cap G=\varnothing}\D(\del^{F\cup G}\psi^{H}_{L})\Big|^2\bigg)^\frac12.
\end{multline}
Next, the energy estimate for~\eqref{eq:aux-wnh} yields
\begin{equation*}
\sum_{\sharp G=k}\int_{ Q_L}\Big|\sum_{n\notin G\cup H}\alpha_{n,G}\D( w_{h,n,G})\Big|^2\,\lesssim\,\sum_{\sharp G=k}\int_{ Q_L}\Big|\sum_{n\notin G\cup H}\alpha_{n,G}h_{n,G}\Big|^2,
\end{equation*}
and thus, using the constraints~\eqref{eq:constr-alph-h} on $\alpha,h$, and noting that the $\ell$-separation property~\eqref{eq:ell-separation+} entails that the $h_{n,G}$'s have disjoint supports for different $n$'s,
\begin{eqnarray*}
\sum_{\sharp G=k}\int_{ Q_L}\Big|\sum_{n\notin G\cup H}\alpha_{n,G}\D( w_{h,n,G})\Big|^2
\,\lesssim\,\sum_{\sharp G=k}\sum_{n\notin G\cup H}|\alpha_{n,G}|^2\int_{ Q_L}|h_{n,G}|^2~=~1.
\end{eqnarray*}
Inserting this into~\eqref{eq:pre-est-I1,1}, using the Cauchy--Schwarz inequality, the $\ell$-separation property~\eqref{eq:ell-separation+} in form of the disjointness of the fattened inclusions $\{I_{m,L}+\frac14\ell B\}_m$, using that the number of points of the process $\Pc_L$ in $Q_L$ is bounded by $C(L/\ell)^d$, and changing summation variables, we deduce
\begin{equation}\label{eq:bnd-I11}
L^{-d}I_{1,1}(\alpha,h)^2\,\lesssim\,\mathds1_{k=1,j=0}\ell^{-2d}+\ell^{-d}\Big(T_L^H(k-1,j)
+S_L^H(k,j)\Big).
\end{equation}
We turn to a corresponding estimation for $I_{1,2}(\alpha,h)$. For that purpose, we first note that the disjointness of fattened inclusions $\{I_{m,L}+\frac14\ell B\}_m$ allows to decompose
\begin{multline*}
\sum_{\sharp G=k+1}\sum_{m\in G\setminus H}\int_{I_{m,L}+\frac14\ell B}\Big|\sum_{n\in G\setminus(H\cup\{m\})}\alpha_{n,G\setminus\{n\}}\D( w_{h,n,G\setminus\{n\}})\Big|^2\\
\,\lesssim\,\sum_{\sharp G=k+1}\int_{ Q_L}\Big|\sum_{n\in G\setminus H}\alpha_{n,G\setminus\{n\}}\D( w_{h,n,G\setminus\{n\}})\Big|^2\\
+\sum_{\sharp G=k+1}\sum_{m\in G\setminus H}|\alpha_{m,G\setminus\{m\}}|^2\int_{ Q_L}|\!\D( w_{h,m,G\setminus\{m\}})|^2,
\end{multline*}
and the energy estimate for~\eqref{eq:aux-wnh} then yields
\begin{multline*}
\sum_{\sharp G=k+1}\sum_{m\in G\setminus H}\int_{I_{m,L}+\frac14\ell B}\Big|\sum_{n\in G\setminus(H\cup\{m\})}\alpha_{n,G\setminus\{n\}}\D( w_{h,n,G\setminus\{n\}})\Big|^2\\
\,\lesssim\,\sum_{\sharp G=k+1}\int_{ Q_L}\Big|\sum_{n\in G\setminus H}\alpha_{n,G\setminus\{n\}}h_{n,G\setminus\{n\}}\Big|^2
+\sum_{\sharp G=k+1}\sum_{m\in G\setminus H}|\alpha_{m,G\setminus\{m\}}|^2\int_{ Q_L}|h_{m,G\setminus\{m\}}|^2,
\end{multline*}
from which we deduce, using the constraints~\eqref{eq:constr-alph-h} on $\alpha,h$ and recalling that the $h_{n,G}$'s have disjoint supports for different $n$'s,
\begin{multline*}
\sum_{\sharp G=k+1}\sum_{m\in G\setminus H}\int_{I_{m,L}+\frac14\ell B}\Big|\sum_{n\in G\setminus(H\cup\{m\})}\alpha_{n,G\setminus\{n\}}\D( w_{h,n,G\setminus\{n\}})\Big|^2\\
\,\lesssim\,\sum_{\sharp G=k+1}\sum_{n\in G\setminus H}|\alpha_{n,G\setminus\{n\}}|^2
\,=\,\sum_{\sharp G=k}\sum_{n\notin G\cup H}|\alpha_{n,G}|^2~=~1.
\end{multline*}
With this estimate at hand, we may now repeat the same argument as for~\eqref{eq:bnd-I11} and we obtain
\begin{equation}\label{eq:bnd-I12}
L^{-d}I_{1,2}(\alpha,h)^2\,\lesssim\,
\mathds1_{k=0,j=1}\ell^{-2d}
+\ell^{-d}\Big(T_L^H(k,j-1)+S_L^H(k+1,j-1)\Big).
\end{equation}
Likewise, the second term $I_2(\alpha,h)$ in~\eqref{eq:bnd-RLp-4-2} is easily estimated as follows,
\begin{multline}\label{eq:bnd-RLp-4-4}
L^{-d}I_2(\alpha,h)^2\,\lesssim\,\mathds1_{k=0,j=1}\ell^{-2d}
+\ell^{-d}\Big(T_L^H(k,j-1)+T_L^H(k+1,j-2)\\
+S_L^H(k+1,j-1)+S_L^H(k+2,j-2)\Big).
\end{multline}
Combining these different estimates, that is, \eqref{eq:I1-I11I12}, \eqref{eq:bnd-I11}, \eqref{eq:bnd-I12}, and~\eqref{eq:bnd-RLp-4-4}, inserting them into~\eqref{eq:bnd-RLp-4-1}, and recalling~\eqref{eq:stand-reg-ut+}, the claim~\eqref{eq:recurr-T} follows.

\medskip
\step4 Conclusion.\\
By a direct double induction argument, starting with~\eqref{eq:st1-nabphiH+}, the recurrence relation~\eqref{eq:recurr-T} entails,
for all $H\subset\N$ and $k,j\ge0$,
\begin{multline}\label{eq:bnd-RLp-4}
T_{L}^H(k,j)\,\lesssim\,
\mathds1_{k=j=0}\ell^{-2d}+\mathds1_{k+j\ge1}(C\ell^{-d})^{2(k+j)+1}\\
+\sum_{l=0}^{k+j-1}(C\ell^{-d})^{2(l+1)}\sum_{i=0}^{2(l+1)}S_{L}^H(k+i-l,j-i).
\end{multline}
Combined with the other recurrence relation~\eqref{induct-details}, this yields
\begin{multline*}
S_{L}^H(k,j)\,\lesssim\,\mathds1_{k=j=0}\ell^{-d}+\mathds1_{k+j\ge1}(C\ell^{-d})^{2(k+j)-1}
+S_{L}^H(k+1,j-1)\\
+\sum_{l=0}^{k+j-1}(C\ell^{-d})^{2(l+1)}\sum_{i=0}^{2l+2}S_{L}^H(k+i-l,j-i)\\
+\sum_{l=0}^{k+j-2}(C\ell^{-d})^{2(l+1)}\sum_{i=0}^{2l+3}S_{L}^H(k+i-l-1,j-i).
\end{multline*}
For $\ell\gg1$, occurrences of $S_{L}^H(k,j)$ in the right-hand side can be absorbed into the left-hand side, and we are then left with
\begin{multline*}
S_{L}^H(k,j)\,\lesssim\,\mathds1_{k=j=0}\ell^{-d}+\mathds1_{k+j\ge1}(C\ell^{-d})^{2(k+j)-1}
+S_{L}^H(k+1,j-1)+S_{L}^H(k+2,j-2)\\
+\sum_{l=1}^{k+j-1}(C\ell^{-d})^{2(l+1)}\sum_{i=0}^{2l+2}S_{L}^H(k+i-l,j-i)\\
+\sum_{l=0}^{k+j-2}(C\ell^{-d})^{2(l+1)}\sum_{i=0}^{2l+3}S_{L}^H(k+i-l-1,j-i).
\end{multline*}
By a double induction argument, this relation leads to the conclusion
\[S_{L}^{H}(k,j)\lesssim\left\{\begin{array}{lll}
\ell^{-d}&:&k=j=0,\\
(C\ell^{-d})^{2(k+j)-1}&:&k,j\ge0,~~k+j\ge1.
\end{array}\right.\]
Combining this with~\eqref{eq:bnd-RLp-4} further yields
\[T_{L}^{H}(k,j)\lesssim\left\{\begin{array}{lll}
\ell^{-2d}&:&k=j=0,\\
(C\ell^{-d})^{2(k+j)+1}&:&k,j\ge0,~~k+j\ge1.
\end{array}\right.\]
Recalling that the case $\ell\simeq1$ was already covered in~\eqref{eq:bnd-ap-S0}, this finally concludes the proof of Theorem~\ref{prop:apest-re}.
\qed

\subsection{Uniform cluster estimates}\label{sec:proof-thm:bounds}
This section is devoted to the proof of Theorem~\ref{thm:bounds}(i), based on the interpolating $\ell^1-\ell^2$ energy estimates of Theorem~\ref{prop:apest-re}.
We focus on the bound~\eqref{e.unif-bd-Bb-R} on the remainder $R_{L}^{k+1}$, while the corresponding bounds on cluster coefficients follow along the same lines.
For $k\ge1$, after changing summation variables, the definition~\eqref{eq:form-clust-rem} of the remainder can be written as
\begin{equation*}
E:R_{L}^{k+1}E
\,=\,\tfrac12L^{-d}\sum_n\,\expecM{\int_{\partial I_{n,L}}\Big(\sum_{\sharp F=k\atop n\notin F}\del^{F}\psi_{L}^\varnothing\Big)\cdot\sigma_{L}\nu}.
\end{equation*}
Using the boundary conditions and the incompressibility constraint
to smuggle in arbitrary constants in the different factors, as in the proof of~\eqref{eq:subtr-press-cond},
using the Cauchy--Schwarz inequality, and then appealing to the trace estimates of Lemma~\ref{lem:trace}, we find
\begin{equation*}
|E:R_{L}^{k+1}E|
\,\lesssim\,L^{-d}\,\expecM{\sum_{n}\int_{I_{n,L}}\Big|\sum_{\sharp F=k\atop n\notin F}\D(\del^{F}\psi_{L}^\varnothing)\Big|^2}^\frac12\expecM{\sum_{n}\int_{I_{n,L}+\rho B}|\!\D(\psi_L)+E|^2}^\frac12.
\end{equation*}
Recalling the disjointness of the fattened inclusions $\{I_{n,L}+\rho B\}_n$, recognizing the definition of $S_L$ and $T_L^\varnothing$, and using that in case $\ell\gg1$ the $\ell$-separation property~\eqref{eq:ell-separation+} entails that the number of points of the process $\Pc_L$ in $ Q_L$ is bounded by $C(L/\ell)^d$, we are led to
\begin{equation*}
|R^{k+1}_{L}|
~\lesssim~\expecm{T_{L}^\varnothing(0,k)}^\frac12\Big(\ell^{-d}+\expec{S_L(0,0)}\Big)^\frac12,
\end{equation*}
and the conclusion~\eqref{e.unif-bd-Bb-R} then follows from Theorem~\ref{prop:apest-re}.
\qed

\subsection{Convergence of finite-volume approximations}\label{sec:proof-eq:form-Bj}
This section is devoted to the proof of the convergence result~\eqref{eq:form-Bj} in Theorem~\ref{thm:bounds}.
The idea is as follows: if $\{\Bb_L^j\}_j$ could be viewed as derivatives of $\Bb_L$ in some sense, then the convergence of $\Bb_L$ as $L\uparrow\infty$ and the uniform bounds on $\{\Bb_L^j\}_j$ would ensure the convergence of latter.
We split the proof into two steps, first appealing to a probabilistic argument to view $\{\Bb_L^j\}_j$ as true derivatives, and then concluding by means of standard real analysis.

\medskip
\step1 Dilution by random deletion.\\
Given $p\in[0,1]$, we consider a sequence $\{b_n^{(p)}\}_n$ of iid Bernoulli variables, independent of~$\Pc,\Ic$, with parameter
\[p\,=\,\prm{b_n^{(p)}=1},\]
and we define the corresponding decimated process
\begin{equation}\label{eq:p-decim-def}
\Pc^{(p)}\,:=\,\{x_n\}_{n\in N^{(p)}},\qquad\Ic^{(p)}\,:=\,\textstyle\bigcup_{n\in N^{(p)}}I_n,\qquad\text{where}~~N^{(p)}:=\{n:b_n^{(p)}=1\}.
\end{equation}
Similarly, in the periodized setting~\eqref{e.set-perio}, we set
\begin{equation*}
\Pc_{L}^{(p)}\,:=\,\{x_{n,L}\}_{n\in N^{(p)}},\qquad
\Ic_{L}^{(p)}\,:=\,\textstyle\bigcup_{n\in \Nb^{(p)}}I_{n,L}.
\end{equation*}
By definition, the decimated processes $\Pc^{(p)},\Ic^{(p)}$ satisfy~\ref{H0} and~\ref{Gunif} whenever $\Pc,\Ic$ do, and their periodized versions $\Pc_L^{(p)},\Ic_L^{(p)}$ satisfy the same separation and stabilization properties as $\Pc_L,\Ic_L$ in Section~\ref{sec:finitevol-approx}.
We use the notation $\Bb^{(p)},\Bb_L^{(p)},\{\Bb_L^{(p),j}\}_j,\{R_L^{(p),k+1}\}_k$ for the effective viscosity, its periodized approximation, cluster coefficients, and cluster remainders associated with decimated processes $\Ic^{(p)},\Ic_L^{(p)}$.
As a corollary of \cite[Theorem~1]{DG-19}, as in~\eqref{eq:conv-BL-B-hom}, we have for all $p\in[0,1]$,
\begin{equation}\label{eq:convL-p}
\lim_{L \uparrow \infty} \Bb_{L}^{(p)}\,=\,\Bb^{(p)}.
\end{equation}
In the next two substeps, we shall further prove for all $k,j\ge1$,
\begin{eqnarray}
\Bb_L^{(p),j}&=&p^j\Bb_L^j,\label{eq:BLpj-porder}\\
|R_L^{(p),k+1}|&\le&(Cp\ell^{-d})^{k+1}.\label{eq:bnd-RLpk}
\end{eqnarray}
Combined with the cluster expansion~\eqref{eq:preconcl-exp}, this yields for all $L$ and $k\ge1$,
\begin{equation}\label{eq:pre-concl-anal}
\bigg|\Bb_{L}^{(p)}-\Big(\Id+\sum_{j=1}^k\tfrac{p^j}{j!}\Bb_{L}^{j}\Big)\bigg|\,\le\,(Cp \ell^{-d})^{k+1},
\end{equation}
which entails that $\Bb_L^j$ can be seen as the $j$th derivative of the map $p\mapsto \Bb_L^{(p)}$ at $p=0$.
(Note that this estimate further shows that this map is real-analytic; we shall later come back to this observation as part of Theorem~\ref{th:analytic}.)

\medskip
\substep{1.1} Proof of~\eqref{eq:BLpj-porder}.\\
By definition of decimated processes, the cluster formula~\eqref{eq:form-clust} for $\Bb_L^{(p),j}$ can be written as
\[E:\Bb_L^{(p),j}E\,=\,j!\sum_{\sharp F=j}\expecM{\mathds1_{F\subset N^{(p)}} \fint_{Q_L}\delta^F\big(|\D(\psi_L^\varnothing)+E|^2\big)}.\]
As $N^{(p)}$ is independent of $\Ic$ and as $\prm{F\subset N^{(p)}}=\prm{b_n^{(p)}=1,~\forall n\in F}=p^{\sharp F}$, we get
\begin{equation}
E:\Bb_L^{(p),j}E\,=\,j!p^j\sum_{\sharp F=j}\expecM{\fint_{Q_L}\delta^F\big(|\D(\psi_L^\varnothing)+E|^2\big)}\,=\,p^jE:\Bb_L^j E,
\end{equation}
that is, \eqref{eq:BLpj-porder}.

\medskip
\substep{1.2} Proof of~\eqref{eq:bnd-RLpk}.\\
Let $k\ge1$. By definition of decimated processes, the remainder formula~\eqref{eq:form-clust-rem} for $R_L^{(p),k+1}$ can be written as
\begin{equation*}
E:R_{L}^{(p),k+1}E\,=\,\tfrac12L^{-d}\!\!\sum_{\sharp F=k+1}\sum_{n\in F}\expecM{\mathds1_{F\subset N^{(p)}}\int_{\partial I_{n,L}}\!\!\del^{F\setminus\{n\}}\psi_{L}^\varnothing\cdot \sigma_{L}^{(p)}\nu},
\end{equation*}
or equivalently, using the constraint $F\subset N^{(p)}$ to replace $\sigma_L^{(p)}=\sigma_L^{N^{(p)}}$ by $\sigma_L^{N^{(p)}\cup F}$,
\begin{equation*}
E:R_{L}^{(p),k+1}E\,=\,\tfrac12L^{-d}\!\!\sum_{\sharp F=k+1}\sum_{n\in F}\expecM{\mathds1_{F\subset N^{(p)}}\int_{\partial I_{n,L}}\!\!\del^{F\setminus\{n\}}\psi_{L}^\varnothing\cdot \sigma_{L}^{N^{(p)}\cup F}\nu}.
\end{equation*}
In this expression, the integral
\[\int_{\partial I_{n,L}}\!\!\del^{F\setminus\{n\}}\psi_{L}^\varnothing\cdot \sigma_{L}^{N^{(p)}\cup F}\nu\]
does not depend on the value of $\{b_n^{(p)}\}_{n\in F}$ and is thus independent of
\[\mathds1_{F\subset N^{(p)}}\,=\,\textstyle\prod_{n\in F}\mathds1_{b_n^{(p)}=1},\]
hence we are led to
\begin{equation}\label{eq:reform-RLpk}
E:R_{L}^{(p),k+1}E\,=\,\tfrac12p^{k+1}L^{-d}\!\!\sum_{\sharp F=k+1}\sum_{n\in F}\expecM{\int_{\partial I_{n,L}}\!\!\del^{F\setminus\{n\}}\psi_{L}^\varnothing\cdot \sigma_{L}^{N^{(p)}\cup F}\nu}.
\end{equation}
It remains to estimate the right-hand side and deduce~\eqref{eq:bnd-RLpk}, which is easily done by adapting the proof of Theorem~\ref{thm:bounds}(i) in Section~\ref{sec:proof-thm:bounds}.
For that purpose, we first note that, for all $F\subset\N$, using that $\sum_{H'\subset H}(-1)^{|H'|}=0$ if $H\ne\varnothing$, we have
\begin{eqnarray*}
\sum_{G\subset F}\delta^G\sigma_L^{(p)}
&=&\sum_{G\subset F}\sum_{G'\subset G}(-1)^{|G\setminus G'|}\sigma_L^{N^{(p)}\cup G'}\\
&=&\sum_{G'\subset F}\Big(\sum_{G''\subset F\setminus G'}(-1)^{|G''|}\Big)\,\sigma_L^{N^{(p)}\cup G'}\\
&=&\sigma_L^{N^{(p)}\cup F},
\end{eqnarray*}
so that formula~\eqref{eq:reform-RLpk} can be decomposed as follows, after changing summation variables,
\[E:R_L^{(p),k+1}E\,=\,\tfrac12p^{k+1}L^{-d}\sum_{\sharp F=k}\sum_{n\notin F}\sum_{G\subset F\cup\{n\}}\expecM{\int_{\partial I_{n,L}}\delta^F\psi^\varnothing_L\cdot\delta^G\sigma_L^{(p)}\nu}.\]
Using the following identity, for all maps $f$ and all $n\notin F$,
\[\sum_{G\subset F\cup\{n\}}f(G)\,=\,\sum_{G\subset F}f(G)+\sum_{G\subset F}f(G\cup\{n\}),\]
we deduce
\[E:R_L^{(p),k+1}E\,=\,\tfrac12p^{k+1}L^{-d}\sum_{\sharp F=k}\sum_{G\subset F}\sum_{n\notin F}\expecM{\int_{\partial I_{n,L}}\delta^F\psi_L^\varnothing\cdot\Big(\delta^G\sigma_L^{(p)}+\delta^{G\cup\{n\}}\sigma_L^{(p)}\Big)\nu},\]
or equivalently, further changing summation variables,
\begin{multline*}
E:R_L^{(p),k+1}E\,=\,\tfrac12p^{k+1}L^{-d}\sum_{j=0}^k\sum_{\sharp G=j}\sum_{n\notin G}\E\bigg[\int_{\partial I_{n,L}}\bigg(\sum_{\sharp F=k-j\atop F\cap(G\cup\{n\})=\varnothing}\delta^{F\cup G}\psi_L^\varnothing\bigg)\\
\cdot\Big(\delta^G\sigma_L^{(p)}+\delta^{G\cup\{n\}}\sigma_L^{(p)}\Big)\nu\bigg].
\end{multline*}
Using the boundary conditions for $\delta^G\sigma_L^{(p)}+\delta^{G\cup\{n\}}\sigma_L^{(p)}=\delta^G\sigma_L^{N^{(p)}\cup\{n\}}$ and using the incompressibility constraint to smuggle in arbitrary constants in the different factors, as in the proof of~\eqref{eq:subtr-press-cond}, and then appealing to the trace estimates of Lemma~\ref{lem:trace}, we find
\begin{multline*}
|E:R_L^{(p),k+1}E|\,\lesssim\,p^{k+1}L^{-d}\sum_{j=0}^k\sum_{\sharp G=j}\sum_{n\notin G}\expecM{\int_{I_{n,L}}\Big|\sum_{\sharp F=k-j\atop F\cap (G\cup\{n\})=\varnothing}\D(\delta^{F\cup G}\psi_L^\varnothing)\Big|^2}^\frac12\\
\times\expecM{\mathds1_{j=0}+\int_{I_{n,L}+\rho B}|\!\D(\delta^G\psi_L^{(p)})|^2+|\!\D(\delta^{G\cup\{n\}}\psi_L^{(p)})|^2}^\frac12.
\end{multline*}
Recalling the disjointness of fattened inclusions $\{I_{n,L}+\rho B\}_n$, recognizing the definition of $S_L^{(p)}$ and $T_L^\varnothing$, and using that in case $\ell\gg1$ the $\ell$-separation property~\eqref{eq:ell-separation+} entails that the number of points of the process $\Pc_L$ in $ Q_L$ is bounded by $C(L/\ell)^d$, we deduce
\begin{multline*}
|E:R_L^{(p),k+1}E|\\
\,\lesssim\,p^{k+1}\sum_{j=0}^k\expecm{T_L^\varnothing(j,k-j)}^\frac12
\Big(\mathds1_{j=0}\ell^{-d}+\expecm{S_L^{(p)}(j,0)}+\expecm{S_L^{(p)}(j+1,0)}\Big)^\frac12.
\end{multline*}
Now appealing to Theorem~\ref{prop:apest-re}, the claim~\eqref{eq:bnd-RLpk} follows.

\medskip
\step2 Conclusion.\\
While the uniform estimates of Theorem~\ref{thm:bounds}(i) ensure that the sequence $\{\Bb^j_L\}_{L\ge1}$ converges as $L\uparrow\infty$ up to extraction of a subsequence, we shall use their interpretation as derivatives of the map $p\mapsto\Bb_L^{(p)}$ at $p=0$, together with some real analysis, to deduce
the convergence of the full sequence.
We argue by induction:
given $k\ge0$, we assume that the limits \mbox{$\Bb^{j}=\lim_{L\uparrow\infty}\Bb_{L}^{j}$} exist for all $1\le j\le k$,
and we shall then prove that the limit
\[\Bb^{k+1}=\lim_{L\uparrow\infty}\Bb_{L}^{k+1}\]
also exists. As $\Bb_{L}^{k+1}$ is bounded uniformly in $L$ by Theorem~\ref{thm:bounds}(i), it admits a limit~$\Cc^{k+1}$ as $L\uparrow\infty$ up to extraction of a subsequence. Passing to the limit along this subsequence in~\eqref{eq:pre-concl-anal}, with $k$ replaced by $k+1$, and using~\eqref{eq:convL-p} and the induction assumptions, we get for all~$p$,
\begin{equation}\label{e:super-analytic}
\bigg|\Bb^{(p)}-\Big(\Id+\sum_{j=1}^k\tfrac{p^j}{j!}\Bb^{j}+\tfrac{p^{k+1}}{(k+1)!}\Cc^{k+1}\Big)\bigg|\,\le\, (Cp)^{k+2},
\end{equation}
which proves that $\Cc^{k+1}$ satisfies
\[\Cc^{k+1}\,=\,\lim_{p\downarrow0} \tfrac{(k+1)!}{p^{k+1}}\bigg(\Bb^{(p)}-\Big(\Id+\sum_{j=1}^k\tfrac{p^j}{j!}\Bb^{j}\Big)\bigg),\]
where in particular the limit exists. Since the right-hand side does not depend on the choice of the extracted subsequence, we deduce that the limit $\Cc^{k+1}$ is uniquely defined, hence the limit $\Bb^{k+1}:=\Cc^{k+1}=\lim_{L\uparrow\infty}\Bb_{L}^{k+1}$ actually exists. By induction, this concludes the proof of the convergence result~\eqref{eq:form-Bj} in Theorem~\ref{thm:bounds}.
\qed

\subsection{Non-uniform cluster estimates}\label{sec:non-unif}
This section is devoted to the proof of Theorem~\ref{thm:bounds}(ii).
Taking inspiration from~\cite[Section~5.A]{D-thesis}, we proceed by a direct analysis of Green representation formulas for corrector differences. More precisely, we introduce operators $\{\Jc_{L;H}^n\}_{n,H}$ that describe the fluid velocity generated by localized force dipoles in the presence of a finite number of rigid inclusions: these are viewed as Stokeslets for the problem with rigid inclusions and lead to a useful decomposition of corrector differences, cf.~\eqref{eq:cordiff-JLHn} below.
The following lemma defines such operators and states their optimal decay properties, which are shown to coincide with the decay for the explicit Stokeslet associated with the problem in free space without rigid particles. This result is a particular case of Lemma~\ref{*lem:decay-re}, the proof of which is postponed to Appendix~\ref{append:ell}.

\begin{lem}[Decay of Stokeslets with rigid inclusions]\label{lem:decay}
Let Assumptions~\ref{H0} and~\ref{Gunif} hold,
let $H\subset\N$ be finite and $n\notin H$, and let $(\zeta,P)$ satisfies the following Stokes problem in a neighborhood of $I_{n,L}$,
\begin{equation}\label{eq:zetaP}
\left\{\begin{array}{ll}
-\triangle\zeta+\nabla P=0,&\text{in $(I_{n,L}+\rho B)\setminus I_{n,L}$},\\
\Div(\zeta)=0,&\text{in $(I_{n,L}+\rho B)\setminus I_{n,L}$},\\
\D(\zeta)=0,&\text{in $I_{n,L}$},\\
\int_{\partial I_{n,L}}\sigma(\zeta,P)\nu=0,&\\
\int_{\partial I_{n,L}}\Theta(x-x_{n,L})\cdot\sigma(\zeta,P)\nu=0,&\forall \Theta\in\Md^\Skew.
\end{array}\right.
\end{equation}
Denote by $\Jc_{L;H}^n\zeta\in H^1_\per(Q_L)^d$ the solution of the following Stokes problem,
\begin{equation}\label{eq:def-ILHn}
\left\{\begin{array}{ll}
-\triangle\Jc_{L;H}^n\zeta+\nabla \Qc_{L;H}^n\zeta=-\delta_{\partial I_{n,L}}\sigma(\zeta,P)\nu,&\text{in $Q_L\setminus\Ic_L^H$},\\
\Div(\Jc_{L;H}^n\zeta)=0,&\text{in $Q_L\setminus\Ic_L^H$},\\
\D(\Jc_{L;H}^n\zeta)=0,&\text{in $\Ic_L^H$},\\
\int_{\partial I_{m,L}}\sigma(\Jc_{L;H}^n\zeta,\Qc_{L;H}^n\zeta)\nu=0,&\forall m\in H,\\
\int_{\partial I_{m,L}}\Theta(x-x_{m,L})\cdot \sigma(\Jc_{L;H}^n\zeta,\Qc_{L;H}^n\zeta)\nu=0,&\forall m\in H,\,\forall\Theta\in\Md^\Skew.
\end{array}\right.
\end{equation}
Then, we have for all $z\in Q_L$,
\begin{equation}\label{eq:decay-borderline-re}
\Big(\int_{B(z)}|\!\D(\Jc_{L;H}^n\zeta)|^2\Big)^\frac12\,\lesssim_{\sharp H}\,\langle (z-x_{n,L})_L\rangle^{-d}\Big(\int_{I_{n,L}+\rho B}|\!\D(\zeta)|^2\Big)^\frac12.\qedhere
\end{equation}
\end{lem}

The above definition of operators $\{\Jc_{L;H}^n\}_{n,H}$ is motivated by the following observation:
for all~$F,H\subset\N$ with~$F$ finite and nonempty, equations~\eqref{eq:diff-st} for corrector differences entail, in these terms,
\begin{equation}\label{eq:cordiff-JLHn}
\del^{F}\psi_{L}^H\,=\,\sum_{n\in F\setminus H}\Jc^n_{L;H}\del^{F\setminus\{n\}}(\psi^{H\cup\{n\}}_{L}+Ex).
\end{equation}
Iterating this identity allows to write $\del^{F}\psi_L^H$ as a combination of iterations of $\Jc_{L;H}^n$'s, which are viewed as elementary single-particle contributions.
With the above result at hand, we may now conclude with the proof of Theorem~\ref{thm:bounds}(ii).

\begin{proof}[Proof of Theorem~\ref{thm:bounds}(ii)]
We focus on the bound~\eqref{e.non-unif-bd-Bb-R} on the remainder $R_{L}^{k+1}$, while the corresponding bound on cluster coefficients follows along the same lines.
We split the proof into two steps.

\medskip
\step1 Estimation of corrector differences.\\
For all finite $F,H\subset\N$ with $F$ nonempty, and for all $n\in\N$,
recalling the decomposition~\eqref{eq:cordiff-JLHn} for corrector differences,
Lemma~\ref{lem:decay} yields
\begin{multline*}
\Big(\int_{I_{n,L}+\rho B}|\!\D(\del^{F}\psi_{L}^{H})|^2\Big)^\frac12
\\\lesssim_{\sharp H}\,\sum_{m\in F\setminus H}\langle(x_{n,L}-x_{m,L})_L\rangle^{-d}\Big(\int_{I_{m,L}+\rho B}\big|\!\D\!\big(\del^{F\setminus\{m\}}(\psi^{H\cup\{m\}}_{L}+Ex)\big)\big|^2\Big)^\frac12.
\end{multline*}
Iterating this bound, and recalling that the energy estimate~\eqref{eq:det-bnd-phiFk00} gives for all finite $G\subset \N$
\begin{equation*}
\int_{ Q_L}|\!\D(\psi_{L}^G)|^2\,\lesssim\,\sharp G,
\end{equation*}
we deduce for all $n$, setting $k:=\sharp F\ge1$,
\begin{multline}\label{eq:bnd-delphi-1}
\Big(\int_{I_{n,L}+\rho B}|\!\D(\del^F\psi_{L}^\varnothing)|^2+|\!\D(\del^F\psi_{L}^{\{n\}})|^2\Big)^\frac12\\
\,\lesssim_k\,\sum_{n_1,\ldots,n_k\in F}^{\ne}\langle (x_{n,L}-x_{n_1,L})_L\rangle^{-d}\langle (x_{n_1,L}-x_{n_2,L})_L\rangle^{-d}\ldots\langle (x_{n_{k-1},L}-x_{n_k,L})_L\rangle^{-d}.
\end{multline}

\medskip\noindent
\step2 Conclusion.\\
The starting point is the estimate~\eqref{e.DGV-control} in Theorem~\ref{thm:expansion} for the cluster remainder,
\begin{equation}\label{e.DGV-control-re}
|E:R_{L}^{k+1}E|\,\lesssim\,A_k^\circ+\sum_{j=1}^kA_{j,k},
\end{equation}
in terms of
\begingroup\allowdisplaybreaks
\begin{eqnarray}
\hspace{-1cm}A_{k}^\circ\!\!&:=&\!\!\expecM{L^{-d}\sum_{n}\int_{I_{n,L}}\Big|\sum_{\sharp F=k\atop n\notin F}\D(\del^{F}\psi_{L}^\varnothing)\Big|^2},\label{eq:def-Tk0}\\
\hspace{-1cm}A_{j,k}\!\!&:=&\!\!\E\bigg[L^{-d}\sum_{n}\bigg(\int_{I_{n,L}}\Big|\sum_{\sharp F=k\atop n\notin F}\D(\del^{F}\psi_{L}^\varnothing)\Big|^2\bigg)^\frac12\nonumber\\
&&\hspace{2cm}\times\bigg(\int_{I_{n,L}+\rho B}\Big|\sum_{\sharp F=j-1\atop n\notin F}\D\big(\del^{F}(\psi^{\{n\}}_{L}+Ex)\big)\Big|^2\bigg)^\frac12\bigg].\nonumber
\end{eqnarray}
We shall prove for all $1\le j\le k$,
\begin{eqnarray}
A_k^\circ&\lesssim_k&\sum_{l=0}^{k}\lambda_{k+l+1}(\Pc) (\log L)^{2l},\label{eq:bnd-Tk0}\\
A_{j,k}&\lesssim_k&\sum_{l=0}^{j-1}\lambda_{k+l+1}(\Pc) (\log L)^{2l+k-j+1}.\label{eq:bnd-Tk0+}
\end{eqnarray}
\endgroup
Inserting this into~\eqref{e.DGV-control-re}, the conclusion~\eqref{e.non-unif-bd-Bb-R} follows.
We split the proof into two further substeps, separately proving~\eqref{eq:bnd-Tk0} and~\eqref{eq:bnd-Tk0+}.

\medskip
\substep{2.1} Proof of~\eqref{eq:bnd-Tk0}.\\
Let $k\ge1$. The deterministic bound~\eqref{eq:bnd-delphi-1} yields
\begin{equation}\label{eq:sum-delF-decay}
\sum_{\sharp F=k}\Big(\int_{I_{n,L}}|\!\D(\del^F\psi_{L}^\varnothing)|^2+|\!\D(\del^F\psi_{L}^{\{n\}})|^2\Big)^\frac12
\,\lesssim_{k}\,\sum_{n_1,\ldots,n_k}^{\ne}D_L(x_{n,L},x_{n_1,L},\ldots,x_{n_k,L}),
\end{equation}
where we have set
\[D_L(y_0,y_{1},\ldots,y_{k})\,:=\,\prod_{j=0}^{k-1}\langle(y_j-y_{j+1})_L\rangle^{-d}.\]
Inserting this in the definition~\eqref{eq:def-Tk0} of $A_k^\circ$, expanding the square, separating the different intersection patterns, and reformulating in terms of multi-points densities, cf.~\eqref{eq:def-fj}, we are led to
\begin{multline*}
A_k^\circ\,\lesssim_k\,\sum_{l=0}^kL^{-d}\int_{{(Q_L)}^{k+l+1}}
D_L(x,x_1,\ldots,x_k)
D_L(x,x_1,\ldots,x_{k-l},y_1,\ldots,y_l)\\
\times f_{k+l+1}(x,x_1,\ldots,x_k,y_1,\ldots,y_l)\,dx\,dx_1\ldots dx_k\,dy_1\ldots dy_l,
\end{multline*}
hence, in terms of multi-point intensities, appealing to Lemma~\ref{lem:gen-cond-lambd}(iii),
\begin{multline*}
A_k^\circ\,\lesssim_k\,\sum_{l=0}^k\lambda_{k+l+1}(\Pc)\,L^{-d}\int_{{(Q_L)}^{k+l+1}}
D_L(x,x_1,\ldots,x_k)\\
\times D_L(x,x_1,\ldots,x_{k-l},y_1,\ldots,y_l)\,
dx\,dx_1\ldots dx_k\,dy_1\ldots dy_l.
\end{multline*}
First evaluating integrals over $x_{k-l+1},\ldots,x_k,y_1,\ldots,y_l$, and noting that
\begin{equation*}
\int_{Q_L}\langle(x-y)_L\rangle^{-d}\,dy\,\lesssim\,\log L,
\end{equation*}
we find
\begin{equation*}
A_k^\circ\,\lesssim_k\,\sum_{l=0}^k\lambda_{k+l+1}(\Pc)(\log L)^{2l}L^{-d}\int_{{(Q_L)}^{k-l+1}}
D_L(x,x_1,\ldots,x_{k-l})^2\,dx\,dx_1\ldots dx_{k-l}.
\end{equation*}
Now evaluating the remaining integrals, noting that the square yields an integrable decay,
the claim~\eqref{eq:bnd-Tk0} follows.

\medskip
\substep{2.2} Proof of~\eqref{eq:bnd-Tk0+}.\\
Let $k\ge j\ge1$.
Inserting~\eqref{eq:sum-delF-decay} into the definition~\eqref{eq:def-Tk0} of $A_{j,k}$, expanding the square, and separating the different intersection patterns, we now find
\begin{multline*}
A_{j,k}\,\lesssim_k\,\sum_{l=0}^{j-1}L^{-d}\int_{{(Q_L)}^{k+l+1}}
D_L(x,x_1,\ldots,x_k)
D_L(x,x_1,\ldots,x_{j-l-1},y_1,\ldots,y_l)\\
\times f_{k+l+1}(x,x_1,\ldots,x_k,y_1,\ldots,y_l)\,dx\,dx_1\ldots dx_k\,dy_1\ldots dy_l,
\end{multline*}
where for notational convenience we define $D_L(x):=1$.
This integral can be evaluated exactly as in the proof of~\eqref{eq:bnd-Tk0} and the claim~\eqref{eq:bnd-Tk0+} follows.
\end{proof}

\newpage

\section{Renormalization of cluster formulas
}\label{sec:intermezzo}

This section is devoted to the proof of infinite-volume cluster estimates with optimal dependence on multi-point intensities $\{\lambda_j(\Pc)\}_j$.
It amounts to improving on the non-uniform cluster estimates~\eqref{e.non-unif-bd-Bb-R} in Theorem~\ref{thm:bounds}, which captures the `short-range' dependence on multi-point intensities but displays a logarithmic divergence in the large-volume limit.
This requires a better understanding of cluster formulas and of the underlying compensations that make them well-defined in the large-volume limit.

\subsection{Main results}
We explore two different routes for the renormalization of infinite-volume cluster formulas, leading to two complementary results, cf.~Theorems~\ref{prop:quant} and~\ref{th:renorm-B23} below. We also discuss the optimality of our cluster estimates, cf.~Theorem~\ref{lem:optimality}.

\subsubsection{Implicit renormalization}
Our first route relies on a slight algebraic quantification of the convergence of periodic approximations, cf.~assumption~\ref{QA} below: it implies a corresponding convergence rate for periodized cluster formulas, cf.~\eqref{e:quant-conv-BjL} below, which in turn allows to remove the logarithmic divergence in the non-uniform cluster estimates of Theorem~\ref{thm:bounds}.
This result is particularly general given that the quantitative periodization assumption~\ref{QA} holds under a mere algebraic $\alpha$-mixing condition for $\Ic$, cf.~Remark~\ref{rem:Mix} below.
The obtained cluster estimates~\eqref{eq:bnd-BLj-D2} differ from the canonical short-range setting of Lemma~\ref{lem:short} by some logarithmic factors, which are expected to be optimal in general in link with the long-range nature of hydrodynamic interactions, cf.~Theorem~\ref{lem:optimality} below. The proof is displayed in Section~\ref{sec:append}.

\begin{theor1}[Implicit renormalization of cluster formulas]\label{prop:quant}
On top of Assumptions~\ref{H0} and~\ref{Gunif},
let the following hold:
\begin{enumerate}[\quad~~~(a)]
\renewcommand{\labelenumi}{\emph{\textbf{(QPE)}}}
\renewcommand{\theenumi}{{\textbf{(QPE)}}}
\item\label{QA} \emph{Quantitative periodization assumption:} There exist $C,\gamma>0$ such that we have $|\Bb_L^{(p)}-\Bb^{(p)}|\,\le\,CL^{-\gamma}$ for all~$L\ge 1$ and $p\in[0,1]$, where $\Bb_L^{(p)},\Bb^{(p)}$ refer to the random deletion procedure introduced in Section~\ref{sec:proof-eq:form-Bj}, cf.~\eqref{eq:p-decim-def}.
\end{enumerate}
Then, we have the following estimates for the coefficients and the remainder of the infinite-volume cluster expansion defined by~\eqref{eq:form-Bj} in Theorem~\ref{thm:bounds}: for all $k, j\ge1$,
\begin{eqnarray}
|\Bb^j|&\lesssim_{j}&\lambda_j(\Pc)|\!\log \lambda_j(\Pc)|^{j-1},\label{eq:bnd-BLj-D2}
\\
|R^{k+1}| &\lesssim_k &\sum_{l=k+1}^{2k+1}\lambda_{l}(\Pc) |\!\log \lambda_{k+1}(\Pc)|^{l-1}.\nonumber
\end{eqnarray}
In addition, the convergence result~\eqref{eq:form-Bj} for finite-volume approximations can be quantified: for all $L$ and $k, j\ge1$,
\begin{equation}\label{e:quant-conv-BjL}
|\Bb^j_L-\Bb^j|\,\lesssim_j\,L^{-2^{-j}\gamma},\qquad
|R^{k+1}_L-R^{k+1}|\,\lesssim_k\,L^{-2^{-k}\gamma},
\end{equation}
where $\gamma$ is the exponent in~\emph{\ref{QA}}.
\end{theor1}

\begin{rem}[Quantitative periodization assumption]\label{rem:Mix}
The validity of Assumption~\ref{QA} can be shown to follow
from a slight quantitative mixing condition for the inclusion process~$\Ic$, such as the following:
\begin{enumerate}[\quad~~(a)]
\renewcommand{\labelenumi}{{\textbf{(Mix)}}}
\renewcommand{\theenumi}{{\textbf{(Mix)}}}
\item\label{Mix} \emph{Algebraic $\alpha$-mixing condition:}
There exist $C,\beta>0$ such that for all Borel subsets~$U,V\subset\R^d$ and all events $A\subset\sigma(\Ic|_U)$ and $B\in\sigma(\Ic|_V)$ we have
\begin{equation}\label{e.alpha-mixing-rd-del}
|\pr{A\cap B}-\pr A\pr B\!|\,\le\,C\dist(U,V)^{-\beta}.
\end{equation}
\end{enumerate}
More precisely, this condition~\ref{Mix} implies the validity of~\ref{QA} for some \mbox{$0<\gamma \ll \beta$} (depending on $\beta,d$) and for all $0\le p \le 1$ (since random deletion preserves~\eqref{e.alpha-mixing-rd-del}).
This follows  by-now from standard quantitative homogenization theory: we refer to Appendix~\ref{sec:quant}, where we adapt the techniques developed by Armstrong, Kuusi, Mourrat, and Smart~\cite{AS,Armstrong-Mourrat-16,AKM-book} to the present fluid context.
\end{rem}

The above result provides optimal cluster estimates and its proof is extremely short, cf.~Section~\ref{sec:append}.
Yet, it has three main disadvantages, which call for a more detailed analysis.
\begin{enumerate}[---]
\item {\it No explicit renormalization:} While infinite-volume cluster formulas take the form of diverging series, cf.~Section~\ref{sec:need-renorm}, cluster coefficients are defined as limits of finite-volume approximations, cf.~\eqref{eq:form-Bj}.
Using straightforward cancellations, we showed that the first-order cluster coefficient $\Bb^1$ can be represented by a summable integral, cf.~Proposition~\ref{prop:B1}. A similar explicit renormalization was formally performed for the second-order coefficient $\Bb^2$ by Batchelor and Green~\cite{BG-72}, based on more subtle cancellations. The implicit renormalization approach sheds no light on such questions.
We aim to recover the Batchelor--Green renormalized formula for $\Bb^2$ rigorously, as also discussed in~\cite{GVH,GVM-20,GV-20}, and to investigate how explicit renormalizations can be pursued to higher orders.
\smallskip
\item {\it Mixing assumption:} In view of cluster formulas in Theorem~\ref{thm:expansion}, bounds on the cluster coefficient $\Bb^j$ should only require assumptions on the $j$-point density. Likewise, in view of~\eqref{e.DGV-control}, bounds on the remainder $R_L^{k+1}$ should only require assumptions on the $2k$-point density.
Instead, assumptions~\ref{QA} and~\ref{Mix} boldly involve the whole law of the inclusion process $\Ic$, which we aim to refine.
\smallskip
\item {\it Convergence rates:} As the above approach builds on a convergence rate for periodic approximations of the effective  viscosity $\Bb$, cf.~\ref{QA}, it does not exploit the fact that cluster formulas only involve a finite number of particles at a time and are thus significantly simpler than $\Bb$ itself. In particular, convergence rates for periodic approximations of cluster coefficients are not expected to be worse than for approximations of~$\Bb$ (on the contrary!), while the above result~\eqref{e:quant-conv-BjL} displays an exponential degradation of the rates for higher-order coefficients.
\end{enumerate}

\subsubsection{Explicit renormalization}
Our second route to renormalization of cluster formulas aims to remedy the above three issues and we proceed by an explicit analysis of cancellations.
As in Proposition~\ref{prop:B1}, we assume for convenience that particles have independent shapes, cf.~\ref{B1}, which makes cluster formulas somewhat simpler.
While for $\Bb^1$ and~$\Bb^2$ relatively simple cancellations are enough to turn cluster formulas into summable integrals, higher-order coefficients require a much deeper analysis:
we are led to introducing a diagrammatic decomposition of corrector differences that allows to capture relevant cancellations.
This fully resolves the higher-order renormalization question that was still open in the physics community.
We refer in particular to Section~\ref{sec:explicit} for an explicit display of renormalized formulas for~$\Bb^2$ and~$\Bb^3$, cf.~Proposition~\ref{prop:B2-ren} and~\ref{prop:B3-ren}: we recover the Batchelor--Green formula for~$\Bb^2$ and provide the first renormalized formula for~$\Bb^3$.
Incidentally, these results only require assumptions on finite-order multi-point densities (instead of mixing assumptions) and Dini-type decay (instead of algebraic), which is beyond the reach of quantitative homogenization methods (and thus of our implicit renormalization).
Renormalized formulas allow to recover the same cluster estimates~\eqref{eq:bnd-BLj-D2} as obtained above via implicit renormalization and to further prove essentially optimal convergence rates for finite-volume approximations: the convergence rate~\eqref{eq:pr-convBL2} for~$\Bb^j$ below only degrades logarithmically when increasing~$j$ (as opposed to the exponential degradation in~\eqref{e:quant-conv-BjL}), and it is always better (as it should) than the rate for approximations of the effective viscosity~$\Bb$ itself (cf.~$\gamma\ll\beta$ in Remark~\ref{rem:Mix}).
The proof is displayed in Section~\ref{sec:explicit}.

\begin{theor1}[Explicit renormalization of cluster formulas]\label{th:renorm-B23}
On top of Assumptions~\ref{H0} and~\ref{Gunif}, let the independence assumption~\emph{\ref{B1}} hold for particle shapes, as well as the following, for some rate $\omega\in C^\infty_b(\R^+)$:
\begin{enumerate}[\quad\,~~~(a)]
\renewcommand{\labelenumi}{\emph{\textbf{(Mix$_\omega$)}}}
\renewcommand{\theenumi}{{\textbf{(Mix$_\omega$)}}}
\item\label{Mix-om} \emph{$\alpha$-Mixing assumption with rate $\omega$:}
For all Borel subsets $U,V\subset\R^d$ and all events $A\subset\sigma(\Ic|_U)$ and $B\in\sigma(\Ic|_V)$, we have
\begin{equation*}
\qquad|\pr{A\cap B}-\pr A\pr B\!|\,\le\,\omega(\dist(U,V)).
\end{equation*}
\end{enumerate}
Then, the following hold.
\begin{enumerate}[(i)]
\item For all $j\ge2$, provided $\omega$ satisfies the Dini type condition \mbox{$\int_1^\infty t^{-1}(\log t)^{j-2}\omega(t)\,dt<\infty$}, the infinite-volume cluster coefficient $\Bb^j$ can be described by means of summable integrals as detailed in Section~\ref{sec:explicit}.
\smallskip\item In case of an algebraic mixing rate $\omega(t)\le Ct^{-\beta}$ for some $C,\beta>0$, renormalized formulas lead to the same cluster estimates~\eqref{eq:bnd-BLj-D2} for all $k, j\ge1$.
In addition, the following holds for finite-volume approximations: for all $L$ and $j\ge1$,
\begin{equation}\label{eq:pr-convBL2}
\quad|\Bb^j_L-\Bb^j|\,\lesssim_j\,\tfrac{(\log L)^{j-1}}{L^{\beta\wedge1}}.
\end{equation}
\end{enumerate}
Finally, assumption~\emph{\ref{Mix-om}} can be replaced by corresponding assumptions on the $j$-point density for results on $\Bb^j$, and on the $(2k+1)$-point density for results on $R^{k+1}$.
\end{theor1}

\subsubsection{Optimality of cluster estimates} 
The following result states that logarithmic factors in cluster estimates~\eqref{eq:bnd-BLj-D2} are optimal in general. These factors contrast with the canonical short-range setting of Lemma~\ref{lem:short}: they are related to the long-range nature of hydrodynamic interactions and appear due to the lack of $\Ld^\infty$-boundedness of Calder\'on--Zygmund operators.
We focus on the second-order coefficient $\Bb^2$ for illustration, but, starting from renormalized formulas, the argument could be extended to higher orders as well. The proof is displayed in Section~\ref{sec:opti-B2}.

\begin{theor1}[Optimality of estimates on $\Bb_2$]\label{lem:optimality}$  $
\begin{enumerate}[(i)]
\item \emph{Isotropic setting:}
On top of Assumptions~\ref{H0}, \ref{Gunif}, and~\emph{\ref{B1}}, assume that the $2$-point correlation function $h_2(x,y):=f_2(x,y)-\lambda(\Pc)^2$ satisfies the following decay assumption,
\begin{equation}\label{eq:decay-h2-om}
\quad\iint_{B(x)\times B(y)}|h_2|\,\le\,\omega(|x-y|),
\end{equation}
with some rate $\omega$ satisfying the Dini condition $\int_1^\infty t^{-1}\omega(t)\,dt<\infty$.
If in addition the point process~$\Pc$ is statistically isotropic, which entails that the correlation function is radial, then the following improved estimate holds,
\[\quad|\Bb^2|\,\lesssim\, \lambda_2(\Pc).\]
\item \emph{Optimality in the general setting:} There exists an inclusion process $\Ic$ that satisfies Assumptions~\ref{H0}, \ref{Gunif}, \emph{\ref{B1}}, and~\eqref{eq:decay-h2-om}, as well as the local independence condition $\lambda_2(\Pc)\simeq\lambda(\Pc)^2\ll1$, such that we have
\[\quad|\Bb^2|\,\simeq\, \lambda_2(\Pc) |\!\log \lambda_2(\Pc)|.\qedhere\]
\end{enumerate}
\end{theor1}

\subsection{Implicit renormalization of cluster formulas}\label{sec:append}

This section is devoted to the short proof of Theorem~\ref{prop:quant}, which we split into two steps. We start with the quantitative convergence result~\eqref{e:quant-conv-BjL} for finite-volume approximations of cluster coefficients, which we obtain by quantifying the argument for the corresponding qualitative result~\eqref{eq:form-Bj} in Section~\ref{sec:proof-eq:form-Bj}. The claimed cluster estimates~\eqref{eq:bnd-BLj-D2} then follow by optimization.

\medskip
\step1 Suboptimal convergence result: proof of~\eqref{e:quant-conv-BjL}.\\
Starting from the cluster expansion~\eqref{eq:preconcl-exp} in Theorem~\ref{thm:expansion}, the triangle inequality yields for all $k\ge0$,
\[|R^{k+1}_L-R^{k+1}| \le |\Bb_L-\Bb|+\sum_{j=1}^k |\Bb_L^j-\Bb^j|,\]
so that the convergence rate for the remainder in~\eqref{e:quant-conv-BjL} follows from Assumption~\ref{QA} together with the convergence rate for cluster coefficients.
It remains to prove the latter, that is, for all~$j\ge1$,
\begin{equation}\label{eq:conv-BLj-pr}
|\Bb_L^j-\Bb^j|\,\lesssim_j\,L^{-2^{-j}\gamma}.
\end{equation}
For that purpose, we quantify the induction argument in the proof of the corresponding qualitative convergence result~\eqref{eq:form-Bj} in Section~\ref{sec:proof-eq:form-Bj}.
Let $k\ge0$ and assume that~\eqref{eq:conv-BLj-pr} holds for all $1\le j\le k$.
Taking the same notation as in Section~\ref{sec:proof-eq:form-Bj} for the random deletion procedure, we recall the cluster expansion~\eqref{eq:pre-concl-anal},
for all~$L,p$,
\[\bigg|\Bb_{L}^{(p)}-\Big(\Id+\sum_{j=1}^{k+1}\tfrac{p^j}{j!}\Bb^j_L\Big)\bigg|\,\le\,(Cp)^{k+2}.\]
Hence, comparing to the corresponding estimate in the large-volume limit, we find
\[\bigg|(\Bb_{L}^{(p)}-\Bb^{(p)})-\sum_{j=1}^{k+1}\tfrac{p^j}{j!}(\Bb^j_L-\Bb^j)\bigg|\,\le\,(Cp)^{k+2}.\]
Isolating the difference $\Bb_L^{k+1}-\Bb^{k+1}$, and using Assumption~\ref{QA} and the induction hypothesis to estimate other contributions, we deduce
\begin{eqnarray*}
|\Bb_L^{k+1}-\Bb^{k+1}| &\le  &\tfrac{(k+1)!}{p^{k+1}}\bigg((Cp)^{k+2}+|\Bb^{(p)}_L-\Bb^{(p)}|+\sum_{j=1}^{k}\tfrac{p^j}{j!}|\Bb^{j}_L-\Bb^j|\bigg)\\
&\lesssim_{k}& p+\sum_{j=0}^k p^{j-k-1} L^{-2^{-j}\gamma} .
\end{eqnarray*}
The choice $p=L^{-2^{-k-1}\gamma}$ then yields $|\Bb^{k+1}_L-\Bb^{k+1}|\lesssim_{k}L^{-2^{-k-1}\gamma}$, and the claim~\eqref{eq:conv-BLj-pr} follows by induction for all $j\ge1$.

\medskip
\step2 Uniform cluster estimates: proof of~\eqref{eq:bnd-BLj-D2}.\\
Combining the non-uniform estimates~\eqref{e.non-unif-bd-Bb-R} of Theorem~\ref{thm:bounds}
with the suboptimal convergence result~\eqref{e:quant-conv-BjL}, we find for all $k\ge j\ge1$,
\begin{eqnarray*}
|\Bb^j|&\lesssim_j&L^{-2^{-j}\gamma}+\lambda_j(\Pc)(\log L)^{j-1},\\
|R^{k+1}|&\lesssim_j&L^{-2^{-k}\gamma}+\sum_{l=k}^{2k}\lambda_{l+1}(\Pc)(\log L)^{l},
\end{eqnarray*}
and the conclusion~\eqref{eq:bnd-BLj-D2} follows from the choice $L^{-2^{-j}\gamma}=\lambda_j(\Pc)$ or $L^{-2^{-k}\gamma}=\lambda_{k+1}(\Pc)$, respectively.
\qed

\subsection{Preliminary to explicit renormalization}\label{sec:explicit-prel}
Before turning to the explicit renormalization of cluster formulas and to the proof of Theorem~\ref{th:renorm-B23}, we start with some preliminary definitions and technical tools: we define multi-point correlation functions, which provide a convenient framework to weaken the $\alpha$-mixing condition, we revisit the decomposition~\eqref{eq:cordiff-JLHn} for corrector differences in terms of elementary single-particle contributions, and we state several crucial estimates on the latter.
 
\subsubsection{Multi-point correlation functions}
Multi-point correlation functions $\{h_j\}_j$ of the point process $\Pc$ can be defined inductively from the multi-point densities $\{f_j\}_j$, cf.~\eqref{eq:def-fj}, via the following relations:\footnote{Incidentally, these relations
are known as Mayer's {\it cluster expansions} --- although unrelated to the kind of cluster expansions otherwise studied in this work.}
for all $j\ge1$,
\begin{equation}\label{eq:dens-correl}
f_j(x_1,\ldots,x_j)=\sum_\pi\prod_{H\in\pi}h_{\sharp H}(x_H),
\end{equation}
where $\pi$ runs over all partitions of the index set $\{1,\ldots,j\}$, where $H$ runs over all cells of the partition $\pi$,
and where for $H=\{i_1,\ldots,i_l\}$ we set $x_H:=(x_{i_1},\ldots,x_{i_l})$. For the first values of $k$, these relations read
\begin{eqnarray*}
f_1(z)&=&h_1(z)~\,=\,~\lambda(\Pc),\\
f_{2}(y,z)&=&\lambda(\Pc)^2+h_{2}(y,z),\\
f_{3}(x,y,z)&=&\lambda(\Pc)^3+\lambda(\Pc)\big(h_2(x,y)+h_2(y,z)+h_2(z,x)\big)+h_3(x,y,z),
\end{eqnarray*}
from which $h_1,h_2,h_3$ are easily extracted.
More generally, note that the inductive definition~\eqref{eq:dens-correl} can be explicitly inverted: for all $j\ge1$, we find
\begin{equation}\label{e.inver-cor-func}
h_j(x_1,\ldots,x_j)\,:=\,\sum_{\pi}(\sharp\pi-1)!\,(-1)^{\sharp\pi-1}\prod_{H\in\pi}f_{\sharp H}(x_H),
\end{equation}
where $\pi$ runs over all partitions of the index set $\{1,\ldots,j\}$ and where $\sharp\pi$ stands for the number of cells $H\in\pi$.
The $j$-point correlation function~$h_j$ is thus a symmetric function on the product $(\R^d)^j$ and is a polynomial combination of multi-point densities $(f_i)_{i\le j}$.
The definition of multi-point intensities~\eqref{eq:high-intens/Re} then entails the following bounds on correlations, for all~$j\ge1$,
\begin{equation}\label{eq:high-intens-correl}
\sup_{z_1,\ldots,z_j}\fint_{Q_\ell(z_1)\times\ldots\times Q_\ell(z_j)}|h_j|\,\lesssim_j\,\overline \lambda_j(\Pc),
\end{equation}
where we recall the notation~\eqref{e.mul-int-eff}.
It is easily checked that the $\alpha$-mixing assumption~\ref{Mix-om} implies the decay of correlation functions in the following quantitative sense. Since we could not find any precise reference in the literature, we include a short proof below for completeness.

\begin{lem}\label{lem:decay-correl}
Assume that the point process $\Pc$ satisfies the $\alpha$-mixing condition~\emph{\ref{Mix-om}} with a non-increasing rate $\omega\in C^\infty_b(\R^+)$. Then, correlation functions satisfy for all $j\ge2$ and $x_1,\ldots,x_j\in\R^d$,
\begin{equation}\label{eq:decay-correl}
\int_{B(x_1)\times\ldots\times B(x_j)}|h_j|\,\le\,C^jj!\,\min_{i\ne l}\omega\big((\tfrac1{j}|x_i-x_l|-2)_+\big).\qedhere
\end{equation}
\end{lem}

In this view, it is natural to consider a ``truncated'' version of the $\alpha$-mixing condition~\ref{Mix-om} in form of the decay of a finite number of correlation functions only. This is the natural setting for cluster estimates.
\begin{enumerate}[\quad\,~~~(a)]
\renewcommand{\labelenumi}{\textbf{(Mix$_\omega^n$)}}
\renewcommand{\theenumi}{{\textbf{(Mix$_\omega^n$)}}}
\item\label{Mix-om-n} \emph{Mixing assumption with rate $\omega$ to order $n$:} Multi-point correlation functions satisfy for all $2\le j\le n$ and $x_1,\ldots,x_j\in\R^d$,
\begin{equation*}
\qquad\int_{B(x_1)\times\ldots\times B(x_j)}|h_j|\,\le\,\min_{i\ne l}\omega(|x_i-x_l|).
\end{equation*}
\end{enumerate}

\begin{proof}[Proof of Lemma~\ref{lem:decay-correl}]
We argue by induction: given $j\ge2$, we assume that the claimed decay estimate~\eqref{eq:decay-correl} is already known to hold for $h_2,\ldots,h_{j-1}$, and we prove that it also holds for $h_j$.
Let $x_1,\ldots,x_j\in\R^d$ be fixed.
The conclusion~\eqref{eq:decay-correl} is trivial when $\max_{i\ne l}\frac1j|x_i-x_l|\le2$, and we may thus assume $\max_{i\ne l}\tfrac1j|x_i-x_l|>2$.
Up to relabeling the points, we may further assume that there is $1\le j_*< j$ such that
\begin{gather}
\textstyle|x_1-x_{j}|=\max_{i\ne l}|x_i-x_l|,\nonumber\\
|x_i-x_l|\ge\tfrac1j|x_1-x_j|\,>\,2\quad\text{for all $1\le i\le j_*<l\le j$.}\label{eq:prop-x1-j}
\end{gather}
(The latter condition is obtained by dividing the space between $x_1$ and $x_j$ into $j$ stripes of width $\tfrac1j|x_1-x_j|$, by selecting the one that contains none of the points $x_i$'s with $1< i< j$, and by distinguishing the points on either side of this stripe.)
Let \mbox{$\phi\in C((\R^d)^{j_*})$} and \mbox{$\phi'\in C((\R^d)^{j-j_*})$} be supported in $B(x_1)\times\ldots\times B(x_{j_*})$ and in~$B(x_{j_*+1})\times\ldots\times B(x_{j})$, respectively, with $\|\phi\|_{\Ld^\infty((\R^d)^{j_*})}=\|\phi'\|_{\Ld^\infty((\R^d)^{j-j_*})}=1$.
Appealing to a standard covariance inequality, see e.g.~\cite[Lemma~1.2.3]{Doukhan-94}, the $\alpha$-mixing condition~\ref{Mix-om} then yields
\begin{multline}\label{eq:alphamix-appl}
{\bigg|\covM{\sum_{n_1,\ldots,n_{j_*}}^{\ne}\phi(x_{n_1},\ldots,x_{n_{j_*}})}{\sum_{n_{j_*+1},\ldots,n_{j}}^{\ne}\phi'(x_{n_{j_*+1}},\ldots,x_{n_{j}})}\bigg|}\\
~\le~4\,\omega\bigg(\dist\Big(\bigcup_{i=1}^{j_*}B(x_i),\bigcup_{i=j_*+1}^{j}B(x_i)\Big)\bigg)
~\le~4\,\omega\big(\tfrac1j|x_1-x_j|-2\big).
\end{multline}
Now we expand the covariance in terms of multi-point densities:
in view of~\eqref{eq:prop-x1-j} and
of the support condition for $\phi,\phi'$, we find that the product $\phi(x_{n_1},\ldots,x_{n_{j_*}})\phi'(x_{n_{j_*+1}},\ldots,x_{n_{j}})$ vanishes whenever $n_i=n_l$ for some $1\le i\le j_*<l\le j$, hence
\begin{multline}\label{eq:alphamix-appl-re}
\covM{\sum_{n_1,\ldots,n_{j_*}}^{\ne}\phi(x_{n_1},\ldots,x_{n_{j_*}})}{\sum_{n_{j_*+1},\ldots,n_{j}}^{\ne}\phi'(x_{n_{j_*+1}},\ldots,x_{n_{j}})}\\
~=~\int_{(\R^d)^j}(\phi\otimes\phi')\,(f_j-f_{j_*}\otimes f_{j-j_*}).
\end{multline}
Recalling the relation~\eqref{eq:dens-correl} for density functions in terms of correlations, we get
\[(f_j-f_{j_*}\otimes f_{j-j_*})(z_1,\ldots,z_j)\,=\,\sum_\pi\mathds1_{\exists H\in\pi:H\cap\{1,\ldots,j_*\}\ne\varnothing\ne H\cap\{j_*+1,\ldots,j\}}\prod_{H\in\pi}h_{\sharp H}(z_H).\]
Combining this with~\eqref{eq:alphamix-appl} and~\eqref{eq:alphamix-appl-re}, and isolating the contribution of the $j$-point correlation~$h_j$ (obtained for $\sharp\pi=1$), we are led to
\begin{multline*}
\Big|\int_{(\R^d)^j}(\phi\otimes\phi')h_j\Big|\,\le\,
\sum_{\pi:\sharp\pi>1}\mathds1_{\exists H\in\pi:H\cap\{1,\ldots,j_*\}\ne\varnothing\ne H\cap\{j_*+1,\ldots,j\}}\prod_{H\in\pi}\int_{B(x_H)}|h_{\sharp H}|\\
+4\,\omega\big(\tfrac1j|x_1-x_j|-2\big),
\end{multline*}
where for $H=\{i_1,\ldots,i_l\}$ we set $B(x_H):=B(x_{i_1})\times\ldots\times B(x_{i_l})$.
In view of~\eqref{eq:prop-x1-j}, the induction hypothesis for $\{h_l\}_{l<j}$ entails
\begin{multline*}
\Big|\int_{(\R^d)^j}(\phi\otimes\phi')h_j\Big|\,\le\,
\sum_{\ell=2}^j\sum_{i_1+\ldots+i_\ell=j}\binom{j}{i_1,\ldots,i_\ell}\prod_{s=1}^\ell \big(C^{i_s}i_s!\,\omega(\tfrac1j|x_1-x_j|-2)^{i_s}\big)\\
+4\,\omega\big(\tfrac1j|x_1-x_j|-2\big),
\end{multline*}
from which we easily infer $|\int_{(\R^d)^j}(\phi\otimes\phi')h_j|\le C^jj!\,\omega(\frac1j|x_1-x_j|-2)$. By the arbitrariness of $\phi,\phi'$ and of $x_1,\dots,x_j$, the conclusion~\eqref{eq:decay-correl} follows for $h_j$.
\end{proof}

\subsubsection{Estimates on single-particle contributions}\label{sec:Stokeslet-est-rig}
For notational simplicity, we henceforth assume that particles are spherical with unit radius, \mbox{$I_n=B(x_n)$}; the adaptation to the general case~\ref{B1} with independent particle shapes is straightforward.
As we shall see, the explicit renormalization of $\Bb^j$ is particularly intricate for~\mbox{$j\ge3$} since cancellations are not as apparent as they are for the first two orders: it will require to decompose corrector differences into elementary single-particle contributions in the spirit of~\eqref{eq:cordiff-JLHn}.
We start by slightly changing the point of view for correctors, focussing on particle positions rather than on particle indices in the notation:
given a set $Y\subset Q_L$ of ``background'' positions 
such that
\begin{equation}\label{eq:dist-Y}
\dist(B(y),B(y'))\,>\,2\rho,\qquad\dist(B(y),\partial Q_L)\,>\,\rho,\qquad\text{for all $y,y'\in Y,~ y\ne y'$},
\end{equation}
we denote by $\psi_L^Y\in H^1_\per(Q_L)^d$
the solution of the following periodic corrector problem, using the short-hand notation $\sigma_L^Y:=\sigma(\psi_L^Y+Ex,\Sigma_L^Y)$,
\[\left\{\begin{array}{ll}
-\triangle \psi_L^Y+\nabla\Sigma_L^Y=0,&\text{in $ Q_L\setminus\cup_{y\in  Y}B(y)$},\\
\Div(\psi_L^Y)=0,&\text{in $ Q_L\setminus\cup_{y\in  Y}B(y)$},\\
\D(\psi_L^Y+Ex)=0,&\text{in $\cup_{y\in  Y}B(y)$},\\
\int_{\partial B(y)}\sigma_L^Y\nu=0,&\forall y\in  Y,\\
\int_{\partial B(y)}\Theta(x-y)\cdot\sigma_L^Y\nu=0,&\forall \Theta\in\Md^\Skew,~\forall y\in  Y.
\end{array}\right.\]
Next, similarly as in~\eqref{eq:def-diff}, for any $z\in Q_L$ and any finite subset~\mbox{$Z\subset Q_L$}, provided that the union set $\{z\}\cup Z\cup Y$ satisfies~\eqref{eq:dist-Y}, we can define corrector differences
\[\delta^{\{z\}}\psi_L^{Y}\,:=\,\psi_L^{\{z\}\cup Y}-\psi_L^{Y},\qquad \delta^{Z}\psi_L^{Y}\,:=\,\sum_{W\subset Z}(-1)^{|Z\setminus W|}\psi_L^{W\cup Y}.\]
Compared with the notation that we use elsewhere in this memoir,
this means for all index sets~$F,H\subset\N$,
\[\psi_L^H\,\equiv\,\psi_L^{\{x_{n,L}\}_{n\in H}},\qquad\delta^F\psi_L^H\,\equiv\,\delta^{\{x_{n,L}\}_{n\in F}}\psi_L^{\{x_{n,L}\}_{n\in H}}.\]
For $Y=\{y_1,\ldots,y_m\}$ and $Z=\{z_1,\ldots,z_n\}$, we shall also write for convenience
\begin{equation}\label{eq:corr-xy-notation}
\psi_L^{y_1,\ldots,y_m}\,:=\,\psi_L^{Y},\qquad \delta^{z_1,\ldots,z_n}\psi_L^{y_1,\ldots,y_m}\,:=\,\delta^Z\psi_L^Y.
\end{equation}
Recall that Lemma~\ref{eq:cordiff-form} states that the corrector difference $\del^{Z}\psi_{L}^Y$ satisfies
\begin{equation}\label{eq:corr-diff-repeat}
-\triangle\del^{Z}\psi_{L}^Y+\nabla\del^{Z}\Sigma_{L}^Y
\,=\,-\sum_{z\in Z}\delta_{\partial B(z)}\del^{Z\setminus\{z\}}\sigma^{Y\cup\{z\}}_{L}\nu \qquad \text{ in } Q_L\setminus\cup_{y\in  Y}B(y),
\end{equation}
together with the rigidity contraint $\D(\del^{Z}\psi_{L}^Y)=0$ in $\cup_{y\in Y}B(y)$ and with associated boundary conditions.
In view of this equation, as in~\eqref{eq:cordiff-JLHn}, we can decompose corrector differences into elementary single-particle contributions that we express in terms of operators $\{\Jc_{L;Y}^z\}_{z,Y}$ defined as follows: given a ``tagged'' position~$z\in Q_L$,
given a pair $(\zeta,P)\in H^1(B_{1+\rho}(z))^d\times\Ld^2(B_{1+\rho}(z)\setminus B(z))$ satisfying the following Stokes equations in a neighborhood of $B(z)$,
\begin{equation}\label{eq:zetaP-re}
\left\{\begin{array}{ll}
-\triangle\zeta+\nabla P=0,&\text{in $B_{1+\rho}(z)\setminus B(z)$},\\
\Div(\zeta)=0,&\text{in $B_{1+\rho}(z)\setminus B(z)$},\\
\D(\zeta)=0,&\text{in $B(z)$},\\
\int_{\partial B(z)}\sigma(\zeta,P)\nu=0,&\\
\int_{\partial B(z)}\Theta(x-z)\cdot\sigma(\zeta,P)\nu=0,&\forall\Theta\in\Md^\Skew,
\end{array}\right.
\end{equation}
and given a finite subset $Y\subset Q_L$ of ``background'' positions such that $\{z\}\cup Y$ satisfies~\eqref{eq:dist-Y},
we denote by $\Jc_{L;Y}^z\zeta\in H^1_\per(Q_L)^d$ the solution of the following Stokes problem,
\begin{equation*}
\left\{\begin{array}{ll}
-\triangle \Jc^{z}_{L;Y}\zeta+\nabla \Qc^{z}_{L;Y}\zeta=-\delta_{\partial B(z)}\sigma(\zeta,P)\nu,&\text{in $ Q_L\setminus \cup_{y\in Y}B(y)$},\\
\Div(\Jc^{z}_{L;Y}\zeta)=0,&\text{in $ Q_L\setminus \cup_{y\in Y}B(y)$},\\
\D(\Jc^{z}_{L;Y}\zeta)=0,&\text{in $\cup_{y\in Y}B(y)$},\\
\int_{\partial B(y)}\sigma(\Jc_{L;Y}^z\zeta,\Qc_{L;Y}^z\zeta)\nu=0,&\forall y\in Y,\\
\int_{\partial B(y)}\Theta (x-y)\cdot\sigma(\Jc_{L;Y}^z\zeta,\Qc_{L;Y}^z\zeta)\nu=0,&\forall\Theta\in\Md^\Skew,~\forall y\in Y.
\end{array}\right.
\end{equation*}
These operators $\{\Jc_{L;Y}^z\}_{z,Y}$ describe the fluid velocity generated by localized force dipoles in the presence of a finite number of rigid inclusions and are thus viewed as Stokeslets for the problem with rigid inclusions.
In view of our upcoming analysis (see in particular cancellation properties in Lemma~\ref{lem:new-cancel} below), we further extend the definition of $\Jc_{L;Y}^z$
when the support~$B(z)$ of the force dipole intersects rigid inclusions $\cup_{y\in Y}B(y)$ or the cell boundary~$\partial Q_L$, which was excluded above by assuming that $\{z\}\cup Y$ satisfies~\eqref{eq:dist-Y}. A convenient way to proceed is as follows: given $z\in Q_L$ and $Y\subset Q_L$ with only $Y$ satisfying~\eqref{eq:dist-Y}, we define $\Jc_{L;Y}^z\zeta\in H^1_\per(Q_L)^d$ as the solution of the following Stokes problem,
\begin{equation}\label{eq:def-JLy|z}
\left\{\begin{array}{ll}
-\triangle \Jc^{z}_{L;Y}\zeta+\nabla \Qc^{z}_{L;Y}\zeta=-\delta_{\partial B^L(z)}\sigma(\zeta,P)\nu,&\text{in $ Q_L\setminus \cup_{y\in Y}B(y)$},\\
\Div(\Jc^{z}_{L;Y}\zeta)=0,&\text{in $ Q_L\setminus \cup_{y\in Y}B(y)$},\\
\D(\Jc^{z}_{L;Y}\zeta)=0,&\text{in $\cup_{y\in Y\setminus Y_z}B(y)$},\\
\int_{\partial B(y)}\sigma(\Jc_{L;Y}^z\zeta,\Qc_{L;Y}^z\zeta)\nu=0,&\forall y\in Y\setminus Y_z,\\
\vspace{-0.4cm}&\\
\int_{\partial B(y)}\Theta (x-y)\cdot\sigma(\Jc_{L;Y}^z\zeta,\Qc_{L;Y}^z\zeta)\nu=0,&\forall\Theta\in\Md^\Skew,~\forall y\in Y\setminus Y_z,
\\
\vspace{-0.38cm}&\\
\Jc^{z}_{L;Y}\zeta=V_z+\Theta_z (x-z),&\text{in $\cup_{y\in  Y_z}B(y)$,}\\
&\quad\text{for some $V_z \in \R^d, \Theta_z \in\Md^\Skew$},
\\
\sum_{y\in  Y_z} \int_{\partial B(y)}\sigma(\Jc_{L;Y}^z\zeta,\Qc_{L;Y}^z\zeta)\nu
\\
\hspace{1cm}=\sum_{y\in Y_z}\int_{B(y)\cap\partial B^L(z)}\sigma(\zeta,P)\nu,&
\\
\vspace{-0.38cm}&\\
\sum_{y\in  Y_z} \int_{\partial B(y)}\Theta (x-z)\cdot\sigma(\Jc_{L;Y}^z\zeta,\Qc_{L;Y}^z\zeta)\nu&\\
\hspace{1cm}=\sum_{y\in Y_z}\int_{B(y)\cap\partial B^L(z)}\Theta(x-z)\cdot\sigma(\zeta,P)\nu,&\forall\Theta\in\Md^\Skew,
\end{array}\right.
\end{equation}
where $B^L(z):=(B(z)+L\Z^d)\cap Q_L$ stands for the periodization of the ball $B(z)$ in~$Q_L$,
where we have set $Y_z:= \{y\in Y:B(y)\cap B^L(z)\ne\varnothing\}$, and where we have implicitly extended $(\zeta,P)$ periodically to $B_{1+\rho}(z)+L\Z^d$. We emphasize that these equations are equivalent to the previous simpler ones when $\{z\}\cup Y$ satisfies~\eqref{eq:dist-Y} (hence $Y_z=\varnothing$).
The solution~$\Jc_{L;Y}^z\zeta$ is only defined up to a rigid motion in $ Q_L$, which we fix by further choosing
\[\int_{ Q_L}\Jc_{L;Y}^z\zeta~=~0,\qquad\int_{ Q_L}\nabla\Jc_{L;Y}^z\zeta~\in~\Md_0^\Sym.\]
Note that $\Jc_{L;Y}^z\zeta$  depends of course on the pair~$(\zeta,P)$, not only on $\zeta$, but we leave the pressure field implicit in the notation for convenience. We further define
\[\Jc_{L}^z\zeta\,:=\,\Jc_{L;\varnothing}^z\zeta,\]
for which the defining Stokes problem~\eqref{eq:def-JLy|z} reduces to
\begin{equation}\label{eq:def-JLy}
-\triangle \Jc^{z}_{L}\zeta+\nabla \Qc^{z}_{L}\zeta=-\delta_{\partial B^L(z)}\sigma(\zeta,P)\nu,\qquad
\Div(\Jc^{z}_{L}\zeta)=0,\qquad\text{in $ Q_L$},
\end{equation}
and we define $\Jc_{Y}^z\zeta,\Jc^z\zeta$ as the corresponding operators on whole space, that is, with $B^L(z)$ and $Q_L$ replaced by $B(z)$ and $\R^d$, respectively, in~\eqref{eq:def-JLy|z} and~\eqref{eq:def-JLy}.
In these terms, as in~\eqref{eq:cordiff-JLHn}, given $Y,Z\subset Q_L$, provided that $Y\cup Z$ satisfies~\eqref{eq:dist-Y}, the equation~\eqref{eq:corr-diff-repeat} for corrector differences allows to decompose
\begin{equation}\label{eq:cordiff-JLHn-re}
\delta^Z\psi_L^Y\,=\,\sum_{z\in Z}\Jc_{L;Y}^z\delta^{Z\setminus\{z\}}(\psi_L^{\{z\}\cup Y}+Ex).
\end{equation}
The above definition~\eqref{eq:def-JLy|z} of $\Jc_{L;Y}^z$, with the particular choice of the extension to all~\mbox{$z\in Q_L$}, is dictated by the following key observation. This constitutes the precise cancellation property that we shall repeatedly use for the explicit renormalization of cluster formulas.

\begin{lem}[Cancellation property]\label{lem:new-cancel}
For any $Y\subset Q_L$ satisfying~\eqref{eq:dist-Y}, and for any function~$\zeta$ satisfying~\eqref{eq:zetaP-re} around $z=0$, we have for $\zeta^z:=\zeta(\cdot-z)$,
\begin{equation*}
\int_{Q_L}\big(\Jc_{L;Y}^z\zeta^z\big)\,dz\,=\,0.
\qedhere
\end{equation*}
\end{lem}

\begin{proof}
Integrating equations~\eqref{eq:def-JLy|z} for $\Jc^z_{L;Y}\zeta^z$ over $z$, and noting that
\begin{eqnarray*}
\int_{Q_L}\Big(\sum_{y\in  Y_z}\int_{B(y)\cap \partial B^L(z)}\sigma(\zeta^z,P^z)\nu\Big)\,dz
&=&\sharp Y|B|\int_{\partial B}\sigma(\zeta,P)\nu\\
&=&0,
\end{eqnarray*}
and similarly
\begin{eqnarray*}
\lefteqn{\int_{Q_L}\Big(\delta_{\partial B^L(z) \setminus \cup_{y\in Y} B(y)}\sigma(\zeta^z,P^z)\nu\Big)\,dz}
\\
&=&\int_{Q_L} \int_{\partial B^L(z)}\sigma(\zeta^z,P^z)\nu\,dz-\int_{Q_L}\Big(\sum_{y\in  Y_z}\int_{B(y)\cap \partial B^L(z)}\sigma(\zeta^z,P^z)\nu\Big)\,dz
\\
&=&|Q_L\setminus\cup_{y\in Y} B(y)|\int_{\partial B}\sigma(\zeta,P)\nu\,=\,0,
\end{eqnarray*}
the conclusion follows from the uniqueness of the solution to the Stokes problem~\eqref{eq:def-JLy|z}.
\end{proof}

Next, we establish optimal decay estimates for these operators $\{\Jc_{L;Y}^z\}_{z,Y}$, which are shown to coincide with the decay for the explicit Stokeslets $\{\Jc_L^z\}_z$ associated with the problem in free space without rigid inclusions.
This result corresponds to Lemma~\ref{lem:decay} and the proof is postponed to Appendix~\ref{append:ell} in form of~Lemma~\ref{*lem:decay-re}.

\begin{lem}[Decay of Stokeslets with rigid inclusions]\label{lem:decay-re}
Let $z\in Q_L$, let $(\zeta,P)$ satisfy~\eqref{eq:zetaP-re} at $z$,
and let $Y\subset Q_L$ satisfy~\eqref{eq:dist-Y}.
Then, we have for all~\mbox{$x\in Q_L$},
\begin{align*}
\Big(\int_{B^L(x)}|\!\D(\Jc_{L;Y}^z\zeta)|^2\Big)^\frac12&~\lesssim_{\sharp Y}~\langle(x-z)_L\rangle^{-d}\Big(\int_{B_{1+\rho}(z)}|\!\D(\zeta)|^2\Big)^\frac12,
\\
\Big(\int_{B(x)}|\!\D(\Jc_{Y}^z\zeta)|^2\Big)^\frac12&~\lesssim_{\sharp Y}~\langle x-z\rangle^{-d}\Big(\int_{B_{1+\rho}(z)}|\!\D(\zeta)|^2\Big)^\frac12.\qedhere
\end{align*}
\end{lem}

Finally, since we aim at finite-volume approximation error estimates, we need to quantify the difference
 $\Jc_{L;Y}^z-\Jc_{Y}^z$ between periodized and whole-space Stokeslets.
The proof is postponed to Appendix~\ref{append:ell} in form of~Lemma~\ref{*lem:period-est}.
We emphasize that the stated bounds are not optimal, but will be good enough for our purposes.

\begin{lem}[Periodization error]\label{lem:period-est}
Let $z\in Q_L$, let $(\zeta,P)$ satisfy~\eqref{eq:zetaP-re} at $z$, and let~\mbox{$Y\subset Q_L$} such that $\{z\}\cup Y$ satisfies~\eqref{eq:dist-Y}.
Then, we have for all $x\in Q_{L}$,
\begin{multline*}
\Big(\int_{B^L_{1+\rho}(x)}|\!\D(\Jc_{L;Y}^z\zeta-\Jc_{Y}^z\zeta)|^2\Big)^\frac12
\,\lesssim_{\sharp Y}\,\Big(\int_{B_{1+\rho}(z)}|\!\D(\zeta)|^2\Big)^\frac12\\
\times\Big(\mathds{1}_{|x-z|>\frac L4} \langle (x-z)_L\rangle^{-d}
+\mathds{1}_{|x-z| \le \frac L4}\dist(Y \setminus\{x,z\},\partial Q_L)^{-d}\Big),
 \end{multline*}
where we set for notational convenience $\dist(\varnothing, \partial Q_L):=L$,
and where we denote by $B^L_r(z)=(B_r(z)+L\Z^d)\cap Q_L$ the periodization of the ball~$B_r(z)$ in~$Q_L$.
In addition,
\begin{equation*}
\Big(\int_{B^L_{1+\rho}(x)}|\!\D(\psi_L^Y-\psi^Y)|^2\Big)^\frac12\,\lesssim_{\sharp Y} \, \Big(\langle\dist(x,\partial Q_L)\rangle+\langle\dist(Y\setminus\{x\},\partial Q_L)\rangle\Big)^{-d} .\qedhere
\end{equation*}
\end{lem}

\subsection{Explicit renormalization of cluster formulas}\label{sec:explicit}
This section is devoted to the proof of Theorem~\ref{th:renorm-B23}.
We first describe the explicit renormalization of the second and third cluster coefficients~$\Bb^2$ and~$\Bb^3$,
cf.~Propositions~\ref{prop:B2-ren} and~\ref{prop:B3-ren} below, before turning to the general case, cf.~Proposition~\ref{prop:Bgen-ren}.
For notational simplicity, we assume that particles are spherical with unit radius, \mbox{$I_n=B(x_n)$}, but we emphasize that the general case follows along the same lines under the independence assumption~\ref{B1}. More precisely, it suffices to replace each occurence of spherical particles below by iid random shapes and to further take the expectation with respect to the latter; we omit the detail.

\subsubsection{Explicit renormalization of $\Bb^2$: Batchelor--Green formula}
We start with the analysis of $\Bb^2$ and rigorously establish the so-called Batchelor--Green formula~\cite{BG-72}.

\begin{prop}[Batchelor--Green renormalization of $\Bb^2$]\label{prop:B2-ren}
Let~\ref{H0} and~\ref{Gunif} hold, and assume for simplicity that particles are spherical with unit radius, $I_n=B(x_n)$.
Let also the mixing assumption~\emph{\ref{Mix-om-n}} hold to order $n=2$ with some non-increasing rate \mbox{$\omega\in C^\infty_b(\R^+)$} satisfying the Dini condition $\int_1^\infty \tfrac1t\,\omega(t)\,dt<\infty$, as well as the doubling condition $\omega(2t) \simeq \omega(t)$ for all $t\ge0$.
Then, the infinite-volume second-order cluster coefficient~$\Bb^2$ defined in~\eqref{eq:form-Bj} can be expressed as follows,
\begin{equation}\label{eq:lim-form-B2}
E:\Bb^2E\,=\,\int_{\R^d}\Big(\int_{\partial B}\psi^z\cdot\sigma^0\nu\Big)h_2(0,z)\,dz
+\int_{\R^d}\Big(\int_{\partial B}\psi^z\cdot\delta^{z}\sigma^{0}\nu\Big)f_2(0,z)\,dz,
\end{equation}
where both integrals are absolutely converging and where we use the notation~\eqref{eq:corr-xy-notation}.
In addition, the following estimates hold:
\begin{enumerate}[(i)]
\item \emph{Uniform cluster estimate:}
\[|\Bb^2_L|\,\lesssim\,\lambda_2(\Pc)+\int_1^\infty \tfrac1t\big(\omega(t)\wedge\lambda_2(\Pc)\big)\,dt,\]
hence, in case of an algebraic weight $\omega(t)\le Ct^{-\beta}$ for some $C,\beta>0$,
\[|\Bb^2_L|\,\lesssim\,\lambda_2(\Pc)|\!\log\lambda(\Pc)|.\]
\item \emph{Periodization error estimate:}
\begin{equation*}
|\Bb^2_L-\Bb^2|\,\lesssim\,
(\omega(L)+ \tfrac1L)  \log L+\int_1^\infty\tfrac1{t+L}\,\omega(t)\,dt.
\end{equation*}
\item \emph{Uniform remainder estimate:} If~\emph{\ref{Mix-om-n}} further holds with $n=3$, then
\begin{equation*}
|R_L^{2}|\,\lesssim\,\lambda_2(\Pc)
+\int_1^\infty\tfrac{1}t\big(\omega(t)\wedge\lambda_2(\Pc)\big)\,dt
+\int_1^\infty\tfrac{\log t}t\big(\omega(t)\wedge\lambda_3(\Pc)\big)\,dt,
\end{equation*}
hence, in case of an algebraic weight $\omega(t)\le Ct^{-\beta}$ for some $C,\beta>0$,
\begin{equation*}
|R_{L}^{2}|\,\lesssim\,\lambda_2(\Pc)|\!\log\lambda(\Pc)|+\lambda_3(\Pc)|\!\log\lambda(\Pc)|^2.\qedhere
\end{equation*}
\end{enumerate}
\end{prop}

\begingroup\allowdisplaybreaks
\begin{proof}
We split the proof into four steps.
Given $E\in\Md_0^\Sym$ with $|E|=1$, for notational convenience, we write $\Bb_L^2$, $\Bb^2$, and $R^2_L$ for $E:\Bb_L^2E$, $E:\Bb^2E$, and $E:R^2_LE$.

\medskip
\step{1} Reformulation of $\Bb^2_L$:
\begin{multline}\label{eq:BL2-rewr}
\Bb_L^2\,=\,L^{-d}\iint_{Q_{L,\rho}\times Q_{L,\rho}}\Big(\int_{\partial B(y)}\psi^{z}_{L}\cdot\sigma^{y}_{L}\nu\Big)\,h_2(y,z)\,dydz\\
+L^{-d}\iint_{Q_{L,\rho}\times Q_{L,\rho}}\Big(\int_{\partial B(y)}\psi^{z}_{L}\cdot\delta^{z}\sigma^{y}_{L}\nu\Big)\,f_2(y,z)\,dydz\\
-\lambda(\Pc)^2L^{-d}\iint_{Q_{L,\rho}\times (Q_L\setminus Q_{L,\rho})}\Big(\int_{\partial B(y)}\psi^{z}_{L}\cdot\sigma^{y}_{L}\nu\Big)\,dydz,
\end{multline}
where we recall the short-hand notation $Q_{L,\rho}=Q_{L-2(\ell\vee(1+\rho))}$, cf.~\eqref{e.set-perio}.

\medskip\noindent
By definition, cf.~\eqref{eq:form-clust2}, the finite-volume approximation $\Bb^2_L$ is given by
\[\Bb_L^2\,=\,L^{-d}\sum_{m\ne n}\expecM{\int_{\partial B(x_{n,L})}\psi_L^{\{m\}}\cdot\sigma^{\{m,n\}}_L\nu}.\]
Decomposing $\sigma^{\{m,n\}}_L=\sigma^{\{n\}}_L+\delta^{\{m\}}\sigma^{\{n\}}_L$, this turns into
\[\Bb_L^2\,=\,L^{-d}\sum_{m\ne n}\expecM{\int_{\partial B(x_{n,L})}\psi^{\{m\}}_L\cdot\sigma^{\{n\}}_L\nu}
+L^{-d}\sum_{m\ne n}\expecM{\int_{\partial B(x_{n,L})}\psi^{\{m\}}_L\cdot\delta^{\{m\}}\sigma^{\{n\}}_L\nu}.\]
In terms of multi-point densities, cf.~\eqref{eq:def-fj}, recalling the choice of the finite-volume approximation with $\Pc_L=\{x_n:x_n\in\Qd\}$, cf.~\eqref{e.set-perio}, and using the notation~\eqref{eq:corr-xy-notation},
we can rewrite
\begin{multline}\label{eq:pre-BL2-rewr}
\Bb_L^2\,=\,
L^{-d}\iint_{Q_{L,\rho}\times Q_{L,\rho}}\Big({\int_{\partial B(y)}\psi^{z}_{L}\cdot\sigma^{y}_{L}\nu}\Big)\,f_2(y,z)\,dydz\\
+L^{-d}\iint_{Q_{L,\rho}\times Q_{L,\rho}}\Big({\int_{\partial B(y)}\psi^{z}_{L}\cdot\delta^{z}\sigma^{y}_{L}\nu}\Big)\,f_2(y,z)\,dydz,
\end{multline}
and it remains to further analyze the first right-hand side term.
For that purpose, we note that $\psi_{L}^z=\psi_{L}^0(\cdot-z)$ and $\sigma_{L}^y=\sigma_{L}^0(\cdot-y)$, so that
\[\int_{\partial B(y)}\psi^{z}_{L}\cdot\sigma^{y}_{L}\nu\,=\,\int_{\partial B}\psi^{0}_{L}(\cdot+y-z)\cdot\sigma^{0}_{L}\nu.\]
Integrating over $z$, using the periodicity of $\psi_{L}^0$, and recalling that $\int_{\partial B}\sigma_{L}^0\nu=0$, we deduce
\begin{equation}\label{eq:cancel-BG}
\int_{Q_L}\Big(\int_{\partial B(y)}\psi^{z}_{L}\cdot\sigma^{y}_{L}\nu\Big)dz\,=\,0.
\end{equation}
Decomposing $f_2(y,z)=\lambda(\Pc)^2+h_2(y,z)$ in terms of the $2$-point correlation function $h_2$, and then using this cancellation property~\eqref{eq:cancel-BG} to reformulate the first right-hand side term in~\eqref{eq:pre-BL2-rewr}, the claim~\eqref{eq:BL2-rewr} follows.

\medskip
\step{2} Uniform estimate: proof of~(i).\\
Using the boundary conditions and the incompressibility constraints to smuggle in arbitrary constants in the different factors, as in the proof of~\eqref{eq:subtr-press-cond},
and appealing to the trace estimates of Lemma~\ref{lem:trace-0}, we find
\begin{eqnarray}
\Big|\int_{\partial B(y)}\psi_L^z\cdot\sigma_L^y\nu\Big|&\lesssim&\Big(\int_{B(y)}|\!\D(\psi_L^z)|^2\Big)^\frac12\Big(\int_{B_{1+\rho}(y)}|\!\D(\psi_L^y)|^2\Big)^\frac12,\nonumber\\
\Big|\int_{\partial B(y)}\psi_L^z\cdot\delta^z\sigma_L^y\nu\Big|&\lesssim&\Big(\int_{B(y)}|\!\D(\psi_L^z)|^2\Big)^\frac12\Big(\int_{B_{1+\rho}(y)}|\!\D(\delta^z\psi_L^y)|^2\Big)^\frac12.\label{eq:util-trace-psisig}
\end{eqnarray}
Hence, applying the decay estimates of Lemma~\ref{lem:decay-re} to $\psi_L^z=\Jc_{L}^z(\psi_L^z+Ex)$ and to~$\delta^z\psi_L^y=\Jc_{L;y}^z(\psi_L^{y,z}+Ex)$, combined with the energy estimate~\eqref{eq:det-bnd-phiFk00}, we get
\begin{eqnarray}
\Big|\int_{\partial B(y)}\psi_L^z\cdot\sigma_L^y\nu\Big|&\lesssim&\langle(y-z)_L\rangle^{-d},\nonumber\\
\Big|\int_{\partial B(y)}\psi_L^z\cdot\delta^z\sigma_L^y\nu\Big|&\lesssim&\langle(y-z)_L\rangle^{-2d}.\label{eq:est-psi-delsig-2d}
\end{eqnarray}
Formula~\eqref{eq:BL2-rewr} for $\Bb_L^2$ can then be estimated as follows,
\begin{multline*}
|\Bb_L^2|\,\lesssim\,L^{-d}\iint_{Q_{L}\times Q_{L}}\langle (y-z)_L\rangle^{-d}|h_2(y,z)|\,dydz\\
+L^{-d}\iint_{Q_{L}\times Q_{L}}\langle (y-z)_L\rangle^{-2d}f_2(y,z)\,dydz
+\lambda(\Pc)^2L^{-d}\iint_{Q_{L}\times (Q_L\setminus Q_{L,\rho})}\langle (y-z)_L\rangle^{-d}\,dydz.
\end{multline*}
In terms of the two-point intensity, recalling that $\overline\lambda_2(\Pc)=\lambda_2(\Pc)$ by Lemma~\ref{lem:gen-cond-lambd}(ii) in view of~\ref{Mix-om-n}, we can estimate the $2$-point correlation function as follows: appealing both to~\eqref{eq:high-intens-correl} and to the decay assumption~\ref{Mix-om-n}, and arguing as in Lemma~\ref{lem:gen-cond-lambd}(iii), we find
\begin{multline}\label{eq:example-Mix-om-n}
\iint_{Q_L\times Q_L}\langle(y-z)_L\rangle^{-d}|h_2(y,z)|\,dydz\\
\,\lesssim\,\iint_{Q_L\times Q_L}\langle(y-z)_L\rangle^{-d}\big(\omega(|y-z|)\wedge\lambda_2(\Pc)\big)\,dydz.
\end{multline}
The above then becomes
\begin{multline*}
|\Bb_L^2|\,\lesssim\,L^{-d}\iint_{Q_{L}\times Q_{L}}\langle (y-z)_L\rangle^{-d}\big(\omega(|y-z|)\wedge\lambda_2(\Pc)\big)\,dydz\\
+\lambda_2(\Pc)\bigg(L^{-d}\iint_{Q_{L}\times Q_{L}}\langle (y-z)_L\rangle^{-2d}dydz
+L^{-d}\iint_{Q_{L}\times (Q_L\setminus Q_{L,\rho})}\langle (y-z)_L\rangle^{-d}\,dydz\bigg).
\end{multline*}
As $\omega$ is non-increasing and as $|(y-z)_L|\le |y-z|$, the first right-hand side term is bounded by
\begin{eqnarray*}
\lefteqn{L^{-d}\iint_{Q_{L}\times Q_{L}}\langle (y-z)_L\rangle^{-d}\big(\omega(|y-z|)\wedge\lambda_2(\Pc)\big)\,dydz}
\\
&\le&L^{-d}\iint_{Q_{L}\times Q_{L}}\langle (y-z)_L\rangle^{-d}\big(\omega(|(y-z)_L|)\wedge\lambda_2(\Pc)\big)\,dydz
\\
&\lesssim& \int_{Q_{L}} \langle z\rangle^{-d}\big(\omega(|z|)\wedge\lambda_2(\Pc)\big)\,dz\\
&\lesssim&\int_1^\infty \tfrac1t\big(\omega(t)\wedge\lambda_2(\Pc)\big)\,dt,
\end{eqnarray*}
and the conclusion~(i) follows after similarly estimating the other terms.

\medskip
\step3 Convergence result: proof of~(ii).\\
Comparing identities~\eqref{eq:lim-form-B2} and~\eqref{eq:BL2-rewr}, we have
\begin{equation}\label{eq:decomp-B2L-B2}
|\Bb^2_L-\Bb^2|\,\le\,A_L^1+A_L^2+A_L^3,
\end{equation}
where we have set for abbreviation
\begin{eqnarray*}
A_L^1&:=&\bigg|L^{-d}\iint_{Q_{L,\rho}\times Q_{L,\rho}}\Big(\int_{\partial B(y)}\psi^{z}_{L}\cdot\sigma^{y}_{L}\nu\Big)\,h_2(y,z)\,dydz\\
&&\hspace{4cm}-\int_{\R^d}\Big(\int_{\partial B}\psi^z\cdot\sigma^0\nu\Big)h_2(0,z)\,dz\bigg|,\\
A_L^2&:=&\bigg|L^{-d}\iint_{Q_{L,\rho}\times Q_{L,\rho}}\Big(\int_{\partial B(y)}\psi^{z}_{L}\cdot\delta^{z}\sigma^{y}_{L}\nu\Big)\,f_2(y,z)\,dydz\\
&&\hspace{4cm}-\int_{\R^d}\Big(\int_{\partial B}\psi^z\cdot\delta^{z}\sigma^{0}\nu\Big)f_2(0,z)\,dz\bigg|,\\
A_L^3&:=&\lambda(\Pc)^2L^{-d}\iint_{Q_{L,\rho}\times (Q_L\setminus Q_{L,\rho})}\Big|\int_{\partial B(y)}\psi^{z}_{L}\cdot\sigma^{y}_{L}\nu\Big|\,dydz.
\end{eqnarray*}
We estimate these three contributions separately and we start with $A_L^1$.
Noting that stationarity yields $h_2(y,z)=h_2(0,z-y)$, and using that $\psi^z=\psi^{z-y}(\cdot-y)$  and $\sigma^y=\sigma^0(\cdot-y)$, we can write
\begin{eqnarray*}
\lefteqn{L^{-d}\iint_{Q_{L,\rho}\times Q_{L,\rho}}\Big(\int_{\partial B(y)}\psi^{z}\cdot\sigma^{y}\nu\Big)\,h_2(y,z)\,dydz}\\
&=&L^{-d}\iint_{Q_{L,\rho}\times Q_{L,\rho}}\Big(\int_{\partial B}\psi^{z-y}\cdot\sigma^{0}\nu\Big)\,h_2(0,z-y)\,dydz\\
&=&\int_{\R^d}L^{-d}|\Qd\cap(\Qd+z)|\Big(\int_{\partial B}\psi^{z}\cdot\sigma^{0}\nu\Big)\,h_2(0,z)\,dz,
\end{eqnarray*}
and thus, setting for abbreviation $\gamma^2_{L,\rho}(z):=L^{-d}|\Qd\cap(\Qd+z)|$, we get by the triangle inequality,
\begin{align}
&A_L^1\,\lesssim\,
\int_{\R^d}\big(1-\gamma^2_{L,\rho}(z)\big)\Big|\int_{\partial B}\psi^{z}\cdot\sigma^0\nu\Big||h_2(0,z)|\,dz\label{eq:estim-AL1-B2-1}\\
&~+L^{-d}\iint_{Q_{L}\times Q_{L}} \bigg(\Big|\int_{\partial B(y)}(\psi^{z}_{L}-\psi^z)\cdot\sigma^{y}_{L}\nu\Big|+\Big|\int_{\partial B(y)}\psi^{z}\cdot(\sigma^y_L-\sigma^y)\nu\Big|\bigg)|h_2(y,z)|\,dydz.\nonumber
\end{align}
Appealing to the trace estimates of Lemma~\ref{lem:trace}, decomposing
\begin{eqnarray*}
\psi^{z}_{L}-\psi^z&=&\Jc^z_L(\psi^z_L+Ex)-\Jc^z(\psi^z+Ex)\\
&=&\Jc^z(\psi^z_L-\psi^z)+(\Jc^z_{L}-\Jc^z)(\psi^z_L+Ex),
\end{eqnarray*}
using the decay estimates of Lemma~\ref{lem:decay-re}, the periodization error estimates of Lemma~\ref{lem:period-est}, and the energy estimate~\eqref{eq:det-bnd-phiFk00},
we find
\begin{eqnarray*}
\lefteqn{\Big|\int_{\partial B(y)}(\psi^{z}_{L}-\psi^z)\cdot\sigma^{y}_{L}\nu\Big|}\\
&\lesssim&\Big(\int_{B(y)}|\!\D(\psi^{z}_{L}-\psi^z)|^2\Big)^\frac12\Big(\int_{Q_L}|\!\D(\psi^{y}_{L}+Ex)|^2\Big)^\frac12\\
&\lesssim&\Big(\int_{B(y)}|\!\D(\Jc^z(\psi^{z}_{L}-\psi^z))|^2+|\!\D((\Jc_L^z-\Jc^z)(\psi_L^z+Ex))|^2\Big)^\frac12\\
&\lesssim&\langle y-z\rangle^{-d}\Big(\int_{B_{1+\rho}(z)}|\!\D(\psi^{z}_{L}-\psi^z)|^2\Big)^\frac12\\
&&\quad+\Big(\mathds1_{|y-z|>\frac L4}\langle(y-z)_L\rangle^{-d}+\mathds1_{|y-z|\le\frac L4}L^{-d}\Big)\Big(\int_{Q_L}|\!\D(\psi_L^z+Ex)|^2\Big)^\frac12\\
&\lesssim&\mathds1_{|y-z|>\frac L4}\langle(y-z)_L\rangle^{-d}+\mathds1_{|y-z|\le\frac L4}L^{-d},
\end{eqnarray*}
and similarly,
\begin{eqnarray*}
\Big|\int_{\partial B}\psi^{z}\cdot\sigma^0\nu\Big|&\lesssim&\langle z\rangle^{-d},\\
\Big|\int_{\partial B(y)}\psi^{z}\cdot(\sigma^y_L-\sigma^y)\nu\Big|&\lesssim&L^{-d}\langle y-z\rangle^{-d}.
\end{eqnarray*}
Inserting these estimates into~\eqref{eq:estim-AL1-B2-1}, we get
\begin{multline*}
A_L^1\,\lesssim\,
\int_{\R^d}\big(1-\gamma^2_{L,\rho}(z)\big)\langle z\rangle^{-d}|h_2(0,z)|\,dz\\
+L^{-d}\iint_{Q_{L}\times Q_{L}} \Big(\mathds{1}_{|y-z|>\frac L4} \langle (y-z)_L\rangle^{-d}+
 \mathds{1}_{|y-z| \le \frac L4}L^{-d}\Big) \,|h_2(y,z)|\,dydz.
\end{multline*}
Using the decay assumption~\ref{Mix-om-n} for $h_2$, noting that
\begin{equation}\label{eq:QL+QL}
1-\gamma^2_{L,\rho}(z)\,=\,1-L^{-d}|\Qd\cap(\Qd+z)|\,\lesssim\,\tfrac{|z|}L\wedge1,
\end{equation}
and using that $\int_{Q_L}\langle y\rangle^{-d}dy\lesssim\log L$,
we conclude after straightforward simplifications,
\begin{eqnarray}
A_L^1&\lesssim&\int_{\R^d}\big(\tfrac{|z|}L\wedge1\big)\langle z\rangle^{-d}\omega(|z|)\,dz+\omega(L)\log L+L^{-d}\int_{Q_{2L}}\omega(|z|)\,dz\nonumber\\
&\lesssim&\omega(L)\log L+\int_0^\infty\tfrac1{t+L}\,\omega(t)\,dt.\label{eq:decomp-B2L-B2-A1}
\end{eqnarray}
We turn to the estimation of the second term $A_L^2$ in~\eqref{eq:decomp-B2L-B2}.
By stationarity, as above, we find
\begin{multline*}
A_L^2\,\lesssim\,\int_{\R^d}\big(1-\gamma^2_{L,\rho}(z)\big)\Big|\int_{\partial B}\psi^z\cdot\delta^{z}\sigma^{0}\nu\Big|f_2(0,z)\,dz\\
+L^{-d}\iint_{Q_{L}\times Q_{L}}\bigg(\Big|\int_{\partial B(y)}(\psi^{z}_{L}-\psi^{z})\cdot\delta^{z}\sigma^{y}_{L}\nu\Big|+\Big|\int_{\partial B(y)}\psi^{z}\cdot(\delta^{z}\sigma^{y}_{L}-\delta^{z}\sigma^{y})\nu\Big|\bigg)\,f_2(y,z)\,dydz.
\end{multline*}
Recalling $\psi_L^z-\psi^z=\Jc^z(\psi_L^z-\psi^z)+(\Jc_L^z-\Jc^z)(\psi_L^z+Ex)$, further decomposing
\begin{eqnarray*}
\delta^{z}\psi^{y}_{L}-\delta^{z}\psi^{y}&=&\Jc_{L;y}^{z}(\psi^{y,z}_{L}+Ex)-\Jc_y^z(\psi^{y,z}+Ex)\\
&=&\Jc_{y}^{z}(\psi^{y,z}_{L}-\psi^{y,z})+(\Jc_{L;y}^{z}-\Jc_y^z)(\psi^{y,z}_{L}+Ex),
\end{eqnarray*}
and using the trace estimates of Lemma~\ref{lem:trace}, the decay estimates of Lemma~\ref{lem:decay-re}, the periodization error estimates of Lemma~\ref{lem:period-est}, and the energy estimate~\eqref{eq:det-bnd-phiFk00}, we find
\begin{eqnarray*}
\Big|\int_{\partial B}\psi^z\cdot\delta^{z}\sigma^{0}\nu\Big|&\lesssim&\langle z\rangle^{-2d},\\
\Big|\int_{\partial B(y)}(\psi^{z}_{L}-\psi^{z})\cdot\delta^{z}\sigma^{y}_{L}\nu\Big|&\lesssim&\langle(y-z)_L\rangle^{-d}\big(d_L(y)+d_L(z)\big)^{-d}\\
\Big|\int_{\partial B(y)}\psi^{z}\cdot(\delta^{z}\sigma^{y}_{L}-\delta^{z}\sigma^{y})\nu\Big|&\lesssim&\langle y-z\rangle^{-2d}\big(d_L(y)+d_L(z)\big)^{-d}+L^{-d}\langle (y-z)_L\rangle^{-d}.
\end{eqnarray*}
where we have set for abbreviation $d_L(z):=\langle\dist(z,\partial Q_L)\rangle$.
Inserting these estimates into the above, we get
\begin{equation*}
A_L^2\,\lesssim\,\int_{\R^d}\big(1-\gamma^2_{L,\rho}(z)\big)\langle z\rangle^{-2d}f_2(0,z)\,dz\\
+L^{-d}\iint_{Q_{L}\times Q_{L}}\langle(y-z)_L\rangle^{-d}d_L(y)^{-d}f_2(y,z)\,dydz.
\end{equation*}
In terms of the two-point intensity, appealing to Lemma~\ref{lem:gen-cond-lambd}(iii),
recalling~\eqref{eq:QL+QL}, and using that $\int_{Q_L}\langle y\rangle^{-d}dy\lesssim\log L$ and $\int_{Q_L}d_L(y)^{-d}dy\lesssim L^{d-1}$, we deduce
\begin{eqnarray}
A_L^2&\lesssim&\lambda_2(\Pc)\int_{\R^d}\big(\tfrac{|z|}L\wedge1\big)\langle z\rangle^{-2d}dz
+\lambda_2(\Pc)L^{-d}\iint_{Q_{L}\times Q_{L}}\langle(y-z)_L\rangle^{-d}d_L(y)^{-d}\,dydz\nonumber\\
&\lesssim&\lambda_2(\Pc)\tfrac{\log L}L.\label{eq:decomp-B2L-B2-A2}
\end{eqnarray}
It remains to estimate the last term $A_L^3$ in~\eqref{eq:decomp-B2L-B2}.
Using again the trace estimates of Lemma~\ref{lem:trace}, the decay estimates of Lemma~\ref{lem:decay-re}, and the energy estimate~\eqref{eq:det-bnd-phiFk00}, we find
\[\Big|\int_{\partial B(y)}\psi^{z}_{L}\cdot\sigma^{y}_{L}\nu\Big|\,\lesssim\,\langle(y-z)_L\rangle^{-d},\]
and thus
\begin{eqnarray}
A_L^3&\lesssim&\lambda(\Pc)^2\,L^{-d}\iint_{Q_{L,\rho}\times (Q_L\setminus Q_{L,\rho})}\langle(y-z)_L\rangle^{-d}\,dydz\nonumber\\
&\lesssim&\lambda(\Pc)^2\,\tfrac{\log L}L.\label{eq:decomp-B2L-B2-A2-re}
\end{eqnarray}
Combining this with~\eqref{eq:decomp-B2L-B2}, \eqref{eq:decomp-B2L-B2-A1}, and~\eqref{eq:decomp-B2L-B2-A2}, the conclusion~(ii) follows.

\medskip
\step{4} Uniform remainder estimate: proof of~(iii).\\
The starting point is the refined estimate~\eqref{e.DGV-control-bis} on remainders, which reads in this case
\begin{multline*}
|R_{L}^{2}|\,\le\,\expecM{L^{-d}\sum_{n}\int_{B(x_{n,L})}\Big|\sum_{m:m\ne n}\D(\psi_{L}^{\{m\}})\Big|^2}\\
+\bigg|\,\expecM{L^{-d}\sum_{n}\int_{B(x_{n,L})}\Big(\sum_{m:m\ne n}\D(\psi_{L}^{\{m\}})\Big):\D(\hat\psi^{\{n\}}_{n,L})}\bigg|,
\end{multline*}
where we recall that $\hat\psi^{\{n\}}_{n,L}$ is defined by~\eqref{e.DG+6}.
Expanding the square and separating the different intersection patterns, this can be rewritten as follows,
in terms of multi-point densities,
\begin{multline*}
|R_{L}^{2}|\,\le\,L^{-d}\iint_{(Q_{L,\rho})^2}\Big(\int_{B(x)}|\!\D(\psi_{L}^y)|^2\Big)f_2(x,y)\,dxdy\\
+L^{-d}\bigg|\iiint_{(Q_{L,\rho})^3}\Big(\int_{B(x)}\D(\psi_{L}^y):\D(\psi_{L}^z)\Big)f_3(x,y,z)\,dxdydz\bigg|\\
+L^{-d}\bigg|\iint_{(Q_{L,\rho})^2}\Big(\int_{B(x)}\D(\psi_{L}^y):\D(\hat\psi^{x}_{x,L})\Big)f_2(x,y)\,dxdy\bigg|,
\end{multline*}
where we use the obvious notation for $\hat\psi^{x}_{x,L}$ such that $\hat \psi_{x_n,L}^{x_n}:=\hat \psi_{n,L}^{\{n\}}$.
Replacing $f_2,f_3$ by their expansions~\eqref{eq:dens-correl} in terms of correlation functions,
and noting that several contributions can be turned into boundary terms by application of the identity $\int_{Q_L}\D(\psi_L^y)\,dy=0$, we obtain
\begingroup\allowdisplaybreaks
\begin{align*}
|R_{L}^{2}|\,&\lesssim\,L^{-d}\iint_{(Q_{L,\rho})^2}\Big(\int_{B(x)}|\!\D(\psi_{L}^y)|^2\Big)f_2(x,y)\,dxdy\\
&+L^{-d}\bigg|\iiint_{(Q_{L,\rho})^3}\Big(\int_{B(x)}\D(\psi_{L}^y):\D(\psi_{L}^z)\Big)\big(\lambda(\Pc)h_2(y,z)+h_3(x,y,z)\big)\,dxdydz\bigg|\\
&+L^{-d}\bigg|\iint_{(Q_{L,\rho})^2}\Big(\int_{B(x)}\D(\psi_{L}^y):\D(\hat\psi^{x}_{x,L})\Big)h_2(x,y)\,dxdy\bigg|\\
&+\lambda(\Pc)L^{-d}\bigg|\iiint_{(Q_{L,\rho})^2\times(Q_L\setminus Q_{L,\rho})}\Big(\int_{B(x)}\D(\psi_{L}^y):\D(\psi_{L}^z)\Big)h_2(x,y)\,dxdydz\bigg|\\
&+\lambda(\Pc)^3L^{-d}\bigg|\iiint_{Q_{L,\rho}\times(Q_L\setminus Q_{L,\rho})^2}\Big(\int_{B(x)}\D(\psi_{L}^y):\D(\psi_{L}^z)\Big)\,dxdydz\bigg|\\
&+\lambda(\Pc)^2L^{-d}\bigg|\iint_{Q_{L,\rho}\times(Q_L\setminus Q_{L,\rho})}\Big(\int_{B(x)}\D(\psi_{L}^y):\D(\hat\psi^{x}_{x,L})\Big)\,dxdy\bigg|.
\end{align*}
\endgroup
Using~\eqref{eq:apest-tildepsin} to estimate $\hat\psi_{x,L}^x$ in terms of $\psi_L^x$,
\[\int_{B(x)}|\!\D(\hat\psi_{x,L}^x)|^2\,\lesssim\,\int_{B_{1+\rho}(x)}|\!\D(\psi_L^x)+E|^2\,\lesssim\,1,\]
and appealing to the decay estimates of Lemma~\ref{lem:decay-re}, we deduce
\begingroup\allowdisplaybreaks
\begin{align*}
|R_{L}^{2}|\,&\lesssim\,L^{-d}\iint_{(Q_{L,\rho})^2}\langle(x-y)_L\rangle^{-2d}f_2(x,y)\,dxdy\\
&+L^{-d}\iiint_{(Q_{L,\rho})^3}\langle(x-y)_L\rangle^{-d}\langle(x-z)_L\rangle^{-d}\Big(\lambda(\Pc)|h_2(y,z)|+|h_3(x,y,z)|\Big)\,dxdydz\\
&+L^{-d}\iint_{(Q_{L,\rho})^2}\langle(x-y)_L\rangle^{-d}|h_2(x,y)|\,dxdy\\
&+\lambda(\Pc)L^{-d}\iiint_{(Q_{L,\rho})^2\times(Q_L\setminus Q_{L,\rho})}\langle(x-y)_L\rangle^{-d}\langle(x-z)_L\rangle^{-d}|h_2(x,y)|\,dxdydz\\
&+\lambda(\Pc)^3L^{-d}\iiint_{Q_{L,\rho}\times(Q_L\setminus Q_{L,\rho})^2}\langle(x-y)_L\rangle^{-d}\langle(x-z)_L\rangle^{-d}\,dxdydz\\
&+\lambda(\Pc)^2L^{-d}\iint_{Q_{L,\rho}\times(Q_L\setminus Q_{L,\rho})}\langle(x-y)_L\rangle^{-d}\,dxdy.
\end{align*}
\endgroup
In terms of multi-point intensities, appealing to Lemma~\ref{lem:gen-cond-lambd}(ii)--(iii), and using both~\eqref{eq:high-intens-correl} and the decay assumption~\ref{Mix-om-n} to estimate correlation functions similarly as in~\eqref{eq:example-Mix-om-n}, the conclusion~(iii) follows after straightforward computations.
\end{proof}
\endgroup

\subsubsection{Explicit renormalization of $\Bb^3$}
The explicit renormalization of~$\Bb^2$ above is solely based on the simple and neat cancellation property~\eqref{eq:cancel-BG}. Higher-order cluster formulas require more subtle cancellations,
which can only be captured after suitably decomposing corrector differences in terms of elementary single-particle contributions as in~\eqref{eq:cordiff-JLHn-re}.
Before turning to the general case and proving Theorem~\ref{th:renorm-B23}, we start with a detailed account of the third-order cluster coefficient~$\Bb^3$, which contains all the necessary new ingredients.

\begingroup\allowdisplaybreaks
\begin{prop}[Renormalization of $\Bb^3$]\label{prop:B3-ren}
Let~\ref{H0} and~\ref{Gunif} hold, and assume for simplicity that particles are spherical with unit radius, $I_n=B(x_n)$.
Let also the mixing assumption~\emph{\ref{Mix-om-n}} hold to order $n=3$ with some non-increasing rate $\omega\in C^\infty_b(\R^+)$ satisfying the Dini type condition $\int_1^\infty \tfrac{\log t}t\omega(t)\,dt<\infty$, as well as the doubling condition $\omega(2t) \simeq \omega(t)$ for all $t\ge0$.
Then, the infinite-volume third-order cluster coefficient $\Bb^3$ defined in~\eqref{eq:form-Bj} can be expressed as follows,
\begin{align}\label{eq:lim-form-B3}
E:\Bb^3 E\,&=\,
3\iint_{\R^d\times\R^d}\Big({\int_{\partial B}\big(\Jc^y\Jc_{y}^z(\psi^z+Ex)\big) \cdot\sigma^{0}\nu}\Big)\big(\lambda(\Pc) h_2(0,z)+h_3(0,y,z)\big)\,dydz\nonumber\\
&+3\iint_{\R^d\times\R^d}\!\Big({\int_{\partial B}(\Jc^y\Jc_{y}^z\delta^y\psi^{z}) \cdot\sigma^{0}\nu}\Big)\big(f_3(0,y,z)-\lambda(\Pc)f_2(y,z)\big)\,dydz\nonumber\\
&+3\iint_{\R^d\times\R^d}\Big(\int_{\partial B} \big(\Jc^y\Jc_{y}^z(\psi^z+Ex)\big)\cdot\delta^y\sigma^0\nu\Big)\big(f_3(0,y,z)-\lambda(\Pc)f_2(0,y)\big)\,dydz\nonumber\\
&+3\iint_{\R^d\times\R^d}\Big({\int_{\partial B}(\Jc^z\delta^y\psi^z)\cdot\delta^y\sigma^0\nu}\Big)f_3(0,y,z)\,dydz\nonumber\\
&+3\iint_{\R^d\times\R^d}\Big({\int_{\partial B}(\Jc^y\Jc_{y}^z\delta^y\psi^z)\cdot\delta^y\sigma^0\nu}\Big)f_3(0,y,z)\,dydz\nonumber\\
&+\tfrac32\iint_{\R^d\times\R^d}\Big(\int_{\partial B}\delta^{y,z}\psi^\varnothing\cdot\delta^{y,z}\sigma^{0}\nu\Big)f_3(0,y,z)\,dydz,
\end{align}
where all the integrals are absolutely convergent and where we use the notation~\eqref{eq:corr-xy-notation}.
In addition, the following estimates hold:
\begin{enumerate}[(i)]
\item \emph{Uniform cluster estimate:}
\[|\Bb^3_L|\,\lesssim\,\lambda_3(\Pc)+\int_1^\infty\tfrac{\log t}{t}\big(\omega(t)\wedge\lambda_3(\Pc)\big)dt,\]
hence, in case of an algebraic weight $\omega(t)\le Ct^{-\beta}$ for some $C,\beta>0$,
\[|\Bb^3_L|\,\lesssim\,\lambda_3(\Pc)|\!\log\lambda(\Pc)|^2.\]
\item \emph{Periodization error estimate:} 
\begin{equation*}
 |\Bb^3_L-\Bb^3|\,\lesssim\,
\tfrac{\log L}L+\omega(L)(\log L)^2+\int_1^\infty\tfrac{\log t}{t+L}\,\omega(t)\,dt.
\end{equation*}
\item \emph{Uniform remainder estimate:} If~\emph{\ref{Mix-om-n}} further holds with $n=5$, then
\begin{equation*}
|R_L^3|\,\lesssim\,\lambda_3(\Pc)
+\sum_{j=3}^5\int_1^\infty\tfrac{(\log t)^{j-2}}t\big(\omega(t)\wedge\lambda_j(\Pc)\big)\,dt,
\end{equation*}
hence, in case of an algebraic weight $\omega(t)\le Ct^{-\beta}$ for some $C,\beta>0$,
\begin{equation*}
|R_L^3|\,\lesssim\,\lambda_3(\Pc)|\!\log\lambda(\Pc)|^2+\lambda_4(\Pc)|\!\log\lambda(\Pc)|^3+\lambda_5(\Pc)|\!\log\lambda(\Pc)|^4.\qedhere
\end{equation*}
\end{enumerate}
\end{prop}

\begin{proof}
We split the proof into four steps.
Given $E\in\Md_0^\Sym$ with $|E|=1$, for notational convenience, we write $\Bb_L^3$, $\Bb^3$, and $R^3_L$ for $E:\Bb_L^3E$, $E:\Bb^3E$, and $E:R^3_LE$.

\medskip
\step{1} Reformulation of $\Bb_L^3$
\begin{align}\label{eq:BL3-rewr}
&\Bb^3_L\,=\,
E_L^{3}\\
&+3L^{-d}\!\!\iiint_{(Q_{L,\rho})^3}\!\Big({\int_{\partial B(x)}(\Jc_L^y\Jc_{L;y}^z\bar \psi_L^z) \cdot\sigma^{x}_{L}\nu}\Big)\big(\lambda(\Pc) h_2(x,z)+h_3(x,y,z)\big)\,dxdydz\nonumber\\
&+3L^{-d}\!\!\iiint_{(Q_{L,\rho})^3}\!\Big({\int_{\partial B(x)}(\Jc_L^y\Jc_{L;y}^z\Jc_{L;z}^y \bar\psi_L^{y,z}) \cdot\sigma^{x}_{L}\nu}\Big)\big(f_3(x,y,z)-\lambda(\Pc)f_2(y,z)\big)\,dxdydz\nonumber\\
&+3L^{-d}\iiint_{(Q_{L,\rho})^3}\Big(\int_{\partial B(x)} (\Jc_L^y\Jc_{L;y}^z\bar\psi_L^z)\cdot\delta^y\sigma^x_L\nu\Big)\big(f_3(x,y,z)-\lambda(\Pc)f_2(x,y)\big)\,dxdydz\nonumber\\
&+3L^{-d}\iiint_{(Q_{L,\rho})^3}\Big({\int_{\partial B(x)}(\Jc_L^z\Jc_{L;z}^y\bar\psi_L^{y,z})\cdot\delta^y\sigma^x_L\nu}\Big)f_3(x,y,z)\,dxdydz\nonumber\\
&+3L^{-d}\iiint_{(Q_{L,\rho})^3}\Big({\int_{\partial B(x)}(\Jc_L^y\Jc_{L;y}^z\Jc_{L;z}^y\bar\psi_L^{y,z})\cdot\delta^y\sigma^x_L\nu}\Big)f_3(x,y,z)\,dxdydz\nonumber\\
&+\tfrac32L^{-d}\iiint_{(Q_{L,\rho})^3}\Big(\int_{\partial B(x)}\delta^{y,z}\psi_L^\varnothing\cdot\delta^{y,z}\sigma^{x}_{L}\nu\Big)f_3(x,y,z)\,dxdydz,\nonumber
\end{align}
where we henceforth use the short-hand notation $\bar \psi_L^Y:=\psi_L^Y+Ex$,
and where $E_L^{3}$ stands for boundary terms,
\begin{align*}
E_L^{3}~&:=~3\lambda(\Pc)^3L^{-d}\iiint_{\Qd\times(Q_L\setminus\Qd)^2}\Big({\int_{\partial B(x)}(\Jc_L^y\Jc_{L;y}^z\bar\psi_L^z) \cdot\sigma^{x}_{L}\nu}\Big)\,dxdydz\\
&-3\lambda(\Pc)L^{-d}\iiint_{(\Qd)^2\times(Q_L\setminus \Qd)}\Big({\int_{\partial B(x)}(\Jc_L^y\Jc_{L;y}^z\bar\psi_L^z) \cdot\sigma^{x}_{L}\nu}\Big) h_2(x,y)\,dxdydz\\
&-3\lambda(\Pc)L^{-d}\iiint_{(Q_L\setminus\Qd)\times(\Qd)^2}\Big({\int_{\partial B(x)}(\Jc_L^y\Jc_{L;y}^z\bar\psi_L^z) \cdot\sigma^{x}_{L}\nu}\Big)h_2(y,z)\,dxdydz\\
&-3\lambda(\Pc)L^{-d}\iiint_{(Q_L\setminus\Qd)\times(\Qd)^2}\Big({\int_{\partial B(x)}(\Jc_L^y\Jc_{L;y}^z\Jc_{L;z}^y\bar\psi_L^{y,z}) \cdot\sigma^{x}_{L}\nu}\Big)f_2(y,z)\,dxdydz\\
&-3\lambda(\Pc)L^{-d}\iiint_{(Q_{L,\rho})^2\times (Q_L\setminus Q_{L,\rho})}\Big(\int_{\partial B(x)} (\Jc_L^y\Jc_{L;y}^z\bar\psi_L^z)\cdot\delta^y\sigma^x_L\nu\Big)f_2(x,y)\,dxdydz.
\end{align*}

\medskip\noindent
By definition, cf.~\eqref{eq:form-clust2}, the finite-volume approximation $\Bb^3_L$ is given by
\[\Bb^3_L\,=\,\tfrac32L^{-d}\sum_{n,m,p}^{\ne}\expecM{\int_{\partial B(x_{n,L})}\delta^{\{m,p\}}\psi_L^\varnothing\cdot\sigma^{\{n,m,p\}}_L\nu}.\]
Decomposing
\[\sigma^{\{n,m,p\}}_L=\sigma^{\{n\}}_L+\delta^{\{m\}}\sigma^{\{n\}}_L+\delta^{\{p\}}\sigma^{\{n\}}_L+\delta^{\{m,p\}}\sigma^{\{n\}}_L,\]
this becomes by symmetry,
\begin{multline*}
\Bb^3_L\,=\,
\tfrac32L^{-d}\sum_{n,m,p}^{\ne}\expecM{\int_{\partial B(x_{n,L})}\delta^{\{m,p\}}\psi_L^\varnothing\cdot\sigma^{\{n\}}_L\nu}\\
+3L^{-d}\sum_{n,m,p}^{\ne}\expecM{\int_{\partial B(x_{n,L})}\delta^{\{m,p\}}\psi_L^\varnothing\cdot\delta^{\{m\}}\sigma^{\{n\}}_L\nu}\\
+\tfrac32L^{-d}\sum_{n,m,p}^{\ne}\expecM{\int_{\partial B(x_{n,L})}\delta^{\{m,p\}}\psi_L^\varnothing\cdot\delta^{\{m,p\}}\sigma^{\{n\}}_L\nu}.
\end{multline*}
In terms of multi-point densities, cf.~\eqref{eq:def-fj}, recalling the choice of the finite-volume approximation with $\Pc_L=\{x_n:x_n\in\Qd\}$, cf.~\eqref{e.set-perio}, we can rewrite
\begin{multline}\label{eq:pre-BL3-rewr}
\Bb^3_L\,=\,
\tfrac32L^{-d}\iiint_{(Q_{L,\rho})^3}\Big(\int_{\partial B(x)}\delta^{y,z}\psi_L^\varnothing\cdot\sigma^{x}_{L}\nu\Big)f_3(x,y,z)\,dxdydz\\
+3L^{-d}\iiint_{(Q_{L,\rho})^3}\Big(\int_{\partial B(x)}\delta^{y,z}\psi_L^\varnothing\cdot\delta^y\sigma^x_L\nu\Big)f_3(x,y,z)\,dxdydz\\
+\tfrac32L^{-d}\iiint_{(Q_{L,\rho})^3}\Big(\int_{\partial B(x)}\delta^{y,z}\psi_L^\varnothing\cdot\delta^{y,z}\sigma^{x}_{L}\nu\Big)f_3(x,y,z)\,dxdydz.
\end{multline}
It remains to further analyze the first two right-hand side terms and we split the proof into two further substeps.
To capture cancellations, we shall expand $\delta^{y,z}\psi_L^\varnothing$ and $\delta^y\psi^x_L$ in terms of one-particle contributions as in~\eqref{eq:cordiff-JLHn-re}.

\medskip
\substep{1.1} Proof that
\begin{eqnarray}\label{eq:pre-BL3-rewr.3-rewragain}
\lefteqn{L^{-d}\iiint_{(Q_{L,\rho})^3}\Big({\int_{\partial B(x)}\delta^{y,z}\psi_L^\varnothing\cdot\sigma^{x}_{L}\nu}\Big)f_3(x,y,z)\,dxdydz}\\
&=&E_L^{3,1}+2L^{-d}\!\!\iiint_{(Q_{L,\rho})^3}\!\Big({\int_{\partial B(x)}(\Jc_L^y\Jc_{L;y}^z\bar\psi_L^z) \cdot\sigma^{x}_{L}\nu}\Big)\big(\lambda(\Pc) h_2(x,z)+h_3(x,y,z)\big)\,dxdydz\nonumber\\
&&+2L^{-d}\!\!\iiint_{(Q_{L,\rho})^3}\!\Big({\int_{\partial B(x)}(\Jc_L^y\Jc_{L;y}^z\Jc_{L;z}^y\bar\psi_L^{y,z}) \cdot\sigma^{x}_{L}\nu}\Big)\big(f_3(x,y,z)-\lambda(\Pc)f_2(y,z)\big)\,dxdydz,\nonumber
\end{eqnarray}
where we recall the short-hand notation $\bar\psi_L^Y=\psi_L^Y+Ex$,
and where $E_L^{3,1}$ stands for boundary terms,
\begin{align*}
E_L^{3,1}~&:=~2\lambda(\Pc)^3L^{-d}\iiint_{\Qd\times(Q_L\setminus\Qd)^2}\Big({\int_{\partial B(x)}(\Jc_L^y\Jc_{L;y}^z\bar\psi_L^z) \cdot\sigma^{x}_{L}\nu}\Big)\,dxdydz\\
&-2\lambda(\Pc)L^{-d}\iiint_{(\Qd)^2\times(Q_L\setminus \Qd)}\Big({\int_{\partial B(x)}(\Jc_L^y\Jc_{L;y}^z\bar\psi_L^z) \cdot\sigma^{x}_{L}\nu}\Big) h_2(x,y)\,dxdydz\\
&-2\lambda(\Pc)L^{-d}\iiint_{(Q_L\setminus\Qd)\times(\Qd)^2}\Big({\int_{\partial B(x)}(\Jc_L^y\Jc_{L;y}^z\bar\psi_L^z) \cdot\sigma^{x}_{L}\nu}\Big)h_2(y,z)\,dxdydz\\
&-2\lambda(\Pc)L^{-d}\iiint_{(Q_L\setminus\Qd)\times(\Qd)^2}\Big({\int_{\partial B(x)}(\Jc_L^y\Jc_{L;y}^z\Jc_{L;z}^y\bar\psi_L^{y,z}) \cdot\sigma^{x}_{L}\nu}\Big)f_2(y,z)\,dxdydz.
\end{align*}
In view of~\eqref{eq:cordiff-JLHn-re}, corrector differences can be decomposed as
\begin{eqnarray}
\delta^{y,z}\psi_L^\varnothing
&=&\Jc_L^y\delta^z\psi_L^y+\Jc_L^z\delta^y\psi_L^z\nonumber\\
&=&\Jc_L^y\Jc_{L;y}^z\bar\psi_L^{y,z}+\Jc_L^z\Jc_{L;z}^y\bar\psi_L^{y,z},\label{eq:decomp-JJ-pre}
\end{eqnarray}
and thus, further writing
\begin{eqnarray*}
\psi_L^{y,z}&=&\psi_L^z+\delta^y\psi_L^z~~=~~\psi_L^z+\Jc_{L;z}^y\bar \psi_L^{y,z}\\
&=&\psi_L^y+\delta^z\psi_L^y~~=~~\psi_L^y+\Jc_{L;y}^z\bar \psi_L^{y,z},
\end{eqnarray*}
we deduce
\begin{equation}\label{eq:decomp-JJ}
\delta^{y,z}\psi_L^\varnothing
\,=\,\Jc_L^y\Jc_{L;y}^z\bar\psi_L^{z}+\Jc_L^z\Jc_{L;z}^y\bar\psi_L^{y}
+\Jc_L^y\Jc_{L;y}^z\Jc_{L;z}^y\bar\psi_L^{y,z}+\Jc_L^z\Jc_{L;z}^y\Jc_{L;y}^z\bar\psi_L^{y,z}.
\end{equation}
From this decomposition, we get by symmetry
\begin{eqnarray}\label{eq:pre-BL3-rewr.3}
\lefteqn{L^{-d}\iiint_{(Q_{L,\rho})^3}\Big({\int_{\partial B(x)}\delta^{y,z}\psi_L^\varnothing\cdot\sigma^{x}_{L}\nu}\Big)f_3(x,y,z)\,dxdydz}\\
&=&2L^{-d}\iiint_{(Q_{L,\rho})^3}\Big({\int_{\partial B(x)}(\Jc_L^y\Jc_{L;y}^z\bar\psi_L^z) \cdot\sigma^{x}_{L}\nu}\Big)f_3(x,y,z)\,dxdydz\nonumber\\
&&+2L^{-d}\iiint_{(Q_{L,\rho})^3}\Big({\int_{\partial B(x)}(\Jc_L^y\Jc_{L;y}^z\Jc_{L;z}^y\bar\psi_L^{y,z}) \cdot\sigma^{x}_{L}\nu}\Big)f_3(x,y,z)\,dxdydz.\nonumber
\end{eqnarray}
We are now in position to exploit cancellations properties.
First, we note that Lemma~\ref{lem:new-cancel} yields, recalling $\psi_L^z=\psi_L^0(\cdot-z)$,
\begin{equation}\label{eq:cancel-intJc}
\int_{Q_L}(\Jc_{L;y}^z\bar\psi_L^z)\,dz\,=\,0,
\end{equation}
and thus, for all $x,y\in\R^d$,
\begin{equation*}
\int_{Q_L}\Big(\int_{\partial B(x)}(\Jc_L^y\Jc_{L;y}^z\bar\psi_L^z)\cdot\sigma_L^x\nu\Big)\,dz\,=\,0.
\end{equation*}
In addition, similarly to what was done in~\eqref{eq:cancel-BG} for the renormalization of $\Bb_L^2$, writing
\[\int_{\partial B(x)}(\Jc_L^y\Jc_{L;y}^z \bar\psi_L^z) \cdot\sigma^{x}_{L}\nu\,=\,\int_{\partial B}(\Jc_L^y\Jc_{L;y}^z\bar\psi_L^z)(\cdot+x) \cdot\sigma^0_L\nu,\]
and using the condition $\int_{\partial B}\sigma_L^0\nu=0$, we find
\begin{equation}\label{eq:cancel-intx0}
\int_{Q_L}\Big(\int_{\partial B(x)}(\Jc_L^y\Jc_{L;y}^z\bar\psi_L^z) \cdot\sigma^{x}_{L}\nu\Big)\,dx
\,=\,0,
\end{equation}
and likewise,
\[\int_{Q_L}\Big(\int_{\partial B(x)}(\Jc_L^y\Jc_{L;y}^z\Jc_{L;z}^y\bar\psi_L^{y,z}) \cdot\sigma^{x}_{L}\nu\Big)\,dx\,=\,0.\]
Turning back to~\eqref{eq:pre-BL3-rewr.3}, replacing $f_3$ by its expansion~\eqref{eq:dens-correl} in terms of correlation functions, and using these three cancellation properties, the claim~\eqref{eq:pre-BL3-rewr.3-rewragain} follows.

\medskip
\substep{1.2} Proof that
\begin{eqnarray}\label{eq:pre-BL3-rewr.2plop}
\lefteqn{L^{-d}\iiint_{(Q_{L,\rho})^3}\Big({\int_{\partial B(x)}\delta^{y,z}\psi_L^\varnothing\cdot\delta^y\sigma^x_L\nu}\Big)f_3(x,y,z)\,dxdydz}\\
&=&E_L^{3,2}+L^{-d}\!\!\iiint_{(Q_{L,\rho})^3}\!\!\Big(\int_{\partial B(x)} (\Jc_L^y\Jc_{L;y}^z\bar\psi_L^z)\cdot\delta^y\sigma^x_L\nu\Big)\!\big(f_3(x,y,z)-\lambda(\Pc)f_2(x,y)\big)dxdydz\nonumber\\
&&+L^{-d}\iiint_{(Q_{L,\rho})^3}\Big({\int_{\partial B(x)}(\Jc_L^z\Jc_{L;z}^y\bar\psi_L^{y,z})\cdot\delta^y\sigma^x_L\nu}\Big)f_3(x,y,z)\,dxdydz\nonumber\\
&&+L^{-d}\iiint_{(Q_{L,\rho})^3}\Big({\int_{\partial B(x)}(\Jc_L^y\Jc_{L;y}^z\Jc_{L;z}^y\bar\psi_L^{y,z})\cdot\delta^y\sigma^x_L\nu}\Big)f_3(x,y,z)\,dxdydz,\nonumber
\end{eqnarray}
where $E_L^{3,2}$ stands for a boundary term,
\[E_L^{3,2}\,:=\,-\lambda(\Pc)L^{-d}\iiint_{(Q_{L,\rho})^2\times (Q_L\setminus Q_{L,\rho})}\Big(\int_{\partial B(x)} (\Jc_L^y\Jc_{L;y}^z\bar\psi_L^z)\cdot\delta^y\sigma^x_L\nu\Big)f_2(x,y)\,dxdydz.\]
Inserting this into~\eqref{eq:pre-BL3-rewr}, together with~\eqref{eq:pre-BL3-rewr.3-rewragain},
the claim~\eqref{eq:BL3-rewr} follows.

\medskip\noindent
We turn to the proof of~\eqref{eq:pre-BL3-rewr.2plop}. As this term benefits from some additional decay due to the factor $\delta^y\sigma_L^x$, we only need the following (asymetric) simpler version of~\eqref{eq:decomp-JJ},
\begin{equation}\label{eq:decomp-JJ-short}
\delta^{y,z}\psi_L^\varnothing
\,=\,\Jc_L^y\Jc_{L;y}^z\bar\psi_L^{z}+\Jc_L^z\Jc_{L;z}^y\bar\psi_L^{y,z}
+\Jc_L^y\Jc_{L;y}^z\Jc_{L;z}^y\bar\psi_L^{y,z},
\end{equation}
which leads us to
\begin{eqnarray}\label{eq:pre-BL3-rewr.2}
\lefteqn{L^{-d}\iiint_{(Q_{L,\rho})^3}\Big({\int_{\partial B(x)}\delta^{y,z}\psi_L^\varnothing\cdot\delta^y\sigma^x_L\nu}\Big)f_3(x,y,z)\,dxdydz}\\
&=&L^{-d}\iiint_{(Q_{L,\rho})^3}\Big({\int_{\partial B(x)} (\Jc_L^y\Jc_{L;y}^z\bar\psi_L^z)\cdot\delta^y\sigma^x_L\nu}\Big)f_3(x,y,z)\,dxdydz\nonumber\\
&&+L^{-d}\iiint_{(Q_{L,\rho})^3}\Big({\int_{\partial B(x)}(\Jc_L^z\Jc_{L;z}^y\bar\psi_L^{y,z})\cdot\delta^y\sigma^x_L\nu}\Big)f_3(x,y,z)\,dxdydz\nonumber\\
&&+L^{-d}\iiint_{(Q_{L,\rho})^3}\Big({\int_{\partial B(x)}(\Jc_L^y\Jc_{L;y}^z\Jc_{L;z}^y\bar\psi_L^{y,z})\cdot\delta^y\sigma^x_L\nu}\Big)f_3(x,y,z)\,dxdydz,\nonumber
\end{eqnarray}
and it remains to analyze the first right-hand side term. For that purpose,
we use again the elementary cancellation property~\eqref{eq:cancel-intJc}, now in form of
\[\int_{Q_L}\Big(\int_{\partial B(x)} (\Jc_L^y\Jc_{L;y}^z\bar\psi_L^z)\cdot\delta^y\sigma^x_L\nu\Big)dz\,=\,0,\]
which entails
\begin{eqnarray*}
\lefteqn{L^{-d}\iiint_{(Q_{L,\rho})^3}\Big(\int_{\partial B(x)} (\Jc_L^y\Jc_{L;y}^z\bar\psi_L^z)\cdot\delta^y\sigma^x_L\nu\Big)f_3(x,y,z)\,dxdydz}\\
&=&L^{-d}\iiint_{(Q_{L,\rho})^3}\Big(\int_{\partial B(x)} (\Jc_L^y\Jc_{L;y}^z\bar\psi_L^z)\cdot\delta^y\sigma^x_L\nu\Big)\big(f_3(x,y,z)-\lambda(\Pc)f_2(x,y)\big)\,dxdydz\\
&&-\lambda(\Pc)L^{-d}\iiint_{(Q_{L,\rho})^2\times (Q_L\setminus Q_{L,\rho})}\Big(\int_{\partial B(x)} (\Jc_L^y\Jc_{L;y}^z\bar\psi_L^z)\cdot\delta^y\sigma^x_L\nu\Big)f_2(x,y)\,dxdydz,
\end{eqnarray*}
and the claim~\eqref{eq:pre-BL3-rewr.2plop} follows.

\medskip
\step{2} Uniform estimates: proof of~(i).\\
As in the proof of Proposition~\ref{prop:B2-ren}(i), appealing to the trace estimates of Lemma~\ref{lem:trace-0}, the decay estimates of Lemma~\ref{lem:decay-re}, and the energy estimate~\eqref{eq:det-bnd-phiFk00}, formula~\eqref{eq:BL3-rewr} for $\Bb_L^3$ can be estimated as follows,
\begin{align*}
|\Bb^3_L|\,&\lesssim\,
|E_L^{3}|\\
&+L^{-d}\!\!\iiint_{(Q_L)^3}\langle (x-y)_L\rangle^{-d}\langle(y-z)_L\rangle^{-d}\big(\lambda(\Pc) |h_2(x,z)|+|h_3(x,y,z)|\big)\,dxdydz\nonumber\\
&+L^{-d}\!\!\iiint_{(Q_L)^3}\langle(x-y)_L\rangle^{-d}\langle(y-z)_L\rangle^{-2d}|f_3(x,y,z)-\lambda(\Pc)f_2(y,z)|\,dxdydz\nonumber\\
&+L^{-d}\iiint_{(Q_L)^3}\langle(x-z)_L\rangle^{-d}\langle(z-y)_L\rangle^{-d}\langle(y-x)_L\rangle^{-d}f_3(x,y,z)\,dxdydz\nonumber\\
&+L^{-d}\iiint_{(Q_L)^3}\langle(x-y)_L\rangle^{-2d}\langle(y-z)_L\rangle^{-2d}f_3(x,y,z)\,dxdydz,
\end{align*}
and, for boundary terms,
\begin{align*}
|E_L^{3}|~&\lesssim~
\lambda(\Pc)^3L^{-d}\iiint_{Q_L\times(Q_L\setminus\Qd)^2}\langle(x-y)_L\rangle^{-d}\langle(y-z)_L\rangle^{-d}\,dxdydz\\
&+\lambda(\Pc)L^{-d}\iiint_{(Q_L\setminus \Qd)\times(Q_L)^2}\langle(x-y)_L\rangle^{-d}\langle(y-z)_L\rangle^{-d}|h_2(y,z)|\,dxdydz\\
&+\lambda(\Pc)L^{-d}\iiint_{(Q_L\setminus\Qd)\times(Q_L)^2}\langle(x-y)_L\rangle^{-d}\langle(y-z)_L\rangle^{-2d}f_2(y,z)\,dxdydz.
\end{align*}
In terms of multi-point intensities, appealing to Lemma~\ref{lem:gen-cond-lambd},
using both~\eqref{eq:high-intens-correl} and the decay assumption~\ref{Mix-om-n} to estimate correlation functions similarly as in~\eqref{eq:example-Mix-om-n},
we deduce after straightforward computations
\begin{eqnarray}
|\Bb^3_L|&\lesssim&
|E_L^{3}|
+\lambda_3(\Pc)+\int_1^\infty\tfrac{\log t}{t}\big(\omega(t)\wedge\lambda_3(\Pc)\big)dt,\nonumber\\
|E_L^{3}|&\lesssim&
\tfrac{\log L}L\bigg(\lambda_3(\Pc)+\int_1^L\tfrac1t\big(\omega(t)\wedge\lambda_3(\Pc)\big)dt\bigg),
\label{eq:estim-EL3}
\end{eqnarray}
and the conclusion~(i) follows.

\medskip
\step{3} Convergence result: proof of~(ii).\\
In terms of $\gamma^3_{L,\rho}(y,z):=L^{-d}|\Qd\cap(y+\Qd)\cap(z+\Qd)|$,
using stationarity and recalling that $\delta^y\psi^z=\Jc_z^y\bar\psi^{y,z}$,
the formula~\eqref{eq:lim-form-B3} for the infinite-volume cluster coefficient $\Bb^3$ takes the equivalent form
\begin{align*}
\Bb^3&=\,E^4_L\\
&+3L^{-d}\!\!\iiint_{(Q_{L,\rho})^3}\!\Big({\int_{\partial B(x)}(\Jc^y\Jc_{y}^z\bar \psi^z) \cdot\sigma^{x}\nu}\Big)\big(\lambda(\Pc) h_2(x,z)+h_3(x,y,z)\big)\,dxdydz\nonumber\\
&+3L^{-d}\!\!\iiint_{(Q_{L,\rho})^3}\!\Big({\int_{\partial B(x)}(\Jc^y\Jc_{y}^z\Jc_{z}^y \bar\psi^{y,z}) \cdot\sigma^{x}\nu}\Big)\big(f_3(x,y,z)-\lambda(\Pc)f_2(y,z)\big)\,dxdydz\nonumber\\
&+3L^{-d}\iiint_{(Q_{L,\rho})^3}\Big(\int_{\partial B(x)} (\Jc^y\Jc_{y}^z\bar\psi^z)\cdot\delta^y\sigma^x\nu\Big)\big(f_3(x,y,z)-\lambda(\Pc)f_2(x,y)\big)\,dxdydz\nonumber\\
&+3L^{-d}\iiint_{(Q_{L,\rho})^3}\Big({\int_{\partial B(x)}(\Jc^z\Jc_{z}^y\bar\psi^{y,z})\cdot\delta^y\sigma^x\nu}\Big)f_3(x,y,z)\,dxdydz\nonumber\\
&+3L^{-d}\iiint_{(Q_{L,\rho})^3}\Big({\int_{\partial B(x)}(\Jc^y\Jc_{L;y}^z\Jc_{z}^y\bar\psi^{y,z})\cdot\delta^y\sigma^x\nu}\Big)f_3(x,y,z)\,dxdydz\nonumber\\
&+\tfrac32L^{-d}\iiint_{(Q_{L,\rho})^3}\Big(\int_{\partial B(x)}\delta^{y,z}\psi^\varnothing\cdot\delta^{y,z}\sigma^{x}\nu\Big)f_3(x,y,z)\,dxdydz\nonumber,
\end{align*}
where
\begin{align*}
E^4_L&:=\,3\iint_{\R^d\times\R^d}(1-\gamma^3_{L,\rho}(y,z))\Big({\int_{\partial B}(\Jc^y\Jc_{y}^z\bar \psi^z) \cdot\sigma^{0}\nu}\Big)\big(\lambda(\Pc) h_2(0,z)+h_3(0,y,z)\big)\,dydz\\
&+3\iint_{\R^d\times\R^d}(1-\gamma^3_{L,\rho}(y,z))\!\Big({\int_{\partial B}(\Jc^y\Jc_{y}^z\delta^y\psi^{z}) \cdot\sigma^{0}\nu}\Big)\big(f_3(0,y,z)-\lambda(\Pc)f_2(y,z)\big)\,dydz\nonumber\\
&+3\iint_{\R^d\times\R^d}(1-\gamma^3_{L,\rho}(y,z))\Big(\int_{\partial B} (\Jc^y\Jc_{y}^z\bar\psi^z)\cdot\delta^y\sigma^0\nu\Big)\big(f_3(0,y,z)-\lambda(\Pc)f_2(0,y)\big)\,dydz\nonumber\\
&+3\iint_{\R^d\times\R^d}(1-\gamma^3_{L,\rho}(y,z))\Big({\int_{\partial B}(\Jc^z\delta^y\psi^z)\cdot\delta^y\sigma^0\nu}\Big)f_3(0,y,z)\,dydz\nonumber\\
&+3\iint_{\R^d\times\R^d}(1-\gamma^3_{L,\rho}(y,z))\Big({\int_{\partial B}(\Jc^y\Jc_{y}^z\delta^y\psi^z)\cdot\delta^y\sigma^0\nu}\Big)f_3(0,y,z)\,dydz\nonumber\\
&+\tfrac32\iint_{\R^d\times\R^d}(1-\gamma^3_{L,\rho}(y,z))\Big(\int_{\partial B}\delta^{y,z}\psi^\varnothing\cdot\delta^{y,z}\sigma^{0}\nu\Big)f_3(0,y,z)\,dydz,\nonumber
\end{align*}
so that the identity~\eqref{eq:BL3-rewr} for $\Bb_L^3$ entails
\begin{align}
&\Bb^3_L-\Bb^3\,=\,
E_L^{3}-E_L^4\label{eq:decomp-B3L-B3}\\
&+3L^{-d}\!\!\iiint_{(Q_{L,\rho})^3}\!\Big(\int_{\partial B(x)}(\Jc^y_L\Jc_{L;y}^z\bar \psi^z_L) \cdot\sigma^{x}_L\nu-(\Jc^y\Jc_{y}^z\bar \psi^z) \cdot\sigma^{x}\nu\Big)\nonumber\\
&\hspace{7cm}\times\big(\lambda(\Pc) h_2(x,z)+h_3(x,y,z)\big)\,dxdydz\nonumber\\
&+3L^{-d}\!\!\iiint_{(Q_{L,\rho})^3}\!\Big({\int_{\partial B(x)}(\Jc^y_L\Jc_{L;y}^z\Jc_{L;z}^y \bar\psi^{y,z}_L) \cdot\sigma^{x}_L\nu-(\Jc^y\Jc_{y}^z\Jc_{z}^y \bar\psi^{y,z}) \cdot\sigma^{x}\nu}\Big)\nonumber\\
&\hspace{7cm}\times\big(f_3(x,y,z)-\lambda(\Pc)f_2(y,z)\big)\,dxdydz\nonumber\\
&+3L^{-d}\iiint_{(Q_{L,\rho})^3}\Big(\int_{\partial B(x)} (\Jc^y_L\Jc_{L;y}^z\bar\psi^z_L)\cdot\delta^y\sigma^x_L\nu- (\Jc^y\Jc_{y}^z\bar\psi^z)\cdot\delta^y\sigma^x\nu\Big)\nonumber\\
&\hspace{7cm}\times\big(f_3(x,y,z)-\lambda(\Pc)f_2(x,y)\big)\,dxdydz\nonumber\\
&+3L^{-d}\iiint_{(Q_{L,\rho})^3}\Big({\int_{\partial B(x)}(\Jc^z_L\Jc_{L;z}^y\bar\psi^{y,z}_L)\cdot\delta^y\sigma^x_L\nu-(\Jc^z\Jc_{z}^y\bar\psi^{y,z})\cdot\delta^y\sigma^x\nu}\Big)f_3(x,y,z)\,dxdydz\nonumber\\
&+3L^{-d}\iiint_{(Q_{L,\rho})^3}\Big({\int_{\partial B(x)}(\Jc^y_L\Jc_{L;y}^z\Jc_{L;z}^y\bar\psi^{y,z}_L)\cdot\delta^y\sigma^x_L\nu-(\Jc^y\Jc_{y}^z\Jc_{z}^y\bar\psi^{y,z})\cdot\delta^y\sigma^x\nu}\Big)\nonumber\\
&\hspace{9cm}\times f_3(x,y,z)\,dxdydz\nonumber\\
&+\tfrac32L^{-d}\iiint_{(Q_{L,\rho})^3}\Big(\int_{\partial B(x)}\delta^{y,z}\psi^\varnothing_L\cdot\delta^{y,z}\sigma^{x}_L\nu-\delta^{y,z}\psi^\varnothing\cdot\delta^{y,z}\sigma^{x}\nu\Big)f_3(x,y,z)\,dxdydz\nonumber.
\end{align}
The first boundary contribution $E_L^3$ is already estimated in~\eqref{eq:estim-EL3}.
Noting that
\begin{equation*}
1-\gamma_{L,\rho}^3(y,z)
\,\lesssim\,\tfrac{\langle y\rangle}L\wedge1+\tfrac{\langle z\rangle}L\wedge1,
\end{equation*}
using the trace estimates of Lemma~\ref{lem:trace},
the decay estimates of Lemma~\ref{lem:decay-re}, the energy estimate~\eqref{eq:det-bnd-phiFk00}, and Lemma~\ref{lem:gen-cond-lambd}(iii),
and further using~\eqref{eq:high-intens-correl} and the decay assumption~\ref{Mix-om-n} to estimate correlation functions similarly as in~\eqref{eq:example-Mix-om-n},
we obtain for the second boundary contribution in~\eqref{eq:decomp-B3L-B3},
\[E_L^4\,\lesssim\,\tfrac1L\int_1^L(\log t)\,\omega(t)\,dt+\int_L^\infty\tfrac{\log t}{t}\,\omega(t)\,dt.\]
It remains to estimate the remaining six right-hand side terms in~\eqref{eq:decomp-B3L-B3}.
We focus on the first term, which is the most involved, and we skip the detail for the last five ones. We split the proof into two further substeps.

\medskip
\substep{3.1} First periodization error term in~\eqref{eq:decomp-B3L-B3}: proof that
\begin{multline}\label{eq:B3L-term1-period}
\bigg|L^{-d}\!\!\iiint_{(\Qd)^3}\!\Big(\int_{\partial B(x)}(\Jc^y_L\Jc_{L;y}^z\bar \psi^z_L) \cdot\sigma^{x}_L\nu-(\Jc^y\Jc_{y}^z\bar \psi^z)\cdot\sigma^{x}\nu\Big)\\
\times\big(\lambda(\Pc) h_2(x,z)+h_3(x,y,z)\big)\,dxdydz\bigg|\\
 ~~\lesssim~~\omega(L)(\log L)^2+\tfrac{1}L\int_1^L(\log t) \omega(t)\,dt.
\end{multline}
Decomposing
\begin{multline*}
{(\Jc_L^{y}\Jc_{L;y}^{z}\bar \psi_L^{z}) \cdot\sigma^x_{L}\nu
-(\Jc^{y}\Jc_{y}^{z}\bar \psi^{z}) \cdot\sigma^x}\nu
\,=\,(\Jc_L^{y}\Jc_{L;y}^{z}\bar \psi_L^{z}) \cdot(\sigma^x_{L}-\sigma^x)\nu
+(\Jc_L^{y}-\Jc^{y})\Jc_{L;y}^{z}\bar \psi_L^{z}\cdot \sigma^x\nu
\\
+\Jc^{y}(\Jc_{L;y}^{z}-\Jc_{y}^{z})\bar \psi_L^{z}\cdot \sigma^x\nu
+\Jc^{y}\Jc_{y}^{z}(\psi_L^{z}-\psi^{z})\cdot \sigma^x\nu,
\end{multline*}
and appealing to the trace estimates of Lemma~\ref{lem:trace},
we find
\begin{multline*}
 \Big|{\int_{\partial B(x)}(\Jc^y_L\Jc_{L;y}^z\bar \psi^z_L) \cdot\sigma^{x}_L\nu-(\Jc^y\Jc_{y}^z\bar \psi^z)\cdot\sigma^{x}\nu} \Big|
\\
\lesssim\,
  \Big(\int_{B(x)} |\!\D(\Jc_L^{y}\Jc_{L;y}^{z}\bar \psi_L^{z})|^2 \Big)^\frac12\Big(\int_{B_{1+\rho}(x)} |\!\D(\psi_L^{x}-\psi^x)|^2 \Big)^\frac12   
 \\
 +   \Big(\int_{B(x)} |\!\D((\Jc_L^{y}-\Jc^{y})\Jc_{L;y}^{z}\bar \psi_L^{z})|^2 \Big)^\frac12\Big(\int_{B_{1+\rho}(x)} |\!\D(\psi^{x})|^2 \Big)^\frac12   
 \\
+ \Big(\int_{B(x)} |\!\D(\Jc^{y}(\Jc_{L;y}^{z}-\Jc_{y}^{z})\bar \psi_L^{z})|^2 \Big)^\frac12\Big(\int_{B_{1+\rho}(x)} |\!\D(\psi^{x})|^2 \Big)^\frac12 \\
+   \Big(\int_{B(x)} |\!\D(\Jc^{y}\Jc_{y}^{z}(\psi_L^{z}-\psi^{z}))|^2 \Big)^\frac12\Big(\int_{B_{1+\rho}(x)} |\!\D(\psi^{x})|^2 \Big)^\frac12 .
\end{multline*}
We then appeal to the decay estimates of Lemma~\ref{lem:decay-re}, to the periodization error estimates of Lemma~\ref{lem:period-est}, and to the energy estimate~\eqref{eq:det-bnd-phiFk00}, in form of
\begin{eqnarray*}
\Big(\int_{B_{1+\rho}(x)} |\!\D(\psi^x_L-\psi^x)|^2 \Big)^\frac12&\lesssim& L ^{- d} ,\\
\Big(\int_{B_{1+\rho}(x)} |\!\D(\Jc_L^{y}\Jc_{L;y}^{z}\bar \psi_L^{z})|^2 \Big)^\frac12&\lesssim& \langle (x-y)_L\rangle^{-d}   \langle (y-z)_L\rangle^{-d},\\
\Big(\int_{B_{1+\rho}(x)} |\!\D((\Jc_L^{y}-\Jc^{y})\Jc_{L;y}^{z}\bar \psi_L^{z})|^2 \Big)^\frac12&\lesssim& 
\langle (y-z)_L\rangle^{-d}\\
&&\times\Big(\mathds{1}_{|x-y|>\frac L4} \langle (x-y)_L\rangle^{-d}+
 \mathds{1}_{|x-y| \le \frac L4}L^{-d}\Big),\\
\Big(\int_{B_{1+\rho}(x)} |\!\D(\Jc^{y}(\Jc_{L;y}^{z}-\Jc_{y}^{z})\bar \psi_L^{z})|^2 \Big)^\frac12&\lesssim& \langle x-y\rangle^{- d}\\
&&\times\Big(\mathds{1}_{|y-z|>\frac L4} \langle (y-z)_L\rangle^{-d}+ \mathds{1}_{|y-z| \le \frac L4}L^{-d}\Big),
 \\
\Big(\int_{B_{1+\rho}(x)} |\!\D(\Jc^{y}\Jc_{y}^{z}(\psi_L^{z}-\psi^{z}))|^2 \Big)^\frac12&\lesssim& L^{-d}\langle x-y\rangle^{- d} \langle (y-z)_L\rangle^{- d},
\end{eqnarray*}
so that the above becomes
\begin{multline*}
 \Big|{\int_{\partial B(x)}(\Jc^y_L\Jc_{L;y}^z\bar \psi^z_L) \cdot\sigma^{x}_L\nu-(\Jc^y\Jc_{y}^z\bar \psi^z)\cdot\sigma^{x}\nu} \Big|\\
\lesssim\,
\mathds{1}_{|x-y|>\frac L4} \langle (x-y)_L\rangle^{-d}\langle (y-z)_L\rangle^{-d}
+\mathds{1}_{|x-y| \le \frac L4}L^{-d}\langle (y-z)_L\rangle^{-d}\\
+\mathds{1}_{|y-z|>\frac L4}\langle x-y\rangle^{-d}\langle (y-z)_L\rangle^{-d}
+\mathds{1}_{|y-z| \le \frac L4}L^{-d}\langle x-y\rangle^{-d}.
\end{multline*}
Further using~\eqref{eq:high-intens-correl} and the decay assumption~\ref{Mix-om-n} to estimate correlation functions similarly as in~\eqref{eq:example-Mix-om-n}, we get by symmetry
\begin{multline}\label{eq:B3L-term1-period-pre}
\bigg|L^{-d}\!\!\iiint_{(\Qd)^3}\!\Big(\int_{\partial B(x)}(\Jc^y_L\Jc_{L;y}^z\bar \psi^z_L) \cdot\sigma^{x}_L\nu-(\Jc^y\Jc_{y}^z\bar \psi^z)\cdot\sigma^{x}\nu\Big)\\
\times\big(\lambda(\Pc) h_2(x,z)+h_3(x,y,z)\big)\,dxdydz\bigg|\\
~\lesssim~L^{-d}\iiint_{(Q_L)^3}\mathds1_{|x-y|>\frac L4}\langle(x-y)_L\rangle^{-d}\langle(y-z)_L\rangle^{-d}\,\omega(|x-z|)\,dxdydz\\
+L^{-2d}\iiint_{(Q_L)^3}\mathds1_{|x-y|\le\frac L4}\langle(y-z)_L\rangle^{-d}\,\omega(|x-z|)\,dxdydz.
\end{multline}
We start with the first right-hand side term, which is the most delicate one to estimate.
By properties of $\omega$, we may decompose 
\[\omega(|x-z|)\, \lesssim\, \mathds1_{|x-z|>\frac L4}\,\omega(L) + \mathds1_{|x-z|\le\frac L4}\,\omega(|x-z|).\]
The first contribution with $|x-z|>\frac L4$ is easily estimated,
\begin{equation*}
\omega(L)L^{-d}\iiint_{(Q_L)^3}\mathds1_{|x-y|>\frac L4}\mathds1_{|x-z|>\frac L4}\langle(x-y)_L\rangle^{-d}\langle(y-z)_L\rangle^{-d}\,dxdydz
\,\lesssim\,\omega(L)(\log L)^2,
\end{equation*}
and we turn to the contribution with $|x-z|\le\frac L4$. For that purpose, we interpolate between two bounds for the integral with respect to $y$,
\begin{eqnarray*}
\lefteqn{\int_{Q_L} \mathds{1}_{|x-y|>\frac L4} \mathds1_{|x-z|\le\frac L4}\langle (x-y)_L\rangle^{- d} \langle (y-z)_L\rangle^{-d}  \, dy}\\
&\lesssim&\mathds1_{|x-z|\le\frac L4}\Big( \langle( x-z)_L\rangle^{-d} \log(2+|(x-z)_L|)\Big) \wedge \langle\dist(\{x,z\},\partial Q_L)\rangle^{-d}\\
&\lesssim&\Big( \langle x-z\rangle+\dist(\{x,z\},\partial Q_L)\Big)^{-d}\log(2+|x-z|),
\end{eqnarray*}
where we further used that $(x-z)_L=x-z$ if $|x-z|<\frac L4$. By symmetry, we then get
\begin{eqnarray}
\lefteqn{L^{-d}\iiint_{(Q_L)^3}\mathds1_{|x-y|>\frac L4}\mathds1_{|x-z|\le\frac L4}\langle(x-y)_L\rangle^{-d}\langle(y-z)_L\rangle^{-d}\,\omega(|x-z|)\,dxdydz}\nonumber\\
&\lesssim&L^{-d}\iint_{(Q_L)^2}\Big(\langle x-z\rangle+\dist(\{x\},\partial Q_L)\Big)^{-d}\,\log(2+|x-z|)\,\omega(|x-z|)\,dxdz\nonumber\\
&\lesssim& \tfrac{1}L \int_1^L\log(t)\omega(t)\,dt,\label{eq:estimploplop}
\end{eqnarray}
where the last bound follows from a straightforward computation, carefully distinguishing between the cases $\langle x-z\rangle\ge\dist(\{x\},\partial Q_L)$ and $\langle x-z\rangle\le\dist(\{x\},\partial Q_L)$.
Indeed, on the one hand, the part with $\langle x-z\rangle\ge\langle\dist(\{x\},\partial Q_L)\rangle$ can be estimated by
\begin{align*}
&{L^{-d}\iint_{(Q_L)^2}\mathds1_{\dist(\{x\},\partial Q_L)\le\langle x-z\rangle}\langle x-z\rangle^{-d}\log(2+|x-z|)\omega(|x-z|)\,dxdz}\\
&\qquad~\lesssim~L^{-d}\int_{Q_{2L}}(L^{d-1}\langle y\rangle)\,\langle y\rangle^{-d}\log(2+|y|)\omega(|y|)dy\\
&\qquad~\lesssim~\tfrac1L\int_1^L(\log r)\omega(r)\,dr,
\end{align*}
and on the other hand the part with $\langle x-z\rangle\le\dist(\{x\},\partial Q_L)$ is estimated by
\begin{align*}
&{L^{-d}\int_{Q_L}\langle\dist(\{x\},\partial Q_L)\rangle^{-d}\Big(\int_{\{z\in Q_L:\langle x-z\rangle\le\dist(\{x\},\partial Q_L)\}}\log(2+|x-z|)\omega(|x-z|)\,dz\Big)dx}\\
&\qquad\lesssim~L^{-d}\int_1^Lr^{d-1}\langle L-r\rangle^{-d}\Big(\int_1^{L-r}s^{d-1}(\log s)\omega(s)\,ds\Big)dr\\
&\qquad\lesssim~\tfrac1L\int_{L/2}^L\langle L-r\rangle^{-d}\Big(\int_1^{L-r}s^{d-1}(\log s)\omega(s)\,ds\Big)dr+L^{-d}\int_1^{L}s^{d-1}(\log s)\omega(s)ds\\
&\qquad\lesssim~\tfrac1L\int_1^{L/2}s^{d-1}(\log s)\omega(s)\Big(\int_{L/2}^{L-s}\langle L- r\rangle^{-d}dr\Big)ds+L^{-d}\int_1^{L}s^{d-1}(\log s)\omega(s)\,ds\\
&\qquad\lesssim~\tfrac1L\int_1^{L}(\log s)\omega(s)\,ds,
\end{align*}
which yields the bound~\eqref{eq:estimploplop}.
It remains to estimate the second right-hand side term in~\eqref{eq:B3L-term1-period-pre}, for which we directly find
\begin{multline*}
L^{-2d}\iiint_{(Q_L)^3}\mathds1_{|x-y|\le\frac L4}\langle(y-z)_L\rangle^{-d}\,\omega(|x-z|)\,dxdydz\\
\,\lesssim\,(\log L)L^{-d}\int_1^Lt^{d-1}\omega(t)\,dt\,\lesssim\,\tfrac1L\int_1^L(\log t)\omega(t)\,dt.
\end{multline*}
Inserting these different estimates into~\eqref{eq:B3L-term1-period-pre}, the claim~\eqref{eq:B3L-term1-period} follows using the doubling property of $\omega$.

\medskip
\substep{3.2} Conclusion.\\
The next four terms 
of \eqref{eq:decomp-B3L-B3} are estimated similarly as the first periodization error term, and we skip most details for brevity. We solely briefly comment on the last term in~\eqref{eq:decomp-B3L-B3}, which is slightly different as it involves second-order corrector differences.
We claim that
\begin{multline}\label{eq:estim-lastterm-B3L-B3}
\bigg|L^{-d}\iiint_{(Q_{L,\rho})^3}\Big(\int_{\partial B(x)}\delta^{y,z}\psi^\varnothing_L\cdot\delta^{y,z}\sigma^{x}_L\nu-\delta^{y,z}\psi^\varnothing\cdot\delta^{y,z}\sigma^{x}\nu\Big)f_3(x,y,z)\,dxdydz\bigg|\\
\,\lesssim\,\lambda_3(\Pc)\tfrac1L.
\end{multline}
By~\eqref{eq:decomp-JJ-pre}, and arguing similarly as for $\delta^{y,z}\psi_L^{x}$, we can decompose
\begin{eqnarray*}
\delta^{y,z}\psi_L^\varnothing
&=&\Jc_L^y\Jc_{L;y}^z\bar\psi_L^{y,z}+\Jc_L^z\Jc_{L;z}^y\bar\psi_L^{y,z},\\
\delta^{y,z}\psi_L^{x}
&=&\Jc_{L;x}^y\Jc_{L;x,y}^z\bar\psi_L^{x,y,z}+\Jc_{L;x}^z\Jc_{L;x;z}^y\bar\psi_L^{x,y,z},
\end{eqnarray*}
so that, proceeding as in Substep~3.1 above, we can write $\delta^{y,z}\psi_L^\varnothing\cdot \delta^{y,z}\sigma_L^{x}-\delta^{y,z}\psi^\varnothing\cdot \delta^{y,z}\sigma^{x}$ as the difference of four terms, which can each be written 
as a telescopic sum of six terms involving ``elementary'' periodization errors.
The most delicate of those terms is the following one,
\begin{multline*}
L^{-d}\iiint_{(Q_{L})^3}\Big(\int_{B(x)} |\!\D(\Jc^y\Jc_{y}^z\bar\psi^{y,z})|^2\Big)^\frac12 
\\
\times \Big(\int_{B_{1+\rho}(x)} |\!\D(\Jc_{x}^z\Jc_{x,z}^y(\psi_L^{x,y,z}-\psi^{x,y,z}))|^2\Big)^\frac12 f_3(x,y,z)\,dxdydz.
\end{multline*}
By the decay estimates of Lemma~\ref{lem:decay-re}, the periodization error estimates of Lemma~\ref{lem:period-est}, and the energy estimate~\eqref{eq:det-bnd-phiFk00}, and further appealing to Lemma~\ref{lem:gen-cond-lambd}(iii) to control $f_3$ in terms of the multi-point intensity,
this term is bounded by
\begin{equation*}
L^{-d}\lambda_3(\Pc)\iiint_{(Q_{L})^3} \langle x-y \rangle^{-d} \langle x-z \rangle^{-d}  \langle y-z \rangle^{-2d} 
\langle \dist(y,\partial Q_L) \rangle^{-d}\,dxdydz 
\,\lesssim\, \lambda_3(\Pc) \tfrac1L.
\end{equation*}
All the other terms can be bounded similarly, and the claim~\eqref{eq:estim-lastterm-B3L-B3} follows.
This concludes the proof of~(ii).

\medskip
\step4 Analysis of remainders.\\
Starting from~\eqref{e.DGV-control-bis}, expanding the products, and separating the different intersection patterns, we are led to the following, in terms of multi-point intensities,
\begin{eqnarray*}
|R_{L}^3|&\lesssim&L^{-d}\bigg|\int_{(Q_{L,\rho})^5}\Big(\int_{B(x)}\D(\delta^{y,z}\psi_L^\varnothing):\D(\delta^{y',z'}\psi_L^\varnothing)\Big)\,f_5(x,y,z,y',z')\,dxdydzdy'dz'\bigg|\\
&&+L^{-d}\bigg|\int_{(Q_{L,\rho})^4}\Big(\int_{B(x)}\D(\delta^{y,z}\psi_L^\varnothing):\D(\delta^{y,z'}\psi_L^\varnothing)\Big)\,f_4(x,y,z,z')\,dxdydzdz'\bigg|\\
&&+L^{-d}\bigg|\int_{(Q_{L,\rho})^3}\Big(\int_{B(x)}\D(\del^{y,z}\psi_{L}^\varnothing):\D(\del^{z'}\hat\psi^{x}_{x,L})\Big)\,f_4(x,y,z,z')\,dxdydzdz'\bigg|\\
&&+L^{-d}\int_{(Q_{L,\rho})^3}\Big(\int_{B(x)}|\!\D(\delta^{y,z}\psi_L^\varnothing)|^2\Big)\,f_3(x,y,z)\,dxdydz\\
&&+L^{-d}\bigg|\int_{(Q_{L,\rho})^3}\Big(\int_{B(x)}\D(\del^{y,z}\psi_{L}^\varnothing):\D(\del^{y}\hat\psi^{x}_{x,L})\Big)\,f_3(x,y,z)\,dxdydz\bigg|\\
&&+L^{-d}\bigg|\int_{(Q_{L,\rho})^3}\Big(\int_{B(x)}\D(\del^{y,z}\psi_{L}^\varnothing):\D(\hat\psi^{x}_{x,L})\Big)\,f_3(x,y,z)\,dxdydz\bigg|.
\end{eqnarray*}
As in the analysis of $\Bb_L^3$ in Step~1, cancellations are unravelled by decomposing $\delta^{y,z}\psi_L^\varnothing$ in terms of single-particle contributions. We leave the details to the reader.
\end{proof}
\endgroup

\subsubsection{Higher-order explicit renormalization}
Finally, we turn to the general higher-order case. The obtained renormalized formulas are not displayed in the statement as they take the form of  intricate diagrammatic expansions and require notation that will be introduced in the proof.
 
\medskip

\begin{prop}[Higher-order renormalizations]\label{prop:Bgen-ren}
Let~\ref{H0} and~\ref{Gunif} hold, and assume for simplicity that particles are spherical with unit radius, $I_n=B(x_n)$.
Let also the mixing assumption~\emph{\ref{Mix-om-n}} hold to order $n=k+1\ge2$ with rate $\omega\in C^\infty_b(\R^+)$ satisfying the Dini type condition $\int_1^\infty \tfrac1t(\log t)^{k-1}\omega(t)\,dt<\infty$.
Then, the infinite-volume $(k+1)$th-order cluster coefficient~$\Bb^{k+1}$, cf.~\eqref{eq:form-Bj}, can be expressed by means of absolutely convergent integrals.
In addition, in case of an algebraic rate $\omega(t)\le Ct^{-\beta}$ for some $C,\beta>0$, the following hold.
\begin{enumerate}[(i)]
\item \emph{Uniform estimate:}
\[|\Bb^{k+1}_L|\,\lesssim\,\lambda_{k+1}(\Pc)|\!\log\lambda(\Pc)|^k.\]
\item \emph{Convergence result:}
\begin{equation*}
|\Bb^{k+1}_L-\Bb^{k+1}|\,\lesssim\,\tfrac{(\log L)^{k}}{L^{\beta\wedge1}}.
\end{equation*}
\item \emph{Uniform remainder estimate:} If~\emph{\ref{Mix-om-n}} further holds with $n=2k+1$ with algebraic weight $\omega(t)\le Ct^{-\beta}$, then
\begin{equation}\label{e.neededDVG}
|R_L^{k+1}|\,\lesssim\,\sum_{j=k}^{2k}\lambda_{j+1}(\Pc)|\!\log\lambda(\Pc)|^{j}.\qedhere
\end{equation}
\end{enumerate}
\end{prop}

\begin{proof}
Let $k\ge1$ be fixed. By definition, cf.~\eqref{eq:form-clust2}, the periodic approximation $\Bb^{k+1}_L$ is given by
\begin{equation*}
\Bb^{k+1}_{L}\,=\,\tfrac12(k+1)!\,L^{-d}\sum_{\sharp F=k+1}\sum_{n\in F}\expecM{\int_{\partial B(x_{n,L})}\del^{F\setminus\{n\}}\psi_{L}^\varnothing\cdot\sigma_{L}^{F}\nu},
\end{equation*}
and thus, decomposing $\sigma^F_L=\sum_{G\subset F\setminus\{n\}}\delta^G\sigma_L^{\{n\}}$ for $n\in F$, we get by symmetry,
\begin{equation*}
\Bb^{k+1}_{L}\,=\,\tfrac{k+1}2\sum_{l=0}^{k}\binom{k}{l}L^{-d}\sum_{n_0,n_1,\ldots,n_{k}}^{\ne}\expecM{\int_{\partial B(x_{n_0,L})}\del^{\{n_1,\ldots,n_k\}}\psi_{L}^\varnothing\cdot\delta^{\{n_{l+1},\ldots,n_{k}\}}\sigma_{L}^{\{n_0\}}\nu}.
\end{equation*}
In terms of multi-point densities, cf.~\eqref{eq:def-fj}, recalling the choice~\eqref{e.set-perio} of the finite-volume approximation, and using the notation~\eqref{eq:corr-xy-notation}, this becomes
\begin{multline}\label{eq:decomp-BLj1-intens}
\Bb^{k+1}_{L}\,=\,\tfrac{k+1}2\sum_{l=0}^{k}\binom{k}{l}L^{-d}\int_{(\Qd)^{k+1}}\bigg(\int_{\partial B(x_0)}\del^{x_1,\ldots,x_k}\psi_{L}^\varnothing\cdot\delta^{x_{l+1},\ldots,x_k}\sigma_{L}^{x_0}\nu\bigg)\\
\times f_{k+1}(x_0,\ldots,x_k)\,dx_0\ldots dx_k.
\end{multline}
We now need to capture enough cancellations to make these integrals absolutely summable uniformly in the large-volume limit.
For that purpose, similarly to what we did for $\Bb^3_L$ in the proof of Proposition~\ref{prop:B3-ren}, we shall proceed to a suitable expansion of $\del^{x_1,\ldots,x_k}\psi_{L}^\varnothing$ in terms one-particle contributions. For general order $k$, the expansion is conveniently expressed in terms of diagrams.
We split the proof into six steps.

\medskip
\step{1} Diagrammatic decomposition of $\delta^{x_1,\ldots,x_k}\psi_L^\varnothing$.\\
We start with some terminology and notation:
\begin{enumerate}[---]
\item We use the standard notation $[k]:=\{1,\ldots,k\}$ and for any subset $S\subset[k]$ we define~$x_S:=(x_i)_{i\in S}$.
\smallskip\item Given a sequence $I=(i_1,\ldots,i_l)$ of indices, the first index $i_1$ is called the {\it root} of $I$, the last index $i_l$ is its {\it endpoint}, and~$l$ is its {\it length}. The associated index set is denoted by $\langle I\rangle:=\{i_1,\ldots,i_l\}$ and we define the {\it cardinality} of $I$ as $\sharp I:=\sharp\langle I\rangle$. An index $i$ is then said to belong to $I$ (for short,~$i\in I$) if it belongs to the index set $\langle I\rangle$, and we define $x_I:=x_{\langle I\rangle}=(x_i)_{i\in\langle I\rangle}$.
Two sequences $I$ and $J$ are said to be {\it disjoint} if there is no index belonging to both, that is, $\langle I\rangle\cap\langle J\rangle=\varnothing$.
\smallskip\item The {\it concatenation} of two sequences $I=(i_1,\ldots,i_l)$ and $J=(j_1,\ldots,j_m)$ is denoted by $I\uplus J:=(i_1,\ldots,i_{l},j_1,\ldots,j_m)$.
\smallskip\item A {\it string} of indices is defined as any sequence of distinct indices with length $\ge1$.
\smallskip\item Given a string $I=(i_1,\ldots,i_l)$ and an index set $S$, we define the {\it elementary contribution of~$I$ given~$S$} as the following composition of operators,
\begin{equation}\label{eq:string-contr}
\calR^I_{L;S}(x_{[k]})\,:=\,\Jc_{L;x_S}^{x_{i_1}}\Jc_{L;x_{i_1}}^{x_{i_2}}\Jc_{L;x_{i_1},x_{i_2}}^{x_{i_3}}\ldots\Jc_{L;x_{i_1},\ldots,x_{i_{l-1}}}^{x_{i_l}},
\end{equation}
where we recall that the $\Jc_{L,Y}^z$'s are defined in~\eqref{eq:def-JLy|z}.
\smallskip\item A {\it block} is defined as any sequence $B$ of indices that takes the form
\begin{equation}\label{eq:def-block}
B=(b)\uplus I_1\uplus\ldots\uplus I_r,
\end{equation}
where $r\ge0$ (for $r=0$ we simply have $B=(b)$) and where $I_1,\ldots,I_r$ are strings of length $\ge2$ with the following property:
for all $1\le j\le r$, the endpoint of~$I_j$ belongs to $(b)\uplus I_1\uplus\ldots\uplus I_{j-1}$ but other elements of~$I_j$ do not.
\smallskip\item Given a block $B=(b)\uplus I_1\uplus\ldots\uplus I_r$ and an index set $S$, we define the {\it elementary contribution of~$B$ given~$S$} as the following composition of operators,
\begin{equation}\label{eq:block-contr}
\calC^B_{L;S}(x_{[k]})\,:=\,\Jc_{L;x_S}^{x_b}\calR^{I_1}_{L;\{b\}}(x_{[k]})\calR^{I_2}_{L;\langle(b)\uplus I_1\rangle}(x_{[k]})\ldots\calR^{I_n}_{L;\langle(b)\uplus I_1\uplus\ldots\uplus I_{n-1}\rangle}(x_{[k]}).
\end{equation}
\end{enumerate}
In these terms, we claim that $\delta^{x_1,\ldots,x_j}\psi_L^\varnothing$ can be decomposed as follows,
\begin{equation}\label{eq:diagram}
\delta^{x_1,\ldots,x_k}\psi_L^\varnothing\,=\,\sum_{r=1}^k~\sum_{B_1,\ldots, B_r}\calC^{B_1}_{L;\varnothing}(x_{[k]})\calC^{B_2}_{L;\langle B_1\rangle}(x_{[k]})\ldots \calC^{B_r}_{L;\langle B_{r-1}\rangle}(x_{[k]})\,\bar \psi_L^{x_{B_r}},
\end{equation}
where we recall the short-hand notation $\bar \psi_L^{x_{B_r}}= \psi_L^{x_{B_r}}+Ex$, and where the sum $\sum_{B_1,\ldots, B_r}$ runs over all $r$-tuples of disjoint blocks $B_1,\ldots,B_r$ such that
\[\langle B_1\uplus\ldots\uplus B_r\rangle=[k].\]
Note that this sum~\eqref{eq:diagram} is obviously finite, uniformly in $L$.
Any sequence of indices of~$[k]$ can be viewed as a walk on the vertex set~$[k]$, thus inducing a (traversable) graph on~$[k]$ where edges are defined by successive elements of the sequence (with possible multiplicities).
In this view, each term in~\eqref{eq:diagram} can be conveniently represented by a corresponding diagram, cf.~Figure~\ref{fig:block-decomp}; as we shall see, these graphical representations will prove crucial in estimating the different terms.

\begin{figure}
{\begin{center}
{\small\begin{tikzpicture}[scale=0.9]
\begin{scope}[every node/.style={circle,thick,draw}]
    \node (0) at (-2,0) {0};
    \node (1) at (0,0) {1};
    \node (2) at (0,1.5) {2};
    \node (3) at (2,0) {3};
    \node (4) at (3,1.5) {4};
    \node (5) at (4,0) {5};
    \node (6) at (2,-1.5) {6};
    \node (7) at (4,-1.5) {7};
    \node (8) at (6,0) {8};
    \node (9) at (8.5,0) {9};
    \node (10) at (7.5,1.5) {10};
    \node (11) at (9.5,1.5) {11};
\end{scope}
\begin{scope}[>={Stealth[black]},
every node/.style={fill=white,circle},
every edge/.style={draw=black,very thick}]
    \path [->] (0) edge (1);
    \path [->] (1) edge[bend left=30] (2);
    \path [->] (2) edge[bend left=30] (1);
    \path [->] (1) edge (3);
    \path [->] (3) edge (4);
    \path [->] (4) edge (5);
    \path [->] (5) edge (3);
    \path [->] (3) edge (6);
    \path [->] (6) edge (7);
    \path [->] (7) edge (5);
    \path [->] (5) edge (8);
    \path [->] (8) edge (9);
    \path [->] (8) edge (9);
    \path [->] (9) edge (10);
    \path [->] (10) edge (11);
    \path [->] (11) edge (9);
\end{scope}
\begin{scope}[every node/.style={circle,fill,inner sep=0pt,minimum size=0.8pt}]
     \node (AA') at (-0.75,-2) {};
     \node (AA'') at (-0.75,-2.4) {};
     \node (BB') at (1,-2) {};
     \node (BB'') at (1,-2.4) {};
     \node (CC') at (5,-2) {};
     \node (CC'') at (5,-2.4) {};
     \node (DD') at (7,-2) {};
     \node (DD'') at (7,-2.4) {};
     \node (EE') at (10,-2) {};
     \node (EE'') at (10,-2.4) {};
\end{scope}
\begin{scope}[every node/.style={circle,fill,inner sep=0pt,minimum size=0pt}]
     \node (A) at (-0.75,-2.2) {};
     \node (A') at (-0.75,-2) {};
     \node (A'') at (-0.75,-2.4) {};
     \node (B) at (1,-2.2) {};
     \node (B') at (1,-2) {};
     \node (B'') at (1,-2.4) {};
     \node (C) at (5,-2.2) {};
     \node (C') at (5,-2) {};
     \node (C'') at (5,-2.4) {};
     \node (D) at (7,-2.2) {};
     \node (D') at (7,-2) {};
     \node (D'') at (7,-2.4) {};
     \node (E) at (10,-2.2) {};
     \node (E') at (10,-2) {};
     \node (E'') at (10,-2.4) {};
\end{scope}
\begin{scope}[>={Stealth[black]},
every edge/.style={draw=black,thick}]
    \path [-] (A) edge node[below,pos=0.5] {$B_1$} (B);
    \path [-] (B) edge node[below,pos=0.5] {$B_2$} (C);
    \path [-] (C) edge node[below,pos=0.5] {$B_3$} (D);
    \path [-] (D) edge node[below,pos=0.5] {$B_4$} (E);
    \path [-] (A') edge (A'');
    \path [-] (B') edge (B'');
    \path [-] (C') edge (C'');
    \path [-] (D') edge (D'');
    \path [-] (E') edge (E'');
\end{scope}
\end{tikzpicture}}
\caption{\label{fig:block-decomp}
Each term in~\eqref{eq:diagram} can be represented by means of a directed graph on the index set $\{0\}\cup[k]$, where edges are given by pairs of consecutive elements in $B_1\uplus\ldots\uplus B_r$ with possible multiplicities and where we include the edge $(0,b)$ with $b$ the root of $B_1$.
In this way, an edge to $i$ corresponds to an operator $\Jc_{L;x_T}^{x_i}$ in~\eqref{eq:diagram}.
For instance, the above diagram is associated to blocks $B_1=(1)\uplus(2,1)$, $B_2=(3)\uplus(4,5,3)\uplus(6,7,5)$, $B_3=(8)$, and $B_4=(9)\uplus(10,11,9)$.}
\end{center}}
\end{figure}

\medskip\noindent
We turn to the proof of~\eqref{eq:diagram}.
More precisely, we shall  prove the following seemingly simpler statement: for all disjoint index sets $S,T\subset[k]$ with $S\ne\varnothing$, we have
\begin{equation}\label{eq:diagram-ter}
\delta^{x_S}\psi_L^{x_T}\,=\,\sum_{b\in S}\Jc^{x_b}_{L;x_T}\delta^{x_{S\setminus\{b\}}}\bar \psi_L^{x_b}+\sum_{I\text{ string}\atop\langle I\rangle\subset S}~\sum_{c\in T}\calR^{I\uplus(c)}_{L;T}\underbrace{\delta^{x_{S\setminus \langle I\rangle}}\bar\psi_L^{x_I,x_T}}_{\displaystyle \clubsuit}.
\end{equation}
We quickly argue that this indeed implies~\eqref{eq:diagram}.
First, we iteratively replace the corrector difference $\clubsuit$ in~\eqref{eq:diagram-ter}, using \eqref{eq:diagram-ter} itself, to the effect that
\begin{eqnarray*}
\lefteqn{\delta^{x_S}\psi_L^{x_T}}\\
&=&\sum_{l\ge0}\sum_{I_1,\ldots,I_l\,\text{disjoint strings}\atop\langle I_1\rangle,\ldots,\langle I_l\rangle\subset S}\sum_{b\in S\setminus\langle I_1\uplus\ldots\uplus I_l\rangle}\sum_{c_1,\ldots,c_l\atop\forall j: c_j\in T\cup\langle I_1\uplus\ldots\uplus I_{j-1}\rangle}\calR_{L;T}^{I_1\uplus(c_1)}(x_{[k]})\calR_{L;T\cup\langle I_1\rangle}^{I_2\uplus(c_2)}(x_{[k]})\\
&&\hspace{3.5cm}\ldots\calR_{L;T\cup\langle I_1\uplus\ldots\uplus I_{l-1}\rangle}^{I_l\uplus(c_l)}(x_{[k]})\Jc_{L;x_{T\cup\langle I_1\uplus\ldots\uplus I_{l}\rangle}}^{x_b}\delta^{x_{S\setminus\langle I_1\uplus\ldots\uplus I_l\uplus(b)\rangle}}{\bar \psi_L^{x_b}}\\
&&\vspace{-1cm}\\
&&+\sum_{l\ge1}\sum_{I_1,\ldots,I_l\,\text{disjoint strings}\atop\langle I_1\rangle\cup\ldots\cup\langle I_l\rangle= S}\sum_{c_1,\ldots,c_l\atop\forall j: c_j\in T\cup\langle I_1\uplus\ldots\uplus I_{j-1}\rangle}\calR_{L;T}^{I_1\uplus(c_1)}(x_{[k]})\calR_{L;T\cup\langle I_1\rangle}^{I_2\uplus(c_2)}(x_{[k]})\\
&&\hspace{8cm}\ldots\calR_{L;T\cup\langle I_1\uplus\ldots\uplus I_{l-1}\rangle}^{I_l\uplus(c_l)}(x_{[k]}){\bar \psi_L^{x_{S\cup T}}}.
\end{eqnarray*}
In particular, recalling the notation~\eqref{eq:string-contr} for elementary contributions $\calR_{L;S}^I$, and recognizing the definition~\eqref{eq:block-contr} of block contributions, we deduce for all disjoint index sets $S,T\subset[k]$ with $S\ne\varnothing$,
\begin{multline*}
\sum_{b\in S}\Jc_{L;x_T}^{x_b}\delta^{x_{S\setminus\{b\}}}{\bar \psi_L^{x_b}}\\
\,=\,
\sum_{B\,\text{block}\atop\langle B\rangle\subset S}\calC^{B}_{L;T}(x_{[k]})\bigg(\sum_{b\in S\setminus\langle B\rangle}\Jc_{L;x_{\langle B\rangle}}^{x_b}\delta^{x_{S\setminus (\langle B\rangle\cup\{b\})}}{\bar \psi_L^{x_{b}}}\bigg)
+\sum_{B\,\text{block}\atop\langle B\rangle= S}\calC^{B}_{L;T}(x_{[k]}){\bar \psi_L^{x_S}}.
\end{multline*}
Iterating this identity, starting from~\eqref{eq:cordiff-JLHn-re} in form of
\begin{equation*}
\delta^{x_S}\psi_L^{\varnothing}\,=\,
\sum_{b\in S}\Jc_{L}^{x_b}\delta^{x_{S\setminus\{b\}}}{\bar \psi_L^{x_b}},
\end{equation*}
the claim~\eqref{eq:diagram} follows.

\medskip\noindent
We are left with the proof of~\eqref{eq:diagram-ter}. Given disjoint index sets $S,T\subset[k]$ with $S\ne\varnothing$, in view of~\eqref{eq:cordiff-JLHn-re}, corrector differences can be decomposed as
\begin{equation*}
\delta^{x_S}\psi_L^{x_T}\,=\,\sum_{b\in S}\Jc^{x_b}_{L;x_T}\delta^{x_{S\setminus\{b\}}}{\bar \psi_L^{x_b,x_T}}.
\end{equation*}
Decomposing {$\bar \psi_L^{x_b,x_T}=\bar \psi_L^{x_b}+(\psi_L^{x_b,x_T}-\psi_L^{x_b})$}, this becomes
\begin{equation}\label{eq:pre-diagram-ter}
\delta^{x_S}\psi_L^{x_T}\,=\,\sum_{b\in S}\Jc^{x_b}_{L;x_T}\delta^{x_{S\setminus\{b\}}}{\bar\psi_L^{x_b}}+\sum_{b\in S}\Jc^{x_b}_{L;x_T}\delta^{x_{S\setminus\{b\}}}(\psi_L^{x_b,x_T}-\psi_L^{x_b}),
\end{equation}
and it remains to further decompose the last right-hand side term.
For that purpose, in view of Lemma~\ref{eq:cordiff-form}, for all $D \subset S$ with $D\ne \varnothing$, we note that $\delta^{x_{S\setminus D}}(\psi_L^{x_D,x_T}-\psi_L^{x_D})$ satisfies
\begin{multline*}
-\triangle\delta^{x_{S\setminus D}}(\psi_L^{x_D,x_T}-\psi_L^{x_D})+\nabla\delta^{x_{S\setminus D}}\Big(\Sigma_L^{x_D,x_T}\mathds1_{ Q_L\setminus\cup_{i\in D \cup T}B(x_i)}- \Sigma_L^{x_D}\mathds1_{ Q_L\setminus \cup_{b \in D} B(x_{b})}\Big)\\
\,=\,-\sum_{b \in D} \delta_{\partial B(x_b)}\delta^{x_{S\setminus D}}{(\sigma_L^{x_D,x_T}-\sigma_L^{x_D})}\\
-\sum_{i\in S\setminus D}\delta_{\partial B(x_i)}\delta^{x_{S\setminus D \cup\{i\}}}{(\sigma_L^{x_D,x_i,x_T}-\sigma_L^{x_D,x_i})}
-\sum_{i\in T}\delta_{\partial B(x_i)}\delta^{x_{S\setminus D}}{\sigma_L^{x_D,x_T}},
\end{multline*}
which then allows to write
\[\delta^{x_{S\setminus D}}(\psi_L^{x_D,x_T}-\psi_L^{x_D})\,=\,\sum_{i\in S\setminus D}\Jc_{L;x_D}^{x_i}\delta^{x_{S\setminus D\cup \{i\}}}(\psi_L^{x_D,x_i,x_T}-\psi_L^{x_D,x_i})+\sum_{i\in T}\Jc^{x_i}_{L;x_D}\delta^{x_{S\setminus D}} \bar \psi_L^{x_D,x_T}.\]
Using iteratively this identity for $D$ exhausting $S$, we obtain upon recognizing the definition~\eqref{eq:string-contr} of elementary contributions,
\[\sum_{b\in S}\Jc_{L;x_T}^{x_b}\delta^{x_{S\setminus\{b\}}}(\psi_L^{x_b,x_T}-\psi_L^{x_b})\,=\,
\sum_{I\text{ string}\atop\langle I\rangle\subset S}~\sum_{c\in T}\calR^{I\uplus(c)}_{L;T}(x_{[k]})\delta^{x_{S\setminus \langle I\rangle}}\psi_L^{x_I,x_T}.\]
Inserting this into~\eqref{eq:pre-diagram-ter}, the claim~\eqref{eq:diagram-ter} follows.

\medskip
\step2 Estimation of block contributions and graphical notation.\\
Let $B$ be a block of indices with root $b$ and endpoint $f$.
By definition of elementary block contributions, cf.~\eqref{eq:block-contr}, for any index set $S$ that is disjoint from $\langle B\rangle$, Lemma~\ref{lem:decay-re} yields
\begin{equation}\label{eq:decomp-CLB-integral}
 \Big(\int_{B(z)}|\nabla\calC_{L;S}^{B}(x_{[k]})\zeta|^2\Big)^\frac12 
\,\lesssim\, \langle (z-x_{b})_L\rangle^{-d}D_B(x_B)\Big(\int_{B(x_{f})}|\nabla\zeta|^2\Big)^\frac12,
\end{equation}
where for any sequence $C=(c_1,\ldots,c_m)$ we define
\begin{equation}\label{eq:def-DBxB}
D_C(x_C)\,:=\,\langle(x_{c_1}-x_{c_2})_L\rangle^{-d}\ldots\langle(x_{c_{m-1}}-x_{c_m})_L\rangle^{-d}.
\end{equation}
As such contributions will be combined in intricate ways in the sequel, we introduce a convenient graphical notation. 
Integration variables are represented by small black circles and frozen variables by small white circles. The index of a frozen variable is occasionally indicated in the corresponding circle. A solid line between two vertices $i$ and $j$ represents a factor $\langle(x_i-x_j)_L\rangle^{-d}$. In particular, multiple edges correspond to powers of this factor. For instance, we have
\[{\tiny\begin{tikzpicture}[baseline={([yshift=-.8ex]current bounding box.center)},scale=0.6]
\begin{scope}[every node/.style={circle,thick,draw}]
    \node (1) at (0,0) {$1$};
    \node (4) at (2,0) {$4$};
\end{scope}
\begin{scope}[every node/.style={circle,fill,inner sep=0pt,minimum size=3pt}]
    \node (2) at (1,0.7) {};
    \node (3) at (1,-0.7) {};
\end{scope}
\begin{scope}[>={Stealth[black]},
every node/.style={fill=white,circle},
every edge/.style={draw=black,very thick}]
    \path [-] (1) edge (2);
    \path [-] (1) edge (3);
    \path [-] (2) edge[bend left=20] (3);
    \path [-] (2) edge[bend left=-20] (3);
    \path [-] (3) edge (4);
\end{scope}
\end{tikzpicture}}
\,=\,
\int_{(Q_L)^2}\langle(x_1-x_2)_L\rangle^{-d}\langle(x_2-x_3)_L\rangle^{-2d}\langle(x_3-x_1)_L\rangle^{-d}\langle(x_3-x_4)_L\rangle^{-d}\,dx_2dx_3.\]
When evaluating integrals with borderline factors $\langle(x_i-x_j)_L\rangle^{-d}$, we naturally obtain logarithmic factors, for which we shall use the short-hand notation
\begin{equation}\label{eq:def-Lc}
\Lc_L((z_i)_{i\in J})\,:=\,\log\Big(2+\max_{i,j\in J}|(z_i-z_j)_L|\Big).
\end{equation}
This is combined into our graphical notation as follows: a symbolic prefactor $\Lc$ in front of a diagram indicates that a factor $\Lc_L(x_S)$ is to be included into the corresponding integral, where~$S$ stands for the set of all implemented indices. For instance, for any power $\mu\ge0$, we have
\begin{multline*}
\Lc^\mu\,{\tiny\begin{tikzpicture}[baseline={([yshift=-.8ex]current bounding box.center)},scale=0.6]
\begin{scope}[every node/.style={circle,thick,draw}]
    \node (1) at (0,0) {$1$};
    \node (4) at (2,0) {$4$};
\end{scope}
\begin{scope}[every node/.style={circle,fill,inner sep=0pt,minimum size=3pt}]
    \node (2) at (1,0.7) {};
    \node (3) at (1,-0.7) {};
\end{scope}
\begin{scope}[>={Stealth[black]},
every node/.style={fill=white,circle},
every edge/.style={draw=black,very thick}]
    \path [-] (1) edge (2);
    \path [-] (1) edge (3);
    \path [-] (2) edge[bend left=20] (3);
    \path [-] (2) edge[bend left=-20] (3);
    \path [-] (3) edge (4);
\end{scope}
\end{tikzpicture}}
\,=\,
\int_{(Q_L)^2}\Lc_L(x_1,x_2,x_3,x_4)^\mu\\
\times\langle(x_1-x_2)_L\rangle^{-d}\langle(x_2-x_3)_L\rangle^{-2d}\langle(x_3-x_1)_L\rangle^{-d}\langle(x_3-x_4)_L\rangle^{-d}\,dx_2dx_3.
\end{multline*}
Noting that for any $\gamma>0$ and $\mu\ge0$ a direct evaluation of integrals yields
\begin{equation*}
\int_{Q_L}\Lc_L(x,y,z)^\mu\langle(x-y)_L\rangle^{-d}\langle(y-z)_L\rangle^{-\gamma}\,dy
\,\lesssim\,\left\{\begin{array}{lll}
\langle(x-z)_L\rangle^{-d}\Lc_L(x,z)^{\mu}&:&\gamma>d,\\
\langle(x-z)_L\rangle^{-d}\Lc_L(x,z)^{\mu+1}&:&\gamma= d,\\
\langle(x-z)_L\rangle^{-\gamma}\Lc_L(x,z)^{\mu+1}&:&\gamma< d,
\end{array}\right.
\end{equation*}
we deduce with our graphical notation
\begin{eqnarray}\label{eq:bound-integ-kernel-diag}
\Lc^\mu\,{\tiny\begin{tikzpicture}[baseline={([yshift=-.8ex]current bounding box.center)},scale=0.6]
\begin{scope}[every node/.style={circle,draw,fill=white,inner sep=0pt,minimum size=3pt}]
    \node (1) at (0,0) {};
    \node (3) at (2,0) {};
\end{scope}
\begin{scope}[every node/.style={circle,fill,inner sep=0pt,minimum size=3pt}]
    \node (2) at (1,0) {};
\end{scope}
\begin{scope}[>={Stealth[black]},
every node/.style={fill=white,circle},
every edge/.style={draw=black,very thick}]
    \path [-] (1) edge (2);
    \path [-] (2) edge (3);
\end{scope}
\end{tikzpicture}}
&\lesssim&
\Lc^{\mu+1}\,
{\tiny\begin{tikzpicture}[baseline={([yshift=-.8ex]current bounding box.center)},scale=0.6]
\begin{scope}[every node/.style={circle,draw,fill=white,inner sep=0pt,minimum size=3pt}]
    \node (1) at (0,0) {};
    \node (3) at (1,0) {};
\end{scope}
\begin{scope}[>={Stealth[black]},
every node/.style={fill=white,circle},
every edge/.style={draw=black,very thick}]
    \path [-] (1) edge (3);
\end{scope}
\end{tikzpicture}}\\
\Lc^\mu\,{\tiny\begin{tikzpicture}[baseline={([yshift=-.8ex]current bounding box.center)},scale=0.6]
\begin{scope}[every node/.style={circle,draw,fill=white,inner sep=0pt,minimum size=3pt}]
    \node (1) at (0,0) {};
    \node (3) at (2,0) {};
\end{scope}
\begin{scope}[every node/.style={circle,fill,inner sep=0pt,minimum size=3pt}]
    \node (2) at (1,0) {};
\end{scope}
\begin{scope}[>={Stealth[black]},
every node/.style={fill=white,circle},
every edge/.style={draw=black,very thick}]
    \path [-] (1) edge[bend left=30] (2);
    \path [-] (1) edge[bend left=-30] (2);
    \path [-] (2) edge (3);
\end{scope}
\end{tikzpicture}}
&\lesssim&
\Lc^\mu\,{\tiny\begin{tikzpicture}[baseline={([yshift=-.8ex]current bounding box.center)},scale=0.6]
\begin{scope}[every node/.style={circle,draw,fill=white,inner sep=0pt,minimum size=3pt}]
    \node (1) at (0,0) {};
    \node (3) at (1,0) {};
\end{scope}
\begin{scope}[>={Stealth[black]},
every node/.style={fill=white,circle},
every edge/.style={draw=black,very thick}]
    \path [-] (1) edge (3);
\end{scope}
\end{tikzpicture}}\nonumber
\end{eqnarray}
which allows for instance to estimate graphically
\[{\tiny\begin{tikzpicture}[baseline={([yshift=-.8ex]current bounding box.center)},scale=0.6]
\begin{scope}[every node/.style={circle,thick,draw}]
    \node (1) at (0,0) {$1$};
    \node (4) at (2,0) {$4$};
\end{scope}
\begin{scope}[every node/.style={circle,fill,inner sep=0pt,minimum size=3pt}]
    \node (2) at (1,0.7) {};
    \node (3) at (1,-0.7) {};
\end{scope}
\begin{scope}[>={Stealth[black]},
every node/.style={fill=white,circle},
every edge/.style={draw=black,very thick}]
    \path [-] (1) edge (2);
    \path [-] (1) edge (3);
    \path [-] (2) edge[bend left=20] (3);
    \path [-] (2) edge[bend left=-20] (3);
    \path [-] (3) edge (4);
\end{scope}
\end{tikzpicture}}
~\lesssim~
{\tiny\begin{tikzpicture}[baseline={([yshift=-.8ex]current bounding box.center)},scale=0.6]
\begin{scope}[every node/.style={circle,thick,draw}]
    \node (1) at (0,0) {$1$};
    \node (4) at (2,0) {$4$};
\end{scope}
\begin{scope}[every node/.style={circle,fill,inner sep=0pt,minimum size=3pt}]
    \node (3) at (1,-0.7) {};
\end{scope}
\begin{scope}[every node/.style={circle,inner sep=0pt,minimum size=0pt}]
    \node (2) at (1,0.7) {};
\end{scope}
\begin{scope}[>={Stealth[black]},
every node/.style={fill=white,circle},
every edge/.style={draw=black,very thick}]
    \path [-] (1) edge[bend left=20] (3);
    \path [-] (1) edge[bend left=-20] (3);
    \path [-] (3) edge (4);
\end{scope}
\end{tikzpicture}}
~\lesssim~
{\tiny\begin{tikzpicture}[baseline={([yshift=-.8ex]current bounding box.center)},scale=0.6]
\begin{scope}[every node/.style={circle,thick,draw}]
    \node (1) at (0,0) {$1$};
    \node (4) at (1.4,0) {$4$};
\end{scope}
\begin{scope}[>={Stealth[black]},
every node/.style={fill=white,circle},
every edge/.style={draw=black,very thick}]
    \path [-] (1) edge (4);
\end{scope}
\end{tikzpicture}}
~=~
\langle(x_1-x_4)_L\rangle^{-d}.
\]
The counting of logarithmic factors in the sequel will be quite trivial as we shall notice that at most one logarithmic factor appears each time a vertex disappears in the graphical representation. This rough bound can often be improved, but it suffices for our purposes.

\medskip\noindent
We need to add one more ingredient to our graphical notation. Indeed, in the sequel, we replace the density function $f_{k+1}$ in~\eqref{eq:decomp-BLj1-intens} by its expansion~\eqref{eq:dens-correl} in terms of correlation functions, and we estimate the latter by appealing both to the decay assumption~\ref{Mix-om-n} and to the uniform bound~\eqref{eq:high-intens-correl}.
This leads us to combine products of the form~$D_C(x_C)$ with products of factors of the form $\omega((x_i-x_j)_L)\wedge\lambda_{p_i}(\Pc)$ for some $p_i\ge0$. In our graphical notation, such a factor is represented by a dashed line between vertices~$i$ and~$j$. In principle, the value $p_i$ should be included in the notation to precise the value of the edge. For convenience, we rather use a simplified notation: for a diagram with $s$ dashed lines, a symbolic prefactor $\lambda^\circ_{k}$ indicates that the dashed lines correspond to factors \mbox{$(\omega((\cdot)_L)\wedge\lambda_{p_i}(\Pc))_{1\le i\le s}$} with any $p_1,\ldots,p_s\ge1$ satisfying $p_1+\ldots+p_s=k$, and we take the sum over the different possible choices of such $p_i$'s. For instance,
\begin{multline*}
\lambda_{k}^\circ\,{\tiny\begin{tikzpicture}[baseline={([yshift=-.8ex]current bounding box.center)},scale=0.6]
\begin{scope}[every node/.style={circle,thick,draw}]
    \node (1) at (0,0) {$1$};
    \node (4) at (2,0) {$4$};
\end{scope}
\begin{scope}[every node/.style={circle,fill,inner sep=0pt,minimum size=3pt}]
    \node (2) at (1,0.7) {};
    \node (3) at (1,-0.7) {};
\end{scope}
\begin{scope}[>={Stealth[black]},
every node/.style={fill=white,circle},
every edge/.style={draw=black,very thick}]
    \path [-] (1) edge (3);
    \path [-] (2) edge[bend left=-20] (3);
    \path [-] (3) edge (4);
\end{scope}
\begin{scope}[>={Stealth[black]},
every edge/.style={draw=black,very thick,dashed}]
    \path [-] (1) edge (2);
    \path [-] (2) edge[bend left=20] (3);
\end{scope}
\end{tikzpicture}}
\,=\,
\sum_{p_1,p_2\ge1\atop p_1+p_2=k}\int_{(Q_L)^2}\big(\omega((x_1-x_2)_L)\wedge\lambda_{p_1}(\Pc)\big)\big(\omega((x_2-x_3)_L)\wedge\lambda_{p_2}(\Pc)\big)\\
\times\langle(x_2-x_3)_L\rangle^{-d}\langle(x_3-x_1)_L\rangle^{-d}\langle(x_3-x_4)_L\rangle^{-d}\,dx_2dx_3.
\end{multline*}
In addition, a symbolic prefactor $\lambda_{k}'$ in front of a diagram with~$s$ dashed lines indicates that the whole expression is multiplied by a factor $\lambda_{p_0}(\Pc)$ and that the dashed lines correspond to factors $(\omega((\cdot)_L)\wedge\lambda_{p_i}(\Pc))_{1\le i\le s}$ with any $p_0,\ldots,p_s\ge1$ satisfying $p_0+\ldots+p_s=k$, where we again take the sum over all possible choices.
In other words,
\[\big(\lambda_{k}'\big)
\,=\,
\sum_{p=1}^k\lambda_{p}(\Pc)\times\big(\lambda_{k-p}^\circ\big).\]
As obviously $\omega((\cdot)_L)\wedge\lambda_{p_i}(\Pc)\le\lambda_{p_i}(\Pc)$, we get with our notation
\begin{equation}\label{eq:rule-lambda-double}
\lambda_{k}^\circ\,{\tiny\begin{tikzpicture}[baseline={([yshift=-.8ex]current bounding box.center)},scale=0.6]
\begin{scope}[every node/.style={circle,draw,fill=white,inner sep=0pt,minimum size=3pt}]
    \node (1) at (0,0) {};
    \node (2) at (1,0) {};
\end{scope}
\begin{scope}[>={Stealth[black]},
every node/.style={fill=white,circle},
every edge/.style={draw=black,very thick}]
    \path [-] (1) edge[bend left=30] (2);
\end{scope}
\begin{scope}[>={Stealth[black]},
every edge/.style={draw=black,very thick,dashed}]
    \path [-] (1) edge[bend left=-30] (2);
\end{scope}
\end{tikzpicture}}
~\le~
\lambda_k'\,{\tiny\begin{tikzpicture}[baseline={([yshift=-.8ex]current bounding box.center)},scale=0.6]
\begin{scope}[every node/.style={circle,draw,fill=white,inner sep=0pt,minimum size=3pt}]
    \node (1) at (0,0) {};
    \node (2) at (1,0) {};
\end{scope}
\begin{scope}[>={Stealth[black]},
every node/.style={fill=white,circle},
every edge/.style={draw=black,very thick}]
    \path [-] (1) edge (2);
\end{scope}
\end{tikzpicture}}
~=~
\lambda_k(\Pc)\,{\tiny\begin{tikzpicture}[baseline={([yshift=-.8ex]current bounding box.center)},scale=0.6]
\begin{scope}[every node/.style={circle,draw,fill=white,inner sep=0pt,minimum size=3pt}]
    \node (1) at (0,0) {};
    \node (2) at (1,0) {};
\end{scope}
\begin{scope}[>={Stealth[black]},
every node/.style={fill=white,circle},
every edge/.style={draw=black,very thick}]
    \path [-] (1) edge (2);
\end{scope}
\end{tikzpicture}}
\end{equation}
Next, we combine this with the notation~$\Lc$ for logarithmic factors:
in front of a diagram with a prefactor $\lambda_k^\circ$, a symbolic prefactor~$\Lc$ indicates that either a factor $\Lc(x_S)$ is to be included into the corresponding integral, where~$S$ stands for the set of implemented indices, or that one of the factors $\omega((x_i-x_j)_L)\wedge\lambda_{p_i}(\Pc)$ is to be replaced by
\[\big(\omega((x_i-x_j)_L)\Lc_L(x_i,x_j)\big)\wedge\big(\lambda_{p_i}(\Pc)|\!\log\lambda(\Pc)|\big),\]
and we take the sum over the two choices.
Powers of $\Lc$ are defined accordingly: for instance, for any $\mu\ge0$,
\begin{multline*}
\Lc^\mu\lambda_{k}^\circ\,{\tiny\begin{tikzpicture}[baseline={([yshift=-.8ex]current bounding box.center)},scale=0.6]
\begin{scope}[every node/.style={circle,thick,draw}]
    \node (1) at (0,0) {$1$};
    \node (4) at (2,0) {$4$};
\end{scope}
\begin{scope}[every node/.style={circle,fill,inner sep=0pt,minimum size=3pt}]
    \node (2) at (1,0.7) {};
    \node (3) at (1,-0.7) {};
\end{scope}
\begin{scope}[>={Stealth[black]},
every node/.style={fill=white,circle},
every edge/.style={draw=black,very thick}]
    \path [-] (1) edge (3);
    \path [-] (2) edge[bend left=-20] (3);
    \path [-] (3) edge (4);
\end{scope}
\begin{scope}[>={Stealth[black]},
every edge/.style={draw=black,very thick,dashed}]
    \path [-] (1) edge (2);
    \path [-] (2) edge[bend left=20] (3);
\end{scope}
\end{tikzpicture}}
\,=\,
\sum_{\mu_0,\mu_1,\mu_2\ge0\atop\mu_0+\mu_1+\mu_2=\mu}
\sum_{p_0,p_1,p_2\ge1\atop p_0+p_1+p_2=k}\int_{(Q_L)^2}\Lc_L(x_1,x_2,x_3,x_4)^{\mu_0}\\
\times\Big(\big(\omega((x_1-x_2)_L)\Lc_L(x_1,x_2)^{\mu_1}\big)\wedge\big(\lambda_{p_1}(\Pc)|\!\log\lambda(\Pc)|^{\mu_1}\big)\Big)\\
\times \Big(\big(\omega((x_2-x_3)_L)\Lc(x_2,x_3)^{\mu_2}\big)\wedge\big(\lambda_{p_2}(\Pc)|\!\log\lambda(\Pc)|^{\mu_2}\big)\Big)\\
\times\langle(x_2-x_3)_L\rangle^{-d}\langle(x_3-x_1)_L\rangle^{-d}\langle(x_3-x_4)_L\rangle^{-d}\,dx_2dx_3.
\end{multline*}
When $\Lc$ is in front of a diagram with a prefactor $\lambda_k'$, we add the possibility of multiplying the whole expression by a factor $|\!\log\lambda(\Pc)|$: for all $\mu\ge0$, this means
\[\big(\Lc^\mu\lambda_k'\big)\,=\,\sum_{\nu=0}^\mu\sum_{p=1}^k\lambda_p(\Pc)|\!\log\lambda(\Pc)|^\nu\times\big(\Lc^{\mu-\nu}\lambda_{k-p}^\circ\big).\]
In case of an algebraic rate $\omega(t)=Ct^{-\beta}$ for some $\beta\in(0,d)$,
a direct evaluation of integrals yields, for all $\lambda>0$ and $\mu,\nu\ge0$,
\begin{multline*}
\int_{Q_L}\Lc_L(x,y,z)^\mu\langle(x-y)_L\rangle^{-d}\Big(\big(\omega((y-z)_L)\Lc_L(y,z)^\nu\big)\wedge\big(\lambda|\!\log\lambda|^\nu\big)\Big)\,dy\\
\,\lesssim\,\big(\omega((x-z)_L)\Lc_L(x-z)^{\mu+\nu+1}\big)\wedge\big(\lambda|\!\log\lambda|^{\mu+\nu+1}\big).
\end{multline*}
With our graphical notation, recalling Lemma~\ref{lem:gen-cond-lambd}(ii), this estimate and similar computations yield for all $\mu\ge0$,
\begin{eqnarray}\label{eq:bound-integ-kernel-diag-dash}
\Lc^\mu\lambda_k^\circ
\,{\tiny\begin{tikzpicture}[baseline={([yshift=-.8ex]current bounding box.center)},scale=0.6]
\begin{scope}[every node/.style={circle,draw,fill=white,inner sep=0pt,minimum size=3pt}]
    \node (1) at (0,0) {};
    \node (3) at (2,0) {};
\end{scope}
\begin{scope}[every node/.style={circle,fill,inner sep=0pt,minimum size=3pt}]
    \node (2) at (1,0) {};
\end{scope}
\begin{scope}[>={Stealth[black]},
every node/.style={fill=white,circle},
every edge/.style={draw=black,very thick}]
    \path [-] (1) edge (2);
\end{scope}
\begin{scope}[>={Stealth[black]},
every edge/.style={draw=black,very thick,dashed}]
    \path [-] (2) edge (3);
\end{scope}
\end{tikzpicture}}
&\lesssim&
\Lc^{\mu+1}\lambda_k^\circ\,
{\tiny\begin{tikzpicture}[baseline={([yshift=-.8ex]current bounding box.center)},scale=0.6]
\begin{scope}[every node/.style={circle,draw,fill=white,inner sep=0pt,minimum size=3pt}]
    \node (1) at (0,0) {};
    \node (3) at (1,0) {};
\end{scope}
\begin{scope}[>={Stealth[black]},
every edge/.style={draw=black,very thick,dashed}]
    \path [-] (1) edge (3);
\end{scope}
\end{tikzpicture}}
\\
\nonumber
\Lc^\mu\lambda_k^\circ
\,{\tiny\begin{tikzpicture}[baseline={([yshift=-.8ex]current bounding box.center)},scale=0.6]
\begin{scope}[every node/.style={circle,draw,fill=white,inner sep=0pt,minimum size=3pt}]
    \node (1) at (0,0) {};
    \node (3) at (2,0) {};
\end{scope}
\begin{scope}[every node/.style={circle,fill,inner sep=0pt,minimum size=3pt}]
    \node (2) at (1,0) {};
\end{scope}
\begin{scope}[>={Stealth[black]},
every node/.style={fill=white,circle},
every edge/.style={draw=black,very thick}]
    \path [-] (1) edge[bend left=30] (2);
    \path [-] (1) edge[bend left=-30] (2);
\end{scope}
\begin{scope}[>={Stealth[black]},
every edge/.style={draw=black,very thick,dashed}]
    \path [-] (2) edge (3);
\end{scope}
\end{tikzpicture}}
&\lesssim&
\Lc^{\mu}\lambda_k^\circ\,
{\tiny\begin{tikzpicture}[baseline={([yshift=-.8ex]current bounding box.center)},scale=0.6]
\begin{scope}[every node/.style={circle,draw,fill=white,inner sep=0pt,minimum size=3pt}]
    \node (1) at (0,0) {};
    \node (3) at (1,0) {};
\end{scope}
\begin{scope}[>={Stealth[black]},
every edge/.style={draw=black,very thick,dashed}]
    \path [-] (1) edge (3);
\end{scope}
\end{tikzpicture}}
\\
\nonumber
\Lc^\mu\lambda_k^\circ
\,{\tiny\begin{tikzpicture}[baseline={([yshift=-.8ex]current bounding box.center)},scale=0.6]
\begin{scope}[every node/.style={circle,draw,fill=white,inner sep=0pt,minimum size=3pt}]
    \node (1) at (0,0) {};
    \node (3) at (2,0) {};
\end{scope}
\begin{scope}[every node/.style={circle,fill,inner sep=0pt,minimum size=3pt}]
    \node (2) at (1,0) {};
\end{scope}
\begin{scope}[>={Stealth[black]},
every node/.style={fill=white,circle},
every edge/.style={draw=black,very thick}]
    \path [-] (1) edge[bend left=30] (2);
    \path [-] (2) edge (3);
\end{scope}
\begin{scope}[>={Stealth[black]},
every edge/.style={draw=black,very thick,dashed}]
    \path [-] (1) edge[bend left=-30] (2);
\end{scope}
\end{tikzpicture}}
&\lesssim&
\Lc^{\mu+1}\lambda_k'\,
{\tiny\begin{tikzpicture}[baseline={([yshift=-.8ex]current bounding box.center)},scale=0.6]
\begin{scope}[every node/.style={circle,draw,fill=white,inner sep=0pt,minimum size=3pt}]
    \node (1) at (0,0) {};
    \node (3) at (1,0) {};
\end{scope}
\begin{scope}[>={Stealth[black]},
every node/.style={fill=white,circle},
every edge/.style={draw=black,very thick}]
    \path [-] (1) edge (3);
\end{scope}
\end{tikzpicture}}
\\
\nonumber
\Lc^\mu\lambda_k^\circ
\,{\tiny\begin{tikzpicture}[baseline={([yshift=-.8ex]current bounding box.center)},scale=0.6]
\begin{scope}[every node/.style={circle,draw,fill=white,inner sep=0pt,minimum size=3pt}]
    \node (1) at (0,0) {};
    \node (3) at (2,0) {};
\end{scope}
\begin{scope}[every node/.style={circle,fill,inner sep=0pt,minimum size=3pt}]
    \node (2) at (1,0) {};
\end{scope}
\begin{scope}[>={Stealth[black]},
every node/.style={fill=white,circle},
every edge/.style={draw=black,very thick}]
    \path [-] (1) edge[bend left=30] (2);
\end{scope}
\begin{scope}[>={Stealth[black]},
every edge/.style={draw=black,very thick,dashed}]
    \path [-] (2) edge (3);
    \path [-] (1) edge[bend left=-30] (2);
\end{scope}
\end{tikzpicture}}
&\lesssim&
\Lc^{\mu+1}\lambda_k'\,
{\tiny\begin{tikzpicture}[baseline={([yshift=-.8ex]current bounding box.center)},scale=0.6]
\begin{scope}[every node/.style={circle,draw,fill=white,inner sep=0pt,minimum size=3pt}]
    \node (1) at (0,0) {};
    \node (3) at (1,0) {};
\end{scope}
\begin{scope}[>={Stealth[black]},
every edge/.style={draw=black,very thick,dashed}]
    \path [-] (1) edge (3);
\end{scope}
\end{tikzpicture}}
\end{eqnarray}
which allows for instance to estimate graphically
\begin{multline*}
\Lc^\mu\lambda_{k}^\circ\,{\tiny\begin{tikzpicture}[baseline={([yshift=-.8ex]current bounding box.center)},scale=0.6]
\begin{scope}[every node/.style={circle,thick,draw}]
    \node (1) at (0,0) {$1$};
    \node (4) at (2,0) {$4$};
\end{scope}
\begin{scope}[every node/.style={circle,fill,inner sep=0pt,minimum size=3pt}]
    \node (2) at (1,0.7) {};
    \node (3) at (1,-0.7) {};
\end{scope}
\begin{scope}[>={Stealth[black]},
every node/.style={fill=white,circle},
every edge/.style={draw=black,very thick}]
    \path [-] (1) edge (3);
    \path [-] (2) edge[bend left=-20] (3);
    \path [-] (3) edge (4);
\end{scope}
\begin{scope}[>={Stealth[black]},
every edge/.style={draw=black,very thick,dashed}]
    \path [-] (1) edge (2);
    \path [-] (2) edge[bend left=20] (3);
\end{scope}
\end{tikzpicture}}
~\lesssim~
\Lc^{\mu+1}\lambda_{k}'\,{\tiny\begin{tikzpicture}[baseline={([yshift=-.8ex]current bounding box.center)},scale=0.6]
\begin{scope}[every node/.style={circle,thick,draw}]
    \node (1) at (0,0) {$1$};
    \node (4) at (2,0) {$4$};
\end{scope}
\begin{scope}[every node/.style={circle,fill,inner sep=0pt,minimum size=3pt}]
    \node (3) at (1,-0.7) {};
\end{scope}
\begin{scope}[every node/.style={circle,inner sep=0pt,minimum size=0pt}]
    \node (2) at (1,0.7) {};
\end{scope}
\begin{scope}[>={Stealth[black]},
every node/.style={fill=white,circle},
every edge/.style={draw=black,very thick}]
    \path [-] (1) edge[bend left=-20] (3);
    \path [-] (3) edge (4);
\end{scope}
\begin{scope}[>={Stealth[black]},
every edge/.style={draw=black,very thick,dashed}]
    \path [-] (1) edge[bend left=20] (3);
\end{scope}
\end{tikzpicture}}
~\lesssim~
\Lc^{\mu+2}\lambda_{k}'\,{\tiny\begin{tikzpicture}[baseline={([yshift=-.8ex]current bounding box.center)},scale=0.6]
\begin{scope}[every node/.style={circle,thick,draw}]
    \node (1) at (0,0) {$1$};
    \node (4) at (1.4,0) {$4$};
\end{scope}
\begin{scope}[>={Stealth[black]},
every node/.style={fill=white,circle},
every edge/.style={draw=black,very thick}]
    \path [-] (1) edge (4);
\end{scope}
\end{tikzpicture}}\\
~\le~\lambda_k(\Pc)\langle(x_1-x_4)_L\rangle^{-d}\big(|\!\log\lambda(\Pc)|+\Lc_L(x_1,x_4)\big)^{\mu+2}.
\end{multline*}
This notation will be used abundantly in the sequel.

\medskip
\step3 Partial integration on blocks.\\
Let $B$ be a block of indices with root $b$ and endpoint $f$.
We shall establish the following key estimate for partial integrals on~$B$: for all $\alpha,\beta\in \langle B\rangle$,
\begin{multline}\label{eq:integ-DB-plop}
\int_{(Q_L)^{\sharp \langle B\rangle\setminus \{b,f,\alpha,\beta\}}}D_B(x_B)\,dx_{\langle B\rangle\setminus \{b,f,\alpha,\beta\}}
\,\lesssim_B\,\Lc_L(x_b,x_f,x_\alpha,x_\beta)^{\sharp \langle B\rangle\setminus \{b,f,\alpha,\beta\}}\\
\times\Big(D_{(\alpha,b,f,\beta)}(x_{[k]})+D_{(\alpha,f,b,\beta)}(x_{[k]})\langle(x_b-x_f)_L\rangle^{-d}\Big)\\
\wedge\Big(D_{(\alpha,f,b,\beta)}(x_{[k]})+D_{(\alpha,b,f,\beta)}(x_{[k]})\langle(x_b-x_f)_L\rangle^{-d}\Big).
\end{multline}
%
With the above graphical notation, if $b,f,\alpha,\beta$ are all distinct, this can be written as
\begin{equation}\label{eq:integ-DB-plop-diagr-bis}
{\tiny\begin{tikzpicture}[baseline={([yshift=-.8ex]current bounding box.center)},scale=0.6]
\filldraw[gray] (0,0) rectangle (2,2);
\begin{scope}[every node/.style={circle,thick,draw,fill=white}]
    \node (1) at (0,1) {$b$};
    \node (2) at (2,1) {$f$};
    \node (3) at (1,2) {$\alpha$};
    \node (4) at (1,0) {$\beta$};
\end{scope}
\end{tikzpicture}}
~\lesssim~
\Lc^{\sharp B-4}\bigg(
{\tiny\begin{tikzpicture}[baseline={([yshift=-.8ex]current bounding box.center)},scale=0.6]
\begin{scope}[every node/.style={circle,thick,draw}]
    \node (1) at (0,1) {$b$};
    \node (2) at (2,1) {$f$};
    \node (3) at (1,2) {$\alpha$};
    \node (4) at (1,0) {$\beta$};
\end{scope}
\begin{scope}[>={Stealth[black]},
every node/.style={fill=white,circle},
every edge/.style={draw=black,very thick}]
    \path [-] (1) edge (2);
    \path [-] (1) edge (3);
    \path [-] (4) edge (2);
\end{scope}
\end{tikzpicture}}
+
{\tiny\begin{tikzpicture}[baseline={([yshift=-.8ex]current bounding box.center)},scale=0.6]
\begin{scope}[every node/.style={circle,thick,draw}]
    \node (1) at (0,1) {$b$};
    \node (2) at (2,1) {$f$};
    \node (3) at (1,2) {$\alpha$};
    \node (4) at (1,0) {$\beta$};
\end{scope}
\begin{scope}[>={Stealth[black]},
every node/.style={fill=white,circle},
every edge/.style={draw=black,very thick}]
    \path [-] (1) edge[bend left=10] (2);
    \path [-] (2) edge[bend left=10] (1);
    \path [-] (1) edge (4);
    \path [-] (3) edge (2);
\end{scope}
\end{tikzpicture}}
\bigg)
\wedge\bigg(
{\tiny\begin{tikzpicture}[baseline={([yshift=-.8ex]current bounding box.center)},scale=0.6]
\begin{scope}[every node/.style={circle,thick,draw}]
    \node (1) at (0,1) {$b$};
    \node (2) at (2,1) {$f$};
    \node (3) at (1,2) {$\alpha$};
    \node (4) at (1,0) {$\beta$};
\end{scope}
\begin{scope}[>={Stealth[black]},
every node/.style={fill=white,circle},
every edge/.style={draw=black,very thick}]
    \path [-] (1) edge (2);
    \path [-] (2) edge (3);
    \path [-] (4) edge (1);
\end{scope}
\end{tikzpicture}}
+
{\tiny\begin{tikzpicture}[baseline={([yshift=-.8ex]current bounding box.center)},scale=0.6]
\begin{scope}[every node/.style={circle,thick,draw}]
    \node (1) at (0,1) {$b$};
    \node (2) at (2,1) {$f$};
    \node (3) at (1,2) {$\alpha$};
    \node (4) at (1,0) {$\beta$};
\end{scope}
\begin{scope}[>={Stealth[black]},
every node/.style={fill=white,circle},
every edge/.style={draw=black,very thick}]
    \path [-] (1) edge[bend left=10] (2);
    \path [-] (2) edge[bend left=10] (1);
    \path [-] (2) edge (4);
    \path [-] (3) edge (1);
\end{scope}
\end{tikzpicture}}
\bigg),
\end{equation}
where henceforth gray squares stand for integration on a generic block.
As these estimates will be abundantly used in the sequel, we also display the important special case when one further integrates over $b$ or $f$, which is deduced by applying~\eqref{eq:bound-integ-kernel-diag},
\begin{equation}\label{eq:block-integr-b=f}
{\tiny\begin{tikzpicture}[baseline={([yshift=-.8ex]current bounding box.center)},scale=0.6]
\filldraw[gray] (0,0) rectangle (2,2);
\begin{scope}[every node/.style={circle,thick,draw,fill=white}]
    \node (1) at (0,1) {$b$};
    \node (3) at (1,2) {$\alpha$};
    \node (4) at (1,0) {$\beta$};
\end{scope}
\end{tikzpicture}}\hspace{0.28cm}
~\lesssim~
\Lc^{\sharp B-3}\,
{\tiny\begin{tikzpicture}[baseline={([yshift=-.8ex]current bounding box.center)},scale=0.6]
\begin{scope}[every node/.style={circle,thick,draw}]
    \node (1) at (0,1) {$b$};
    \node (3) at (1,2) {$\alpha$};
    \node (4) at (1,0) {$\beta$};
\end{scope}
\begin{scope}[>={Stealth[black]},
every node/.style={fill=white,circle},
every edge/.style={draw=black,very thick}]
    \path [-] (1) edge (3);
    \path [-] (4) edge (1);
\end{scope}
\end{tikzpicture}}
\qquad\text{and}\qquad
{\tiny\begin{tikzpicture}[baseline={([yshift=-.8ex]current bounding box.center)},scale=0.6]
\filldraw[gray] (0,0) rectangle (2,2);
\begin{scope}[every node/.style={circle,thick,draw,fill=white}]
    \node (1) at (2,1) {$f$};
    \node (3) at (1,2) {$\alpha$};
    \node (4) at (1,0) {$\beta$};
\end{scope}
\end{tikzpicture}}
~\lesssim~
\Lc^{\sharp B-3}\,
{\tiny\begin{tikzpicture}[baseline={([yshift=-.8ex]current bounding box.center)},scale=0.6]
\begin{scope}[every node/.style={circle,thick,draw}]
    \node (1) at (2,1) {$f$};
    \node (3) at (1,2) {$\alpha$};
    \node (4) at (1,0) {$\beta$};
\end{scope}
\begin{scope}[>={Stealth[black]},
every node/.style={fill=white,circle},
every edge/.style={draw=black,very thick}]
    \path [-] (1) edge (3);
    \path [-] (4) edge (1);
\end{scope}
\end{tikzpicture}}
\end{equation}
and we further display the special cases when $\alpha=\beta$,
\begingroup\allowdisplaybreaks
\begin{eqnarray}\label{eq:block-integr-alph=beta}
{\tiny\begin{tikzpicture}[baseline={([yshift=-.8ex]current bounding box.center)},scale=0.6]
\filldraw[gray] (0,0) rectangle (2,2);
\begin{scope}[every node/.style={circle,thick,draw,fill=white}]
    \node (1) at (0,1) {$b$};
    \node (2) at (2,1) {$f$};
    \node (3) at (1,2) {$\alpha$};
\end{scope}
\end{tikzpicture}}
&\lesssim&
\Lc^{\sharp B-3}\,
{\tiny\begin{tikzpicture}[baseline={([yshift=-.8ex]current bounding box.center)},scale=0.6]
\begin{scope}[every node/.style={circle,thick,draw}]
    \node (1) at (0,0) {$b$};
    \node (2) at (2,0) {$f$};
    \node (3) at (1,1.2) {$\alpha$};
\end{scope}
\begin{scope}[>={Stealth[black]},
every node/.style={fill=white,circle},
every edge/.style={draw=black,very thick}]
    \path [-] (1) edge (2);
    \path [-] (1) edge (3);
\end{scope}
\end{tikzpicture}}
\wedge
{\tiny\begin{tikzpicture}[baseline={([yshift=-.8ex]current bounding box.center)},scale=0.6]
\begin{scope}[every node/.style={circle,thick,draw}]
    \node (1) at (0,0) {$b$};
    \node (2) at (2,0) {$f$};
    \node (3) at (1,1.2) {$\alpha$};
\end{scope}
\begin{scope}[>={Stealth[black]},
every node/.style={fill=white,circle},
every edge/.style={draw=black,very thick}]
    \path [-] (1) edge (2);
    \path [-] (2) edge (3);
\end{scope}
\end{tikzpicture}}
\\
\label{eq:block-integr-alph=beta}
{\tiny\begin{tikzpicture}[baseline={([yshift=-.8ex]current bounding box.center)},scale=0.6]
\filldraw[gray] (0,0) rectangle (2,2);
\begin{scope}[every node/.style={circle,thick,draw,fill=white}]
    \node (1) at (0,1) {$b$};
    \node (3) at (1,2) {$\alpha$};
\end{scope}
\end{tikzpicture}}\hspace{0.28cm}
&\lesssim&
\Lc^{\sharp B-2}\,
{\tiny\begin{tikzpicture}[baseline={([yshift=-.8ex]current bounding box.center)},scale=0.6]
\begin{scope}[every node/.style={circle,thick,draw}]
    \node (1) at (0,0) {$b$};
    \node (3) at (1,1.2) {$\alpha$};
\end{scope}
\begin{scope}[>={Stealth[black]},
every node/.style={fill=white,circle},
every edge/.style={draw=black,very thick}]
    \path [-] (1) edge[bend left=10] (3);
    \path [-] (3) edge[bend left=10] (1);
\end{scope}
\end{tikzpicture}}
\\
\label{eq:block-integr-alph=beta-rev}
{\tiny\begin{tikzpicture}[baseline={([yshift=-.8ex]current bounding box.center)},scale=0.6]
\filldraw[gray] (0,0) rectangle (2,2);
\begin{scope}[every node/.style={circle,thick,draw,fill=white}]
    \node (2) at (2,1) {$f$};
    \node (3) at (1,2) {$\alpha$};
\end{scope}
\end{tikzpicture}}
&\lesssim&
\Lc^{\sharp B-2}\,
{\tiny\begin{tikzpicture}[baseline={([yshift=-.8ex]current bounding box.center)},scale=0.6]
\begin{scope}[every node/.style={circle,thick,draw}]
    \node (2) at (2,1) {$f$};
    \node (3) at (1,2) {$\alpha$};
\end{scope}
\begin{scope}[>={Stealth[black]},
every node/.style={fill=white,circle},
every edge/.style={draw=black,very thick}]
    \path [-] (3) edge[bend left=10] (2);
    \path [-] (2) edge[bend left=10] (3);
\end{scope}
\end{tikzpicture}}
\\
\label{eq:block-integr-b=alpha=bet}
{\tiny\begin{tikzpicture}[baseline={([yshift=-.8ex]current bounding box.center)},scale=0.6]
\filldraw[gray] (0,0) rectangle (2,2);
\begin{scope}[every node/.style={circle,thick,draw,fill=white}]
    \node (1) at (0,1) {$b$};
    \node (2) at (2,1) {$f$};
\end{scope}
\end{tikzpicture}}
&\lesssim&
\Lc^{\sharp B-2}\,
{\tiny\begin{tikzpicture}[baseline={([yshift=-.8ex]current bounding box.center)},scale=0.6]
\begin{scope}[every node/.style={circle,thick,draw}]
    \node (1) at (0,1) {$b$};
    \node (2) at (2,1) {$f$};
\end{scope}
\begin{scope}[>={Stealth[black]},
every node/.style={fill=white,circle},
every edge/.style={draw=black,very thick}]
    \path [-] (1) edge[bend left=10] (2);
    \path [-] (1) edge[bend left=-10] (2);
\end{scope}
\end{tikzpicture}}
\end{eqnarray}
\endgroup
We shall in fact prove a much more precise estimates, see~\eqref{eq:bnd-4} below, but the above convenient estimates will be enough for our purposes.
Powers of the logarithmic factor in each of these estimates is equal to the difference between the numbers of vertices in the left-hand side and in the right-hand side: indeed, in view of~\eqref{eq:bound-integ-kernel-diag}, each vertex that is integrated yields at most one logarithmic factor. This could in fact be improved in~\eqref{eq:block-integr-alph=beta}--\eqref{eq:block-integr-b=alpha=bet} based on~\eqref{eq:bnd-4}, but we shall not need such refinements.

\medskip\noindent
Before turning to the proof, we make a notational comment. A special role is of course played in the above estimates by the root $b$ and by the endpoint~$f$ of the block. In the sequel, even when vertices are not labeled explicitly, as e.g.~in~\eqref{eq:bnd-4}, we take the convention that the root and the endpoint are always drawn respectively on the left and on the right sides of the square (or at one of these two locations in case they coincide), while all other distinguished vertices are drawn indistinctly on the upper and lower sides.

\medskip\noindent
We turn to the proof of~\eqref{eq:integ-DB-plop}.
For that purpose, we shall study geometric properties of the graph $\Gg$ associated with the block~$B$. Letting $B=(b_1,b_2,\ldots,b_l)$ with $b_1=b$ and $b_l=f$, we recall that we define vertices of $\Gg$ as the elements of the index set $\langle B\rangle=\{b_i\}_{1\le i\le l}$, and edges of $\Gg$ as pairs of consecutive indices $(b_i,b_{i+1})$ with $1\le i< l$. Note that $\Gg$ is connected and may have multiple edges but no self-loop. 
We shall repeatedly use the following observation: as edges of~$\Gg$ are defined from the block $B$, we note that~$b$ and~$f$ have odd degree $\ge3$ and that other vertices have even degree $\ge2$ (where the degree of a vertex is the number of unoriented edges containing that vertex; see e.g.~Figure~\ref{fig:block-path}).
We split the proof into three further substeps.


\medskip
\substep{3.1} Cyclic estimate.\\
In the spirit of~\eqref{eq:def-DBxB}, for a graph $\Hg$ on the index set $[k]$, we define
\begin{equation}\label{eq:notation-DH}
D_\Hg(x_{[k]})\,:=\,\prod_{(i,j)\in\Hg}\langle(x_i-x_j)_L\rangle^{-d},
\end{equation}
where the notation $(i,j)\in\Hg$ means that $(i,j)$ is an edge of $\Hg$.
Provided that $\Hg$ is Eulerian (that is, provided that $\Hg$ is connected and that each vertex has even degree), we claim that for all vertices $\alpha,\beta,\gamma\in\langle\Hg\rangle$,
\begin{multline}\label{eq:cyclic-unif}
\int_{(Q_L)^{\sharp\Hg\setminus\{\alpha,\beta,\gamma\}}}D_\Hg(x_{[k]})\,dx_{\langle \Hg\rangle\setminus\{\alpha,\beta,\gamma\}}\\
\,\lesssim_{\Hg}\,
\Lc^{\sharp\Hg\setminus\{\alpha,\beta,\gamma\}}\bigg(~
{\tiny\begin{tikzpicture}[baseline={([yshift=-.8ex]current bounding box.center)},scale=0.6]
\begin{scope}[every node/.style={circle,thick,draw}]
    \node (1) at (0,0) {$\alpha$};
    \node (2) at (1.8,0) {$\gamma$};
    \node (3) at (0.9,1.4) {$\beta$};
\end{scope}
\begin{scope}[>={Stealth[black]},
every node/.style={fill=white,circle},
every edge/.style={draw=black,very thick}]
    \path [-] (1) edge (2);
    \path [-] (1) edge (3);
    \path [-] (3) edge (2);
\end{scope}
\end{tikzpicture}}
+
{\tiny\begin{tikzpicture}[baseline={([yshift=-.8ex]current bounding box.center)},scale=0.6]
\begin{scope}[every node/.style={circle,thick,draw}]
    \node (1) at (0,0) {$\alpha$};
    \node (2) at (1.8,0) {$\gamma$};
    \node (3) at (0.9,1.4) {$\beta$};
\end{scope}
\begin{scope}[>={Stealth[black]},
every node/.style={fill=white,circle},
every edge/.style={draw=black,very thick}]
    \path [-] (1) edge[bend left=10] (2);
    \path [-] (1) edge[bend left=-10] (2);
    \path [-] (1) edge[bend left=10] (3);
    \path [-] (1) edge[bend left=-10] (3);
\end{scope}
\end{tikzpicture}}
+
{\tiny\begin{tikzpicture}[baseline={([yshift=-.8ex]current bounding box.center)},scale=0.6]
\begin{scope}[every node/.style={circle,thick,draw}]
    \node (1) at (0,0) {$\alpha$};
    \node (2) at (1.8,0) {$\gamma$};
    \node (3) at (0.9,1.4) {$\beta$};
\end{scope}
\begin{scope}[>={Stealth[black]},
every node/.style={fill=white,circle},
every edge/.style={draw=black,very thick}]
    \path [-] (2) edge[bend left=10] (1);
    \path [-] (2) edge[bend left=-10] (1);
    \path [-] (2) edge[bend left=10] (3);
    \path [-] (2) edge[bend left=-10] (3);
\end{scope}
\end{tikzpicture}}
+
{\tiny\begin{tikzpicture}[baseline={([yshift=-.8ex]current bounding box.center)},scale=0.6]
\begin{scope}[every node/.style={circle,thick,draw}]
    \node (1) at (0,0) {$\alpha$};
    \node (2) at (1.8,0) {$\gamma$};
    \node (3) at (0.9,1.4) {$\beta$};
\end{scope}
\begin{scope}[>={Stealth[black]},
every node/.style={fill=white,circle},
every edge/.style={draw=black,very thick}]
    \path [-] (3) edge[bend left=10] (2);
    \path [-] (3) edge[bend left=-10] (2);
    \path [-] (3) edge[bend left=10] (1);
    \path [-] (3) edge[bend left=-10] (1);
\end{scope}
\end{tikzpicture}}
~\bigg).
\end{multline}
%
We will use standard terminology from graph theory, which we recall here for clarity: a walk is a sequence of edges joining a sequence of vertices; a trail is a walk in which all edges are distinct (taking edge multiplicity into account); a path is a trail in which all vertices are also distinct; a circuit is a trail in which the first and last vertices coincide; a cycle is a circuit in which only the first and last vertices coincide.

\medskip\noindent
We turn to the proof of~\eqref{eq:cyclic-unif} and we argue by induction on the size of $\Hg$.
Assume that~$\alpha,\beta,\gamma$ are distinct (the cases $\alpha=\beta\ne\gamma$ and $\alpha=\beta=\gamma$ can be treated similarly and are skipped for brevity).
The result is straightforward if $\sharp\Hg=3$ as no integral is performed in that case. We turn to the case~$\sharp\Hg>3$.
As $\Hg$ is Eulerian, there is a circuit that covers $\Hg$ (that is, a circuit that visits every edge of $\Hg$ exactly once).
Removing some subcircuits, we deduce that one of the following two possibilities must hold up to a permutation of~$\alpha,\beta,\gamma$:
\begin{enumerate}[(a)]
\item either there is a cycle $\Cg$ visiting $\alpha,\beta,\gamma$;
\smallskip\item or there is a cycle $\Cg_1$ visiting $\alpha,\beta$ and a cycle $\Cg_2$ visiting $\alpha,\gamma$ such that vertices of $\Cg_1$ and $\Cg_2$ are all distinct except $\alpha$.
\end{enumerate}
Both cases can be treated similarly and we focus on the first one for brevity. Let~$\Cg$ be a cycle visiting $\alpha,\beta,\gamma$.
Denote by $\Hg'$ the (possibly empty) subgraph of $\Hg$ induced by the complement of the edge set of the cycle $\Cg$.
As $\Hg$ is Eulerian and as $\Cg$ is a cycle, we notice that $\Hg'$ is the union of Eulerian subgraphs $\Hg'_1,\ldots,\Hg'_s$ that are edge-disjoint. We may then decompose
\[D_{\Hg}(x_{[k]})\,=\,D_{\Cg}(x_{[k]})D_{\Hg_1'}(x_{[k]})\ldots D_{\Hg_s'}(x_{[k]}).\]
For all $1\le i\le s$, there is a vertex $j_i$ of $\Hg_i'$ that also belongs to the cycle $\Cg$. Summing separately over repeated variables, we may then estimate
\begin{multline*}
\int_{(Q_L)^{\sharp\Hg\setminus\{\alpha,\beta,\gamma\}}}D_{\Hg}(x_{[k]})\,dx_{\langle\Hg\rangle\setminus\{\alpha,\beta,\gamma\}}\\
\,\lesssim\,\int_{(Q_L)^{\sharp\Cg\setminus\{\alpha,\beta,\gamma\}}}D_{\Cg}(x_{[k]})\bigg(\prod_{i=1}^s\int_{(Q_L)^{\sharp \Hg_i'-1}}D_{\Hg_i'}(x_{[k]})\,dx_{\langle\Hg_i'\rangle\setminus\{j_i\}}\bigg)dx_{\langle\Cg\rangle\setminus\{\alpha,\beta,\gamma\}}.
\end{multline*}
As the $\Hg_i'$'s are strict Eulerian subgraphs of $\Hg$, an induction argument allows to assume that the claim~\eqref{eq:cyclic-unif} is already known to hold for $\Hg$ replaced by any of the $\Hg_i'$'s. In particular, upon integration, this entails
\[\int_{(Q_L)^{\sharp \Hg_i'-1}}D_{\Hg_i'}(x_{[k]})\,dx_{\langle\Hg_i'\rangle\setminus\{j_i\}}\,\lesssim\,\Lc^{\sharp\Hg_i'-1}.\]
The above then reduces to
\begin{equation*}
\int_{(Q_L)^{\sharp\Hg\setminus\{\alpha,\beta,\gamma\}}}D_{\Hg}(x_{[k]})\,dx_{\langle\Hg\rangle\setminus\{\alpha,\beta,\gamma\}}
\,\lesssim\,{\Lc^{\sum_i (\sharp\Hg_i'-1)}}\int_{(Q_L)^{\sharp\Cg\setminus\{\alpha,\beta,\gamma\}}}D_{\Cg}(x_{[k]})\,dx_{\langle\Cg\rangle\setminus\{\alpha,\beta,\gamma\}},
\end{equation*}
where the right-hand side is now simply an integral of the form
\begin{equation*}
{\tiny\begin{tikzpicture}[baseline={([yshift=-.8ex]current bounding box.center)},scale=0.6]
\begin{scope}[every node/.style={circle,thick,draw}]
    \node (1) at (0,0) {$\alpha$};
    \node (2) at (2,3) {$\beta$};
    \node (3) at (4,0) {$\gamma$};
\end{scope}
\begin{scope}[every node/.style={circle,fill,inner sep=0pt,minimum size=3pt}]
    \node (1a) at (0.7,0) {};
    \node (1b) at (1.2,0) {};
    \node (1e) at (2.8,0) {};
    \node (1f) at (3.3,0) {};
    \node (2a) at (0.35,0.525) {};
    \node (2b) at (0.6,0.9) {};
    \node (2e) at (1.4,2.1) {};
    \node (2f) at (1.625,2.475) {};
    \node (3a) at (3.65,0.525) {};
    \node (3b) at (3.4,0.9) {};
    \node (3e) at (2.6,2.1) {};
    \node (3f) at (2.375,2.475) {};
\end{scope}
\begin{scope}[every node/.style={circle,inner sep=0pt,minimum size=0pt}]
    \node (1c) at (1.5,0) {};
    \node (1d) at (2.5,0) {};
    \node (2c) at (0.75,1.125) {};
    \node (2d) at (1.25,1.875) {};
    \node (3c) at (3.25,1.125) {};
    \node (3d) at (2.75,1.875) {};
\end{scope}
\begin{scope}[>={Stealth[black]},
every node/.style={fill=white,circle},
every edge/.style={draw=black,very thick}]
    \path [-] (1) edge (1c);
    \path [-] (3) edge (1d);
    \path [-] (1) edge (2c);
    \path [-] (2) edge (2d);
    \path [-] (2) edge (3d);
    \path [-] (3) edge (3c);
\end{scope}
\begin{scope}[every node/.style={circle}]
    \node (13) at (2,0) {$\cdot$};
    \node (13) at (2.3,0) {$\cdot$};
    \node (13) at (1.7,0) {$\cdot$};
    \node (12) at (1,1.5) {$\cdot$};
    \node (12) at (1.15,1.725) {$\cdot$};
    \node (12) at (0.85,1.275) {$\cdot$};
    \node (23) at (3,1.5) {$\cdot$};
    \node (23) at (2.85,1.725) {$\cdot$};
    \node (23) at (3.15,1.275) {$\cdot$};
\end{scope}
\end{tikzpicture}}
\end{equation*}
Using~\eqref{eq:bound-integ-kernel-diag} to evaluate the integrals, noting that the number of appearing logarithmic factors is bounded by the length of $\Cg$ minus $3$ and that the length of $\Cg$ is bounded by the total number of vertices
minus the number of vertices not in the cycle (that is, $\sum_i(\sharp\Hg_i'-1)$),
the claim~\eqref{eq:cyclic-unif} follows. More precisely, we obtain in this way the first right-hand side term in~\eqref{eq:cyclic-unif}, while other terms correspond to case~(b) above.

\medskip
\substep{3.2} Path decomposition of the graph $\Gg$ associated with a block $B$.\\
We show that, if $b\ne f$, there exist three edge-disjoint trails $\Lg^1,\Lg^2,\Lg^3$ that cover $\Gg$ (that is, the union of their vertex sets is the vertex set of $\Gg$ and the disjoint union of their edge sets is the edge set of $\Gg$).
We refer to Figure~\ref{fig:block-path} for an illustrative example.

\medskip\noindent
As $b$ and $f$ have odd degree $\ge3$ and as all other vertices of $\Gg$ have even degree, we can find a trail $\Lg^1$ from $b$ to $f$. Without loss of generality, we can assume that $b$ and $f$ are visited only once by $\Lg^1$. Then consider the subgraph $\Gg'$ of $\Gg$ induced by the complement of the edge set of $\Lg^1$. By construction, all vertices of $\Gg'$ now have even degree, and the definition of the block $B$ ensures that $\Gg'$ must be connected. This allows to find two other disjoint trails~$\Lg^2,\Lg^3$ from $b$ to $f$ in $\Gg'$.

\medskip\noindent
Next, assume that a vertex $\alpha\in\langle\Gg\rangle$ is not visited by any of the three constructed trails $\Lg^1,\Lg^2,\Lg^3$. Recalling that $\Gg$ is connected, a degree argument as above ensures that there exists a circuit $\Kg$ from $\alpha$ to itself that is disjoint from the trails $\Lg^1,\Lg^2,\Lg^3$ and that crosses at least one of them.
A detour via $\Kg$ is then easily added to those trails in such a way that they remain disjoint and that at least one of them now visits $\alpha$. Repeating this construction, we are led to edge-disjoint trails $\Lg^1,\Lg^2,\Lg^3$ that visit all vertices of~$\Gg$.

\medskip\noindent
Finally, consider the subgraph $\Gg''$ of $\Gg$ induced by the complement of the union of the edge sets of $\Lg^1,\Lg^2,\Lg^3$. By construction, all vertices of $\Gg''$ have even degree, which allows to write $\Gg''$ as a union of edge-disjoint circuits. Adding detours via these circuits, we can assume that the trails $\Lg^1,\Lg^2,\Lg^3$ cover the whole graph $\Gg$, and the claim follows.

\begin{figure}
{\begin{center}
{\small\begin{tikzpicture}[scale=0.7]
\begin{scope}[every node/.style={circle,thick,draw}]
    \node (1) at (0,0) {$1$};
    \node (2) at (2,2) {$2$};
    \node (3) at (4,0) {$3$};
    \node (4) at (2,-2) {$4$};
    \node (5) at (3.6,1.6) {$5$};
    \node (6) at (8,0) {$6$};
    \node (7) at (5.3,-1.6) {$7$};
\end{scope}
\begin{scope}[>={Stealth[black]},
every node/.style={fill=white,circle},
every edge/.style={draw=black,very thick}]
    \path [-] (1) edge[bend left=10] (2);
    \path [-] (2) edge[bend left=-10] (3);
    \path [-] (3) edge (1);
    \path [-] (1) edge[bend left=-10] (4);
    \path [-] (4) edge[bend left=10] (3);
    \path [-] (3) edge[bend left=-10] (5);
    \path [-] (5) edge (2);
    \path [-] (2) edge[bend left=30] (6);
    \path [-] (6) edge[bend left=-10] (4);
    \path [-] (4) edge[bend left=-10] (7);
    \path [-] (7) edge[bend left=-10] (6);
\end{scope}
\end{tikzpicture}}
\caption{\label{fig:block-path}
This graph represents the block $B=(1)\uplus(2,3,1)\uplus(4,3)\uplus(5,2)\uplus(6,4)\uplus(7,6)$.
The path decomposition of Substep~2.3 can be chosen in this case as $\Lg^1=(1,2,6)$, $\Lg^2=(1,4,7,6)$, and $\Lg^3=(1,3,2,5,3,4,6)$.}
\end{center}}
\end{figure}

\medskip
\substep{3.3} Proof of~\eqref{eq:integ-DB-plop}.\\
Let $\Lg^1,\Lg^2,\Lg^3$ be three covering edge-disjoint trails from $b$ to $f$ as constructed above.
Given a vertex $\alpha$, distinguishing between the number of paths from $b$ to $f$ to which $\alpha$ belongs, and removing cycles, we get the following four possibilities:
\begin{enumerate}[(a)]
\item either there exists a cycle $\Cg$ from $b$ or from $f$ that visits $\alpha$ and there exist three paths $\Kg^1,\Kg^2,\Kg^3$ from $b$ to $f$, such that $\Cg,\Kg^1,\Kg^2,\Kg^3$ are edge-disjoint and cross each other only at $b$ or $f$;
\item or there exists a path $\Kg^1$ from $b$ to $f$ that visits $\alpha$ and there exist two other paths $\Kg^2,\Kg^3$ from $b$ to $f$ that do not, such that $\Kg^1,\Kg^2,\Kg^3$ are edge-disjoint and cross each other only at $b$ or $f$;
\item or there exist two paths $\Kg^1,\Kg^2$ from $b$ to $f$ that visit $\alpha$ and there exists another path $\Kg^3$ from $b$ to $f$ that does not, such that $\Kg^1,\Kg^2,\Kg^3$ are edge-disjoint and cross each other only at $\alpha$, $b$ or $f$;
\item or there exist three paths $\Kg^1,\Kg^2,\Kg^3$ from $b$ to $f$ that visit $\alpha$ and that are edge-disjoint and cross each other only at $\alpha$, $b$ and $f$.
\end{enumerate}
Given another vertex $\beta$, and distinguishing between corresponding cases, we obtain three distinguished paths $\Kg^1,\Kg^2,\Kg^3$ from $b$ to $f$ that may visit or not $\alpha$ and $\beta$, in different possible orders, and we obtain up to two cycles $\Cg^1,\Cg^2$ from $b$ or $f$ visiting $\alpha$ or $\beta$.
The subgraph of $\Gg$ induced by the complement of the union of the edge sets of those three paths and possible cycles is necessarily a disjoint union of Eulerian graphs and can be removed by duplicating variables as in Substep~3.1. It remains to consider the union of those three paths and possible cycles. Considering different patterns and using~\eqref{eq:bound-integ-kernel-diag} to estimate consecutive edges along each path between frozen vertices $b,f,\alpha,\beta$, we are led to
\begin{equation}\label{eq:bnd-4}
{\tiny\begin{tikzpicture}[baseline={([yshift=-.8ex]current bounding box.center)},scale=0.6]
\filldraw[gray] (0,0) rectangle (1,1);
\begin{scope}[every node/.style={circle,draw,fill=white,inner sep=0pt,minimum size=3pt}]
    \node (1) at (0,0.5) {};
    \node (2) at (1,0.5) {};
    \node (3) at (0.5,1) {};
    \node (4) at (0.5,0) {};
\end{scope}
\end{tikzpicture}}
~\lesssim~
\Lc^{\sharp B-4}\bigg(
{\tiny\begin{tikzpicture}[baseline={([yshift=-.8ex]current bounding box.center)},scale=0.8]
\begin{scope}[every node/.style={circle,draw,fill=white,inner sep=0pt,minimum size=3pt}]
    \node (1) at (0,0.5) {};
    \node (2) at (1,0.5) {};
    \node (3) at (0.5,1) {};
    \node (4) at (0.5,0) {};
\end{scope}
\begin{scope}[>={Stealth[black]},
every node/.style={fill=white,circle},
every edge/.style={draw=black,very thick}]
    \path [-] (1) edge (2);
    \path [-] (1) edge (3);
    \path [-] (1) edge (4);
    \path [-] (3) edge (2);
    \path [-] (4) edge (2);
\end{scope}
\end{tikzpicture}}
+
{\tiny\begin{tikzpicture}[baseline={([yshift=-.8ex]current bounding box.center)},scale=0.8]
\begin{scope}[every node/.style={circle,draw,fill=white,inner sep=0pt,minimum size=3pt}]
    \node (1) at (0,0.5) {};
    \node (2) at (1,0.5) {};
    \node (3) at (0.5,1) {};
    \node (4) at (0.5,0) {};
\end{scope}
\begin{scope}[>={Stealth[black]},
every node/.style={fill=white,circle},
every edge/.style={draw=black,very thick}]
    \path [-] (1) edge (2);
    \path [-] (1) edge (3);
    \path [-] (2) edge (4);
    \path [-] (3) edge (4);
\end{scope}
\end{tikzpicture}}
+
{\tiny\begin{tikzpicture}[baseline={([yshift=-.8ex]current bounding box.center)},scale=0.8]
\begin{scope}[every node/.style={circle,draw,fill=white,inner sep=0pt,minimum size=3pt}]
    \node (1) at (0,0.5) {};
    \node (2) at (1,0.5) {};
    \node (3) at (0.5,1) {};
    \node (4) at (0.5,0) {};
\end{scope}
\begin{scope}[>={Stealth[black]},
every node/.style={fill=white,circle},
every edge/.style={draw=black,very thick}]
    \path [-] (1) edge (3);
    \path [-] (3) edge (2);
    \path [-] (1) edge[bend left=20] (4);
    \path [-] (1) edge[bend left=-20] (4);
    \path [-] (2) edge[bend left=20] (4);
    \path [-] (2) edge[bend left=-20] (4);
\end{scope}
\end{tikzpicture}}
+
{\tiny\begin{tikzpicture}[baseline={([yshift=-.8ex]current bounding box.center)},scale=0.8]
\begin{scope}[every node/.style={circle,draw,fill=white,inner sep=0pt,minimum size=3pt}]
    \node (1) at (0,0.5) {};
    \node (2) at (1,0.5) {};
    \node (3) at (0.5,1) {};
    \node (4) at (0.5,0) {};
\end{scope}
\begin{scope}[>={Stealth[black]},
every node/.style={fill=white,circle},
every edge/.style={draw=black,very thick}]
    \path [-] (1) edge[bend left=20] (3);
    \path [-] (1) edge[bend left=-20] (3);
    \path [-] (1) edge[bend left=15] (2);
    \path [-] (1) edge[bend left=-15] (2);
    \path [-] (2) edge (4);
    \path [-] (1) edge (4);
\end{scope}
\end{tikzpicture}}
+
{\tiny\begin{tikzpicture}[baseline={([yshift=-.8ex]current bounding box.center)},scale=0.8]
\begin{scope}[every node/.style={circle,draw,fill=white,inner sep=0pt,minimum size=3pt}]
    \node (1) at (0,0.5) {};
    \node (2) at (1,0.5) {};
    \node (3) at (0.5,1) {};
    \node (4) at (0.5,0) {};
\end{scope}
\begin{scope}[>={Stealth[black]},
every node/.style={fill=white,circle},
every edge/.style={draw=black,very thick}]
    \path [-] (1) edge[bend left=20] (3);
    \path [-] (1) edge[bend left=-20] (3);
    \path [-] (1) edge (2);
    \path [-] (1) edge[bend left=20] (4);
    \path [-] (1) edge[bend left=-20] (4);
    \path [-] (2) edge[bend left=20] (4);
    \path [-] (2) edge[bend left=-20] (4);
\end{scope}
\end{tikzpicture}}
+
{\tiny\begin{tikzpicture}[baseline={([yshift=-.8ex]current bounding box.center)},scale=0.8]
\begin{scope}[every node/.style={circle,draw,fill=white,inner sep=0pt,minimum size=3pt}]
    \node (1) at (0,0.5) {};
    \node (2) at (1,0.5) {};
    \node (3) at (0.5,1) {};
    \node (4) at (0.5,0) {};
\end{scope}
\begin{scope}[>={Stealth[black]},
every node/.style={fill=white,circle},
every edge/.style={draw=black,very thick}]
    \path [-] (1) edge[bend left=20] (3);
    \path [-] (1) edge[bend left=-20] (3);
    \path [-] (1) edge[bend left=20] (2);
    \path [-] (1) edge (2);
    \path [-] (1) edge[bend left=-20] (2);
    \path [-] (2) edge[bend left=20] (4);
    \path [-] (2) edge[bend left=-20] (4);
\end{scope}
\end{tikzpicture}}
+\text{sym.}
\bigg),
\end{equation}
where for brevity ``sym.'' stands for the sum of all other graphs obtained by reflecting the six pictured graphs with respect to the vertical axis, the horizontal axis, or both (which corresponds to permuting $\alpha$ and $\beta$, $b$ and $f$, or both).
In fact, the analysis of all possible patterns produces a larger number of terms, but we claim that all others are bounded by the above.
For instance, another possible pattern corresponds to the case of three paths $\Kg^1,\Kg^2,\Kg^3$ from $b$ to $f$ visiting both $\alpha$ and $\beta$, where $\Kg^1,\Kg^2$ visit $\alpha$ before $\beta$ while $\Kg^3$ visit them in reverse order: we claim that the corresponding contribution can be bounded as follows,
\begin{equation}\label{eq:bound-4-special}
{\tiny\begin{tikzpicture}[baseline={([yshift=-.8ex]current bounding box.center)},scale=0.8]
\begin{scope}[every node/.style={circle,draw,fill=white,inner sep=0pt,minimum size=3pt}]
    \node (1) at (0,0.5) {};
    \node (2) at (1,0.5) {};
    \node (3) at (0.5,1) {};
    \node (4) at (0.5,0) {};
\end{scope}
\begin{scope}[>={Stealth[black]},
every node/.style={fill=white,circle},
every edge/.style={draw=black,very thick}]
    \path [-] (1) edge[bend left=20] (3);
    \path [-] (1) edge[bend left=-20] (3);
    \path [-] (1) edge (4);
    \path [-] (3) edge (4);
    \path [-] (3) edge (2);
    \path [-] (2) edge[bend left=20] (4);
    \path [-] (2) edge[bend left=-20] (4);
\end{scope}
\end{tikzpicture}}
~\lesssim~
{\tiny\begin{tikzpicture}[baseline={([yshift=-.8ex]current bounding box.center)},scale=0.8]
\begin{scope}[every node/.style={circle,draw,fill=white,inner sep=0pt,minimum size=3pt}]
    \node (1) at (0,0.5) {};
    \node (2) at (1,0.5) {};
    \node (3) at (0.5,1) {};
    \node (4) at (0.5,0) {};
\end{scope}
\begin{scope}[>={Stealth[black]},
every node/.style={fill=white,circle},
every edge/.style={draw=black,very thick}]
    \path [-] (1) edge (3);
    \path [-] (1) edge (2);
    \path [-] (3) edge (4);
    \path [-] (2) edge (4);
\end{scope}
\end{tikzpicture}}
\end{equation}
which is indeed bounded by the right-hand side of~\eqref{eq:bnd-4}.
This bound follows from
\begin{equation*}
{\tiny\begin{tikzpicture}[baseline={([yshift=-.8ex]current bounding box.center)},scale=0.8]
\begin{scope}[every node/.style={circle,draw,fill=white,inner sep=0pt,minimum size=3pt}]
    \node (1) at (0,0.5) {};
    \node (2) at (1,0.5) {};
    \node (3) at (0.5,1) {};
\end{scope}
\begin{scope}[>={Stealth[black]},
every node/.style={fill=white,circle},
every edge/.style={draw=black,very thick}]
    \path [-] (1) edge (3);
    \path [-] (2) edge (3);
\end{scope}
\end{tikzpicture}}
~\lesssim~
{\tiny\begin{tikzpicture}[baseline={([yshift=-.8ex]current bounding box.center)},scale=0.8]
\begin{scope}[every node/.style={circle,draw,fill=white,inner sep=0pt,minimum size=3pt}]
   \node (1) at (0,0.5) {};
    \node (2) at (1,0.5) {};
    \node (3) at (0.5,1) {};
\end{scope}
\begin{scope}[>={Stealth[black]},
every node/.style={fill=white,circle},
every edge/.style={draw=black,very thick}]
    \path [-] (1) edge (2);
 \end{scope}
\end{tikzpicture}}
\end{equation*}
which is itself nothing but the triangle inequality
$\langle (x_1-x_3)_L \rangle \le \langle (x_1-x_2)_L \rangle+\langle (x_2-x_3)_L
 \rangle$ post-processed into $\langle (x_1-x_3)_L \rangle \le 2\langle (x_1-x_2)_L \rangle\langle (x_2-x_3)_L
 \rangle$ and put to the power $-d$.
A straightforward similar inspection of all other possible patterns shows that the bound~\eqref{eq:bnd-4} indeed holds; we skip the detail for brevity.

\medskip\noindent
Finally, removing a few edges in~\eqref{eq:bnd-4}, we are led in particular to the claim~\eqref{eq:integ-DB-plop-diagr-bis}. The claims~\eqref{eq:block-integr-b=f}--\eqref{eq:block-integr-b=alpha=bet} follow as straightforward corollaries after integrations using~\eqref{eq:bound-integ-kernel-diag}.

\medskip
\step4 Approximate cancellation of translation-invariant averages on given blocks.\\
Let $B$ be a block of indices with root $b$ and endpoint $f$, and let $S,T$ be disjoint index sets with \mbox{$(S\cup T)\cap\langle B\rangle=\varnothing$}. Let $m:=\sharp B$ and $s:=\sharp S$.
For all $x_B,x_S$, let $\zeta^{x_S}_{L;x_B}\in H^1_\per(Q_L)^d$ satisfy~\eqref{eq:zetaP-re} at~$z=x_{f}$,
and assume that $\zeta_L$ is equivariant under translations in the sense that
\[\zeta^{x_S+[z]_S}_{L;x_B+[z]_B}(\cdot+z)=\zeta^{x_S}_{L;x_B},\qquad\text{for all $z\in\R^d$,}\]
where $[z]_B$ (resp.~$[z]_S$) stands for the element of~$(\R^d)^m$ (resp.~$(\R^d)^s$) with all coordinates equal to $z$.
Then, for any function $h$ on $(Q_L)^{m+s}$ that is translation-invariant in the sense that $h(x_B+[z]_B,x_S+[z]_S)=h(x_B,x_S)$ for all~$z\in\R^d$, we have for any linear functional~$F:H^1_\per(Q_L)^d\to\R$,
\begin{multline}\label{eq:cancel-CB}
\bigg|\int_{(Q_L)^{m+s}}F\big[\calC_{L;T}^B(x_{[k]})\zeta_{L;x_B}^{x_S}\big]\,h(x_B,x_S)\,dx_Bdx_S\bigg|\\
\,\le\,\int_{Q_L}\bigg(\int_{(Q_L+x_b)^{m+s-1}\setminus(Q_L)^{m+s-1}}\Big|F\Big[\calC_{L;T}^B(x_{[k]})\zeta^{x_S}_{L;x_B}\Big]\Big|\\
\times\big(|h_L(x_B,x_S)|+|h(x_B,x_S)|\big)\,dx_{\langle B\rangle\setminus \{b\}}dx_S\bigg)dx_{b},
\end{multline}
where we have defined the periodization $h_L(z):=h(z_L)$ where $z_L\in (Q_L)^{m+s}$ stands for the reduction of $z\in(\R^d)^{m+s}$ modulo $(L\Z^d)^{m+s}$.
Note that we do not obtain an exact cancellation in general for such a symmetric average on a block, but this bound reduces it to a boundary term.

\medskip\noindent
We turn to the proof of~\eqref{eq:cancel-CB}.
Set for abbreviation $C:=\langle B\rangle\setminus \{b\}$. By definition of elementary block contributions, cf.~\eqref{eq:block-contr}, we can write
\begin{equation}\label{eq:def-xi-zeta}
\calC_{L;T}^B(x_{[k]})\zeta_{L;x_B}^{x_S}\,=\,\Jc_{L;x_T}^{x_{b}}\xi_{L;x_{b}}^{x_C;x_S},
\end{equation}
for some function $\xi_{L;x_b}^{x_C;x_S}$ that satisfies~\eqref{eq:zetaP-re} at \mbox{$z=x_{b}$} and is such that $\xi_L$ is equivariant under translations. The left-hand side of~\eqref{eq:cancel-CB} then becomes
\begin{multline*}
\int_{(Q_L)^{m+s}}F\big[\calC_{L;T}^B(x_{[k]})\zeta_{L;x_B}^{x_S}\big]\,h(x_B,x_S)\,dx_Bdx_S\\
\,=\,\int_{(Q_L)^{m+s}}F\big[\Jc_{L;x_T}^{x_{b}}\xi_{L;x_{b}}^{x_C;x_S}\big]\,h(x_{b},x_C,x_S)\,dx_{b}dx_Cdx_S,
\end{multline*}
and thus, using the equivariance of $\xi$ under translations,
\begin{multline*}
\int_{(Q_L)^{m+s}}F\big[\calC_{L;T}^B(x_{[k]})\zeta_{L;x_B}^{x_S}\big]\,h(x_B,x_S)\,dx_Bdx_S\\
\,=\,\int_{(Q_L)^{m+s}}F\Big[\Jc_{L;x_T}^{x_{b}}\big(\xi_{L;0}^{x_C-[x_{b}]_C;x_S-[x_{b}]_S}(\cdot-x_{b})\big)\Big]\,h(x_{b},x_C,x_S)\,dx_{b}dx_Cdx_S.
\end{multline*}
Replacing $h$ by its periodization $h_L$ (which we can on $(Q_L)^{m+s}$), changing variables and using periodicity, the above becomes in these terms,
\begin{multline*}
\int_{(Q_L)^{m+s}}F\big[\calC_{L;T}^B(x_{[k]})\zeta_{L;x_B}^{x_S}\big]\,h(x_B,x_S)\,dx_Bdx_S\\
\,=\,\int_{(Q_L)^{m+s}}F\Big[\Jc_{L;x_T}^{x_{b}}\big(\xi_{L;0}^{x_C;x_S}(\cdot-x_{b})\big)\Big]\,h_L(x_{b},x_C+[x_{b}]_C,x_S+[x_{b}]_S)\,dx_{b}dx_Cdx_S.
\end{multline*}
If $h_L(x_{b},x_C+[x_{b}]_C,x_S+[x_{b}]_S)$ were replaced by $h(0,x_C,x_S)$ in the integrand, the cancellation property of Lemma~\ref{lem:new-cancel} would precisely entail that the integral vanishes (this would have been the case if we had considered a periodization in law of $\Ic$ rather than~\eqref{e.set-perio}). Adding and subtracting $h(0,x_C,x_S)$, we deduce
\begin{eqnarray*}
\lefteqn{\int_{(Q_L)^{m+s}}F\big[\calC_{L;T}^B(x_{[k]})\zeta_{L;x_B}^{x_S}\big]\,h(x_B,x_S)\,dx_Bdx_S}\\
&=&\int_{(Q_L)^{m+s}}F\Big[\Jc_{L;x_T}^{x_{b}}\big(\xi_{L;0}^{x_C;x_S}(\cdot-x_{b})\big)\Big]\\
&&\hspace{2cm}\times\Big(h_L\big(x_{b},x_C+[x_{b}]_C,x_S+[x_{b}]_S\big)-h(0,x_C,x_S)\Big)dx_{b}dx_Cdx_S.
\end{eqnarray*}
If $x_{b},x_C,x_S$ are such that $(x_{b},x_C+[x_{b}]_C,x_S+[x_{b}]_S)\in (Q_L)^{m+s}$, then the definition of the periodization~$h_L$ and the translation invariance of $h$ imply that the integrand vanishes.
This leads us to the bound
\begin{eqnarray*}
\lefteqn{\bigg|\int_{(Q_L)^{m+s}}F\big[\calC_{L;T}^B(x_{[k]})\zeta_{L;x_B}^{x_S}\big]\,h(x_B,x_S)\,dx_Bdx_S\bigg|}\\
&\le&\int_{Q_L}\bigg(\int_{(Q_L)^{m+s-1}\setminus(Q_L-x_{b})^{m+s-1}}\Big|F\Big[\Jc_{L;x_T}^{x_{b}}\big(\xi_{L;0}^{x_C;x_S}(\cdot-x_{b})\big)\Big]\Big|\\
&&\hspace{1.2cm}\times\Big(\big|h_L\big(x_{b},x_C+[x_{b}]_C,x_S+[x_{b}]_S\big)\big|+|h(0,x_C,x_S)|\Big)\,dx_{\langle B\rangle\setminus \{b\}}dx_S\bigg)dx_{b}.
\end{eqnarray*}
Using again~\eqref{eq:def-xi-zeta} and the equivariance of $\xi$, the claim~\eqref{eq:cancel-CB} follows.

\medskip
\step5 Uniform estimates: proof of~(i).\\
The starting point is the decomposition~\eqref{eq:decomp-BLj1-intens} of $\Bb_L^{k+1}$. For brevity, we shall focus on the term corresponding to $l=k$ in~\eqref{eq:decomp-BLj1-intens}, that is,
\begin{equation*}
\Cc^{k+1}_{L}\,:=\,\tfrac{k+1}2L^{-d}\int_{(\Qd)^{k+1}}\bigg(\int_{\partial B(x_0)}\del^{x_1,\ldots,x_k}\psi_{L}^\varnothing\cdot\sigma_{L}^{x_0}\nu\bigg) f_{k+1}(x_0,\ldots,x_k)\,dx_0\ldots dx_k,
\end{equation*}
while the other terms are simpler to estimate due to the additional decay given by the factor~$\delta^{x_{l+1},\ldots,x_k}\sigma_{L}^{x_0}$.
Inserting the diagrammatic decomposition~\eqref{eq:diagram}, we get
\[\Cc^{k+1}_{L}\,=\,\tfrac{k+1}2\sum_{r=1}^k~\sum_{B_1,\ldots, B_r}\Cc^{k+1;r}_L(B_1,\ldots,B_r),\]
where we recall that the sum runs over all $r$-tuples of disjoint blocks $B_1,\ldots,B_r$ such that $\langle B_1\uplus \ldots\uplus B_r\rangle=[k]$, and where we have set for abbreviation
\begin{multline*}
\Cc^{k+1;r}_L(B_1,\ldots,B_r)\\
\,:=\,L^{-d}\int_{(\Qd)^{k+1}}\bigg(\int_{\partial B(x_0)}\Big(\calC^{B_1}_{L;\varnothing}(x_{[k]})\calC^{B_2}_{L;\langle B_1\rangle}(x_{[k]})\ldots \calC^{B_r}_{L;\langle B_{r-1}\rangle}(x_{[k]}){\bar \psi_L^{x_{B_r}}}\Big)\cdot\sigma_{L}^{x_0}\nu\bigg)\\
\times f_{k+1}(x_0,x_{[k]})\,dx_0dx_{[k]}.
\end{multline*}
Let such $B_1,\ldots,B_r$ be fixed.
Replacing $f_{k+1}$ by its expansion~\eqref{eq:dens-correl} in terms of correlation functions, we find
\begin{multline}\label{eq:pre-decomp-pi-CLkr}
\Cc^{k+1;r}_L(B_1,\ldots,B_r)
\,=\,\sum_\pi L^{-d}\int_{(\Qd)^{k+1}}\bigg(\prod_{H\in\pi}h_{\sharp H}(x_H)\bigg)\\
\times\bigg(\int_{\partial B(x_0)}\Big(\calC^{B_1}_{L;\varnothing}(x_{[k]})\calC^{B_2}_{L;\langle B_1\rangle}(x_{[k]})
\ldots \calC^{B_r}_{L;\langle B_{r-1}\rangle}(x_{[k]}){\bar \psi_L^{x_{B_r}}}\Big)\cdot\sigma_{L}^{x_0}\nu\bigg)\,dx_0dx_{[k]},
\end{multline}
where $\pi$ runs over all partitions of the index set $\{0\}\cup[k]$ and where $H$ runs over all cells of the partition $\pi$.

\medskip\noindent
We shall say that a partition $\pi$ of $\{0\}\cup[k]$ is {\it covering} for~$B_1,\ldots,B_r$ if there is no `separating' index $1\le \alpha\le r$ such that each cell $H\in\pi$ is included either in $\{0\}\cup\bigcup_{i=1}^{\alpha-1}\langle B_i\rangle$ or in $\bigcup_{i=\alpha}^r\langle B_i\rangle$. We denote by $\Kc(B_1,\ldots,B_r)$ the set of such partitions.
Using the approximate cancellation property~\eqref{eq:cancel-CB}, and further noting as in~\eqref{eq:cancel-intx0} that
\begin{equation}\label{eq:cancel-CB-another}
\int_{Q_L}\bigg(\int_{\partial B(x_0)}\Big(\calC^{B_1}_{L;\varnothing}(x_{[k]})\calC^{B_2}_{L;\langle B_1\rangle}(x_{[k]})\ldots \calC^{B_r}_{L;\langle B_{r-1}\rangle}(x_{[k]})\psi_L^{x_{B_r}}\Big)\cdot\sigma_{L}^{x_0}\nu\bigg)dx_0\,=\,0,
\end{equation}
we note that only covering partitions produce nontrivial terms in~\eqref{eq:pre-decomp-pi-CLkr}: contributions from non-covering partitions either vanish or are reduced to boundary terms. Therefore, we naturally decompose
\begin{multline*}
|\Cc^{k+1;r}_L(B_1,\ldots,B_r)|\\
\,\le\,\sum_{\pi\in\Kc(B_1,\ldots,B_r)}|\Cc^{k+1;r}_L(B_1,\ldots,B_r;\pi)|+\sum_{\pi\notin\Kc(B_1,\ldots,B_r)}|\Cc^{k+1;r}_L(B_1,\ldots,B_r;\pi)|,
\end{multline*}
where we have set for abbreviation
\begin{multline*}
\Cc^{k+1;r}_L(B_1,\ldots,B_r;\pi)\,:=\,L^{-d}\int_{(\Qd)^{k+1}}\bigg(\prod_{H\in\pi} h_{\sharp H}(x_H)\bigg)\\
\times\bigg(\int_{\partial B(x_0)}\Big(\calC^{B_1}_{L;\varnothing}(x_{[k]})\calC^{B_2}_{L;\langle B_1\rangle}(x_{[k]})
\ldots \calC^{B_r}_{L;\langle B_{r-1}\rangle}(x_{[k]}){\bar \psi_L^{x_{B_r}}}\Big)\cdot\sigma_{L}^{x_0}\nu\bigg)
\,dx_0dx_{[k]}.
\end{multline*}
We split the proof into two further substeps, separately considering the two types of contributions.

\medskip
\substep{5.1} Main contributions: in case of an algebraic rate $\omega(t)\le Ct^{-\beta}$ for some $C,\beta>0$, we have for all $\pi\in\Kc(B_1,\ldots,B_r)$,
\begin{equation}\label{eq:bnd-main-Dk+1npi}
|\Cc^{k+1;r}_L(B_1,\ldots,B_r;\pi)|\,\lesssim\,\lambda_{k+1}(\Pc)|\!\log\lambda(\Pc)|^{k}.
\end{equation}
Without loss of generality, we may assume $\beta\in(0,d)$ (so we can appeal to~\eqref{eq:bound-integ-kernel-diag-dash}).
Using the boundary conditions for~$\psi^{x_0}_L$ and the incompressibility constraints to smuggle in arbitrary constants in the different factors, and then appealing to the trace estimates of Lemma~\ref{lem:trace-0}, we find
\begin{multline*}
|\Cc^{k+1;r}_L(B_1,\ldots,B_r;\pi)|\,\lesssim\,L^{-d}\int_{(\Qd)^{k+1}}\bigg(\prod_{H\in\pi} |h_{\sharp H}(x_H)|\bigg)\\
\times\bigg(\int_{B(x_0)}\Big|\nabla\calC^{B_1}_{L;\varnothing}(x_{[k]})\calC^{B_2}_{L;\langle B_1\rangle}(x_{[k]})
\ldots \calC^{B_r}_{L;\langle B_{r-1}\rangle}(x_{[k]}){\bar \psi_L^{x_{B_r}}}\Big|^2\bigg)^\frac12dx_0dx_{[k]}.
\end{multline*}
For all $1\le l\le r$, denote by $b_l$ the root of $B_l$ and by $f_l$ its endpoint, and set for notational convenience $f_0:=0$. Iterating the bound~\eqref{eq:decomp-CLB-integral}, we then get
\begin{multline}\label{eq:pre-bnd-Dpi}
|\Cc^{k+1;r}_L(B_1,\ldots,B_r;\pi)|\,\lesssim\,L^{-d}\int_{(\Qd)^{k+1}}\bigg(\prod_{H\in\pi} |h_{\sharp H}(x_H)|\bigg)\\
\times\bigg(\prod_{l=1}^{r}\langle(x_{f_{l-1}}-x_{b_{l}})_L\rangle^{-d}D_{B_l}(x_{B_l})\bigg)\,dx_0dx_{[k]}.
\end{multline}
Next, we examine the structure of the product of correlation functions.
Given a covering partition $\pi\in\Kc(B_1,\ldots,B_r)$, we can construct a sequence of intertwined pairings $(m_i,m_i')_{1\le i\le s}$ {(for some integer $s\ge 1$)} such that
\begin{enumerate}[---]
\item {$(m_j)_{1\le j \le s}$ and  $(m'_j)_{1\le j \le s}$ are increasing, $m_1=0$, and $m_s'=r$;}
\smallskip\item {$m_{i-1}'<m_{i+1}$ for all $1< i<s$, and $m_i\le m_{i-1}'$ for all $1<i\le s$;}
\smallskip\item for all $i$ there is a cell $H\in\pi$ such that $H\cap\langle B_{m_i}\rangle\ne\varnothing$ and $H\cap\langle B_{m_i'}\rangle\ne\varnothing$ {(with the understanding that $B_0=\{0\}$)}.  
\end{enumerate}
The construction is as follows:
Starting from $m_1=0$, we define $m_1'$ as the maximum index~$m$ such that there is $H\in\pi$ with $\{0\}\cap H\ne\varnothing$ and $\langle B_m\rangle\cap H\ne\varnothing$, which is well-defined by the covering assumption for $\pi$ with index $\alpha=1$. Once~$m_i$ and~$m_i'$ are defined for some $i\ge1$, if $m_i'<r$, we define $m_{i+1}'$ as the maximum index $m$ such that there is~$H\in\pi$ with $(\{0\}\cup\bigcup_{l\le m_i'}\langle B_l\rangle)\cap H\ne\varnothing$ and $\langle B_m\rangle\cap H\ne\varnothing$, which is well-defined by the covering assumption for $\pi$ with index $\alpha=m_i'+1\le r$ and  satisfies $m_{i+1}'>m_i'$ by construction. Next, we define $m_{i+1}$ as the minimum index $m$ such that there is $H\in\pi$ with $\langle B_m\rangle\cap H\ne\varnothing$ and $\langle B_{m_{i+1}'}\rangle\cap H\ne\varnothing$. We continue the construction until $m_s'=r$ is reached.
We claim that by construction we have $m_{i-1}'<m_{i+1} \le m_{i}'$ for all $i$ (which, since $(m'_j)_j$ is increasing, implies that $(m_j)_j$ is increasing as well). On the one hand, we indeed have $m_{i+1} \le m_i'$ by definition of $m_{i+1}'$. On the other hand, we must have $m_{i+1}>m_{i-1}'$ since the inequality
$m_{i+1}\le m_{i-1}'$ would imply $m_i'=m_{i+1}'$ and contradict the strict monotonicity of the sequence $(m'_j)_j$.

\medskip\noindent
With this construction of intertwined pairings $(m_i,m_i')_{1\le i\le s}$,
we can choose a sequence of distinct blocks $(H_{i})_{1\le i\le s}$ of $\pi$ such that \mbox{$\langle B_{m_i}\rangle \cap H_i\ne\varnothing$} and \mbox{$\langle B_{m_i'}\rangle \cap H_i\ne\varnothing$} for all $i$ {(recall that $B_0=\{0\}$)}. We may then pick indices $j_i,j_i'\in\{0\}\cup[k]$ such that $j_i\in\langle B_{m_i}\rangle \cap H_i$ and $j_i'\in\langle B_{m_i'}\rangle \cap H_i$ for all $i$.
In these terms, appealing both to~\eqref{eq:high-intens-correl} and to the decay assumption~\ref{Mix-om-n} with $n=k+1$,
the product of correlation functions in~\eqref{eq:pre-bnd-Dpi} can be bounded for instance as follows (up to integration, as in~\eqref{eq:example-Mix-om-n}),
\begin{equation}\label{eq:choice-correl}
\prod_{H\in\pi}|h_{\sharp H}(x_H)|\,\lesssim\,\lambda_{p_0}(\Pc)\prod_{i=1}^s\big(\omega(x_{j_i}-x_{j_i'})\wedge\lambda_{p_i}(\Pc)\big),
\end{equation}
for some $p_0,\ldots,p_s\ge1$ with $\sum_{i=0}^sp_i=k+1$.
We then define the following concatenations of blocks between paired indices: for $1\le i< s$,
\begin{eqnarray*}
A_{i}\,:=\,(f_{m_{i-1}'})\uplus B_{m_{i-1}'+1}\uplus\ldots\uplus B_{m_{i+1}-1}\uplus(b_{m_{i+1}}),\quad
A_i'\,:=\,B_{m_{i+1}}\uplus\ldots\uplus B_{m_{i}'},
\end{eqnarray*}
with the convention $m_0':=0$, and
\begin{eqnarray*}
A_s\,:=\,(f_{m_{s-1}'})\uplus B_{m_{s-1}'+1}\uplus\ldots\uplus B_{m_s},\qquad
A_s'\,:=\,\varnothing.
\end{eqnarray*}
In these terms,
inserting~\eqref{eq:choice-correl} into~\eqref{eq:pre-bnd-Dpi}, the integral can be reorganized as
\begin{multline*}
|\Cc^{k+1;r}_L(B_1,\ldots,B_r;\pi)|\,\lesssim\,\lambda_{p_0}(\Pc)\\
\times L^{-d}\int_{(\Qd)^{k+1}}\bigg(\prod_{i=1}^{s}D_{A_{i}}(x_{A_i})D_{A_i'}(x_{A_i'})\big(\omega(x_{j_i}-x_{j_i'})\wedge\lambda_{p_i}(\Pc)\big)\bigg)\,dx_0dx_{[k]}.
\end{multline*}
We emphasize that the coupled indices $j_i$'s and $j_i'$'s belong to $A_i'$'s and can thus intersect $A_i$'s only at their endpoints.
In terms of the graphical representation introduced in {Step~3}, the latter integral can be represented generically in the following way,
\begin{align}\label{eq:represent-integr}
&|\Cc^{k+1;r}_L(B_1,\ldots,B_r;\pi)|\\
&\,\lesssim\,L^{-d}\lambda_{k+1}'~
{\tiny\begin{tikzpicture}[baseline={([yshift=-.8ex]current bounding box.center)},scale=0.5]
\filldraw[gray] (1,-0.5) rectangle (2,0.5);
\filldraw[gray] (4,-0.5) rectangle (5,0.5);
\filldraw[gray] (7,-0.5) rectangle (8,0.5);
\filldraw[gray] (10,-0.5) rectangle (11,0.5);
\filldraw[gray] (13,-0.5) rectangle (14,0.5);
\filldraw[gray] (18,-0.5) rectangle (19,0.5);
\filldraw[gray] (21,-0.5) rectangle (22,0.5);
\filldraw[gray] (24,-0.5) rectangle (25,0.5);
\begin{scope}[every node/.style={circle,fill,inner sep=0pt,minimum size=3pt}]
    \node (1) at (0,0) {};
    \node (2b) at (1,0) {};
    \node (2f) at (2,0) {};
    \node (3b) at (4,0) {};
    \node (3f) at (5,0) {};
    \node (3d) at (4.5,-0.5) {};
    \node (4b) at (7,0) {};
    \node (4f) at (8,0) {};
    \node (4d) at (7.5,0.5) {};
    \node (5b) at (10,0) {};
    \node (5f) at (11,0) {};
    \node (5d) at (10.5,0.5) {};
    \node (6b) at (13,0) {};
    \node (6f) at (14,0) {};
    \node (6d) at (13.5,-0.5) {};
    \node (8b) at (18,0) {};
    \node (8f) at (19,0) {};
    \node (8d) at (18.5,0.5) {};
    \node (9b) at (21,0) {};
    \node (9f) at (22,0) {};
    \node (9d) at (21.5,-0.5) {};
    \node (10b) at (24,0) {};
    \node (10d) at (24.5,0.5) {};
\end{scope}
\begin{scope}[every node/.style={circle}]
    \node (2f+) at (2.8,0) {};
    \node (3b-) at (3.2,0) {};
    \node (3f+) at (5.8,0) {};
    \node (4b-) at (6.2,0) {};
    \node (4f+) at (8.8,0) {};
    \node (5b-) at (9.2,0) {};
    \node (5f+) at (11.8,0) {};
    \node (6b-) at (12.2,0) {};
    \node (6f+) at (14.8,0) {};
    \node (stop1-) at (14.8,1.5) {};
    \node (8b-) at (17.2,0) {};
    \node (8f+) at (19.8,0) {};
    \node (9b-) at (20.2,0) {};
    \node (9f+) at (22.8,0) {};
    \node (10b-) at (23.2,0) {};
    \node (stop2-) at (17.2,-1.5) {};
\end{scope}
\begin{scope}[every node/.style={circle}]
    \node (23) at (3,0) {$\ldots$};
    \node (34) at (6,0) {$\ldots$};
    \node (45) at (9,0) {$\ldots$};
    \node (56) at (12,0) {$\ldots$};
    \node (67) at (15,0) {$\ldots$};
    \node (stop1) at (15,1.5) {$\ldots$};
    \node (78) at (17,0) {$\ldots$};
    \node (89) at (20,0) {$\ldots$};
    \node (910) at (23,0) {$\ldots$};
    \node (stop2) at (17,-1.5) {$\ldots$};
    \node (stop2) at (15,-2) {$\ldots$};
    \node (stop2) at (17,-2) {$\ldots$};
\end{scope}
\begin{scope}[>={Stealth[black]},
every edge/.style={draw=black,very thick}]
    \path [-] (1) edge (2b);
    \path [-] (2f) edge (2f+);
    \path [-] (3b-) edge (3b);
    \path [-] (3f) edge (3f+);
    \path [-] (4b-) edge (4b);
    \path [-] (4f) edge (4f+);
    \path [-] (5b-) edge (5b);
    \path [-] (5f) edge (5f+);
    \path [-] (6b-) edge (6b);
    \path [-] (6f) edge (6f+);
    \path [-] (8b-) edge (8b);
    \path [-] (8f) edge (8f+);
    \path [-] (9b-) edge (9b);
    \path [-] (9f) edge (9f+);
    \path [-] (10b-) edge (10b);
\end{scope}
\begin{scope}[>={Stealth[black]},
every edge/.style={draw=black,very thick,dashed}]
    \path [-] (1) edge[bend left=40] (4d);
    \path [-] (3d) edge[bend left=-20] (6d);
    \path [-] (5d) edge[bend left=10] (stop1-);
    \path [-] (8d) edge[bend left=30] (10d);
    \path [-] (9d) edge[bend left=10] (stop2-);
\end{scope}
\begin{scope}[every node/.style={circle,fill,inner sep=0pt,minimum size=0.8pt}]
     \node (AA') at (0,-1.8) {};
     \node (AA'') at (0,-2.2) {};
     \node (BB') at (4,-1.8) {};
     \node (BB'') at (4,-2.2) {};
     \node (CC') at (8,-1.8) {};
     \node (CC'') at (8,-2.2) {};
     \node (DD') at (10,-1.8) {};
     \node (DD'') at (10,-2.2) {};
     \node (EE') at (14,-1.8) {};
     \node (EE'') at (14,-2.2) {};
     \node (FF') at (18,-1.8) {};
     \node (FF'') at (18,-2.2) {};
     \node (GG') at (22,-1.8) {};
     \node (GG'') at (22,-2.2) {};
     \node (HH') at (25,-1.8) {};
     \node (HH'') at (25,-2.2) {};
\end{scope}
\begin{scope}[every node/.style={circle,fill,inner sep=0pt,minimum size=0pt}]
     \node (A) at (0,-2) {};
     \node (A') at (0,-1.8) {};
     \node (A'') at (0,-2.2) {};
     \node (B) at (4,-2) {};
     \node (B') at (4,-1.8) {};
     \node (B'') at (4,-2.2) {};
     \node (C) at (8,-2) {};
     \node (C') at (8,-1.8) {};
     \node (C'') at (8,-2.2) {};
     \node (D) at (10,-2) {};
     \node (D') at (10,-1.8) {};
     \node (D'') at (10,-2.2) {};
     \node (E) at (14,-2) {};
     \node (E') at (14,-1.8) {};
     \node (E'') at (14,-2.2) {};
     \node (Estop) at (14.5,-2) {};
     \node (Fstop) at (17.5,-2) {};
     \node (F) at (18,-2) {};
     \node (F') at (18,-1.8) {};
     \node (F'') at (18,-2.2) {};
     \node (G) at (22,-2) {};
     \node (G') at (22,-1.8) {};
     \node (G'') at (22,-2.2) {};
     \node (H) at (25,-2) {};
     \node (H') at (25,-1.8) {};
     \node (H'') at (25,-2.2) {};
\end{scope}
\begin{scope}[>={Stealth[black]},
every edge/.style={draw=black,thick}]
    \path [-] (A) edge node[below,pos=0.5] {$A_1$} (B);
    \path [-] (B) edge node[below,pos=0.5] {$A_1'$} (C);
    \path [-] (C) edge node[below,pos=0.5] {$A_2$} (D);
    \path [-] (D) edge node[below,pos=0.5] {$A_2'$} (E);
    \path [-] (F) edge node[below,pos=0.5] {$A_{s-1}'$} (G);
    \path [-] (G) edge node[below,pos=0.5] {$A_s$} (H);
    \path [-] (A') edge (A'');
    \path [-] (B') edge (B'');
    \path [-] (C') edge (C'');
    \path [-] (D') edge (D'');
    \path [-] (E') edge (E'');
    \path [-] (E) edge (Estop);
    \path [-] (F) edge (Fstop);
    \path [-] (F') edge (F'');
    \path [-] (G') edge (G'');
    \path [-] (H') edge (H'');
\end{scope}
\end{tikzpicture}}\nonumber
\end{align}
where we further delineate the concatenations of blocks $A_i$'s and $A_i'$'s.
In particular, note these take the generic forms
\begin{eqnarray*}
A_i&\equiv&{\tiny\begin{tikzpicture}[baseline={([yshift=-.8ex]current bounding box.center)},scale=0.6]
\filldraw[gray] (1,-0.5) rectangle (2,0.5);
\filldraw[gray] (4,-0.5) rectangle (5,0.5);
\begin{scope}[every node/.style={circle,draw,inner sep=0pt,minimum size=3pt}]
    \node (1) at (0,0) {};
    \node (4) at (6,0) {};
\end{scope}
\begin{scope}[every node/.style={circle,fill,inner sep=0pt,minimum size=3pt}]
    \node (2b) at (1,0) {};
    \node (2f) at (2,0) {};
    \node (3b) at (4,0) {};
    \node (3f) at (5,0) {};
\end{scope}
\begin{scope}[every node/.style={circle}]
    \node (23) at (3,0) {$\ldots$};
\end{scope}
\begin{scope}[every node/.style={circle}]
    \node (2f+) at (2.8,0) {};
    \node (3b-) at (3.2,0) {};
\end{scope}
\begin{scope}[>={Stealth[black]},
every edge/.style={draw=black,very thick}]
    \path [-] (1) edge (2b);
    \path [-] (2f) edge (2f+);
    \path [-] (3b-) edge (3b);
    \path [-] (3f) edge (4);
\end{scope}
\end{tikzpicture}}
\\
A_i'&\equiv&
{\tiny\begin{tikzpicture}[baseline={([yshift=-.8ex]current bounding box.center)},scale=0.6]
\filldraw[gray] (1,-0.5) rectangle (2,0.5);
\filldraw[gray] (4,-0.5) rectangle (5,0.5);
\begin{scope}[every node/.style={circle,draw,fill=white,inner sep=0pt,minimum size=3pt}]
    \node (2b) at (1,0) {};
    \node (2d) at (1.5,0.5) {};
    \node (3d) at (4.5,-0.5) {};
    \node (3f) at (5,0) {};
\end{scope}
\begin{scope}[every node/.style={circle,fill,inner sep=0pt,minimum size=3pt}]
    \node (2f) at (2,0) {};
    \node (3b) at (4,0) {};
\end{scope}
\begin{scope}[every node/.style={circle}]
    \node (23) at (3,0) {$\ldots$};
\end{scope}
\begin{scope}[every node/.style={circle}]
    \node (2f+) at (2.8,0) {};
    \node (3b-) at (3.2,0) {};
\end{scope}
\begin{scope}[>={Stealth[black]},
every edge/.style={draw=black,very thick}]
    \path [-] (2f) edge (2f+);
    \path [-] (3b-) edge (3b);
\end{scope}
\end{tikzpicture}}
\quad\text{or}\quad
{\tiny\begin{tikzpicture}[baseline={([yshift=-.8ex]current bounding box.center)},scale=0.6]
\filldraw[gray] (1,-0.5) rectangle (2,0.5);
\begin{scope}[every node/.style={circle,draw,fill=white,inner sep=0pt,minimum size=3pt}]
    \node (2b) at (1,0) {};
    \node (2d) at (1.5,0.5) {};
    \node (3d) at (1.5,-0.5) {};
    \node (3f) at (2,0) {};
\end{scope}
\end{tikzpicture}}
\end{eqnarray*}
where the second possibility for $A_i'$ corresponds to the case when $m_{i+1}=m_i'$.
In order to estimate~\eqref{eq:represent-integr}, we first perform integration on $A_i$'s and $A_i'$'s:
using~\eqref{eq:integ-DB-plop-diagr-bis}--\eqref{eq:block-integr-b=alpha=bet} to estimate the integral on each block, and using~\eqref{eq:bound-integ-kernel-diag} to estimate consecutive edges, we find
\[\begin{array}{rllll}
{\tiny\begin{tikzpicture}[baseline={([yshift=-.8ex]current bounding box.center)},scale=0.6]
\filldraw[gray] (1,-0.5) rectangle (2,0.5);
\filldraw[gray] (4,-0.5) rectangle (5,0.5);
\begin{scope}[every node/.style={circle,draw,inner sep=0pt,minimum size=3pt}]
    \node (1) at (0,0) {};
    \node (4) at (6,0) {};
\end{scope}
\begin{scope}[every node/.style={circle,fill,inner sep=0pt,minimum size=3pt}]
    \node (2b) at (1,0) {};
    \node (2f) at (2,0) {};
    \node (3b) at (4,0) {};
    \node (3f) at (5,0) {};
\end{scope}
\begin{scope}[every node/.style={circle}]
    \node (23) at (3,0) {$\ldots$};
\end{scope}
\begin{scope}[every node/.style={circle}]
    \node (2f+) at (2.8,0) {};
    \node (3b-) at (3.2,0) {};
\end{scope}
\begin{scope}[>={Stealth[black]},
every edge/.style={draw=black,very thick}]
    \path [-] (1) edge (2b);
    \path [-] (2f) edge (2f+);
    \path [-] (3b-) edge (3b);
    \path [-] (3f) edge (4);
\end{scope}
\end{tikzpicture}}
&\lesssim&
\Lc^{[\sharp]}\,
{\tiny\begin{tikzpicture}[baseline={([yshift=-.8ex]current bounding box.center)},scale=0.6]
\begin{scope}[every node/.style={circle,draw,inner sep=0pt,minimum size=3pt}]
    \node (1) at (0,0) {};
    \node (4) at (6,0) {};
\end{scope}
\begin{scope}[every node/.style={circle,fill,inner sep=0pt,minimum size=3pt}]
    \node (2b) at (1,0) {};
    \node (2f) at (2,0) {};
    \node (3b) at (4,0) {};
    \node (3f) at (5,0) {};
\end{scope}
\begin{scope}[every node/.style={circle}]
    \node (23) at (3,0) {$\ldots$};
\end{scope}
\begin{scope}[every node/.style={circle}]
    \node (2f+) at (2.8,0) {};
    \node (3b-) at (3.2,0) {};
\end{scope}
\begin{scope}[>={Stealth[black]},
every edge/.style={draw=black,very thick}]
    \path [-] (1) edge (2b);
    \path [-] (2b) edge[bend left=30] (2f);
    \path [-] (2f) edge[bend left=30] (2b);
    \path [-] (2f) edge (2f+);
    \path [-] (3b-) edge (3b);
    \path [-] (3b) edge[bend left=30] (3f);
    \path [-] (3f) edge[bend left=30] (3b);
    \path [-] (3f) edge (4);
\end{scope}
\end{tikzpicture}}
&\lesssim&
\Lc^{[\sharp]}\,
{\tiny\begin{tikzpicture}[baseline={([yshift=-.8ex]current bounding box.center)},scale=0.6]
\begin{scope}[every node/.style={circle,draw,inner sep=0pt,minimum size=3pt}]
    \node (1) at (0,0) {};
    \node (4) at (1,0) {};
\end{scope}
\begin{scope}[>={Stealth[black]},
every edge/.style={draw=black,very thick}]
    \path [-] (1) edge (4);
\end{scope}
\end{tikzpicture}}\\
\vspace{-0.2cm}&&&&\\
{\tiny\begin{tikzpicture}[baseline={([yshift=-.8ex]current bounding box.center)},scale=0.6]
\filldraw[gray] (1,-0.5) rectangle (2,0.5);
\filldraw[gray] (4,-0.5) rectangle (5,0.5);
\begin{scope}[every node/.style={circle,draw,fill=white,inner sep=0pt,minimum size=3pt}]
    \node (2b) at (1,0) {};
    \node (2d) at (1.5,0.5) {};
    \node (3d) at (4.5,-0.5) {};
    \node (3f) at (5,0) {};
\end{scope}
\begin{scope}[every node/.style={circle,fill,inner sep=0pt,minimum size=3pt}]
    \node (2f) at (2,0) {};
    \node (3b) at (4,0) {};
\end{scope}
\begin{scope}[every node/.style={circle}]
    \node (23) at (3,0) {$\ldots$};
\end{scope}
\begin{scope}[every node/.style={circle}]
    \node (2f+) at (2.8,0) {};
    \node (3b-) at (3.2,0) {};
\end{scope}
\begin{scope}[>={Stealth[black]},
every edge/.style={draw=black,very thick}]
    \path [-] (2f) edge (2f+);
    \path [-] (3b-) edge (3b);
\end{scope}
\end{tikzpicture}}
&\lesssim&
\Lc^{[\sharp]}\,
{\tiny\begin{tikzpicture}[baseline={([yshift=-.8ex]current bounding box.center)},scale=0.6]
\begin{scope}[every node/.style={circle,draw,fill=white,inner sep=0pt,minimum size=3pt}]
    \node (2b) at (1,0) {};
    \node (2d) at (1.5,0.5) {};
    \node (3d) at (4.5,-0.5) {};
    \node (3f) at (5,0) {};
\end{scope}
\begin{scope}[every node/.style={circle,fill,inner sep=0pt,minimum size=3pt}]
    \node (2f) at (2,0) {};
    \node (3b) at (4,0) {};
\end{scope}
\begin{scope}[every node/.style={circle}]
    \node (23) at (3,0) {$\ldots$};
\end{scope}
\begin{scope}[every node/.style={circle}]
    \node (2f+) at (2.8,0) {};
    \node (3b-) at (3.2,0) {};
\end{scope}
\begin{scope}[>={Stealth[black]},
every edge/.style={draw=black,very thick}]
    \path [-] (2f) edge (2f+);
    \path [-] (3b-) edge (3b);
    \path [-] (2b) edge (2d);
    \path [-] (3f) edge (3d);
    \path [-] (2b) edge[bend left =30] (2f);
    \path [-] (2f) edge[bend left =30] (2b);
    \path [-] (3b) edge[bend left =30] (3f);
    \path [-] (3f) edge[bend left =30] (3b);
\end{scope}
\end{tikzpicture}}
&\lesssim&
\Lc^{[\sharp]}\,
{\tiny\begin{tikzpicture}[baseline={([yshift=-.8ex]current bounding box.center)},scale=0.6]
\begin{scope}[every node/.style={circle,draw,fill=white,inner sep=0pt,minimum size=3pt}]
    \node (2b) at (1,0) {};
    \node (2d) at (1.5,0.5) {};
    \node (3d) at (1.5,-0.5) {};
    \node (3f) at (2,0) {};
\end{scope}
\begin{scope}[>={Stealth[black]},
every edge/.style={draw=black,very thick}]
    \path [-] (2b) edge (2d);
    \path [-] (3f) edge (3d);
    \path [-] (2b) edge (3f);
\end{scope}
\end{tikzpicture}}\\
\vspace{-0.2cm}&&&&\\
{\tiny\begin{tikzpicture}[baseline={([yshift=-.8ex]current bounding box.center)},scale=0.6]
\filldraw[gray] (1,-0.5) rectangle (2,0.5);
\begin{scope}[every node/.style={circle,draw,fill=white,inner sep=0pt,minimum size=3pt}]
    \node (2b) at (1,0) {};
    \node (2d) at (1.5,0.5) {};
    \node (3d) at (1.5,-0.5) {};
    \node (3f) at (2,0) {};
\end{scope}
\end{tikzpicture}}
&\lesssim&
\Lc^{[\sharp]}\,\Big(
{\tiny\begin{tikzpicture}[baseline={([yshift=-.8ex]current bounding box.center)},scale=0.6]
\begin{scope}[every node/.style={circle,draw,fill=white,inner sep=0pt,minimum size=3pt}]
    \node (2b) at (1,0) {};
    \node (2d) at (1.5,0.5) {};
    \node (3d) at (1.5,-0.5) {};
    \node (3f) at (2,0) {};
\end{scope}
\begin{scope}[>={Stealth[black]},
every edge/.style={draw=black,very thick}]
    \path [-] (2b) edge (2d);
    \path [-] (2b) edge (3f);
    \path [-] (3f) edge (3d);
\end{scope}
\end{tikzpicture}}
+
{\tiny\begin{tikzpicture}[baseline={([yshift=-.8ex]current bounding box.center)},scale=0.6]
\begin{scope}[every node/.style={circle,draw,fill=white,inner sep=0pt,minimum size=3pt}]
    \node (2b) at (1,0) {};
    \node (2d) at (1.5,0.5) {};
    \node (3d) at (1.5,-0.5) {};
    \node (3f) at (2,0) {};
\end{scope}
\begin{scope}[>={Stealth[black]},
every edge/.style={draw=black,very thick}]
    \path [-] (2b) edge (3d);
    \path [-] (3f) edge (2d);
    \path [-] (2b) edge[bend left=30] (3f);
    \path [-] (2b) edge[bend left=-30] (3f);
\end{scope}
\end{tikzpicture}}
\Big)
&&
\end{array}\]
where henceforth we use the short-hand notation $\Lc^{[\sharp]}$ for a power of the logarithmic factor that can change from an occurrence to another and stands for the difference between the numbers of vertices in the left-hand side and in the right-hand side.
Inserting this into~\eqref{eq:represent-integr}, we are led to
\begin{equation}\label{eq:represent-integr-re}
|\Cc^{k+1;r}_L(B_1,\ldots,B_r;\pi)|
\,\lesssim\, L^{-d}\Lc^{[\sharp]}\lambda_{k+1}'
{\tiny\begin{tikzpicture}[baseline={([yshift=-.8ex]current bounding box.center)},scale=0.6]
\draw[pattern=north east lines] (1,0) rectangle (2,1);
\draw[pattern=north east lines] (3,0) rectangle (4,1);
\draw[pattern=north east lines] (5,0) rectangle (6,1);
\draw[pattern=north east lines] (8,0) rectangle (9,1);
\begin{scope}[every node/.style={circle,fill,inner sep=0pt,minimum size=3pt}]
    \node (1) at (0,0) {};
    \node (2b) at (1,0) {};
    \node (2f) at (2,0) {};
    \node (2d) at (1,1) {};
    \node (2u) at (2,1) {};
    \node (3b) at (3,0) {};
    \node (3f) at (4,0) {};
    \node (3d) at (3,1) {};
    \node (3u) at (4,1) {};
    \node (4b) at (5,0) {};
    \node (4f) at (6,0) {};
    \node (4d) at (5,1) {};
    \node (4u) at (6,1) {};
    \node (5b) at (8,0) {};
    \node (5f) at (9,0) {};
    \node (5d) at (8,1) {};
    \node (5u) at (9,1) {};
    \node (6) at (10,0) {};
\end{scope}
\begin{scope}[every node/.style={circle}]
    \node (4f+) at (6.8,0) {};
    \node (5b-) at (7.2,0) {};
    \node (stop1-) at (6.8,1) {};
    \node (stop2-) at (7.2,1) {};
\end{scope}
\begin{scope}[every node/.style={circle}]
    \node (45) at (7,0) {$\ldots$};
    \node (stop1) at (7,1) {$\ldots$};
\end{scope}
\begin{scope}[>={Stealth[black]},
every edge/.style={draw=black,very thick}]
    \path [-] (1) edge (2b);
    \path [-] (2f) edge (3b);
    \path [-] (3f) edge (4b);
    \path [-] (4f) edge (4f+);
    \path [-] (5b-) edge (5b);
    \path [-] (5f) edge (6);
\end{scope}
\begin{scope}[>={Stealth[black]},
every edge/.style={draw=black,very thick,dashed}]
    \path [-] (1) edge[bend left=10] (2d);
    \path [-] (2u) edge (3d);
    \path [-] (3u) edge (4d);
    \path [-] (4u) edge (stop1-);
    \path [-] (stop2-) edge (5d);
    \path [-] (5u) edge[bend left=10] (6);
\end{scope}
\end{tikzpicture}}
\end{equation}
where for abbreviation hatched boxes are given by
\begin{equation*}
{\tiny\begin{tikzpicture}[baseline={([yshift=-.8ex]current bounding box.center)},scale=0.6]
\draw[pattern=north east lines] (0,0) rectangle (1,1);
\begin{scope}[every node/.style={circle,draw,fill=white,inner sep=0pt,minimum size=3pt}]
    \node (1) at (0,0) {};
    \node (2) at (1,0) {};
    \node (3) at (0,1) {};
    \node (4) at (1,1) {};
\end{scope}
\end{tikzpicture}}
~~:=~~
{\tiny\begin{tikzpicture}[baseline={([yshift=-.8ex]current bounding box.center)},scale=0.6]
\begin{scope}[every node/.style={circle,draw,fill=white,inner sep=0pt,minimum size=3pt}]
    \node (1) at (0,0) {};
    \node (2) at (1,0) {};
    \node (3) at (0,1) {};
    \node (4) at (1,1) {};
\end{scope}
\begin{scope}[>={Stealth[black]},
every edge/.style={draw=black,very thick}]
    \path [-] (1) edge (4);
    \path [-] (2) edge (3);
    \path [-] (1) edge (2);
\end{scope}
\end{tikzpicture}}
+
{\tiny\begin{tikzpicture}[baseline={([yshift=-.8ex]current bounding box.center)},scale=0.6]
\begin{scope}[every node/.style={circle,draw,fill=white,inner sep=0pt,minimum size=3pt}]
    \node (1) at (0,0) {};
    \node (2) at (1,0) {};
    \node (3) at (0,1) {};
    \node (4) at (1,1) {};
\end{scope}
\begin{scope}[>={Stealth[black]},
every edge/.style={draw=black,very thick}]
    \path [-] (1) edge (3);
    \path [-] (2) edge (4);
    \path [-] (1) edge[bend left=20] (2);
    \path [-] (2) edge[bend left=20] (1);
\end{scope}
\end{tikzpicture}}
\end{equation*}
which we obtain by reorganizing the graphs as follows,
\[{\tiny\begin{tikzpicture}[baseline={([yshift=-.8ex]current bounding box.center)},scale=0.6]
\begin{scope}[every node/.style={circle,thick,draw}]
    \node (2) at (0,1) {2};
    \node (3) at (2,1) {3};
    \node (1) at (1,2) {1};
    \node (4) at (1,0) {4};
\end{scope}
\begin{scope}[>={Stealth[black]},
every node/.style={fill=white,circle},
every edge/.style={draw=black,very thick}]
    \path [-] (1) edge (2);
    \path [-] (2) edge (3);
    \path [-] (3) edge (4);
\end{scope}
\end{tikzpicture}}
\,\equiv\,
{\tiny\begin{tikzpicture}[baseline={([yshift=-.8ex]current bounding box.center)},scale=0.6]
\begin{scope}[every node/.style={circle,thick,draw}]
    \node (2) at (0,0) {2};
    \node (4) at (0,1.5) {4};
    \node (1) at (1.5,1.5) {1};
    \node (3) at (1.5,0) {3};
\end{scope}
\begin{scope}[>={Stealth[black]},
every node/.style={fill=white,circle},
every edge/.style={draw=black,very thick}]
    \path [-] (1) edge (2);
    \path [-] (2) edge (3);
    \path [-] (3) edge (4);
\end{scope}
\end{tikzpicture}}
\qquad\text{and}\qquad
{\tiny\begin{tikzpicture}[baseline={([yshift=-.8ex]current bounding box.center)},scale=0.6]
\begin{scope}[every node/.style={circle,thick,draw}]
    \node (3) at (0,1) {3};
    \node (2) at (2,1) {2};
    \node (1) at (1,2) {1};
    \node (4) at (1,0) {4};
\end{scope}
\begin{scope}[>={Stealth[black]},
every node/.style={fill=white,circle},
every edge/.style={draw=black,very thick}]
    \path [-] (1) edge (2);
    \path [-] (2) edge[bend left=15] (3);
    \path [-] (2) edge[bend left=-15] (3);
    \path [-] (3) edge (4);
\end{scope}
\end{tikzpicture}}
\,\equiv\,
{\tiny\begin{tikzpicture}[baseline={([yshift=-.8ex]current bounding box.center)},scale=0.6]
\begin{scope}[every node/.style={circle,thick,draw}]
    \node (3) at (0,0) {3};
    \node (4) at (0,1.5) {4};
    \node (1) at (1.5,1.5) {1};
    \node (2) at (1.5,0) {2};
\end{scope}
\begin{scope}[>={Stealth[black]},
every node/.style={fill=white,circle},
every edge/.style={draw=black,very thick}]
    \path [-] (1) edge (2);
    \path [-] (2) edge[bend left=15] (3);
    \path [-] (2) edge[bend left=-15] (3);
    \path [-] (3) edge (4);
\end{scope}
\end{tikzpicture}}\]
It remains to evaluate the right-hand side in~\eqref{eq:represent-integr-re}.
Using the graphical rules~\eqref{eq:bound-integ-kernel-diag}, \eqref{eq:rule-lambda-double}, and~\eqref{eq:bound-integ-kernel-diag-dash}, and noting that that direct integrations yield
\begin{equation*}
{\tiny\begin{tikzpicture}[baseline={([yshift=-.8ex]current bounding box.center)},scale=0.6]
\begin{scope}[every node/.style={circle,draw,fill=white,inner sep=0pt,minimum size=3pt}]
    \node (2) at (1,0) {};
\end{scope}
\begin{scope}[every node/.style={circle,fill,inner sep=0pt,minimum size=3pt}]
    \node (1) at (0,0) {};
\end{scope}
\begin{scope}[>={Stealth[black]},
every edge/.style={draw=black,very thick}]
    \path [-] (1) edge[bend left=30] (2);
    \path [-] (1) edge[bend left=-30] (2);
\end{scope}
\begin{scope}[>={Stealth[black]},
every edge/.style={draw=black,very thick,dashed}]
    \path [-] (1) edge[bend left=-30] (2);
\end{scope}
\end{tikzpicture}}
~~\lesssim~~
{\tiny\begin{tikzpicture}[baseline={([yshift=-.8ex]current bounding box.center)},scale=0.6]
\begin{scope}[every node/.style={circle,draw,fill=white,inner sep=0pt,minimum size=3pt}]
    \node (2) at (1,0) {};
\end{scope}
\end{tikzpicture}}
\qquad\text{and}\qquad
\Lc^\mu\lambda^\circ_{k+1}\,{\tiny\begin{tikzpicture}[baseline={([yshift=-.8ex]current bounding box.center)},scale=0.6]
\begin{scope}[every node/.style={circle,draw,fill=white,inner sep=0pt,minimum size=3pt}]
    \node (2) at (1,0) {};
\end{scope}
\begin{scope}[every node/.style={circle,fill,inner sep=0pt,minimum size=3pt}]
    \node (1) at (0,0) {};
\end{scope}
\begin{scope}[>={Stealth[black]},
every edge/.style={draw=black,very thick}]
    \path [-] (1) edge[bend left=30] (2);
\end{scope}
\begin{scope}[>={Stealth[black]},
every edge/.style={draw=black,very thick,dashed}]
    \path [-] (1) edge[bend left=-30] (2);
\end{scope}
\end{tikzpicture}}
~~\lesssim~~\Lc^{\mu+1}\lambda'_{k+1}\,
{\tiny\begin{tikzpicture}[baseline={([yshift=-.8ex]current bounding box.center)},scale=0.6]
\begin{scope}[every node/.style={circle,draw,fill=white,inner sep=0pt,minimum size=3pt}]
    \node (2) at (1,0) {};
\end{scope}
\end{tikzpicture}}
\end{equation*}
we can estimate
\begin{eqnarray*}
\Lc^\mu\lambda_{k+1}'\,
{\tiny\begin{tikzpicture}[baseline={([yshift=-.8ex]current bounding box.center)},scale=0.6]
\draw[pattern=north east lines] (1,0) rectangle (2,1);
\begin{scope}[every node/.style={circle,fill,draw,inner sep=0pt,minimum size=3pt}]
    \node (1) at (0,0) {};
    \node (2) at (1,0) {};
    \node (3) at (1,1) {};
    \node (4) at (2,0) {};
    \node (5) at (2,1) {};
\end{scope}
\begin{scope}[every node/.style={circle,draw,fill=white,inner sep=0pt,minimum size=3pt}]
    \node (6) at (3,0) {};
    \node (7) at (3,1) {};
\end{scope}
\begin{scope}[>={Stealth[black]},
every edge/.style={draw=black,very thick}]
    \path [-] (1) edge (2);
    \path [-] (4) edge (6);
\end{scope}
\begin{scope}[>={Stealth[black]},
every edge/.style={draw=black,very thick,dashed}]
    \path [-] (1) edge[bend left=10] (3);
    \path [-] (5) edge (7);
\end{scope}
\end{tikzpicture}}
&=&
\Lc^\mu\lambda_{k+1}'\Big(~
{\tiny\begin{tikzpicture}[baseline={([yshift=-.8ex]current bounding box.center)},scale=0.6]
\begin{scope}[every node/.style={circle,fill,draw,inner sep=0pt,minimum size=3pt}]
    \node (1) at (0,0) {};
    \node (2) at (1,0) {};
    \node (3) at (1,1) {};
    \node (4) at (2,0) {};
    \node (5) at (2,1) {};
\end{scope}
\begin{scope}[every node/.style={circle,draw,fill=white,inner sep=0pt,minimum size=3pt}]
    \node (6) at (3,0) {};
    \node (7) at (3,1) {};
\end{scope}
\begin{scope}[>={Stealth[black]},
every edge/.style={draw=black,very thick}]
    \path [-] (1) edge (2);
    \path [-] (3) edge (4);
    \path [-] (2) edge (5);
    \path [-] (2) edge (4);
    \path [-] (4) edge (6);
\end{scope}
\begin{scope}[>={Stealth[black]},
every edge/.style={draw=black,very thick,dashed}]
    \path [-] (1) edge[bend left=10] (3);
    \path [-] (5) edge (7);
\end{scope}
\end{tikzpicture}}
+
{\tiny\begin{tikzpicture}[baseline={([yshift=-.8ex]current bounding box.center)},scale=0.6]
\begin{scope}[every node/.style={circle,fill,draw,inner sep=0pt,minimum size=3pt}]
    \node (1) at (0,0) {};
    \node (2) at (1,0) {};
    \node (3) at (1,1) {};
    \node (4) at (2,0) {};
    \node (5) at (2,1) {};
\end{scope}
\begin{scope}[every node/.style={circle,draw,fill=white,inner sep=0pt,minimum size=3pt}]
    \node (6) at (3,0) {};
    \node (7) at (3,1) {};
\end{scope}
\begin{scope}[>={Stealth[black]},
every edge/.style={draw=black,very thick}]
    \path [-] (1) edge (2);
    \path [-] (2) edge[bend left=30] (4);
    \path [-] (2) edge[bend left=-30] (4);
    \path [-] (2) edge (3);
    \path [-] (4) edge (5);
    \path [-] (4) edge (6);
\end{scope}
\begin{scope}[>={Stealth[black]},
every edge/.style={draw=black,very thick,dashed}]
    \path [-] (1) edge[bend left=10] (3);
    \path [-] (5) edge (7);
\end{scope}
\end{tikzpicture}}
~\Big)
\\
&\lesssim&
\Lc^{\mu+1}\lambda_{k+1}'\Big(~
{\tiny\begin{tikzpicture}[baseline={([yshift=-.8ex]current bounding box.center)},scale=0.6]
\begin{scope}[every node/.style={circle,fill,draw,inner sep=0pt,minimum size=3pt}]
    \node (2) at (1,0) {};
    \node (3) at (1,1) {};
    \node (4) at (2,0) {};
    \node (5) at (2,1) {};
\end{scope}
\begin{scope}[every node/.style={circle,draw,fill=white,inner sep=0pt,minimum size=3pt}]
    \node (6) at (3,0) {};
    \node (7) at (3,1) {};
\end{scope}
\begin{scope}[>={Stealth[black]},
every edge/.style={draw=black,very thick}]
    \path [-] (3) edge (4);
    \path [-] (2) edge (5);
    \path [-] (2) edge (4);
    \path [-] (4) edge (6);
\end{scope}
\begin{scope}[>={Stealth[black]},
every edge/.style={draw=black,very thick,dashed}]
    \path [-] (2) edge (3);
    \path [-] (5) edge (7);
\end{scope}
\end{tikzpicture}}
+
{\tiny\begin{tikzpicture}[baseline={([yshift=-.8ex]current bounding box.center)},scale=0.6]
\begin{scope}[every node/.style={circle,fill,draw,inner sep=0pt,minimum size=3pt}]
    \node (2) at (1,0) {};
    \node (3) at (1,1) {};
    \node (4) at (2,0) {};
    \node (5) at (2,1) {};
\end{scope}
\begin{scope}[every node/.style={circle,draw,fill=white,inner sep=0pt,minimum size=3pt}]
    \node (6) at (3,0) {};
    \node (7) at (3,1) {};
\end{scope}
\begin{scope}[>={Stealth[black]},
every edge/.style={draw=black,very thick}]
    \path [-] (2) edge[bend left=30] (4);
    \path [-] (2) edge[bend left=-30] (4);
    \path [-] (2) edge[bend left=-30] (3);
    \path [-] (4) edge (5);
    \path [-] (4) edge (6);
\end{scope}
\begin{scope}[>={Stealth[black]},
every edge/.style={draw=black,very thick,dashed}]
    \path [-] (2) edge[bend left=30] (3);
    \path [-] (5) edge (7);
\end{scope}
\end{tikzpicture}}
~\Big)
\\
&\lesssim&
\Lc^{\mu+2}\lambda_{k+1}'\Big(~
{\tiny\begin{tikzpicture}[baseline={([yshift=-.8ex]current bounding box.center)},scale=0.6]
\begin{scope}[every node/.style={circle,fill,draw,inner sep=0pt,minimum size=3pt}]
    \node (2) at (1,0) {};
    \node (4) at (2,0) {};
    \node (5) at (2,1) {};
\end{scope}
\begin{scope}[every node/.style={circle,draw,fill=white,inner sep=0pt,minimum size=3pt}]
    \node (6) at (3,0) {};
    \node (7) at (3,1) {};
\end{scope}
\begin{scope}[>={Stealth[black]},
every edge/.style={draw=black,very thick}]
    \path [-] (2) edge (5);
    \path [-] (2) edge[bend left=30] (4);
    \path [-] (4) edge (6);
\end{scope}
\begin{scope}[>={Stealth[black]},
every edge/.style={draw=black,very thick,dashed}]
    \path [-] (2) edge[bend left=-30] (4);
    \path [-] (5) edge (7);
\end{scope}
\end{tikzpicture}}
+
{\tiny\begin{tikzpicture}[baseline={([yshift=-.8ex]current bounding box.center)},scale=0.6]
\begin{scope}[every node/.style={circle,fill,draw,inner sep=0pt,minimum size=3pt}]
    \node (2) at (1,0) {};
    \node (4) at (2,0) {};
    \node (5) at (2,1) {};
\end{scope}
\begin{scope}[every node/.style={circle,draw,fill=white,inner sep=0pt,minimum size=3pt}]
    \node (6) at (3,0) {};
    \node (7) at (3,1) {};
\end{scope}
\begin{scope}[>={Stealth[black]},
every edge/.style={draw=black,very thick}]
    \path [-] (2) edge[bend left=30] (4);
    \path [-] (2) edge[bend left=-30] (4);
    \path [-] (4) edge (5);
    \path [-] (4) edge (6);
\end{scope}
\begin{scope}[>={Stealth[black]},
every edge/.style={draw=black,very thick,dashed}]
    \path [-] (5) edge (7);
\end{scope}
\end{tikzpicture}}
~\Big)
\\
&\lesssim&
\Lc^{\mu+3}\lambda_{k+1}'\,
{\tiny\begin{tikzpicture}[baseline={([yshift=-.8ex]current bounding box.center)},scale=0.6]
\begin{scope}[every node/.style={circle,fill,draw,inner sep=0pt,minimum size=3pt}]
    \node (4) at (2,0) {};
    \node (5) at (2,1) {};
\end{scope}
\begin{scope}[every node/.style={circle,draw,fill=white,inner sep=0pt,minimum size=3pt}]
    \node (6) at (3,0) {};
    \node (7) at (3,1) {};
\end{scope}
\begin{scope}[>={Stealth[black]},
every edge/.style={draw=black,very thick}]
    \path [-] (4) edge (5);
    \path [-] (4) edge (6);
\end{scope}
\begin{scope}[>={Stealth[black]},
every edge/.style={draw=black,very thick,dashed}]
    \path [-] (5) edge (7);
\end{scope}
\end{tikzpicture}}
\\
&\lesssim&
\Lc^{\mu+4}\lambda_{k+1}'\,
{\tiny\begin{tikzpicture}[baseline={([yshift=-.8ex]current bounding box.center)},scale=0.6]
\begin{scope}[every node/.style={circle,fill,draw,inner sep=0pt,minimum size=3pt}]
    \node (4) at (2,0) {};
\end{scope}
\begin{scope}[every node/.style={circle,draw,fill=white,inner sep=0pt,minimum size=3pt}]
    \node (6) at (3,0) {};
    \node (7) at (3,1) {};
\end{scope}
\begin{scope}[>={Stealth[black]},
every edge/.style={draw=black,very thick}]
    \path [-] (4) edge (6);
\end{scope}
\begin{scope}[>={Stealth[black]},
every edge/.style={draw=black,very thick,dashed}]
    \path [-] (4) edge[bend left=10] (7);
\end{scope}
\end{tikzpicture}}
\end{eqnarray*}
Iterating this estimate, the right-hand side of~\eqref{eq:represent-integr-re} can now be estimated as follows,
\begin{equation*}
|\Cc^{k+1;r}_L(B_1,\ldots,B_r;\pi)|
\,\lesssim\, L^{-d}\Lc^{[\sharp]}\lambda_{k+1}'\,
{\tiny\begin{tikzpicture}[baseline={([yshift=-.8ex]current bounding box.center)},scale=0.6]
\begin{scope}[every node/.style={circle,fill,inner sep=0pt,minimum size=3pt}]
    \node (1) at (0,0) {};
\end{scope}
\end{tikzpicture}}
\end{equation*}
As the number of vertices in the left-hand side is equal to $k+1$ while only one vertex remains in the right-hand side, recalling our notation for $\Lc$ and $\lambda'_k$, and noting that free integration yields ${\tiny\begin{tikzpicture}[baseline={([yshift=-.8ex]current bounding box.center)},scale=0.6]
\begin{scope}[every node/.style={circle,fill,inner sep=0pt,minimum size=3pt}]
    \node (1) at (0,0) {};
\end{scope}
\end{tikzpicture}}=L^d$,
we get
\begin{equation*}
|\Cc^{k+1;r}_L(B_1,\ldots,B_r;\pi)|
\,\lesssim\,L^{-d}\Lc^{k}\lambda_{k+1}'\,{\tiny\begin{tikzpicture}[baseline={([yshift=-.8ex]current bounding box.center)},scale=0.6]
\begin{scope}[every node/.style={circle,fill,inner sep=0pt,minimum size=3pt}]
    \node (1) at (0,0) {};
\end{scope}
\end{tikzpicture}}
\,=\,\lambda_{k+1}(\Pc)|\!\log\lambda(\Pc)|^k,
\end{equation*}
that is,~\eqref{eq:bnd-main-Dk+1npi}.

\medskip\noindent
\substep{5.2} Boundary terms: in case of an algebraic rate $\omega(t)\le Ct^{-\beta}$ for some $C,\beta>0$, we have for all $\pi\notin\Kc(B_1,\ldots,B_r)$,
\begin{equation}\label{eq:boundary-terms-gen}
|\Cc^{k+1;r}_L(B_1,\ldots,B_r;\pi)|\\
\,\lesssim\,
\left\{\begin{array}{lll}
\big(\lambda_{k+1}(\Pc)|\!\log\lambda(\Pc)|^{k-1}\big)\wedge \tfrac{(\log L)^{k-1}}{L^{\beta\wedge1}}&:&\beta\ne1,\\
\vspace{-0.3cm}&&\\
\big(\lambda_{k+1}(\Pc)|\!\log\lambda(\Pc)|^k\big)\wedge \frac{(\log L)^k}{L}&:&\beta=1.
\end{array}\right.
\end{equation}
By definition, given $\pi\notin\Kc(B_1,\ldots,B_r)$, we can consider the largest separating index $\alpha\le r$ such that each cell $H\in\pi$ is included either in $\{0\}\cup\bigcup_{i=1}^{\alpha-1} \langle B_i\rangle$ or in $\bigcup_{i=\alpha}^r\langle B_i\rangle$.
Setting $\hat B:=B_\alpha\uplus\ldots\uplus B_{r}$,
the choice of $\alpha$ ensures that $\pi$ can be restricted to a partition~$\hat\pi$ of the index subset $\langle\hat B\rangle\subset[k]$ such that $\hat\pi$ is covering for $B_\alpha,\ldots,B_r$.
Arguing as in~\eqref{eq:pre-bnd-Dpi} and estimating the integrals over the first blocks $\{0\},B_1,\ldots,B_{\alpha-1}$ brutally as in Section~\ref{sec:non-unif} without taking any advantage of the decay of correlation functions, we get
\begin{multline*}
|\Cc^{k+1;r}_L(B_1,\ldots,B_r;\pi)|\,\lesssim\,(\log L)^\alpha\bigg(\prod_{H\in\pi\setminus\hat\pi}\lambda_{\sharp H}(\Pc)\bigg) \\
\times L^{-d}\int_{Q_L}\bigg(\int_{(Q_L+x_b)^{\sharp\hat B-1}\setminus(Q_L)^{\sharp\hat B-1}}\bigg(\prod_{H\in\hat \pi}|h_{\sharp H}(x_H)|\bigg) D_{\hat B}(x_{\hat B})\,
dx_{\langle\hat B\rangle\setminus\{b\}}\bigg)dx_b
\end{multline*}
It remains to show that the remaining integral is a boundary term that is algebraically small as $L\uparrow\infty$, so that in particular the logarithmic prefactor $(\log L)^\alpha$ plays no role.
For that purpose, we first note that
\[\mathds1_{(Q_L+x_b)^{\sharp\hat B-1}\setminus(Q_L)^{\sharp\hat B-1}}(x_{\langle\hat B\rangle\setminus\{b\}})\le\sum_{j\in\langle\hat B\rangle\setminus\{b\}}\mathds1_{(Q_L+x_b)\setminus Q_L}(x_j),\]
so the above can be bounded by
\begin{multline*}
|\Cc^{k+1;r}_L(B_1,\ldots,B_r;\pi)|\,\lesssim\,(\log L)^\alpha \bigg(\prod_{H\in\pi\setminus\hat\pi}\lambda_{\sharp H}\bigg)\sum_{j\in\langle\hat B\rangle\setminus\{b\}}\\
\times L^{-d}\int_{Q_L}\int_{(Q_L+x_b)\setminus Q_L}\bigg(\int_{(Q_L+x_b)^{\sharp\hat B-2}}\bigg(\prod_{H\in\hat \pi}|h_{\sharp H}(x_H)|\bigg)
D_{\hat B}(x_{\hat B})
dx_{\langle\hat B\rangle\setminus\{b,j\}}\bigg)dx_jdx_b.
\end{multline*}
As $\hat \pi$ is covering for $B_\alpha,\ldots,B_r$, that is, $\hat\pi\in\Kc(B_\alpha,\ldots,B_r)$, similar arguments based on the graphical representation as in Substep~5.1 allow to estimate the integral over $\langle\hat B\rangle\setminus\{b,j\}$, to the effect of
\begin{multline}\label{eq:c-prebound-bdary}
|\Cc^{k+1;r}_L(B_1,\ldots,B_r;\pi)|\\
\,\lesssim\,(\log L)^\alpha
\sum_{j\in\langle\hat B\rangle\setminus\{b\}}L^{-d}\int_{Q_L}\int_{(Q_L+x_b)\setminus Q_L}\Big(\Lc^{\sharp\hat B-2}\lambda_{k+1}'\,
{\tiny\begin{tikzpicture}[baseline={([yshift=-.8ex]current bounding box.center)},scale=0.6]
\begin{scope}[every node/.style={circle,draw,fill=white}]
    \node (1) at (0,0) {$b$};
    \node (2) at (2,0) {$j$};
\end{scope}
\begin{scope}[>={Stealth[black]},
every edge/.style={draw=black,very thick}]
    \path [-] (1) edge[bend left=20] (2);
\end{scope}
\begin{scope}[>={Stealth[black]},
every edge/.style={draw=black,very thick,dashed}]
    \path [-] (1) edge[bend left=-20] (2);
\end{scope}
\end{tikzpicture}}\Big)
\,dx_jdx_b.
\end{multline}
In order to estimate this integral, we note that
\begin{eqnarray*}
\lefteqn{L^{-d}\int_{Q_L}\int_{(Q_L+x_b)\setminus Q_L}\Big(\lambda_{k+1}^\circ\,
{\tiny\begin{tikzpicture}[baseline={([yshift=-.8ex]current bounding box.center)},scale=0.6]
\begin{scope}[every node/.style={circle,draw,fill=white}]
    \node (1) at (0,0) {$b$};
    \node (2) at (2,0) {$j$};
\end{scope}
\begin{scope}[>={Stealth[black]},
every edge/.style={draw=black,very thick}]
    \path [-] (1) edge[bend left=20] (2);
\end{scope}
\begin{scope}[>={Stealth[black]},
every edge/.style={draw=black,very thick,dashed}]
    \path [-] (1) edge[bend left=-20] (2);
\end{scope}
\end{tikzpicture}}\Big)
\,dx_jdx_b}\\
&=&L^{-d}\int_{Q_L}\int_{(Q_L+x_b)\setminus Q_L}\langle(x_b-x_j)_L\rangle^{-d}\big(\omega((x_b-x_j)_L)\wedge\lambda_{k+1}(\Pc)\big)\,dx_jdx_b\\
&=&L^{-d}\int_{Q_L}\int_{Q_L\setminus (Q_L-x_b)}\langle y\rangle^{-d}\big(\omega(y)\wedge\lambda_{k+1}(\Pc)\big)\,dydx_b\\
&=&L^{-d}\int_{Q_L}\langle y\rangle^{-d}\big(\omega(y)\wedge\lambda_{k+1}(\Pc)\big)\,|Q_L\setminus (Q_L-y)|\,dy,
\end{eqnarray*}
and thus, using~\eqref{eq:QL+QL} in form of $L^{-d}|Q_L\setminus(Q_L-y)|\lesssim\tfrac{|y|}L\wedge1$, in case of an algebraic rate $\omega(t)\le Ct^{-\beta}$ for some $C,\beta>0$,
\begin{eqnarray*}
L^{-d}\int_{Q_L}\int_{(Q_L+x_b)\setminus Q_L}\!\!\Big(\lambda_{p}^\circ\,
{\tiny\begin{tikzpicture}[baseline={([yshift=-.8ex]current bounding box.center)},scale=0.6]
\begin{scope}[every node/.style={circle,draw,fill=white}]
    \node (1) at (0,0) {$b$};
    \node (2) at (2,0) {$j$};
\end{scope}
\begin{scope}[>={Stealth[black]},
every edge/.style={draw=black,very thick}]
    \path [-] (1) edge[bend left=20] (2);
\end{scope}
\begin{scope}[>={Stealth[black]},
every edge/.style={draw=black,very thick,dashed}]
    \path [-] (1) edge[bend left=-20] (2);
\end{scope}
\end{tikzpicture}}\Big)\,dx_jdx_b
&\lesssim&L^{-1}\int_{Q_L}\langle y\rangle^{1-d}\big(\omega(y)\wedge\lambda_{p}(\Pc)\big)\,dy\\
&\lesssim&\left\{\begin{array}{lll}
\lambda_{p}(\Pc)\wedge L^{-\beta\wedge1}&:&\beta\ne1,\\
\big(\lambda_{p}(\Pc)|\!\log\lambda(\Pc)|\big)\wedge \frac{\log L}{L}&:&\beta=1.
\end{array}\right.
\end{eqnarray*}
Now turning back to the right-hand side in~\eqref{eq:c-prebound-bdary}, repeating the above computation after including logarithmic factors, and noting that $\alpha+\sharp\hat B\le \sharp B=k+1$, the claim~\eqref{eq:boundary-terms-gen} follows.

\medskip
\step6 Strategy for~(ii) and~(iii).\\
Both for (ii) and (iii), the arguments are similar to what we already did so far, and require no new insight. 
We omit lengthy details for brevity.

\medskip\noindent
We start with~(ii). In view of the estimation~\eqref{eq:boundary-terms-gen} for boundary terms, it remains to estimate the convergence of terms corresponding to covering partitions in~\eqref{eq:pre-decomp-pi-CLkr} in the large-volume limit. For that purpose, we appeal to the periodization error estimates of Lemma~\ref{lem:period-est}, as in the proof of Proposition~\ref{prop:B2-ren}(ii).

\medskip\noindent
We turn to~(iii). The starting point is the refined estimate~\eqref{e.DGV-control-bis} on $R_L^{k+1}$. In the spirit of the proof of Proposition~\ref{prop:B2-ren}(iii), a decomposition of the right-hand side in~\eqref{e.DGV-control-bis} can be performed in the same way as what we did above for~$\Bb_L^{k+1}$.
\end{proof}

\subsection{Optimality of the bound on $\Bb^2$}\label{sec:opti-B2}

This section is devoted to the proof of Theorem~\ref{lem:optimality}, which shows that logarithmic factors are optimal in general in our estimation of cluster coefficients, e.g.~Proposition~\ref{prop:Bgen-ren}(i).
As will be clear in the proof below, logarithmic factors are related to the lack of continuity of the Helmholtz projection in $\Ld^\infty(\R^d)$.

\begin{proof}[Proof of Theorem~\ref{lem:optimality}]
Let Assumptions~\ref{H0}, \ref{Gunif}, and~\ref{B1} hold, and assume that the correlation function satisfies the Dini condition~\eqref{eq:decay-h2-om}.
We split the proof into two steps.

\medskip
\step1 Proof of~(i).\\
Appealing to Proposition~\ref{prop:B2-ren} in form
of the explicit formula~\eqref{eq:lim-form-B2} for $\Bb^2$, and estimating the second right-hand side term as in the proof of Proposition~\ref{prop:B2-ren}(i), we find
\begin{equation}\label{e.optima10-0}
\bigg|E:\Bb^2E-\int_{\R^d}\expecM{\int_{\partial I^\circ}\psi^z\cdot\sigma^0\nu}h_2(0,z)\,dz\bigg|\,
\lesssim\,\lambda_2(\Pc),
\end{equation}
where $\sigma^0$ and $\psi^z$ are associated respectively with single particles at~$I^\circ$ and at~$z+I^{\circ\prime}$, where~$I^\circ$ and~$I^{\circ\prime}$ are iid copies of the same random shape.
Replacing $\psi^z$ by its Taylor expansion, using the boundary conditions for $\sigma^0$, and using standard decay properties of~$\psi^z$, we find
\[\bigg|\int_{\partial I^\circ}\psi^z\cdot\sigma^0\nu-\D(\psi^z)(0):\int_{\partial I^\circ}\sigma^0\nu\otimes x\bigg|\,\lesssim\,\langle z\rangle^{-d-1}.\]
Inserting this into~\eqref{e.optima10-0} together with~\eqref{eq:high-intens-correl},
and recalling the short-hand notation
\[E:\Bh^1E\,=\,\tfrac12\expecM{\int_{\partial I^{\circ}} Ex\cdot\sigma^\circ \nu}\,=\,\expecM{\int_{\R^d}|\!\D(\psi^\circ)|^2},\]
cf.~\eqref{e.Einstein0}, we get
\begin{equation}\label{e.optima10-1}
\bigg|E:\Bb^2E-(2\Bh^1E):\Big(\int_{\R^d}\expec{\D(\psi^z)(0)}h_2(0,z)\,dz\Big)\bigg|\,
\lesssim\,\lambda_2(\Pc).
\end{equation}
Next, we further analyze $\D(\psi^z)(0)$.
In view of Lemma~\ref{lem:eqn-corH}, we note that $\psi^z$ satisfies in $\R^d$,
\[-\triangle\psi^z+\nabla(\Sigma^z\mathds1_{\R^d\setminus (z+I^{\circ\prime})})\,=\,-\delta_{\partial (z+I^{\circ\prime})}\sigma^z\nu.\]
In terms of the Stokeslet $G$ for the free Stokes equation, Green's representation formula then yields
\[\nabla_i\psi^z(0)\,=\,-\int_{\partial(z+I^{\circ\prime})}\nabla_iG(-\cdot)\,\sigma^z\nu.\]
Replacing $\nabla_iG$ by its Taylor expansion, using the boundary conditions for $\sigma^z$, and using standard decay properties of $G$, we find
\begin{equation*}
\bigg|\nabla_i\psi^z(0)-\nabla_{ij}^2G(-z)\int_{\partial(z+I^{\circ\prime})}(\cdot-z)_j\,\sigma^z\nu\bigg|\,\lesssim\,\langle z\rangle^{-d-1},
\end{equation*}
and therefore, taking the expectation, noting that $\sigma^z=\sigma^0(\cdot-z)$, and recognizing $\Bh^1E$ again,
\begin{equation*}
\Big|\expec{\nabla_i\psi^z_k(0)}-(2\Bh^1E)_{lj}\nabla_{ij}^2G_{kl}(-z)\Big|\,\lesssim\,\langle z\rangle^{-d-1}.
\end{equation*}
Inserting this into~\eqref{e.optima10-1} together with~\eqref{eq:high-intens-correl} again, we get
\begin{equation}\label{e.optima10-2}
\bigg|E:\Bb^2E-(2\Bh^1E)_{lj}(2\Bh^1E)_{ki}\Big(\pv\int_{\R^d}\nabla_{ij}^2G_{kl}(z)\,h_2(0,z)\,dz\Big)\bigg|\,
\lesssim\,\lambda_2(\Pc),
\end{equation}
where the notation $\pv$ stands for the principal value.
%
It remains to analyze the integral term in the left-hand side.
As $h_2$ satisfies the Dini condition~\eqref{eq:decay-h2-om}, this integral is absolutely summable.
Further assuming that the point process $\Pc$ is statistically isotropic, the correlation function $h_2(0,\cdot)$ is radial. By symmetry, this entails $\pv\int_{\R^d}\nabla_{ij}^2G_{kl}(z)\,h_2(0,z)\,dz=0$, and the conclusion~(i) follows.

\medskip
\step2 Proof of~(ii).\\
In view of~\eqref{e.optima10-2}, as $2\Bh^1E$ does not vanish, it suffices to construct a point process~$\Pc$ that satisfies Assumptions~\ref{H0} and~\ref{Gunif}, has decay of correlations~\eqref{eq:decay-h2-om} with algebraic rate~$\omega(t)\le Ct^{-\beta}$ for some $C,\beta>0$,
and satisfies the local independence condition $\lambda_2(\Pc)\simeq\lambda(\Pc)^2\ll1$, such that the integral term in the left-hand side of~\eqref{e.optima10-2} satisfies
\begin{equation}\label{eq:to-prove-constrh2}
\bigg|\pv\int_{\R^d}\nabla^2G(z)\,h_{2}(0,z)\,dz\bigg|\,\gtrsim\,\lambda(\Pc)^2|\!\log\lambda(\Pc)|,
\end{equation}
with the logarithmic factor.
We shall consider spherical particles, $I^\circ=B$, and we start with the construction of the correlation function $h_2$.
%

\medskip\noindent
For that purpose, first note that we can find a smooth bounded function $g:\mathbb S^{d-1}\to[0,1]$ such that
\begin{equation}\label{eq:choice-g}
\bigg|\int_{\partial B}\nabla^2G(e)\,g(e)\,d\theta(e)\bigg|\,\gtrsim\,1,
\end{equation}
where $d\theta$ stands for the Lebesgue measure on $\partial B$,
and we then define
\[h(z)\,:=\,\frac{g(\tfrac{z}{|z|})}{1+\lambda^2|z|^{2d+1}}.\]
Using~\eqref{eq:choice-g}, a computation in spherical coordinates yields
\begin{eqnarray*}
\lefteqn{\bigg|\int_{|z|>2(1+\rho)}\nabla^2G(z)\,h(z)\,dz\bigg|}\\
&=&\bigg(\int_{2(1+\rho)}^\infty r^{-1}(1+\lambda^2r^{2d+1})^{-1}dr\bigg)\bigg|\int_{\partial B}\nabla^2G(e)\,g(e)\,d\theta(e)\bigg|\\
&\gtrsim&|\!\log\lambda|,
\end{eqnarray*}
which proves~\eqref{eq:to-prove-constrh2} if the point process is chosen with intensity $\lambda(\Pc)=\lambda$ and with two-point correlation function $h_2$ given by
\[h_2(x,y)+\lambda^2\,:=\,\lambda^2(h(x-y)+1)\mathds1_{|x-y|>2(1+\rho)}.\]
In particular, this choice also yields
\begin{equation*}
h_2(0,z)\mathds1_{|z|>2(1+\rho)}\ge0,\qquad\sup_z\int_{Q(z)}|h_2(0,\cdot)|\,\lesssim\,\lambda^2,\qquad |h_2(0,z)|\,\lesssim\,\langle z\rangle^{-2d-1}.
\end{equation*}
It remains to prove that this choice of $h_2$ can be realized as the correlation function of a point process with intensity $\lambda(\Pc)=\lambda$ and satisfying~\ref{H0} and~\ref{Gunif}: this is precisely the subject of Proposition~\ref{lem:real-pointproc} below.
\end{proof}

The construction of a point process with given intensity and given two-point density function is easily done under suitable positivity conditions, e.g.\@ following~\cite{KLS-11}. In the present setting, more care is needed to further ensure stationarity and ergodicity of the constructed point process. Note that we use here a sufficient positivity condition that is much stronger than the one in~\cite{KLS-11}, but is easier to handle and suffices for our purposes.

\begin{prop}[Realizability of point processes]\label{lem:real-pointproc}
Let $\lambda,\rho>0$ and let $h\in\Ld^\infty(\R^d)$ be nonnegative with $h(x)\to0$ uniformly as \mbox{$|x|\uparrow\infty$}.
Then, there exists a strongly mixing stationary point process $\Pc=\{x_n\}_n$ on~$\R^d$ with intensity $\lambda$ and two-point density
\begin{equation}\label{eq:defin-f2-choice-exist}
f_2(x,y):=\lambda^2(h(x-y)+1)\mathds1_{|x-y|>2(1+\rho)},
\end{equation}
such that $|x_n-x_m|\ge2(1+\rho)$ almost surely for all~$n\ne m$.
\end{prop}

\begin{proof}
Let $\calM_\rho$ denote the set of locally finite point sets $\{z_n\}_n$ with $|z_n-z_m|\ge2(1+\rho)$ for all~\mbox{$n\ne m$}. It is easily checked that~$\calM_\rho$ is compact for the topology of convergence of point sets restricted to compact domains (this coincides with the vague topology when viewing point sets $\{z_n\}_n$ as measures $\sum_n\delta_{z_n}$).
Consider the space $V:=C(\calM_\rho)$, and denote by~$V_0$ the dense vector subset of polynomials with continuous coefficients on $\calM_\rho$, that is, the subset of functions $P^N:\calM_\rho\to\R$ of the form
\begin{equation}\label{eq:PNpol}
P^N(\{z_n\}_n)\,=\,P_0^N+\sum_{k=1}^N\sum_{n_1,\ldots,n_k}^{\ne}P_k^N(z_{n_1},\ldots,z_{n_k}),
\end{equation}
with $P_0^N\in\R$ and $P_k^N\in C_c(S_\rho^k)$ for $1\le k\le N$,
where we use the short-hand notation
\[S_\rho^k\,:=\,\big\{(z_1,\ldots,z_k)\in(\R^d)^k:|z_n-z_m|\ge2(1+\rho)\text{ for all $n\ne m$}\big\}.\]
In order to construct a point process with the two-point density $f_2$ given by~\eqref{eq:defin-f2-choice-exist}, we shall further prescribe all its multi-point density functions. For convenience, these are chosen in form of Mayer cluster expansions with vanishing higher-order correlations: for all $k\ge1$,
\begin{multline}\label{eq:def-fk}
f_k(z_1,\ldots,z_k)\,:=\,\lambda^k\,\mathds1_{S_\rho^k}(z_1,\ldots,z_k)\\
\times\bigg(1+\sum_{j=1}^{k/2}\frac1{2^jj!}\sum_{1\le n_1,\ldots,n_{2j}\le k}^{\ne}h(z_{n_1}-z_{n_2})\ldots h(z_{n_{2\ell-1}}-z_{n_{2j}})\bigg).
\end{multline}
Next, we define in these terms a linear map $L:V_0\to\R$ as follows: for any polynomial $P^N$ of the form~\eqref{eq:PNpol}, we set
\begin{equation}\label{eq:form-gP}
L(P^N)\,:=\,P_0^N+\sum_{k=1}^N\int_{S_\rho^k}P_k^Nf_k.
\end{equation}
We argue that $L$ is a positive linear functional on $V_0$, hence it is also continuous on $V_0$ with respect to the topology of $V$.
Indeed, for any polynomial $P^N$ of the form~\eqref{eq:PNpol} with~$P^N\ge0$ pointwise, if we evaluate it at the points of a Poisson point process with intensity $\lambda$, and if we compute the expectation, we find
\[P_0^N+\sum_{k=1}^N\lambda^k\int_{S_\rho^k}P_k^N\,\ge\,0,\]
hence, noting that the positivity of $h$ entails $f_k\ge\lambda^k\mathds1_{S_\rho^k}$ for all~$k\ge1$, we get
\[L(P^N)\,\ge\,P_0^N+\sum_{k=1}^N\lambda^k\int_{S_\rho^k}P_k^N\,\ge\,0,\]
thus proving the claimed positivity.

\medskip\noindent
As $V_0$ is dense in $V$, we can extend $L$ uniquely into a positive linear functional \mbox{$L:V\to\R$}.
Next, by the Riesz--Markov--Kakutani representation theorem, there exists a random element in~$\calM_\rho$, that is, a random point process $\Pc=\{x_n\}_n$, such that
\[\expec{P(\Pc)}\,=\,L(P)\qquad\text{for all $P\in V$}.\]
Testing this relation with polynomials, and using~\eqref{eq:form-gP}, we deduce that for all $k\ge1$ the $k$-point density function of the point process~$\Pc$ coincides with~$f_k$. In particular, it has intensity $f_1=\lambda$ and two-point density $f_2(x,y)=\lambda^2(h(x-y)+1)\mathds1_{|x-y|>2(1+\rho)}$ as desired.
In addition, $L$ is translation-invariant by definition, hence $\Pc$ is stationary.

\medskip\noindent
It remains to check that $\Pc$ is strongly mixing.
For that purpose, we compute the covariance of $\sigma(\Pc)$-measurable random variables.
Choose a polynomial $P^N$ of the form~\eqref{eq:PNpol}, and let $R>0$ be such that $P_k^N$ is supported in $(B_R)^k$ for all $1\le k\le N$.
For $|x|>2R$, as we have $B_R\cap(x+B_R)=\varnothing$, we can compute
\begin{multline*}
\cov{P^N(\Pc+x)}{P^N(\Pc)}\\
\,=\,(P_0^N)^2
+\sum_{k\ge1}\sum_{j=0}^k\int_{S_\rho^k}\big(P_j^N(\cdot+x,\ldots,\cdot+x)\otimes P_{k-j}^N\big)\, \big(f_k-f_j\otimes f_{k-j}\big).
\end{multline*}
The definition~\eqref{eq:def-fk} of $f_k$ easily leads to
\begin{multline*}
\bigg|\int_{S_\rho^k}\big(P_j^N(\cdot+x,\ldots,\cdot+x)\otimes P_{k-j}^N\big)\, \big(f_k-f_j\otimes f_{k-j}\big)\bigg|\\
\,\lesssim\,\Big(\sup_{z\in B_{2R}}h(x+z)\Big)(2\lambda)^k\big(1+\|h\|_{\Ld^\infty(\R^d)}\big)^{\frac k2-1}\|P_j^N\|_{\Ld^1((\R^d)^j)}\|P_{k-j}^N\|_{\Ld^1((\R^d)^{k-j})}.
\end{multline*}
As by assumption $h(x)\to0$ uniformly as $|x|\uparrow\infty$, we get $\cov{P^N(\Pc+x)}{P^N(\Pc)}\to0$. By a density argument, the same holds if $P^N$ is replaced by any element of $V$, which proves that $\Pc$ is strongly mixing.
\end{proof}


\newpage
\section{Conclusion: summing up the main results}\label{sec:gen-dil}

In this last section, we recall, reformulate, and comment on our main findings on the validity of Einstein's formula and of higher-order cluster expansions for the effective viscosity, as obtained in Sections~\ref{sec:Einsteinx3},~\ref{sec:cluster}, and~\ref{sec:intermezzo}.

\subsection{Cluster expansion of the effective viscosity in the model-free setting}\label{sec:cluster-recap}
We start with the validity of Einstein's formula and the associated error estimates as proved in Section~\ref{sec:Einsteinx3}, cf.~Theorem~\ref{theor:Einsteinx3}.
The three important features of this result are the generality in terms of probabilistic assumptions (mere qualitative ergodicity 
under Assumption~\ref{Gunif}), the sharpness of the error estimate~\eqref{e.def-bd-Einstein-recap}, and the possibility for particles to touch (under Assumption~\ref{Gmom} or~\ref{Gperc}).
\begin{theor1}[Einstein's formula]\label{theor:Einsteinx3-recap}
Under Assumption~\ref{H0} and either Assumption~\ref{Gunif}, \ref{Gmom}, or~\ref{Gperc}, for some $\rho>0$ and $\kappa>1$,
we have
\begin{align}\label{e.def-bd-Einstein-recap}
&|\Bb-(\Id+\Bb^1)|\,\lesssim_{\rho}\,\lambda_2(\Pc) \log\big(2+\tfrac{\lambda(\Pc)}{\lambda_2(\Pc)}\big)\\
&\hspace{4cm}+\left\{
\begin{array}{lll}
0&:&\text{in case of~\ref{Gunif}},\\
\Kc_\kappa\, \lambda_2(\Pc)^{1-\frac1\kappa}\lambda(\Pc)^{\frac1\kappa}&:& \text{in case of~\ref{Gmom} or~\ref{Gperc}},
\end{array}
\right.\nonumber
\end{align}
where $\Bb^1$ satisfies
\[|\Bb^1|\simeq \lambda(\Pc),\]
and is defined for all $E\in\Md_0^\Sym$ by
\begin{equation}\label{eq:def-B1-recap}
E:\Bb^1E\,:=\,\sum_n\expecM{\frac{\mathds1_{0\in I_n}}{|I_n|}\int_{\R^d}|\!\D(\psi_E^{\{n\}})|^2},
\end{equation}
where $\psi_E^{\{n\}}$
is the unique decaying solution of the single-particle problem~\eqref{eq:singlepart}.
In particular, the estimate $|\Bb-(\Id+\Bb^1)|=o(\lambda(\Pc))$ holds provided the point process $\Pc$ satisfies the weak local independence condition $\lambda_2(\Pc)=o({\lambda(\Pc)}/{|\!\log\lambda(\Pc)|})$.
\end{theor1}

In order to address the optimality of this estimate, one needs to identify the next term in the expansion.
In Sections~\ref{sec:cluster} and~\ref{sec:intermezzo}, we have further investigated higher-order expansions of the effective viscosity in form of cluster expansions.
The upcoming result, which summarizes Theorems~\ref{prop:quant} and~\ref{th:renorm-B23} in Section~\ref{sec:intermezzo}, gives the optimal order of magnitude of the cluster coefficients and of the remainder. The two important features of this result are the generality of the point processes (to be compared with results in Section~\ref{sec:Dilat/Delet} below) and the sharpness of the estimates. The main achievement is the explicit understanding of the needed renormalizations to all orders,  solving a problem that was still open in the physics community.

\begin{theor1}[Cluster expansion in general dilute setting]\label{th:mod-free1}
On top of Assumptions~\ref{H0} and~\ref{Gunif},
assume that the inclusion process is $\alpha$-mixing in the sense of~\emph{\ref{Mix}} with algebraic rate $\omega$.
Then, for all $k\ge1$, the following holds for the effective viscosity,
\begingroup\allowdisplaybreaks
\begin{gather*}
\bigg|\Bb\,-\,\Id-\sum_{j=1}^k \frac{1}{j!}\Bb^{j}\bigg|\,\lesssim_k\,\sum_{l=k+1}^{2k+1}\lambda_l(\Pc)|\!\log\lambda(\Pc)|^{l-1},\\
|\Bb^j|\,\lesssim_j\,\lambda_j(\Pc)|\!\log\lambda(\Pc)|^{j-1},\quad\text{for all $1\le j\le k$},
\end{gather*}
\endgroup
where the cluster coefficients $\{\Bb^j\}_j$ are defined in~\eqref{eq:form-Bj} by means of finite-volume approximations. If in addition the independence assumption~\emph{\ref{B1}} holds for particle shapes, renormalized formulas can be given for cluster coefficients in form of absolutely convergent multiple integrals, cf.~Section~\ref{sec:explicit}, and the following quantitative convergence result holds for finite-volume approximations $\{\Bb^j_L\}_j$: in case of an algebraic $\alpha$-mixing rate $\omega(t)\le Ct^{-\beta}$ for some $C,\beta>0$,
\[|\Bb_L^j-\Bb^j|\,\lesssim_j\,\tfrac{(\log L)^{j-1}}{L^{\beta\wedge1}}.\qedhere\]
\end{theor1}

Note that the bound on $\Bb_2$ in Theorem~\ref{th:mod-free1} coincides with the estimate on the remainder in Theorem~\ref{theor:Einsteinx3-recap}, which contrasts with the results of Lemma~\ref{lem:short} in the short-range setting by a logarithmic correction.
Optimality of the latter is addressed in Theorem~\ref{lem:optimality}, which we presently recall.

\begin{theor1}[Optimality of estimates on $\Bb_2$]\label{lem:optimality-recap}$  $
\begin{enumerate}[(i)]
\item \emph{Isotropic setting:}
On top of Assumptions~\ref{H0}, \ref{Gunif}, and~\emph{\ref{B1}}, assume that the $2$-point correlation function $h_2(x,y):=f_2(x,y)-\lambda(\Pc)^2$ satisfies the following decay assumption,
\begin{equation*}
\quad\iint_{B(x)\times B(y)}|h_2|\,\le\,\omega(|x-y|),
\end{equation*}
with some rate $\omega$ satisfying the Dini condition $\int_1^\infty t^{-1}\omega(t)\,dt<\infty$.
If in addition the point process~$\Pc$ is statistically isotropic, which entails that the correlation function is radial, then the following improved estimate holds,
\[\quad|\Bb^2|\,\lesssim\, \lambda_2(\Pc).\]
\item \emph{Optimality in the general setting:} There exists an inclusion process $\Ic$ that satisfies Assumptions~\ref{H0}, \ref{Gunif}, \emph{\ref{B1}}, and~\eqref{eq:decay-h2-om}, as well as the local independence condition $\lambda_2(\Pc)\simeq\lambda(\Pc)^2\ll1$, such that we have
\[\quad|\Bb^2|\,\simeq\, \lambda_2(\Pc) |\!\log \lambda_2(\Pc)|.\qedhere\]
\end{enumerate}
\end{theor1}

Based on the explicit renormalization of higher-order cluster coefficients, it appears that Theorem~\ref{lem:optimality-recap}(ii) readily extends to higher orders, demonstrating the optimality of cluster estimates in Theorem~\ref{th:mod-free1}.

\medskip

To conclude this section, let us apply and confront Theorems~\ref{theor:Einsteinx3-recap},~\ref{th:mod-free1}, and~\ref{lem:optimality-recap} to some specific families of inclusion processes displaying multi-point intensities with different scaling laws. 
We start with the construction.

\smallskip\noindent
$\bullet$ \textit{Construction of inclusion processes $\{\Ic_{\beta,\lambda}\}_{\beta,\lambda}$:}
We define a family of point processes $\{\Pc_{\beta,\lambda}\}_{\beta,\lambda}$ with parameters $0\le \beta \le 1$ and $0<\lambda\ll1$ as follows.
Consider a hardcore Poisson process $\Pc'=\{x_n'\}_n$ with radius $6$ and with intensity $\lambda(\Pc')=\lambda$, see e.g.~\cite[Section~3.4]{DG2} using Penrose's graphical construction~\cite{Penrose-01}.
Next, independently choose a sequence $\{y_n\}_n$ of iid random points that are uniformly distributed in $B_4\setminus B_3$, and, given $\beta\in[0,1]$, also independently choose a sequence $\{b_{n,\beta}\}_n$ of iid Bernoulli variables with parameter $\lambda^\beta=\pr{b_{n,\beta}=1}$. 
The desired point processes and spherical inclusion processes are then defined by
\[\Pc_{\beta,\lambda}\,:=\,\Pc'\,\cup\,\big\{x_n'+y_n:b_{n,\beta}=1\big\},\qquad \Ic_{\beta,\lambda}:=\textstyle\bigcup_{x\in\Pc_{\beta,\lambda}}B(x).\]

\noindent
$\bullet$ \textit{Properties of the processes:}
$\Ic_{\beta,\lambda}$ satisfies~\ref{H0} and~\ref{Gunif} (with \mbox{$\rho=1$}) as well as~\ref{B1}.
In addition, the point process $\Pc_{\beta,\lambda}$ is statistically isotropic and $\alpha$-mixing with exponential rate uniformly with respect to $\beta,\lambda$ (e.g.~\cite[Proposition~1.4 (iii)]{DG1} and~\cite[Proposition~3.5]{DG2}), so that Theorem~\ref{th:mod-free1} applies.
A direct computation shows that the multi-point intensities scale as follows,
\begin{equation*}
\lambda(\Pc_{\beta,\lambda})\,\simeq\,\lambda ,
\qquad
\lambda_2(\Pc_{\beta,\lambda})\,\simeq \,\lambda^{1+\beta},
\qquad
\lambda_3(\Pc_{\beta,\lambda}) \, \simeq \,\lambda^{2+\beta},
\end{equation*}
and more generally $\lambda_{2k}(\Pc_{\beta,\lambda})\simeq_k\lambda^{k(1+\beta)}$ and $\lambda_{2k+1}(\Pc_{\beta,\lambda})\simeq_k\lambda^{1+k(1+\beta)}$.
In particular  the minimal local independence condition $\lambda_3(\Pc_{\beta,\lambda})\ll\lambda_2(\Pc_{\beta,\lambda})\ll\lambda(\Pc_{\beta,\lambda})$ holds for $\beta>0$.

\medskip\noindent
$\bullet$ \textit{Second-order cluster expansion:}
We denote by $\Bb_{\beta,\lambda}$ the effective viscosity associated with~$\Ic_{\beta,\lambda}$.
Theorem~\ref{th:mod-free1}  implies that 
\[\big|\Bb_{\beta,\lambda}-\big(\Id+\Bb^{1}_{\beta,\lambda}+\tfrac12\Bb^{2}_{\beta,\lambda}\big)\big| \,\lesssim\, \lambda^{2+\beta}|\!\log \lambda|^2,\]
where $|\Bb^1_{\beta,\lambda}|\simeq \lambda$ and $|\Bb^2_{\beta,\lambda}| \simeq  \lambda^{1+\beta}$ (cf.~\eqref{e.Einstein0} and Theorem~\ref{lem:optimality-recap}(i)).
In particular, discarding~$\Bb^{2}_{\beta,\lambda}$ in the above yields
the following (completely new) sharp error estimate for Einstein's formula in this setting: for all~$0\le \beta \le 1$ and~$\lambda\ll1$,
\[
| \Bb_{\beta,\lambda}-(\Id+\Bb^{1}_{\beta,\lambda})|\,\simeq\,\lambda^{1+\beta}\,\simeq\,|\Bb^1_{\beta,\lambda}|^{1+\beta}.
\]
In this example, Einstein's formula is thus accurate whenever $\beta>0$, which illustrates the full range of the local independence condition $\lambda_2(\Pc_{\beta,\lambda})=o({\lambda(\Pc_{\beta,\lambda})}/{|\!\log\lambda(\Pc_{\beta,\lambda})|})$ in
Theorem~\ref{theor:Einsteinx3-recap}.

\subsection{Summability of the cluster expansions for specific dilution procedures}\label{sec:Dilat/Delet}
Next, we consider the following two specific one-parameter dilution procedures, for which our results can be substantially strengthened using the uniform $\ell^1-\ell^2$ energy estimates of Theorem~\ref{prop:apest-re}: more precisely, logarithmic corrections in cluster estimates can be removed in that case and the full cluster expansion is absolutely converging.
\begin{enumerate}[\hspace{1cm}(1)]
\renewcommand{\labelenumi}{\theenumi}
\renewcommand{\theenumi}{\textbf{(Dilat)}}
\item\label{Ddilat} \emph{Dilution by geometric dilation:}
Given a point process $\Pc=\{x_n\}_n$ and random inclusions $I_n=x_n+I_n^\circ$ satisfying~\ref{H0},
we consider the dilated process~\mbox{$\Pc_s=\{s x_n\}_n$} and the corresponding inclusions $I_{n,s}=s x_n+I_n^\circ$. The latter has minimal distance $\ell(\Pc_s)=s\ell(\Pc)\simeq s$, still satisfies~\ref{H0}, and further satisfies~\ref{Gunif} with minimal interparticle distance
\begin{equation}\label{eq:ell-separation}
\qquad\qquad\inf_{n\ne m}\dist(I_{n,s},I_{m,s})\,\ge\,\inf_{n\ne m}|sx_n-sx_m|-2\,\ge\,s\ell(\Pc)-2\,\gtrsim\,s,
\end{equation}
provided $s\gg1$.
Its multi-point intensities take the form
\[\qquad\qquad\lambda_j(\Pc_s)=s^{-jd}\lambda_j(\Pc)\quad\text{for all $j\ge1$}.\]

\renewcommand{\labelenumi}{\theenumi}
\renewcommand{\theenumi}{\textbf{(Delet)}}
\item\label{Ddelet} \emph{Dilution by random deletion:}
Given a point process $\Pc=\{x_n\}_n$ and random inclusions $I_n=x_n+I_n^\circ$
satisfying~\ref{H0} and~\ref{Gunif}, the Bernoulli deletion scheme consists in keeping each inclusion only with given probability~$p\in[0,1]$.
More precisely, we attach to the inclusions iid Bernoulli variables~$\{b^{(p)}_n\}_n$, independent of $\Pc,\Ic$, with parameter
\[\qquad\qquad p=\prm{b_n^{(p)}=1},\]
and we define the corresponding decimated process
\[\qquad\qquad\Pc^{(p)}:=\{x_n\}_{n\in \Nb^{(p)}},\quad \Ic^{(p)}:=\textstyle\bigcup_{n\in \Nb^{(p)}}I_n,\quad\text{where $\Nb^{(p)}:=\{n:b_n^{(p)}=1\}$}.\]
This decimated process still satisfies~\ref{H0} and~\ref{Gunif},
and its multi-point intensities are given by
\[\qquad\qquad\lambda_j(\Pc^{(p)})=p^{j}\lambda_j(\Pc)\quad\text{for all $j\ge1$}.\]
\end{enumerate}
In these one-parameter settings, dilute expansions of the effective viscosity amount to expansions with respect to the dilution parameters $s^{-1}$ or $p$.
Given a random set of particles $\Ic=\bigcup_nI_n$ centered at the points of $\Pc=\{x_n\}_n$, we shall consider both dilution procedures at once, defining the dilated decimated process
\[\Pc_s^{(p)}:=\{x_{n,s}\}_{n\in N^{(p)}},\qquad\textstyle\Ic_s^{(p)}:=\bigcup_{n\in N^{(p)}}I_{n,s},\qquad x_{n,s}:=s x_n, \qquad I_{n,s}:=x_{n,s}+I_n^\circ.\]
As a consequence of Theorem~\ref{thm:bounds}, together with~\eqref{eq:BLpj-porder}--\eqref{eq:bnd-RLpk} in Section~\ref{sec:proof-eq:form-Bj}, we obtain the following summability result and estimates for the cluster expansion of the effective viscosity $\Bb_s^{(p)}$ associated with $\Ic_s^{(p)}$.
In particular, it shows that the scaling of cluster coefficients coincides in this case with that of Lemma~\ref{lem:short} for the short-range setting: indeed, we have $|\Bb^{(p),j}_s|=p^j|\Bb_s^j|\lesssim_j(ps^{-d})^j\simeq\lambda_j (\Pc^{(p)}_s)$. We emphasize that no mixing assumption is required here.

\begin{theor1}[Cluster expansion for one-parameter dilution procedures]\label{th:analytic}
Under Assumptions~\ref{H0} and~\ref{Gunif}, for the specific dilution models~\emph{\ref{Ddilat}} and~\emph{\ref{Ddelet}} above, with dilation parameter $s$ and Bernoulli parameter $p$, the cluster expansion of the effective viscosity is uniformly summable in the following sense:
there exists a constant $C$ (only depending on~$d,\rho$) such that for all $0\le p s^{-d}<\frac1{C}$ the effective viscosity satisfies
\begin{equation}\label{eq:summable-exp0}
\Bb_{s}^{(p)}\,=\,\Id+\sum_{j=1}^\infty \tfrac{p^j}{j!}\Bb_{s}^{j},\qquad\qquad
|\Bb_{s}^j|\,\le\,j!\,(Cs^{-d})^j\quad\text{for all $j\ge1$},
\end{equation}
where the cluster coefficients $\{\Bb_{s}^j\}_j$ are defined in~\eqref{eq:form-Bj} by means of finite-volume approximations.
\end{theor1}

\begin{rems}[Analyticity with respect to dilution parameters]$ $\label{rem:dilat/delet-anal}
\begin{enumerate}[(a)]
\item
In case of the random deletion model~\ref{Ddelet}, the expansion~\eqref{eq:summable-exp0} yields the local analyticity of $p\mapsto\Bb^{(p)}$ at $p=0$.
Local analyticity can, in fact, be established on the whole interval $0\le p\le1$; the reader is referred to~\cite{DG-16a} for a similar result in the scalar setting using an observation by Mourrat \cite{Mourrat-13}.
\smallskip\item
In case of the dilation model~\ref{Ddilat}, the expansion~\eqref{eq:summable-exp0} does not yield the analyticity of the map $s^{-d}\mapsto\Bb_s$ since the rescaled coefficients $\{s^{dj}\Bb_s^j\}_j$ also depend on $s$. 
By means of multipole expansions, the maps $s^{-1}\mapsto s^{dj}\Bb_s^j$ can be checked to be analytic themselves, as well as $s^{-1}\mapsto\Bb_s$. For a more direct approach to expansions in $s^{-1}$, we refer to the recent work~\cite{Pertinand} in the scalar setting; see also~\cite{Berdichevski-83}.
\qedhere
\end{enumerate}
\end{rems}

To illustrate Remark~\ref{rem:dilat/delet-anal}(b), we display the first term of the monopole expansion for the second-order coefficient $\Bb_s^2$.
In particular, as is natural, we note that $\Bb^2_s$ can be expressed to leading order in terms of the single-particle problem only, and it coincides with the formula obtained in~\cite[Proposition~5.6]{GVH} in case of spherical inclusions.

\begin{prop1}[Leading order of monopole expansion]\label{cor:B2ell-re}
On top of Assumptions~\ref{H0} and \ref{Gunif}, assume that particles have independent shapes, cf.~\emph{\ref{B1}}, and that
the two-point correlation function $h_{2}=f_{2}-\lambda(\Pc)^2$ satisfies the decay assumption
\begin{equation*}
\iint_{B(x)\times B(y)}|h_2|\,\le\,\omega(|x-y|),
\end{equation*}
with some rate $\omega$ satisfying the Dini condition $\int_1^\infty t^{-1}\omega(t)\,dt<\infty$.
Consider the dilated process $\Pc_s$, cf.~\emph{\ref{Ddilat}}, and the associated second-order cluster coefficient $\Bb^2_s$ defined in Theorem~\ref{th:analytic}.
Then, we have
\begin{equation*}
\big|\Bb^2_s-s^{-2d}\Bb^{2,1}\big|\,\lesssim\,s^{-2d-1},
\end{equation*}
and the leading-order contribution $\Bb^{2,1}$ is given by the following reduced formula,
\[E:\Bb^{2,1} E\,=\,(2\Bh^1E):\Big(\pv\int_{\R^d}\mathcal G(z)\,h_{2}(0,z)\,dz\Big)\,(2\Bh^1E),\]
where $\Bh^1E$ is defined in~\eqref{e.Einstein0},
where the notation $\pv$ stands for the principal value,
and where the $4$-tensor field $\Gc$ is given by $M:\Gc(z)M=M_{jk}M_{lm}\nabla^2_{km}G_{jl}(z)$
in terms of the standard Stokeslet
\[G(z)\,=\,\frac{|z|^{2-d}}{2(d-2)|\partial B|}\Big(\Id+(d-2)\frac{z\otimes z}{|z|^2}\Big).\]
In case of spherical particles, $I_n=B(x_n)$, we thus have
\begin{equation*}
E:\Bb^{2,1} E\,=\,(d+2)^2|B|~\pv\int_{\R^d}\bigg(\frac{d+2}2\frac{(z\cdot Ez)^2}{|z|^{d+4}}-\frac{|Ez|^2}{|z|^{d+2}}\bigg)\,h_{2}(0,z)\,dz.
\qedhere
\end{equation*}
\end{prop1}

\begin{proof}
Starting from the renormalized formula~\eqref{eq:lim-form-B2} in Proposition~\ref{prop:B2-ren}, repeating the proof of~\eqref{e.optima10-2} to decompose the first contribution, and using~\eqref{eq:est-psi-delsig-2d} to estimate the second one, we are led to
\begin{multline*}
\bigg|E:\Bb^2_sE-(2\Bh^1E)_{lj}(2\Bh^1E)_{ki}\Big(\pv\int_{\R^d}\nabla^2_{ij}G_{kl}(z)\,h_2(0,z)\,dz\Big)\bigg|\\
\,\lesssim\,\int_{\R^d}\langle z\rangle^{-d-1}|h_{2,s}(0,z)|\,dz+\int_{\R^d}\langle z\rangle^{-2d}f_{2,s}(0,z)\,dz.
\end{multline*}
Using that the two-point density and the correlation for the dilated process $\Pc_s$ take the form
\[f_{2,s}(0,z)=s^{-2d}f_2(0,s^{-1}z),\qquad h_{2,s}(0,z)=s^{-2d}h_2(0,s^{-1}z),\]
and changing variables, the conclusion follows by scaling. In case of spherical particles, we appeal to the proof of Proposition~\ref{prop:B1} for the explicit computation of $\Bh^1E$.
\end{proof}

Finally, we revisit a recent result by G\'erard-Varet~\cite{GV-20} that displays to second order similar estimates as for the random deletion procedure, cf.~\eqref{eq:summable-exp0}, but only assuming some specific structure of the multi-point densities up to order 5, thus contrasting with Theorem~\ref{th:analytic}.
As a corollary of Proposition~\ref{prop:Bgen-ren}, we establish the following result, which
constitutes an extension
of~\cite{GV-20} to higher orders with new, optimal error bounds. Note indeed that for $k=2$ the result~\eqref{e.DGV-revisited} below yields an error bound $O(p^3)$, which improves on the bound $O(p^\frac52)$ obtained in~\cite{GV-20}.
\begin{cor1}\label{cor:DVG}
Let $\Pc$ satisfy Assumptions~\ref{H0} and~\ref{Gunif}, and let $\Ic$ satisfy the independence assumption~\emph{\ref{B1}}.
Given $k\ge2$, assume that there exists $0<p\le 1$ such that the multi-point density functions of $\Pc$ 
can be written as $f_j = p^j f_j^\circ$ for all $1\le j \le 2k+1$, for some functions $(f_j^\circ)_{1\le j\le 2k+1}$.
Further define functions $(h_j^\circ)_{1\le j\le 2k+1}$ through the correlation/density relation~\eqref{e.inver-cor-func} starting from $(f_j^\circ)_{1\le j\le 2k+1}$ and assume that they satisfy~\emph{\ref{Mix-om-n}} to order $n=2k+1$ with algebraic rate $\omega$.
Then, we have
\begin{equation}\label{e.DGV-revisited}
\Big|\Bb-\Id-\sum_{j=1}^k \tfrac1{j!} \Bb^j\Big|\,\lesssim_k \, p^{k+1}, \qquad |\Bb^{j}|\lesssim_j p^j~~\text{for all $1\le j\le k$}.
\end{equation}
where the multiplicative constants are independent of $p$.
\end{cor1}
\begin{proof}
The assumption $f_j=p^jf^\circ_j$ entails $h_j=p^jh^\circ_j$, where $h^\circ_j$ is assumed to satisfy~{\ref{Mix-om-n}}.
Further writing $(f_j^\circ)_{j}$ in terms of $(h_j^\circ)_{j}$ by means of~\eqref{eq:dens-correl},
the assumption~{\ref{Mix-om-n}} for the latter yields
\[\lambda_j^\circ\,:=\,\sup_{z_1,\ldots,z_j}\fint_{Q(z_1)\times\ldots\times Q(z_j)} f_j^\circ\,\lesssim_j\,1,\]
where the bound only depends on $j,\omega$, and on the constant function $f_1^\circ$.
In this setting, the bounds of Proposition~\ref{prop:Bgen-ren}(i)--(iii) now take the form 
\begin{eqnarray*}
|\Bb^j| &\lesssim& p^j \bar \lambda_j^\circ |\log \lambda^\circ|^{j-1},\\
|R^{k+1}| &\lesssim&\sum_{j=k}^{2k+1} p^{j+1} \bar \lambda_{j+1}^\circ |\log \lambda^\circ|^{j} \lesssim p^{k+1},
\end{eqnarray*}
and the conclusion readily follows.
\end{proof}

\newpage\appendix
\section{Stokeslet estimates with rigid inclusions}\label{append:ell}

This appendix is dedicated to the proofs of several estimates on the behavior of the fluid velocity generated by a localized force dipole in the presence of a finite number of rigid inclusions. In other words, it concerns the Stokeslet for the Stokes problem with rigid inclusions, and we shall prove in particular Lemmas~\ref{lem:decay}, \ref{lem:decay-re}, and~\ref{lem:period-est}.

\subsection{Main results}
For convenience, we start by recalling the notation of Section~\ref{sec:Stokeslet-est-rig}.
Given a set $Y\subset Q_L$ of ``background'' positions with
\begin{equation}\label{*eq:dist-Y}
\dist(B(y),B(y'))\,>\,2\rho,\qquad\dist(B(y),\partial Q_L)\,>\,\rho,\qquad\text{for all $y,y'\in Y,~ y\ne y'$},
\end{equation}
we denote by $\psi_L^Y\in H^1_\per(Q_L)^d$ the solution of the following periodic corrector problem, using the short-hand notation $\sigma_L^Y:=\sigma(\psi_L^Y+Ex,\Sigma_L^Y)$,
\[\left\{\begin{array}{ll}
-\triangle \psi_L^Y+\nabla\Sigma_L^Y=0,&\text{in $ Q_L\setminus\cup_{y\in Y}B(y)$},\\
\Div(\psi_L^Y)=0,&\text{in $ Q_L\setminus\cup_{y\in Y}B(y)$},\\
\D(\psi_L^Y+Ex)=0,&\text{in $\cup_{y\in Y}B(y)$},\\
\int_{\partial B(y)}\sigma_L^Y\nu=0,&\forall y\in Y,\\
\int_{\partial B(y)}\Theta(x-y)\cdot\sigma_L^Y\nu=0,&\forall \Theta\in\Md^\Skew,~\forall y\in Y.
\end{array}\right.\]
Next, we turn to our notation for elementary single-particle contributions or so-called Stokeslets $\{\Jc_{L;Y}^z\}_{z,Y}$:
Given a ``tagged'' position~$z\in Q_L$, given $(\zeta,P)\in H^1(B_{1+\rho}(z))^d\times\Ld^2(B_{1+\rho}(z)\setminus B(z))$ satisfying the following Stokes equations in a neighborhood of $B(z)$,
\begin{equation}\label{*eq:zetaP-re}
\left\{\begin{array}{ll}
-\triangle\zeta+\nabla P=0,&\text{in $B_{1+\rho}(z)\setminus B(z)$},\\
\Div(\zeta)=0,&\text{in $B_{1+\rho}(z)\setminus B(z)$},\\
\D(\zeta)=0,&\text{in $B(z)$},\\
\int_{\partial B(z)}\sigma(\zeta,P)\nu=0,&\\
\int_{\partial B(z)}\Theta(x-z)\cdot\sigma(\zeta,P)\nu=0,&\forall\Theta\in\Md^\Skew,
\end{array}\right.
\end{equation}
and given a finite subset $Y\subset Q_L$ of ``background'' positions satisfying~\eqref{*eq:dist-Y},
we define $\Jc_{L;Y}^z\zeta\in H^1_\per(Q_L)^d$ as the solution of the following Stokes problem with force dipole localized around $z$ and rigid inclusions around points of $Y$,
\begin{equation}\label{*eq:def-JLy|z}
\left\{\begin{array}{ll}
-\triangle \Jc^{z}_{L;Y}\zeta+\nabla \Qc^{z}_{L;Y}\zeta=-\delta_{\partial B^L(z)}\sigma(\zeta,P)\nu,&\text{in $ Q_L\setminus \cup_{y\in Y}B(y)$},\\
\Div(\Jc^{z}_{L;Y}\zeta)=0,&\text{in $ Q_L\setminus \cup_{y\in Y}B(y)$},\\
\D(\Jc^{z}_{L;Y}\zeta)=0,&\text{in $\cup_{y\in Y\setminus Y^z}B(y)$},\\
\int_{\partial B(y)}\sigma(\Jc_{L;Y}^z\zeta,\Qc_{L;Y}^z\zeta)\nu=0,&\forall y\in Y\setminus Y_z,\\
\vspace{-0.4cm}&\\
\int_{\partial B(y)}\Theta (x-y)\cdot\sigma(\Jc_{L;Y}^z\zeta,\Qc_{L;Y}^z\zeta)\nu=0,&\forall\Theta\in\Md^\Skew,~\forall y\in Y\setminus Y_z,
\\
\vspace{-0.38cm}&\\
\Jc^{z}_{L;Y}\zeta=V_z+\Theta_z (x-z),&\text{in $\cup_{y\in  Y_z}B(y)$,}\\
&\quad\text{for some $V_z \in \R^d, \Theta_z \in\Md^\Skew$},
\\
\sum_{y\in  Y_z} \int_{\partial B(y)}\sigma(\Jc_{L;Y}^z\zeta,\Qc_{L;Y}^z\zeta)\nu
\\
\hspace{1cm}=\sum_{y\in Y_z}\int_{B(y)\cap\partial B^L(z)}\sigma(\zeta,P)\nu,&
\\
\vspace{-0.38cm}&\\
\sum_{y\in  Y_z} \int_{\partial B(y)}\Theta (x-z)\cdot\sigma(\Jc_{L;Y}^z\zeta,\Qc_{L;Y}^z\zeta)\nu&\\
\hspace{1cm}=\sum_{y\in Y_z}\int_{B(y)\cap\partial B^L(z)}\Theta(x-z)\cdot\sigma(\zeta,P)\nu,&\forall\Theta\in\Md^\Skew,
\end{array}\right.
\end{equation}
where we recall that $B^L(z)=(B(z)+L\Z^d)\cap Q_L$ stands for the periodization of the ball~$B(z)$ in~$Q_L$, where we have set $Y_z:=\{y \in Y:B(y)\cap B^L(z)\ne \varnothing\}$, and where we have implicitly extended $(\zeta,P)$ periodically to $B_{1+\rho}(z)+L\Z^d$.
The solution $\Jc_{L;Y}^z\zeta$ is only defined up to a rigid motion in $Q_L$, which we fix by further choosing
\[\int_{ Q_L}\Jc_{L;Y}^z\zeta~=~0,\qquad\int_{Q_L}\nabla\Jc_{L;Y}^z\zeta~\in~\Md_0^\Sym.\]
Note that $\Jc_{L;Y}^z\zeta$ depends of course on the pair $(\zeta,P)$, not only on $\zeta$, but we leave the pressure field implicit in the notation for convenience.
We refer to Section~\ref{sec:Stokeslet-est-rig} for motivation of the above equations~\eqref{*eq:def-JLy|z}, and we recall that it reduces to the following simpler equations when $\{z\}\cup Y$ satisfies~\eqref{*eq:dist-Y} (meaning that $z$ neither gets close to background positions $Y$  nor to the cell boundary $\partial Q_L$),
\begin{equation}\label{*eq:def-JLy|z-simpl}
\left\{\begin{array}{ll}
-\triangle \Jc^{z}_{L;Y}\zeta+\nabla \Qc^{z}_{L;Y}\zeta=-\delta_{\partial B(z)}\sigma(\zeta,P)\nu,&\text{in $ Q_L\setminus \cup_{y\in Y}B(y)$},\\
\Div(\Jc^{z}_{L;Y}\zeta)=0,&\text{in $ Q_L\setminus \cup_{y\in Y}B(y)$},\\
\D(\Jc^{z}_{L;Y}\zeta)=0,&\text{in $\cup_{y\in Y}B(y)$},\\
\int_{\partial B(y)}\sigma(\Jc_{L;Y}^z\zeta,\Qc_{L;Y}^z\zeta)\nu=0,&\forall y\in Y,\\
\vspace{-0.4cm}&\\
\int_{\partial B(y)}\Theta (x-y)\cdot\sigma(\Jc_{L;Y}^z\zeta,\Qc_{L;Y}^z\zeta)\nu=0,&\forall\Theta\in\Md^\Skew,~\forall y\in Y.
\end{array}\right.
\end{equation}
We further define
\[\Jc_{L}^z\zeta\,:=\,\Jc_{L;\varnothing}^z\zeta,\]
for which the Stokes problem~\eqref{*eq:def-JLy|z} reduces to
\begin{equation}\label{*eq:def-JLy}
-\triangle \Jc^{z}_{L}\zeta+\nabla \Qc^{z}_{L}\zeta=-\delta_{\partial B^L(z)}\sigma(\zeta,P)\nu,
\qquad\Div(\Jc^{z}_{L}\zeta)=0,\qquad\text{in $ Q_L$},
\end{equation}
and we define $\Jc_{Y}^z\zeta,\Jc^z\zeta$ as the corresponding operators on whole space, that is, with $B^L(z)$ and $Q_L$ replaced by $B(z)$ and $\R^d$, respectively, in~\eqref{*eq:def-JLy|z} and~\eqref{*eq:def-JLy}.

With the above notation, we start by recalling the statement of Lemma~\ref{lem:decay-re} regarding the optimal decay properties of the Stokeslets $\{\Jc_{L;Y}^z\}_{z,Y}$.
Note that Lemma~\ref{lem:decay} is a particular case of this result, using notation~\eqref{eq:corr-xy-notation}, when $\{z\}\cup Y$ satisfies~\eqref{*eq:dist-Y}. The proof is displayed in Section~\ref{sec:proof-*lem:decay-re}.

\begin{lem}[Decay of Stokeslets with rigid inclusions]\label{*lem:decay-re}
Let $z\in\R^d$, let $(\zeta,P)$ satisfy~\eqref{*eq:zetaP-re} at $z$, and let $Y\subset Q_L$ satisfy~\eqref{*eq:dist-Y}.
Then, we have for all $x\in Q_L$,
\begin{align}
\Big(\int_{B^L(x)}|\!\D(\Jc_{L;Y}^z\zeta)|^2\Big)^\frac12&~\lesssim_{\sharp Y}~\langle(x-z)_L\rangle^{-d}\Big(\int_{B_{1+\rho}(z)}|\!\D(\zeta)|^2\Big)^\frac12,\label{eq:res-*lem:decay-re}
\\
\Big(\int_{B(x)}|\!\D(\Jc_{Y}^z\zeta)|^2\Big)^\frac12&~\lesssim_{\sharp Y}~\langle x-z\rangle^{-d}\Big(\int_{B_{1+\rho}(z)}|\!\D(\zeta)|^2\Big)^\frac12.\nonumber\qedhere
\end{align}
\end{lem}

A similar argument leads us to the following version of the mean-value property for Stokes equations in the presence of a finite number of rigid inclusions.
The proof is displayed in Section~\ref{sec:proof-*lem:mean-value-inc}.

\begin{lem}[Mean-value property with rigid inclusions]\label{*lem:mean-value-inc}
Let $Y\subset Q_L$ satisfy~\eqref{*eq:dist-Y}
and let $w\in H^1(Q_L)^d$ satisfy the following free steady Stokes equations in $Q_L$,
\begin{equation}\label{*eq:def-w}
\left\{\begin{array}{ll}
-\triangle w +\nabla P =0,&\text{in $Q_L\setminus\cup_{y\in Y}B(y)$},\\
\Div(w)=0,&\text{in $Q_L\setminus\cup_{y\in Y}B(y)$},\\
\D(w)=0,&\text{in $\cup_{y\in Y}B(y)$},\\
\int_{\partial B(y)}\sigma(w,P)\nu=0,&\forall y\in Y,\\
\int_{\partial B(y)}\Theta(x-y)\cdot \sigma(w,P)\nu=0,&\forall y\in Y,\,\forall\Theta\in\Md^\Skew.
\end{array}\right.
\end{equation}
Then, we have for all $B(x) \subset  Q_L$,
\begin{equation}\label{*eq:mean-value-inc}
\int_{B^L(x)}|\!\D(w)|^2 \,\lesssim_{\sharp Y}\,\langle\dist(x,\partial Q_L)\rangle^{-d} \int_{Q_L}|\!\D( w)|^2.\qedhere
\end{equation}
\end{lem}
Finally, we recall the statement of Lemma~\ref{lem:period-est} regarding the error~\mbox{$\Jc_{L;Y}^z-\Jc_Y^z$} between periodized and whole-space Stokeslets.
The proof makes heavy use of the above mean-value property and is displayed in Section~\ref{sec:proof-*lem:period-est}.
The stated bounds are not optimal: finer estimates are given in the proof, but this simplified statement is good enough for our purposes.

\begin{lem}[Periodization error]\label{*lem:period-est}
Let $z\in Q_L$, let $(\zeta,P)$ satisfy~\eqref{*eq:zetaP-re} at $z$,
and let~$Y\subset Q_{L}$ be a finite subset such that $\{z\}\cup Y$ satisfies~\eqref{*eq:dist-Y}.
Then, we have for all~$x\in Q_{L}$
\begin{multline}\label{eq:period-est}
\Big(\int_{B^L_{1+\rho}(x)}|\!\D(\Jc_{L;Y}^z\zeta-\Jc_{Y}^z\zeta)|^2\Big)^\frac12
\,\lesssim_{\sharp Y} \, \Big(\int_{B_{1+\rho}(z)}|\!\D(\zeta)|^2\Big)^\frac12\\
\times\Big(\mathds{1}_{|x-z|>\frac L4} \langle (x-z)_L\rangle^{-d}+
 \mathds{1}_{|x-z| \le \frac L4}\langle\dist(Y \setminus\{x,z\},\partial Q_L)\rangle^{-d}\Big),
 \end{multline}
where we recall the notation $\dist(\varnothing, \partial Q_L)=L$ and $B^L_r(z)=(B_r(z)+L\Z^d)\cap Q_L$.
In addition,
\begin{equation}\label{eq:period-est-lala}
\Big(\int_{B^L_{1+\rho}(x)}|\!\D(\psi_L^Y-\psi^Y)|^2\Big)^\frac12\,\lesssim_{\sharp Y} \, \Big(\langle\dist(x,\partial Q_L)\rangle+\langle\dist(Y\setminus\{x\},\partial Q_L)\rangle\Big)^{-d} .\qedhere
\end{equation}
\end{lem}

In the above three lemmas, the multiplicative constants in the estimates crucially depend on the finite number of rigid particles: in Lemma~\ref{*lem:decay-re}, for instance, a quick inspection of the proof shows that the multiplicative constant can be bounded by $C^{\sharp H}(\sharp H)!^{3/2}$.
Although these deterministic results fail in general for an unbounded number of rigid inclusions, we refer the reader to~\cite{DG-21b} where corresponding results are proved to hold in a suitable annealed sense in case of a stationary and ergodic random ensemble of rigid inclusions.

\subsection{Decay of Stokeslets with rigid inclusions}\label{sec:proof-*lem:decay-re}
This section is devoted to the proof of Lemma~\ref{*lem:decay-re} (hence of Lemmas~\ref{lem:decay} and~\ref{lem:decay-re}).
We argue by comparing~$\Jc_{L;Y}^z\zeta$ to  $\Jc_{L;Y_z}^z\zeta$ (recall $Y_z=\{y\in Y:B(y)\cap B^L(z)\ne\varnothing\}$), which is a variant of the solution~$\Jc_{L}^z\zeta$ of the corresponding problem without rigid inclusions.
Equation~\eqref{*eq:def-JLy|z} for $\Jc_{L;Y_z}^z\zeta$ reads
\begin{gather}\label{*eq:def-JLy-re}
\left\{\begin{array}{ll}
-\triangle\Jc_{L;Y_z}^z\zeta+\nabla\Qc_{L;Y_z}^z\zeta=-\delta_{\partial B^L(z)}\sigma(\zeta,P)\nu,&\text{in $Q_L\setminus\cup_{y\in Y_z}B(y)$},\\
\Div(\Jc_{L;Y_z}^z\zeta)=0,&\text{in $Q_L\setminus\cup_{y\in Y_z}B(y)$},\\
\Jc_{L;Y_z}^z\zeta=V_z+\Theta_z(x-z),&\text{in $\cup_{y\in Y_z}B(y)$}\\
&\quad\text{for some $V_z\in\R^d,\Theta_z\in\Md^\Skew$},\\
\sum_{y\in Y_z}\int_{\partial B(y)}\sigma(\Jc_{L;Y_z}^z\zeta,\Qc_{L;Y_z}^z\zeta)\nu&\\
\hspace{1cm}=\sum_{y\in Y_z}\int_{B(y)\cap\partial B^L(z)}\sigma(\zeta,P)\nu,&\\
\sum_{y\in Y_z}\int_{\partial B(y)}\Theta(x-z)\cdot\sigma(\Jc_{L;Y_z}^z\zeta,\Qc_{L;Y_z}^z\zeta)\nu&\\
\hspace{1cm}=\sum_{y\in Y_z}\int_{B(y)\cap\partial B^L(z)}\Theta(x-z)\cdot\sigma(\zeta,P)\nu,&\forall\Theta\in\Md^\Skew.
\end{array}\right.
\end{gather}
We split the proof into three steps: we first apply elliptic regularity to unravel the decay properties of $\Jc_{L;Y_z}^z\zeta$ in the first step, and then estimate the difference $\Jc_{L;Y}^z\zeta-\Jc_{L;Y_z}^z\zeta$ in the last two steps.
Let $z\in Q_L$, let~$\zeta$ satisfy~\eqref{*eq:zetaP-re} at $z$, and let $Y\subset Q_L$ satisfy~\eqref{*eq:dist-Y}.

\medskip
\step{1} Proof that for all $x\in Q_L$,
\begin{equation}\label{*e.nub-3}
\Big(\int_{B^L(x)} |\!\D(\Jc_{L;Y_z}^z\zeta)|^2\Big)^\frac12 \,\lesssim\,\langle (x-z)_L\rangle^{-d}\Big(\int_{B_{1+\rho}(z)}|\!\D(\zeta)|^2\Big)^\frac12.
\end{equation}
The argument is based on elliptic regularity via a duality argument, in a form that is similar to the proof of Theorem~\ref{prop:apest-re} in Section~\ref{sec:proof-prop:apest-re-2}.
By an energy estimate for $\Jc_{L;Y_z}^z\zeta$, the claim~\eqref{*e.nub-3} is trivial if \mbox{$|(x-z)_L|\lesssim1$}, and we shall focus on the case when
\begin{equation}\label{eq:r-dist-xz}
r\,:=\,\tfrac12|(x-z)_L|\,>\,2(1+\rho).
\end{equation}
By definition~\eqref{*eq:def-JLy-re}, we then note that $\Jc_{L;Y_z}^z\zeta$ satisfies the free steady Stokes equation in~$B^L_r(x)=(B_r(x)+L\Z^d)\cap Q_L$, which is the periodization of the ball $B_r(x)$ in $Q_L$.
Elliptic regularity in form of Lemma~\ref{lem:regularity} then yields
\begin{equation}\label{*e.nub-4}
\int_{B^L(x)} |\!\D(\Jc_{L;Y_z}^z\zeta)|^2 \,\lesssim\, r^{-d} \int_{B^L_r(x)} |\!\D(\Jc_{L;Y_z}^z\zeta)|^2.
\end{equation}
Next, by duality, the right-hand side can be written as
\begin{multline}\label{*e.nub-4.1}
\int_{B_r^L(x)} |\!\D(\Jc_{L;Y_z}^n\zeta)|^2 = \sup\bigg\{\Big(\int_{Q_L}h:\D(\Jc_{L;Y_z}^z\zeta)\Big)^2~:~ h\in\Ld^2(Q_L)^{d\times d}_\Sym,\\
~\|h\|_{\Ld^2(Q_L)}=1,~~\supp h\subset B_r^L(x)\bigg\}.
\end{multline}
Given a test function $h\in\Ld^2(Q_L)^{d\times d}_\Sym$ with $\supp h\subset B_r^L(x)$, let $w_{L;h}\in H^1_\per(Q_L)^d$ be the solution of the auxiliary Stokes problem
\begin{equation}\label{*eq:hwh}
\left\{\begin{array}{ll}
-\triangle w_{L;h}+\nabla Q_{L;h} = \Div(h),&\text{in $Q_L\setminus\cup_{y\in Y_z}B(y)$},\\
\Div(w_{L;h})=0,&\text{in $Q_L\setminus\cup_{y\in Y_z}B(y)$},\\
w_{L;h}=V_z+\Theta_z(x-z),&\text{in $\cup_{y\in Y_z}B(y)$},\\
&\quad\text{for some $V_z\in\R^d,\Theta_z\in\Md^\Skew$,}\\
\sum_{y\in Y_z}\int_{\partial B(y)}\sigma(w_{L;h},Q_{L;h})\nu=0,&\\
\sum_{y\in Y_z}\int_{\partial B(y)}\Theta(x-z)\cdot\sigma(w_{L;h},Q_{L;h})\nu=0,&\forall\Theta\in\Md^\Skew.
\end{array}\right.
\end{equation}
These equations are indeed well-posed since by~\eqref{eq:r-dist-xz} the support $B_r^L(x)$ of the force term~$h$ does not intersect the rigid inclusions $\cup_{y\in Y_z}B(y)$.
In view of Lemma~\ref{lem:eqn-corH}, $w_{L;h}$ satisfies the following relation in $Q_L$,
\begin{equation}\label{*eq:hwh/QL}
-\triangle w_{L;h}+\nabla\big(\mathds1_{Q_L\setminus\cup_{y\in Y_z}B(y)}Q_{L;h}\big)\,=\,\Div(h)
-\sum_{y\in Y_z}\delta_{\partial B(y)}\sigma(w_{L;h},Q_{L;h})\nu.
\end{equation}
Similarly the defining equation~\eqref{*eq:def-JLy-re} for $\Jc_{L;Y_z}^z\zeta$ yields in $Q_L$,
\begin{multline}\label{*eq:def-JLy-re/QL}
-\triangle\Jc_{L;Y_z}^z\zeta+\nabla\big(\mathds1_{Q_L\setminus \cup_{y\in Y_z}B(y)}\Qc_{L;Y_z}^z\big)\,=\,-\mathds1_{Q_L\setminus \cup_{y\in Y_z}B(y)}\delta_{\partial B^L(z)}\sigma(\zeta,P)\nu\\
-\sum_{y\in Y_z}\delta_{\partial B(y)}\sigma(\Jc_{L;Y_z}^z\zeta,\Qc_{L;Y_z}^z\zeta)\nu.
\end{multline}
Testing~\eqref{*eq:hwh/QL} with $\Jc_{L;Y_z}^z\zeta$ and~\eqref{*eq:def-JLy-re/QL} with $w_{L;h}$, we are led to
\begin{multline*}
\int_{Q_L}h:\D(\Jc_{L;Y_z}^z\zeta)
\,=\,\int_{\partial B^L(z)\setminus\cup_{y\in Y_z}B(y)} w_{L;h}\cdot\sigma(\zeta,P)\nu\\
+\sum_{y\in Y_z}\int_{\partial B(y)}w_{L;h}\cdot\sigma(\Jc_{L;Y_z}^z\zeta,\Qc_{L;Y_z}^z\zeta)\nu
+\sum_{y\in Y_z}\int_{\partial B(y)}\Jc_{L;Y_z}^z\zeta\cdot\sigma(w_{L;h},Q_{L;h})\nu,
\end{multline*}
and thus, using the boundary conditions in~\eqref{*eq:def-JLy-re} and~\eqref{*eq:hwh},
\begin{equation*}
\int_{Q_L}h:\D(\Jc_{L;Y_z}^z\zeta)
\,=\,\int_{\partial B^L(z)} w_{L;h}\cdot\sigma(\zeta,P)\nu.
\end{equation*}
Recalling that $(\zeta,P)$ satisfies~\eqref{*eq:zetaP-re} and is implicitly extended by $Q_L$-periodicity,
using the boundary conditions and the incompressibility constraints to smuggle in arbitrary constants in the different factors, as in the proof of~\eqref{eq:subtr-press-cond},
and appealing to the trace estimates of Lemma~\ref{lem:trace-0}, we find
\begin{equation}\label{*eq:preconcl-dual-ILnzeta}
\Big(\int_{Q_L}h:\D(\Jc_{L;Y_z}^z\zeta)\Big)^2
\,\lesssim\,\Big(\int_{B^L(z)}|\!\D(w_{L;h})|^2\Big)\Big(\int_{B_{1+\rho}(z)}|\!\D(\zeta)|^2\Big).
\end{equation}
Since equation~\eqref{*eq:hwh} entails that $w_{L;h}$ satisfies the free steady Stokes equation in $B_r^L(z)$,  elliptic regularity in form of Lemma~\ref{lem:regularity} yields
\[\int_{B^L(z)}|\!\D(w_{L;h})|^2\,\lesssim\,r^{-d}\int_{Q_L}|\!\D(w_{L;h})|^2,\]
and thus, combining this with an energy estimate for~\eqref{*eq:hwh},
\[\int_{B^L(z)}|\!\D(w_{L;h})|^2\,\lesssim\,r^{-d}\int_{Q_L}|h|^2.\]
Combining this with~\eqref{*e.nub-4}, \eqref{*e.nub-4.1}, and~\eqref{*eq:preconcl-dual-ILnzeta}, the claim~\eqref{*e.nub-3} follows.

\medskip
\step{2} Proof that for all $x\in Q_L$,
\begin{equation}
\int_{B^L(x)}|\!\D(\Jc_{L;Y}^z\zeta)|^2\,\lesssim\,\bigg(\sum_{y\in\{x\}\cup (Y\setminus Y_z)}\langle(y-z)_L\rangle^{-2d}\bigg)\int_{B_{1+\rho}(z)}|\!\D(\zeta)|^2.\label{*eq:decay-borderline}
\end{equation}
In view of Lemma~\ref{lem:eqn-corH}, the defining equation~\eqref{*eq:def-JLy|z} for $\Jc_{L;Y}^z\zeta$ yields in $Q_L$,
\begin{multline}\label{*eq:def-JLy|z/QL}
-\triangle\Jc_{L;Y}^z\zeta+\nabla\big(\mathds1_{Q_L\setminus \cup_{y\in Y}B(y)}\Qc_{L;Y}^z\big)\,=\,-\mathds1_{Q_L\setminus \cup_{y\in Y_z}B(y)}\delta_{\partial B^L(z)}\sigma(\zeta,P)\nu\\
-\sum_{y\in Y}\delta_{\partial B(y)}\sigma(\Jc_{L;Y}^z\zeta,\Qc_{L;Y}^z\zeta)\nu.
\end{multline}
Subtracting~\eqref{*eq:def-JLy-re/QL} entails in $ Q_L$
\begin{multline}\label{*eq:def-JLy|z/QL-diff}
-\triangle(\Jc_{L;Y}^z\zeta-\Jc_{L;Y_z}^z\zeta)+\nabla\Big(\mathds1_{Q_L\setminus \cup_{y\in Y}B(y)}\Qc_{L;Y}^z-\mathds1_{Q_L\setminus \cup_{y\in Y_z}B(y)}\Qc_{L;Y_z}^z\Big)\\
\,=\,
-\sum_{y\in Y\setminus Y_z}\delta_{\partial B(y)}\sigma(\Jc_{L;Y}^z\zeta,\Qc_{L;Y}^z\zeta)\nu\\
-\sum_{y\in Y_z}\delta_{\partial B(y)}\Big(\sigma(\Jc_{L;Y}^z\zeta,\Qc_{L;Y}^z\zeta)\nu-\sigma(\Jc_{L;Y_z}^z\zeta,\Qc_{L;Y_z}^z\zeta)\nu\Big).
\end{multline}
Testing this equation with $\Jc_{L;Y}^z\zeta-\Jc_{L;Y_z}^z\zeta$ itself, and using the boundary conditions in~\eqref{*eq:def-JLy|z} and~\eqref{*eq:def-JLy-re}, we obtain the energy identity
\begin{equation*}
2\int_{Q_L}|\!\D(\Jc_{L;Y}^z\zeta-\Jc^z_{L;Y_z}\zeta)|^2
\,=\,
\sum_{y\in Y\setminus Y_z}\int_{\partial B(y)}\Jc^z_{L;Y_z}\zeta\cdot\sigma(\Jc_{L;Y}^z\zeta,\Qc_{L;Y}^z\zeta)\nu.
\end{equation*}
Further using the boundary conditions and the incompressibility constraints to smuggle in arbitrary constants in the different factors, as in the proof of~\eqref{eq:subtr-press-cond},
and appealing to the trace estimates of Lemma~\ref{lem:trace-0}, we deduce
\begin{equation*}
\int_{ Q_L}|\!\D(\Jc_{L;Y}^z\zeta-\Jc_{L;Y_z}^z\zeta)|^2\,\lesssim\,\sum_{y\in Y\setminus Y_z}\Big(\int_{B(y)}|\!\D(\Jc_{L;Y_z}^z\zeta)|^2\Big)^\frac12\Big(\int_{B_{1+\rho}(y)}|\!\D(\Jc_{L;Y}^z\zeta)|^2\Big)^\frac12.
\end{equation*}
Decomposing $\Jc_{L;Y}^z\zeta=(\Jc_{L;Y}^z\zeta-\Jc_{L;Y_z}^z\zeta)+\Jc_{L;Y_z}^z\zeta$ in the last factor, using the triangle inequality and Young's inequality, we are led to
\[\int_{Q_L}|\!\D(\Jc_{L;Y}^z\zeta-\Jc_{L;Y_z}^z\zeta)|^2\,\lesssim\,\sum_{y\in Y\setminus Y_z}\int_{B_{1+\rho}(y)}|\!\D(\Jc_{L;Y_z}^z\zeta)|^2.\]
The triangle inequality then yields for all $x\in Q_L$,
\begin{eqnarray*}
\int_{B^L(x)}|\!\D(\Jc_{L;Y}^z\zeta)|^2&\lesssim&\int_{B^L(x)}|\!\D(\Jc_{L;Y_z}^z\zeta)|^2+\int_{Q_L}|\!\D(\Jc_{L;Y}^z\zeta-\Jc_{L;Y_z}^z\zeta)|^2\\
&\lesssim&\int_{B^L(x)}|\!\D(\Jc_{L;Y_z}^z\zeta)|^2+\sum_{y\in Y\setminus Y_z}\int_{B_{1+\rho}(y)}|\!\D(\Jc_{L;Y_z}^z\zeta)|^2,
\end{eqnarray*}
which yields the claim~\eqref{*eq:decay-borderline} in combination with~\eqref{*e.nub-3}.

\medskip
\step3 Conclusion.\\
We argue by induction on the cardinality of $Y\setminus Y_z$ for~\eqref{eq:res-*lem:decay-re}.
%
If $\sharp (Y\setminus Y_z)=0$, that is, if $Y=Y_z$, the conclusion~\eqref{eq:res-*lem:decay-re} already follows from~\eqref{*e.nub-3}. Given $n\ge1$, we assume that~\eqref{eq:res-*lem:decay-re} holds whenever $\sharp (Y\setminus Y_z)<n$, and we shall show that it also holds when $\sharp (Y\setminus Y_z)=n$.
Let $Y\subset Q_L$ be fixed with $\sharp (Y\setminus Y_z)=n$.
For any $S\subset Y\setminus Y_z$, the same argument as for~\eqref{*eq:def-JLy|z/QL-diff} yields in $Q_L\setminus\cup_{y\in S}B(y)$  
\begin{multline*}
-\triangle(\Jc_{L;Y}^z\zeta-\Jc_{L;Y_z\cup S}^z\zeta)+\nabla\Big(\mathds1_{Q_L\setminus \cup_{y\in Y}B(y)}\Qc_{L;Y}^z-\mathds1_{Q_L\setminus \cup_{y\in Y_z\cup S}B(y)}\Qc_{L;Y_z\cup S}^z\Big)\\
\,=\,
-\sum_{y\in Y\setminus (Y_z\cup S)}\delta_{\partial B(y)}\sigma(\Jc_{L;Y}^z\zeta,\Qc_{L;Y}^z\zeta)\nu\\
-\sum_{y\in Y_z }\delta_{\partial B(y)}\sigma\big(\Jc_{L;Y}^z\zeta-\Jc_{L;Y_z\cup S}^z\zeta,\Qc_{L;Y}^z\zeta-\Qc_{L;Y_z\cup S}^z\zeta\big)\nu.
\end{multline*}
As $\Jc_{L;Y}^z\zeta-\Jc_{L;Y_z\cup S}^z\zeta$ is further rigid in $\cup_{y\in S}B(y)$,
this implies, by definition of~$\{\Jc_{L;S}^y\}_y$,
\begin{equation*}
\Jc_{L;Y}^z\zeta-\Jc_{L;Y_z\cup S}^z\zeta
\,=\,
\sum_{y\in Y\setminus (Y_z\cup S)}\Jc_{L;S}^y\Jc_{L;Y}^z\zeta
+\sum_{y\in Y_z}\Jc_{L;S}^y(\Jc_{L;Y}^z\zeta-\Jc_{L;Y_z\cup S}^z\zeta),
\end{equation*}
which we may further decompose as
\begin{multline*}
\Jc_{L;Y}^z\zeta-\Jc_{L;Y_z\cup S}^z\zeta
\,=\,\sum_{y\in Y\setminus (Y_z\cup S)}\Jc_{L;S}^y(\Jc_{L;Y}^z\zeta-\Jc_{L;Y_z\cup S\cup\{y\}}^z\zeta)\\
+\sum_{y\in Y\setminus (Y_z\cup S)}\Jc_{L;S}^y\Jc_{L;Y_z\cup S\cup\{y\}}^z\zeta
+\sum_{y\in Y_z}\Jc_{L;S}^y(\Jc_{L;Y}^z\zeta-\Jc_{L;Y_z\cup S}^z\zeta).
\end{multline*}
Iterating this identity, we find
\begin{multline*}
\Jc_{L;Y}^z\zeta-\Jc_{L;Y_z}^z\zeta
\,=\,
\sum_{l=1}^n\sum_{y_1,\ldots,y_l\in Y\setminus Y_z}^{\ne}\Jc^{y_1}_L\Jc_{L;\{y_1\}}^{y_2}\ldots \Jc_{L;\{y_1,\ldots,y_{l-1}\}}^{y_l}\Jc_{L;Y_z\cup\{y_1,\ldots,y_l\}}^z\zeta\\
+\sum_{l=1}^n\sum_{y_1,\ldots,y_{l-1}\in Y\setminus Y_z}^{\ne}\sum_{y\in Y_z}\Jc^{y_1}_L\Jc^{y_2}_{L;\{y_1\}}\ldots \Jc_{L;\{y_1,\ldots,y_{l-2}\}}^{y_{l-1}}\Jc_{L;\{y_1,\ldots,y_{l-1}\}}^{y}\\
\times\big(\Jc_{L;Y}^z\zeta-\Jc_{L;Y_z\cup\{y_1,\ldots,y_{l-1}\}}^z\zeta\big).
\end{multline*}
We now appeal to the induction hypothesis in form of~\eqref{eq:res-*lem:decay-re} for the terms  $\Jc_{L;\{y_1,\ldots,y_{j}\}}^{y}$ and $\Jc_{L;Y_z\cup\{y_1,\ldots,y_j\}}^z$ for all~$1\le j< n$ and~$y\in Y$, to the suboptimal decay estimate~\eqref{*eq:decay-borderline} for \mbox{$\Jc_{L;Y_z\cup\{y_1,\ldots,y_n\}}^z$} (which only appears in the first right-hand sum when $l=n$).
Recalling that $|(y-z)_L|\le2$ for all $y\in Y_z$, this yields for all $x\in Q_L$, after straightforward simplifications,
\begin{multline*}
\Big(\int_{B^L(x)}|\!\D(\Jc_{L;Y}^z\zeta-\Jc_{L;Y_z}^z\zeta)|^2\Big)^\frac12
\,\lesssim_n\,\Big(\int_{B_{1+\rho}(z)}|\!\D(\zeta)|^2\Big)^\frac12\\
\times\sum_{l=0}^n\sum_{y_1,\ldots,y_l\in Y\setminus Y_z}^{\ne}\langle(x-y_1)_L\rangle^{-d}\langle(y_1-y_2)_L\rangle^{-d}\ldots\langle(y_{l}-z)_L\rangle^{-d}.
\end{multline*}
The conclusion~\eqref{eq:res-*lem:decay-re} now follows from the bound $\langle(a-b)_L\rangle^{-d}\langle(b-c)_L\rangle^{-d}\lesssim\langle(a-c)_L\rangle^{-d}$ for all $a,b,c\in Q_L$.
\qed

\subsection{Mean-value property with rigid inclusions}\label{sec:proof-*lem:mean-value-inc}
This section is devoted to the proof of Lemma~\ref{*lem:mean-value-inc}.
We split the proof into two steps.
Let $Y\subset Q_L$ satisfy~\eqref{*eq:dist-Y}
and let $(w,P)\in H^1(Q_L)^d\times \Ld^2(Q_L)$ satisfy~\eqref{*eq:def-w} in~$Q_L$.

\medskip
\step1 Proof that for all $x\in Q_L$,
\begin{equation}\label{*eq:mvp-2}
\int_{B^L(x)}|\!\D(w)|^2  \,\lesssim_{\sharp Y}\, \bigg(\sum_{y\in \{x\}\cup Y}\langle \dist(y,\partial Q_L) \rangle^{-d}\bigg) \int_{Q_L} |\!\D( w)|^2.
\end{equation}
For that purpose, we shall compare $w$ to the solution $\tilde w\in w+H^1_\per(Q_L)^d$ of the free steady Stokes equations without rigid particles in $Q_L$,
\begin{equation}\label{*e.mvp-0}
-\triangle \tilde w+\nabla\tilde P =0,\qquad\Div(\tilde w)=0,\qquad\text{in $Q_L$}.
\end{equation}
In view of Lemma~\ref{lem:eqn-corH}, the equations~\eqref{*eq:def-w} for $w$ yield the following relation in $Q_L$,
\[-\triangle w+\nabla(\mathds1_{Q_L\setminus\cup_{y\in Y}B(y)}P)\,=\,-\sum_{y\in Y}\delta_{\partial B(y)}\sigma(w,P)\nu.\]
Subtracting~\eqref{*e.mvp-0}, we deduce that the difference $w-\tilde w \in H^1_\per(Q_L)$ satisfies 
\begin{equation}\label{eq:form-QL-wtilde}
-\triangle(w-\tilde w)+\nabla\big(\mathds1_{Q_L\setminus\cup_{y\in Y}B(y)}P -\tilde P\big)\,=\,-\sum_{y\in Y}\delta_{\partial B(y)}\sigma(w,P)\nu.
\end{equation}
Testing this equation with $w-\tilde w$ and using the boundary conditions in~\eqref{*eq:def-w},
we obtain the energy identity
\[2\int_{Q_L}|\!\D(w-\tilde w)|^2\,=\,\sum_{y\in Y}\int_{\partial B(y)}\tilde w\cdot\sigma(w,\Qc)\nu.\]
Further using the boundary conditions and the incompressibility constraints to smuggle in arbitrary constants in the different factors, as in the proof of~\eqref{eq:subtr-press-cond},
and appealing to the trace estimates of Lemma~\ref{lem:trace-0}, we get
\begin{equation}\label{eq:est-w-tildew}
\int_{Q_L}|\!\D(w-\tilde w)|^2\,\lesssim\,\sum_{y\in Y}\Big(\int_{B(y)}|\!\D(\tilde w)|^2\Big)^\frac12\Big(\int_{B_{1+\rho}(y)}|\!\D(w)|^2\Big)^\frac12.
\end{equation}
Decomposing $w=(w-\tilde w)+\tilde w$ in the last factor, using the triangle inequality and Young's inequality, we are led to
\[\int_{Q_L}|\!\D(w-\tilde w)|^2\,\lesssim\,\sum_{y\in Y}\int_{B_{1+\rho}(y)}|\!\D(\tilde w)|^2.\]
and thus, by the triangle inequality, for all $x\in Q_L$,
\begin{equation}\label{eq:pre-estim-tilde w-w}
\int_{B^L(x)}|\!\D( w)|^2\,\lesssim\,\int_{B^L(x)}|\!\D(\tilde w)|^2+\sum_{y\in Y}\int_{B_{1+\rho}(y)}|\!\D(\tilde w)|^2.
\end{equation}
Rather decomposing $\tilde w=w-(w-\tilde w)$, we note that~\eqref{eq:est-w-tildew} also yields the energy estimate
\begin{equation}\label{eq:apriori-tildew}
\int_{Q_L}|\!\D(\tilde w)|^2\,\lesssim\,\int_{Q_L}|\!\D(w)|^2.
\end{equation}
As $\tilde w$ satisfies the free steady Stokes equations in $Q_L$, cf.~\eqref{*e.mvp-0}, the mean-value property of Lemma~\ref{lem:regularity} yields for all $x\in Q_L$,
\begin{equation*} 
 \int_{B^L(x)} |\!\D( \tilde w)|^2 \,\lesssim\,\langle \dist(x,\partial Q_L) \rangle^{-d} \int_{Q_L} |\!\D( \tilde w)|^2,
\end{equation*}
and thus, combined with~\eqref{eq:apriori-tildew},
\begin{equation}\label{eq:mean-value-tildew}
 \int_{B^L(x)} |\!\D( \tilde w)|^2 \,\lesssim\,\langle \dist(x,\partial Q_L) \rangle^{-d} \int_{Q_L} |\!\D(w)|^2.
\end{equation}
Inserting this into~\eqref{eq:pre-estim-tilde w-w}, the claim~\eqref{*eq:mvp-2} follows.

\medskip
\step2 Conclusion.\\
Given $S\subset Y$, we denote by $w_S\in w+H^1_\per(Q_L)^d$ the solution of the free steady Stokes problem with rigid inclusions at points of $S$ only,
\begin{equation}\label{*e.mvp-3}
\left\{\begin{array}{ll}
-\triangle w_S+\nabla P_S =0,&\text{in $Q_L\setminus\cup_{y\in S}B(y)$},\\
\Div(w_S)=0,&\text{in $Q_L\setminus\cup_{y\in S}B(y)$},\\
\D(w_S)=0,&\text{in $\cup_{y\in S}B(y)$},\\
\int_{\partial B(y)}\sigma(w_S,P_S)\nu=0,&\forall y\in S,\\
\int_{\partial B(y)}\Theta(x-y)\cdot \sigma(w_S,P_S)\nu=0,&\forall y\in S,~\forall\Theta\in\Md^\Skew.
\end{array}\right.
\end{equation}
In particular, we recover $w_Y=w$ and $w_\varnothing=\tilde w$ as defined in~\eqref{*e.mvp-0}.
The result~\eqref{*eq:mvp-2} of Step~1 yields in this case, for all $x\in Q_L$,
\[\int_{B^L(x)}|\!\D(w_S)|^2\,\lesssim_{\sharp S}\,\bigg(\sum_{y\in\{x\}\cup S}\langle\dist(y,\partial Q_L)\rangle^{-d}\bigg)\int_{Q_L}|\!\D(w_S)|^2.\]
Noting that a similar argument as for~\eqref{eq:apriori-tildew} further yields the energy estimate
\[\int_{Q_L}|\!\D(w_S)|^2\,\lesssim\,\int_{Q_L}|\!\D(w)|^2,\]
we deduce for all $x\in Q_L$,
\begin{equation}\label{eq:mvp-2-bis-S}
\int_{B^L(x)}|\!\D(w_S)|^2\,\lesssim_{\sharp S}\,\bigg(\sum_{y\in\{x\}\cup S}\langle\dist(y,\partial Q_L)\rangle^{-d}\bigg)\int_{Q_L}|\!\D(w)|^2.
\end{equation}
We shall now decompose $w$ in terms of this sequence $(w_S)_{S\subset Y}$.
Arguing as for~\eqref{eq:form-QL-wtilde}, we note that for any $S\subset Y$ the following relation holds in $Q_L\setminus\cup_{y\in S}B(y)$,
\[-\triangle(w-w_S)+\nabla(P-P_S)\,=\,-\sum_{y\in Y\setminus S}\delta_{\partial B(y)}\sigma(w,P)\nu.\]
As $w-w_S$ is rigid in $\cup_{y\in S}B(y)$, this allows to decompose
\[w-w_S\,=\,\sum_{y\in Y\setminus S}\Jc_{L;S}^yw,\]
and thus, iterating this identity and starting with $w_\varnothing=\tilde w$,
\[w\,=\,\tilde w+\sum_{l=1}^{\sharp Y}\sum_{y_1,\ldots,y_l\in Y}^{\ne}\Jc_{L}^{y_1}\Jc_{L;\{y_1\}}^{y_2}\ldots\Jc_{L;\{y_1,\ldots,y_{l-1}\}}^{y_l}w_{\{y_1,\ldots,y_l\}}.\]
Appealing to the decay estimates for $\{\Jc_{L;S}^y\}_{y,S}$ in Lemma~\ref{*lem:decay-re}, and to~\eqref{eq:mean-value-tildew} and~\eqref{eq:mvp-2-bis-S},
we get after straightforward simplifications, for all $x\in Q_L$,
\begin{multline*}
\int_{B^L(x)}|\!\D(w)|^2\,\lesssim_{\sharp Y}\,\int_{Q_L}|\!\D(w)|^2
\\
\times\sum_{l=0}^{\sharp Y}\sum_{y_1,\ldots,y_l\in Y}^{\ne}\langle(x-y_1)_L\rangle^{-2d}\langle(y_1-y_2)_L\rangle^{-2d}\ldots\langle(y_{l-1}-y_l)_L\rangle^{-2d}\langle\dist(y_l,\partial Q_L)\rangle^{-d}.
\end{multline*}
Using the trivial bound $\langle(a-b)_L\rangle^{-d}\langle(b-c)_L\rangle^{-d}\lesssim\langle(a-c)_L\rangle^{-d}$ for all $a,b,c\in Q_L$, and noting that the infimum over $c\in\partial Q_L$ further yields
$\langle(a-b)_L\rangle^{-d}\langle\dist(b,\partial Q_L)\rangle^{-d}\lesssim\langle\dist(a,\partial Q_L)\rangle^{-d}$, 
the conclusion~\eqref{*eq:mean-value-inc} follows.
\qed

\subsection{Periodization errors}\label{sec:proof-*lem:period-est}
This section is devoted to the proof of Lemma~\ref{*lem:period-est}.
We split the proof into three steps. Let $z\in Q_L$, let $\zeta$ satisfy~\eqref{*eq:zetaP-re} at $z$, and let $Y\subset Q_L$ be such that $\{z\}\cup Y$ satisfies~\eqref{*eq:dist-Y}.

\medskip
\step1 Proof that for all $x\in Q_L$,
\begin{multline}\label{eq:period-est-raff}
\int_{B^L_{1+\rho}(x)}|\!\D(\Jc_{L;Y}^z\zeta-\Jc_Y^z\zeta)|^2\,\lesssim_{\sharp Y}\,\int_{B_{1+\rho}(z)}|\!\D(\zeta)|^2\\
\times\min\bigg\{\langle(x-z)_L\rangle^{-2d}\wedge\Big(\langle\dist(x,\partial Q_L(a))\rangle^{-d}\langle\dist(z,\partial Q_L(a))\rangle^{-d}\Big)~:~\\
a\in\R^d,~x,z\in Q_L(a),~Y\subset Q_L(a)\bigg\}.
\end{multline}
It suffices to prove this estimate for $a=0$, that is,
\begin{multline*}
\int_{B^L_{1+\rho}(x)}|\!\D(\Jc_{L;Y}^z\zeta-\Jc_Y^z\zeta)|^2\,\lesssim_{\sharp Y}\,\int_{B_{1+\rho}(z)}|\!\D(\zeta)|^2\\
\times\bigg(\langle(x-z)_L\rangle^{-2d}\wedge\Big(\langle\dist(x,\partial Q_L)\rangle^{-d}\langle\dist(z,\partial Q_L)\rangle^{-d}\Big)\bigg),
\end{multline*}
as the claim~\eqref{eq:period-est-raff} then follows by translating the underlying cell $Q_L$, which does indeed not change the equations provided that the translated cell still contains the relevant points~$x,z,Y$.
Further noting that Lemma~\ref{*lem:decay-re} together with the triangle inequality yields
\begin{equation*}
\int_{B^L_{1+\rho}(x)}|\!\D(\Jc_{L;Y}^z\zeta-\Jc_Y^z\zeta)|^2\,\lesssim_{\sharp Y}\,\langle(x-z)_L\rangle^{-2d}\int_{B_{1+\rho}(z)}|\!\D(\zeta)|^2,
\end{equation*}
it only remains to prove for all $x\in Q_L$,
\begin{equation}\label{eq:period-est-raff-0}
\int_{B^L_{1+\rho}(x)}|\!\D(\Jc_{L;Y}^z\zeta-\Jc_Y^z\zeta)|^2\,\lesssim_{\sharp Y}\,\langle\dist(x,\partial Q_L)\rangle^{-d}\langle\dist(z,\partial Q_L)\rangle^{-d}\int_{B_{1+\rho}(z)}|\!\D(\zeta)|^2.
\end{equation}
As $\{z\}\cup Y$ satisfies~\eqref{*eq:dist-Y}, we recall that $\Jc_{L;Y}^z\zeta$ satisfies the simpler Stokes problem~\eqref{*eq:def-JLy|z-simpl} (and likewise for $\Jc_{Y}^z\zeta$).
The difference $\Jc_{L;Y}^z\zeta-\Jc_{Y}^z\zeta$ then satisfies the free steady Stokes equations~\eqref{*eq:def-w}. Applying the mean-value property of Lemma~\ref{*lem:mean-value-inc} to this equation, we get for all $x\in Q_L$,
\begin{equation}\label{eq:diff-JLJ-Y-meanvalue-1}
\int_{B^L_{1+\rho}(x)}|\!\D(\Jc_{L;Y}^z\zeta-\Jc_{Y}^z\zeta)|^2\,\lesssim_{\sharp Y}\,\langle\dist(x,\partial Q_L)\rangle^{-d}\int_{Q_L}|\!\D(\Jc_{L;Y}^z\zeta-\Jc_{Y}^z\zeta)|^2.
\end{equation}
In order to estimate the last integral, taking some inspiration from the proof of~\eqref{diff-estE}, we note that it is convenient to further compare $\Jc_{L;Y}^z\zeta$ and $\Jc_{Y}^z\zeta$ to the solution of the corresponding Neumann problem in $Q_L$: we define $\Jc_{N;Y}^z\zeta\in H^1(Q_L)^d$ as the solution of
\begin{equation}\label{*eq:def-JLy|z-simpl-N}
\left\{\begin{array}{ll}
-\triangle \Jc_{N;Y}^z\zeta+\nabla\Qc_{N;Y}^z\zeta=-\delta_{\partial B^L(z)}\sigma(\zeta,P)\nu,&\text{in $Q_L\setminus\cup_{y\in Y}B(y)$},\\
\Div(\Jc_{N;Y}^z\zeta)=0,&\text{in $Q_L\setminus\cup_{y\in Y}B(y)$},\\
\sigma(\Jc_{N;Y}^z\zeta,\Qc_{N;Y}^z\zeta)\nu=0,&\text{on $\partial Q_L$},\\
\D(\Jc_{N;Y}^z\zeta)=0,&\text{in $\cup_{y\in Y}B(y)$},\\
\int_{\partial B(y)}\sigma(\Jc_{N;Y}^z\zeta,\Qc_{N;Y}^z\zeta)\nu=0,&\forall y\in Y,\\
\int_{\partial B(y)}\Theta(x-y)\cdot\sigma(\Jc_{N;Y}^z\zeta,\Qc_{N;Y}^z\zeta)\nu=0,&\forall y\in Y,~\forall\Theta\in\Md^\Skew.
\end{array}\right.
\end{equation}
In these terms, we start by estimating
\begin{equation}\label{eq:bound-H-H1H2}
\int_{Q_L}|\!\D(\Jc_{L;Y}^z\zeta-\Jc_{Y}^z\zeta)|^2\,\le\,2\int_{Q_L}|\!\D(H_1)|^2+2\int_{Q_L}|\!\D(H_2)|^2,
\end{equation}
where we have set for abbreviation
\[H_1\,:=\,\Jc_{L;Y}^z\zeta-\Jc_{N;Y}^z\zeta,\qquad H_2\,:=\,\Jc_{Y}^z\zeta-\Jc_{N;Y}^z\zeta.\]
We denote by $P_1,P_2$ the corresponding pressure differences.
In view of~\eqref{*eq:def-JLy|z-simpl} and~\eqref{*eq:def-JLy|z-simpl-N},  $(H_1,P_1)$ satisfies
\begin{equation}\label{eq:H1-free}
\left\{\begin{array}{ll}
-\triangle H_1+\nabla P_1=0,&\text{in $Q_L\setminus\cup_{y\in Y}B(y)$},\\
\Div(H_1)=0,&\text{in $Q_L\setminus\cup_{y\in Y}B(y)$},\\
\sigma(H_1,P_1)\nu=\sigma(\Jc_{L;Y}^z\zeta,\Qc_{L;Y}^z\zeta)\nu,&\text{on $\partial Q_L$},\\
\D(H_1)=0,&\text{in $\cup_{y\in Y}B(y)$},\\
\int_{\partial B(y)}\sigma(H_1,P_1)\nu=0,&\forall y\in Y,\\
\int_{\partial B(y)}\Theta(x-y)\cdot\sigma(H_1,P_1)\nu=0,&\forall y\in Y,~\forall\Theta\in\Md^\Skew,
\end{array}\right.
\end{equation}
for which the energy identity takes the form
\[2\int_{Q_L}|\!\D(H_1)|^2\,=\,\int_{\partial Q_L}H_1\cdot\sigma(\Jc_{L;Y}^z\zeta,\Qc_{L;Y}^z\zeta)\nu,\]
hence, recalling $H_1=\Jc_{L;Y}^z\zeta-\Jc_{N;Y}^z\zeta$ and the periodicity of $\Jc_{L;Y}^z\zeta$,
\begin{equation}\label{eq:energy-est-H1}
2\int_{Q_L}|\!\D(H_1)|^2\,=\,-\int_{\partial Q_L}\Jc_{N;Y}^z\zeta\cdot\sigma(\Jc_{L;Y}^z\zeta,\Qc_{L;Y}^z\zeta)\nu.
\end{equation}
By Lemma~\ref{lem:eqn-corH} and~\eqref{*eq:def-JLy-re}, $\Jc_{L;Y}^z\zeta$ satisfies in $Q_L$
\begin{equation*}
-\triangle\Jc_{L;Y}^z\zeta+\nabla\big(\mathds1_{\R^d\setminus \cup_{y\in Y}B(y)}\Qc_{L;Y}^z\big)\,=\,-\delta_{\partial B(z)}\,\sigma(\zeta,P)\nu
-\sum_{y\in Y}\delta_{\partial B(y)}\sigma(\Jc_{L;Y}^z\zeta,\Qc_{L;Y}^z\zeta)\nu,
\end{equation*}
whereas, by~\eqref{*eq:def-JLy|z-simpl-N},  $\Jc_{N;Y}^z\zeta$ satisfies
\begin{equation*}
-\triangle\Jc_{N;Y}^z\zeta+\nabla\big(\mathds1_{\R^d\setminus \cup_{y\in Y}B(y)}\Qc_{N;Y}^z\big)\,=\,-\delta_{\partial B(z)}\sigma(\zeta,P)\nu
-\sum_{y\in Y}\delta_{\partial B(y)}\sigma(\Jc_{N;Y}^z\zeta,\Qc_{N;Y}^z\zeta)\nu.
\end{equation*}
Testing the first relation with $\Jc_{N;Y}^z\zeta$, testing the second one with $\Jc_{L;Y}^z\zeta$, and using boundary conditions, we find
\begin{eqnarray*}
\lefteqn{\int_{\partial Q_L}\Jc^z_{N;Y}\zeta\cdot \sigma(\Jc^z_{L;Y}\zeta,\Qc^z_{L;Y}\zeta)\nu}\\
&=&2\int\D(\Jc^z_{N;Y}\zeta):\D(\Jc^z_{L;Y}\zeta)+\int_{\partial B(z)}\Jc^z_{N;Y}\zeta\cdot\sigma(\zeta,P)\nu\\
&=&\int_{\partial B(z)}(\Jc^z_{N;Y}\zeta-\Jc^z_{L;Y}\zeta)\cdot\sigma(\zeta,P)\nu,
\end{eqnarray*}
so that identity~\eqref{eq:energy-est-H1} becomes
\begin{equation}\label{eq:identity-energy-H1}
2\int_{Q_L}|\!\D(H_1)|^2\,=\,\int_{\partial B(z)}H_1\cdot\sigma(\zeta,P)\nu.
\end{equation}
Using the boundary conditions and the incompressibility constraint
to smuggle in arbitrary constants in the different factors, as in the proof of~\eqref{eq:subtr-press-cond},
and appealing to the trace estimates of Lemma~\ref{lem:trace-0}, we find
\begin{equation*}
\int_{Q_L}|\!\D(H_1)|^2\,\lesssim\,\Big(\int_{B(z)}|\!\D(H_1)|^2\Big)^\frac12\Big(\int_{B_{1+\rho}(z)}|\!\D(\zeta)|^2\Big)^\frac12.
\end{equation*}
Applying the mean-value property of Lemma~\ref{*lem:mean-value-inc} to equation~\eqref{eq:H1-free} for $H_1$, and using Young's inequality, we deduce
\begin{equation}\label{eq:estim-meanvalue-H1}
\int_{Q_L}|\!\D(H_1)|^2\,\lesssim\,\langle\dist(z,\partial Q_L)\rangle^{-d}\int_{B_{1+\rho}(z)}|\!\D(\zeta)|^2.
\end{equation}
Likewise, repeating the argument in favor of~\eqref{eq:identity-energy-H1}, this time for $H_2$, we obtain
\[2\int_{Q_L}|\!\D(H_2)|^2\,=\,\int_{\partial B(z)} H_2\cdot\sigma(\zeta,P)\nu+\int_{\partial Q_L}\Jc_{Y}^z\zeta\cdot\sigma(\Jc_Y^z\zeta,\Qc_Y^z\zeta)\nu,\]
or equivalently, using the free steady Stokes equations for $\Jc_Y^z\zeta$ in $\R^d\setminus Q_L$ and integrating by parts to reformulate the second right-hand side term,
\begin{eqnarray*}
2\int_{Q_L}|\!\D(H_2)|^2&=&\int_{\partial B(z)} H_2\cdot\sigma(\zeta,P)\nu-2\int_{\R^d\setminus Q_L}|\!\D(\Jc_{Y}^z\zeta)|^2\\
&\le&\int_{\partial B(z)} H_2\cdot\sigma(\zeta,P)\nu.
\end{eqnarray*}
Arguing as for $H_1$, we may then deduce
\begin{equation*}
\int_{Q_L}|\!\D(H_2)|^2\,\lesssim\,\langle\dist(z,\partial Q_L)\rangle^{-d}\int_{B_{1+\rho}(z)}|\!\D(\zeta)|^2.
\end{equation*}
Combined with~\eqref{eq:diff-JLJ-Y-meanvalue-1}, \eqref{eq:bound-H-H1H2}, and~\eqref{eq:estim-meanvalue-H1}, this yields the claim~\eqref{eq:period-est-raff-0}.

\medskip
\step2 Proof of~\eqref{eq:period-est}.\\
We claim that the conclusion~\eqref{eq:period-est} is a simple post-processing of~\eqref{eq:period-est-raff}.
As~\eqref{eq:period-est} trivially follows from~\eqref{eq:period-est-raff} if $|x-z|>\frac L4$,
it remains to consider the case when \mbox{$|x-z|\le\frac L4$}.
In that case, we can choose $q\in\frac L4\Z^d$ with $|q|_\infty\le\frac L4$ such that $x,z\in Q_{\frac12L}(q)$.
We then construct a translation vector $a$ componentwise: First, for all directions $1\le i\le d$ with $q_i=0$, we set $a_i:=0$. Second, for all $i$ with $q_i=\frac L4$, we set $a_i:=\dist(Y\setminus\{x,z\},P_L^{i,-})$, where $P_L^{i,-}$ is the cubic facet $\{v\in\partial Q_L:v_i=-\frac L2\}$. Third, for all $i$ with $q_i=-\frac L4$, we set $a_i:=-\dist(Y\setminus\{x,z\},P_L^{i,+})$, where $P_L^{i,+}$ is the facet $\{v\in\partial Q_L:v_i=\frac L2\}$. With this construction of $a$, we find that $Y\setminus\{x,z\}$ is included in the translated cube~$Q_L(a)$ (and actually intersects its boundary). Moreover, we find
\[\dist(x,\partial Q_L(a))\,\ge\,\dist(x,\partial Q_L)+\inf_i|a_i|\,\ge\,\dist(x,\partial Q_L)+\dist(Y\setminus\{x,z\},\partial Q_L),\]
and similarly
\[\dist(z,\partial Q_L(a))\,\ge\,\dist(z,\partial Q_L)+\dist(Y\setminus\{x,z\},\partial Q_L).\]
In particular, we get
\[\langle\dist(x,\partial Q_L(a))\rangle^{-d}\langle\dist(z,\partial Q_L(a))\rangle^{-d}\,\le\,\langle\dist(Y\setminus\{x,z\},\partial Q_L)\rangle^{-2d},\]
so that the conclusion~\eqref{eq:period-est} indeed follows from~\eqref{eq:period-est-raff}.

\medskip
\step3 Proof of~\eqref{eq:period-est-lala}.\\
We shall prove the following refined version of~\eqref{eq:period-est-lala}: for all $x\in Q_L$,
\begin{multline}\label{eq:period-est-lala-raff}
\int_{B^L_{1+\rho}(x)}|\!\D(\psi_L^Y-\psi^Y)|^2\,\lesssim\,\min\Big\{\langle\dist(x,\partial Q_L(a))\rangle^{-d}\langle\dist(Y,\partial Q_L(a))\rangle^{-d}~:~\\
a\in \R^d,~x\in Q_L(a),~Y\subset Q_L(a)\Big\}.
\end{multline}
Arguing similarly as in Step~2, it is easily seen that the translation $a$ can be suitably chosen so that this estimate yields the conclusion~\eqref{eq:period-est-lala}.
In order to prove~\eqref{eq:period-est-lala-raff},
it suffices, in fact, to prove it for $a=0$, that is,
\begin{equation}\label{eq:period-est-lala-0}
\int_{B^L_{1+\rho}(x)}|\!\D(\psi_L^Y-\psi^Y)|^2\,\lesssim\,\langle\dist(x,\partial Q_L)\rangle^{-d}\langle\dist(Y,\partial Q_L)\rangle^{-d},
\end{equation}
as the claim~\eqref{eq:period-est-lala-raff} then follows by translating the underlying cell $Q_L$, which does indeed not change the equations provided that the translated cell still contains~$x,Y$.

\medskip\noindent
We turn to the proof of~\eqref{eq:period-est-lala-0}.
As the difference $\psi_L^Y-\psi^Y$ satisfies a free steady Stokes problem of the form~\eqref{*eq:def-w}, we may apply the mean-value property of Lemma~\ref{*lem:mean-value-inc} to the effect that for all $x\in Q_L$,
\begin{equation}\label{eq:period-est-meanvalue-1}
\int_{B^L_{1+\rho}(x)}|\!\D(\psi_L^Y-\psi^Y)|^2\,\lesssim\,\langle\dist(x,\partial Q_L)\rangle^{-d}\int_{Q_L}|\!\D(\psi_L^Y-\psi^Y)|^2.
\end{equation}
In order to estimate the last integral, we argue similarly as in Step~1 by further comparing $\psi_L^Y,\psi^Y$ to the solution of the corresponding Neumann problem in $Q_L$: we define $\psi_N^Y\in H^1(Q_L)^d$ as the solution of
\[\left\{\begin{array}{ll}
-\triangle\psi_N^Y+\nabla\Sigma_N^Y=0,&\text{in $Q_L\setminus \cup_{y\in Y}B(y)$},\\
\Div(\psi_N^Y)=0,&\text{in $Q_L\setminus \cup_{y\in Y}B(y)$},\\
\sigma(\psi_N^Y,\Sigma_N^Y)\nu=0,&\text{on $\partial Q_L$},\\
\D(\psi_N^Y+Ex)=0,&\text{in $\cup_{y\in Y}B(y)$},\\
\int_{\partial B(y)}\sigma(\psi_N^Y,\Sigma_N^Y)\nu=0,&\forall y\in Y,\\
\int_{\partial B(y)}\Theta(x-y)\cdot\sigma(\psi_N^Y,\Sigma_N^Y)\nu=0,&\forall y\in Y,~\forall\Theta\in\Md^\Skew.
\end{array}\right.\]
In these terms, we start by estimating
\begin{equation}\label{eq:decomp-cor-err-L}
\int_{Q_L}|\!\D(\psi_L^Y-\psi^Y)|^2\,\le\,2\int_{Q_L}|\!\D(G_1)|^2+2\int_{Q_L}|\!\D(G_2)|^2,
\end{equation}
where we have set for abbreviation
\[G_1\,:=\,\psi_L^Y-\psi_N^Y,\qquad G_2\,:=\,\psi^Y-\psi_N^Y.\]
We denote by $R_1,R_2$ the corresponding pressure differences.
Similarly as in Step~1, energy identities take the form
\begin{eqnarray*}
2\int_{Q_L}|\!\D(G_1)|^2&=&\sum_{y\in Y}\int_{\partial B(y)}E(x-y)\cdot\sigma(G_1,R_1)\nu,\\
2\int_{Q_L}|\!\D(G_2)|^2&=&\sum_{y\in Y}\int_{\partial B(y)}E(x-y)\cdot\sigma(G_2,R_2)\nu-2\int_{\R^d\setminus Q_L}|\!\D(\psi^Y)|^2,
\end{eqnarray*}
and we deduce by means of trace estimates, for both $i=1,2$,
\begin{equation*}
\int_{Q_L}|\!\D(G_i)|^2\,\lesssim\,\sum_{y\in Y}\Big(\int_{B_{1+\rho}(y)}|\!\D(G_i)|^2\Big)^\frac12.
\end{equation*}
Hence, applying the mean-value property of Lemma~\ref{*lem:mean-value-inc} to $G_1,G_2$, together with Young's inequality,
\begin{equation*}
\int_{Q_L}|\!\D(G_i)|^2\,\lesssim_{\sharp Y}\,\sum_{y\in Y}\langle\dist(y,\partial Q_L)\rangle^{-d}.
\end{equation*}
Combined with~\eqref{eq:period-est-meanvalue-1} and~\eqref{eq:decomp-cor-err-L}, this yields the claim~\eqref{eq:period-est-lala-0}, and concludes the proof.\qed

\section{Finite-volume approximation of the effective viscosity}\label{sec:quant}
This appendix is devoted to the proof of an algebraic convergence rate for the finite-volume approximation $\Bb_L$ of the effective viscosity $\Bb$ under an algebraic $\alpha$-mixing condition, 
as announced in Remark~\ref{rem:Mix}.

\begin{prop}[Convergence rate for $\Bb_L$]\label{prop:conv-BL}
On top of Assumption~\ref{H0}, assume that the algebraic mixing condition~\emph{\ref{Mix}} holds. Then there exists $\gamma\in(0,\beta)$ (only depending on $d,\rho$ and on the mixing exponent $\beta$) such that for all $L$,
\[|\Bb_L-\Bb|\,\lesssim\,L^{-\gamma}.\qedhere\]
\end{prop}

The proof displayed below closely follows the monograph~\cite{AKM-book} by Armstrong, Kuusi, and Mourrat (albeit in the more general version~\cite{Armstrong-Mourrat-16} for $\alpha$-mixing coefficients) based on the original argument \cite{AS} by Armstrong and Smart.
We identify a suitable subadditive quantity $J$ that satisfies all the requirements of~\cite{Armstrong-Mourrat-16,AKM-book}
in the present Stokes context: the definition~\eqref{e.J} and
Lemma~\ref{lem:J} below constitute the only new insight wrt~\cite{AKM-book}, and the conclusion follows from elementary adaptations of the arguments in~\cite{Armstrong-Mourrat-16,AKM-book}.
Although we could have used the same subadditive quantity as in \cite{AKM-book}, 
we have chosen to use a subadditive quantity $J$ built on the approximations \eqref{e:def-Dir} and \eqref{e:def-Neu} that we used to prove Einstein's formula, that is, in the form of
\eqref{eq:def-tildeBL0} below.
This choice, which is specific to our problem, makes some of the upcoming arguments technically simpler than in \cite{AKM-book}, in particular avoiding the use of convex duality.

\medskip

Let $E\in\Md_0$ be fixed with $|E|=1$.
We say that a bounded domain~$U\subset\R^d$ is \emph{suitable} if $\dist(\Ic\cap U,\partial U)>\rho$.
We consider the following weakly closed subsets of $H^1(U)^d$,
\begin{eqnarray*}
\calH(U)&:=&\big\{ u \in H^1(U)^d\,:\, \Div( \phi)=0,~\text{ and }\D(\phi+Ex)=0\text{ on }\Ic \cap U\big\},\\
\calH_\circ(U)&:=&H^1_0(U)^d\cap\calH(U),
\end{eqnarray*}
and the following minimization problems (note that only the symmetrized gradient $\D(\psi_*(U))$ is uniquely defined in the first line),
\begin{eqnarray}
\psi_*(U)&:=&\arg\min\Big\{  \int_{U} |\!\D(\phi)|^2\,:\,\phi\in \calH(U)\Big\},\label{J2-sup}\\
\psi_\circ(U)&:=&\arg\min\Big\{  \int_{U} |\!\D(\phi)|^2\,:\,\phi\in \calH_0(U)\Big\}.\label{J1-sub}
\end{eqnarray}
Recalling that the fattened inclusions $\{I_n+\rho B\}_n$ are disjoint,
we define the modified cubes
\[U_L(x)\,:=\,\bigg( Q_L(x) \setminus \bigcup_{n:x_n \notin Q_L(x)} (I_n+\rho B)\bigg) \bigcup \bigg(\bigcup_{n:x_n \in Q_L(x)} (I_n+\rho B)\bigg),\]
which satisfy by definition $Q_{L-2(1+\rho)} \subset U_L(x) \subset Q_{L+2(1+\rho)}$ and $\Ic\cap\partial U_L(x)=\varnothing$.
The family $\{U_L(x)\}_{x\in L\Z^d}$ constitutes a partition of $\R^d$.
Setting $U_L=U_L(0)$, we then consider the following alternative finite-volume approximations of the effective viscosity $\Bb$,
\begin{eqnarray}
E: \Bt_{L,*} E&=&1+\expec{\fint_{U_L} |\!\D(\psi_*(U_L))|^2},\nonumber\\
E: \Bt_{L,\circ} E&=&1+\expec{\fint_{U_L} |\!\D(\psi_\circ(U_L))|^2}.\label{eq:def-tildeBL0}
\end{eqnarray}
Since $\calH_\circ(U_L)\subset \calH(U_L)$, we have $E: \Bt_{L,*} E\,\le\,E: \Bt_{L,\circ} E$.
We then define a random set function $J$ for suitable sets $U$ via
\begin{equation}\label{e.J}
J(U)\,:=\, \fint_{U} |\!\D( \psi_\circ(U))|^2- |\!\D( \psi_*(U))|^2.
\end{equation}
The following lemma collects elementary properties of $J$. In particular, item~(iii) states that $U\mapsto |U|J(U)$ is subadditive.
\begin{lem}[Properties of $J$]\label{lem:J}$ $
\begin{enumerate}[(i)]
\item Recalling the definition~\eqref{eq:def-tildeBL0} of finite-volume approximations $\Bt_{L,*},\Bt_{L,\circ}$ of the effective viscosity, there exists $C>0$ such that
\begin{eqnarray}\label{a++1}
E:\Bt_{L,*} E -CL^{-1}~\, \le &E:\Bb E  &\le~\,E:\Bt_{L,\circ} E+CL^{-1},\\
E:\Bt_{L,*} E -CL^{-1} ~\,\le &E:\Bb_{L+2(1+\rho)} E  &\le~\,E:\Bt_{L,\circ} E+CL^{-1}\label{a++2}.
\end{eqnarray}
\item For all suitable $U$,
\begin{eqnarray}\label{J.differences1}
J(U)&=&  \fint_{U} |\!\D (\psi_\circ(U)-\psi_*(U))|^2.
\end{eqnarray}
\item For all disjoint suitable sets $U^1,\dots,U^k$, setting $U=\inter\big(\bigcup_j \overline{U^j}\big)$,
\begin{equation}\label{J.subadd}
|U|J(U)\,\le\, \sum_{j=1}^k |U^j|J(U^j).
\end{equation}
In addition, setting $\delta \psi(U):=\psi_\circ(U)-\psi_*(U)$,
\begin{equation}\label{J.split}
\sum_{j=1}^k   \|\!\D(\delta \psi(U)-\delta \psi(U^j))\|_{\Ld^2(U^j)}^2
\,=\, \sum_{j=1}^k  |U^j| (J(U^j)-J(U)).
\qedhere
\end{equation}
\end{enumerate}
\end{lem}
\begin{proof}
We split the proof into three steps.

\medskip

\step1 Proof of~(i).\\
We start with the proof of~\eqref{a++2}, that is, the comparison of $\Bt_{L,*},\Bt_{L,\circ}$ with the periodic approximation $\Bb_{L+2(1+\rho)}$.
First, we extend $\psi_\circ(U_L)$ by zero on $Q_{L+2(1+\rho)} \setminus U_L^-$, which makes it a  $Q_{L+2(1+\rho)}$-periodic function, and thus an admissible test function in~\eqref{eq:optim-psiL},
\begin{eqnarray*}
{\fint_{Q_{L+2(1+\rho)}} |\!\D(\psi_{L+2(1+\rho)})|^2} &\le& {\fint_{Q_{L+2(1+\rho)}} |\!\D(\psi_\circ(U_L))|^2} 
\\
&=& {\frac{|U_L|}{(L+2(1+\rho))^d}\fint_{U_L^-} |\!\D(\psi_\circ(U_L^-))|^2},
\end{eqnarray*}
which yields, in view of $\big|L^{-d}|U_L|-1\big| \lesssim L^{-1}$,
\begin{equation*}
{\fint_{Q_{L+2(1+\rho)}} |\!\D(\psi_{L+2(1+\rho)})|^2}
~\le~ { \fint_{U_L} |\!\D(\psi_\circ(U_L))|^2}(1+CL^{-1}) .
\end{equation*}
Second, as the restriction $\psi_{L+2(1+\rho)}|_{U_L}$ belongs to $\calH(U_L)$ and is thus an admissible test function in~\eqref{J2-sup}, we 
similarly obtain
\begin{eqnarray*}
{\fint_{U_L} |\!\D(\psi_*(U_L))|^2} &\le& {\fint_{U_L} |\!\D(\psi_{L+2(1+\rho)})|^2}\le{\frac{(L+2(1+\rho))^d}{|U_L|}\fint_{Q_{L+2(1+\rho)}} |\!\D(\psi_{L+2(1+\rho)})|^2}
\\
&\le& (1+CL^{-1}){ \fint_{Q_{L+2(1+\rho)}} |\!\D(\psi_{L+2(1+\rho)})|^2}.
\end{eqnarray*}
The claim~\eqref{a++2} follows from the combination of these two estimates with the following energy bounds, cf.~\eqref{eq:bnd-psi-L2},
\begin{equation}\label{e.energy-bd-lambda}
\expecM{\fint_{Q_{L+2(1+\rho)}} |\!\D(\psi_{L+2(1+\rho)})|^2}+\expecM{  \fint_{U_L} |\!\D(\psi_\circ(U_L))|^2}+\expecm{ |\!\D(\psi)|^2}\,\lesssim \, \lambda(\Pc).
\end{equation}
We turn to the proof of~\eqref{a++1}.
Since the restriction $\psi|_{U_L}$ belongs to $\calH(U_L)$ and is thus an admissible test function in~\eqref{J2-sup}, we find by stationarity of $\D(\psi)$,
\begin{eqnarray}
\expecM{\fint_{U_L} |\!\D(\psi_*(U_L))|^2} &\le& \expecM{\fint_{U_L} |\!\D(\psi)|^2} ~\le~ \expecM{\frac{L^d}{|U_L|}\fint_{Q_L} |\!\D(\psi)|^2}\nonumber
\\ &\le&(1+CL^{-1})\,\expec{|\!\D(\psi)|^2}.\label{e.useful-later1}
\end{eqnarray}
For the converse inequality, we appeal to a cut-and-paste argument. The starting point is the following convergence, cf.~\cite{DG-19},
$$\expec{|\!\D(\psi)|^2} \,=\, \lim_{k\uparrow \infty} \expecM{\fint_{U_{kL}} |\!\D(\psi_\circ(U_{kL}))|^2}.$$
Since $\tilde \psi_{\circ}(U_{kL}):=\sum_{j}  \psi_{\circ}(U_L(z_j)) \mathds1_{U_L(z_j)}$ belongs to $\calH_0(U_{kL})$, where $\{U_L(z_j)\}_j$ is a partition of $U_{kL}$,
we obtain for all $k$, by stationarity of $z\mapsto U_L(z)$,
\begin{eqnarray}
\expecM{\fint_{U_{kL}} |\!\D(\psi_\circ(U_{kL}))|^2} &\le &\sum_{j} \expecM{\frac{|U_L(z_j)|}{|U_{kL}|}\fint_{U_L(z_j)} |\!\D(\psi_\circ(U_L(z_j)))|^2} 
\nonumber
\\
&\le &(1+CL^{-1})\,\expecM{  \fint_{U_L} |\!\D(\psi_\circ(U_L))|^2}. \label{e.useful-later2}
\end{eqnarray}
The claim~\eqref{a++1} follows from the combination of these three properties with the above energy bounds \eqref{e.energy-bd-lambda}.

\medskip

\step2 Proof of~(ii).\\
By definition,
$$
J(U)\,=\,  \fint_{U} \D\big(\psi_\circ(U) -\psi_*(U)\big): \D\big(\psi_\circ(U)+\psi_*(U)\big).
$$
Since $\psi_\circ(U),\psi_*(U) \in\calH(U)$, the difference $\psi_\circ(U)-\psi_*(U)$ is a suitable test function for the Euler-Lagrange equation 
of the minimization problem~\eqref{J2-sup} defining $\psi_*(U)$, which yields
$$
 \int_{U} \D\big(\psi_\circ(U) -\psi_*(U)\big): \D (\psi_*(U))=0,
$$
and the claim~\eqref{J.differences1} follows.

\medskip
\step3 Proof of~(iii).\\
We start with the proof of \eqref{J.subadd}.
Since the minimization problem~\eqref{J1-sub} defines a subadditive set function due to the gluing property of $\calH_0(U)$,
and since the minimization problem~\eqref{J2-sup} defines a superadditive function due to the restriction property of $\calH(U)$, the function $J$ is subadditive 
as the difference of a subadditive and of a superadditive function.

\medskip\noindent
We turn to the proof of~\eqref{J.split}.
The starting point is~\eqref{J.differences1} for $U^j$, which yields
\begin{eqnarray}
|U^j|J(U^j)- \int_{U^j} |\!\D(\delta \psi(U))|^2&=& \int_{U^j} \D(\delta \psi(U^j)-\delta \psi(U)): \D(\delta \psi(U^j)+\delta \psi(U))\nonumber
\\
&=& \int_{U^j}| \D(\delta \psi(U^j)-\delta \psi(U))|^2 \nonumber
\\
&&+\,2\int_{U^j} \D(\delta \psi(U^j)-\delta \psi(U)): \D(\delta \psi(U)) .\label{J.+1}
\end{eqnarray}
We decompose the second right-hand side term into $2\sum_{k=1}^4 I_{k,j}$, in terms of
\begin{eqnarray*}
I_{1,j}&=&\int_{U^j} \D(\psi_\circ(U^j)-\psi_\circ(U)): \D(\psi_\circ(U)),\\
I_{2,j}&=&-\int_{U^j} \D(\psi_\circ(U^j)-\psi_\circ(U)): \D(\psi_*(U)),\\
I_{3,j}&=&\int_{U^j} \D(\psi_*(U)) : \D(\psi_\circ(U)-\psi_*(U)),\\
I_{4,j}&=&-\int_{U^j} \D(\psi_*(U^j)): \D(\psi_\circ(U)-\psi_*(U)).
\end{eqnarray*}
Since $\psi_\circ(U)|_{U^j},\psi_*(U)|_{U^j}\in \calH(U^j)$, the difference $(\psi_\circ(U)-\psi_*(U))|_{U^j}$ is a suitable test function for the Euler-Lagrange equation for $\psi_*(U^j)$, which yields $I_{4,j}=0$.
Likewise, since $\psi_\circ(U),\sum_j \psi_\circ(U^j)\mathds 1_{U^j} \in \calH_0(U)\subset \calH(U)$, we find both $\sum_j I_{1,j}=0$ and $\sum_j I_{2,j}=0$.
In addition, since $\psi_\circ(U),\psi_*(U) \in \calH(U)$, we find $\sum_j I_{3,j}=0$.
This entails
\[\sum_j  \int_{U^j} \D(\delta \psi(U^j)-\delta \psi(U)):\D(\delta \psi(U)) \,=\,0.\]
Summing~\eqref{J.+1} over $j$, inserting the above, and recalling the identity~\eqref{J.differences1}, the claim~\eqref{J.split} follows.
\end{proof}
For all $n \ge 0$ we set $U^n:=U_{3^n}$ and define the discrepancy
\begin{equation}\label{e.def-tau}
\tau_n := \expec{J(U^n)}-\expec{J(U^{n+1})}.
\end{equation}
In contrast with~\cite{AKM-book}, the set $U^n$ is now random, so that subbadditivity does not directly imply $\tau_n \ge 0$.
This is however true up to an error $O(3^{-n})$, as we briefly argue.
Choose a partition $\{U^n_j:=U_{3^n}(z_j)\}_j$ of the set $U^{n+1}$.
Taking the expectation of~\eqref{J.split} applied to this decomposition of $U^{n+1}$, we find
\begin{equation}\label{J+4} 
0\,\le\, \expecM{\sum_j\|\!\D(\delta \psi(U^{n+1})-\delta \psi(U_j^n))\|_{\Ld^2(U_j^n)}^2}
\,=\, \sum_j\expec{|U_j^n|(J(U_j^n)-J(U^{n+1}))},
\end{equation}
whereas by the deterministic bounds $| 3^d|U^n|-|U^{n+1}||\lesssim 3^{n(d-1)}$ and $J(U^n_j)\lesssim 1$
we have for some constant $C\simeq1$,
\begin{equation}\label{J+4bonus} 
 \sum_j\expec{|U_j^n|(J(U_j^n)-J(U^{n+1}))} \,\lesssim\, 3^{nd} \big(\expecm{J(U^n)}-\expecm{J(U^{n+1})}+C3^{-n}\big).
\end{equation}
The combination of \eqref{J+4} and \eqref{J+4bonus} yields the claim in form of 
\begin{equation}\label{J+5}
\bar \tau_n:= \tau_n+C 3^{-n} \ge 0.
\end{equation}

\medskip
The crux of the approach is the following control of the variance of averages of $\D(\delta \psi(U))$ in terms of $\tau_n$.
In view of Lemma~\ref{lem:J}, the proof is identical to that of \cite[Lemma~2.13]{AKM-book} (albeit in the $\alpha$-mixing version of \cite{Armstrong-Mourrat-16},
further arguing as in~\eqref{J+5} and absorbing the additional error term).

\begin{lem}\label{lem:AKM1}
There exist $C,\e>0$ (only depending on $d,\rho,\beta$) such that for all $n$,
\begin{equation*}
\var{\fint_{U^{n}}\D (\delta \psi(U^n)) } \,\le\, C3^{-\e n}+C \sum_{m=0}^n 3^{-\e (n-m)} \bar \tau_m.\qedhere
\end{equation*}
\end{lem}

Recall the following version of Korn's inequality: for any bounded domain $D\subset\R^d$, for all divergence-free fields $v\in\Ld^2(D)$, we have
$$
\inf_{\kappa\in \R^d\atop\Theta \in \Md^{\Skew}} \int_{D} |v(x)-\kappa-\Theta x|^2\,dx ~\lesssim_D~\|\!\D(v)\|_{H^{-1}(D)}^2,
$$
where the multiplicative constant only depends on the regularity of $D$.
In contrast with Poincaré's inequality, the infimum over $\Theta\in\Md^\Skew$ allows to have the symmetrized gradient in the right-hand side instead of the full gradient.
By the so-called multiscale Poincar\'e inequality in~\cite[Proposition~1.12]{AKM-book}, using the above Korn inequality instead of~\cite[Lemma~1.13]{AKM-book}, Lemma~\ref{lem:AKM1} yields the following estimate as in~\cite[Lemma~2.15]{AKM-book}. This is  simpler than the statement in~\cite{AKM-book} since there is no convex duality involved.
\begin{lem}\label{lem:AKM2}
There exist $C,\e>0$ (only depending on $d,\rho,\beta$) such that for all $n$,
\begin{equation*}
\expecM{\inf_{\kappa\in \R^d\atop\Theta \in \Md^{\Skew}}\fint_{U^{n+1}}|\delta \psi(U^{n+1})(x)-\kappa-\Theta x|^2\,dx} \,\le\, C3^{2n}\Big(3^{-\e n}  + \sum_{m=0}^n 3^{-\e (n-m)} \bar \tau_m\Big).\qedhere
\end{equation*}
\end{lem}
Next, we deduce the following estimate on $J$ as in~\cite[Lemma~2.16]{AKM-book} by means of the Caccioppoli inequality.
As the latter inequality in the present Stokes context involves the pressure, the proof slightly differs from~\cite{AKM-book} and is included below.
\begin{lem}\label{lem:AKM3}
There exist $C,\e>0$ (only depending on $d,\rho,\beta$) such that for all $n$,
\[
\expec{J(U^n)}\,\le\,C3^{-\e n}+C\sum_{m=0}^n 3^{-\e(n-m)} \bar \tau_m.\qedhere
\]
\end{lem}
\begin{proof}
Caccioppoli's inequality in form of e.g.~\cite[Section~4.4, Step~1]{DG-21b} yields for all $K\ge1$, for any constants $c\in\R$, $\kappa\in\R^d$, and $\Theta\in\Md^\Skew$,
\begin{multline}\label{eq:bnd-Cacc}
\fint_{U^{n}}|\!\D( \delta \psi(U^{n+1}))|^2
\,\lesssim\,K^23^{-2n}\fint_{U^{n+1}}|\delta \psi(U^{n+1})(x)-\kappa-\Theta x|^2dx\\
+K^{-2}\fint_{U^{n+1}}|\delta\Sigma(U^{n+1})-c|^2\mathds1_{\R^d\setminus\Ic},
\end{multline}
where $\delta \Sigma(U^{n+1})$ is the difference of the pressures associated with $\psi_\circ(U^{n+1})$, $\psi_*(U^{n+1})$.
Appealing to a local pressure estimate in form of e.g.~\cite[Lemma~3.3]{DG-21b}, and recalling Lemma~\ref{lem:J}(ii), we find
\begin{equation}\label{eq:bnd-Cacc++}
\inf_{c\in\R}\fint_{U^{n+1}}|\delta\Sigma(U^{n+1})-c|^2\mathds1_{\R^d\setminus\Ic}\,\lesssim\,\fint_{U^{n+1}}|\!\D(\delta\psi(U^{n+1}))|^2\,=\,J(U^{n+1}).
\end{equation}
Taking the infimum over $c,\kappa,\Theta$ in~\eqref{eq:bnd-Cacc}, taking the expectation, inserting~\eqref{eq:bnd-Cacc++},
and using Lemma~\ref{lem:AKM2}, we obtain
\begin{equation*}
\expecM{\fint_{U^{n}}|\!\D( \delta \psi(U^{n+1}))|^2}
\,\lesssim\,K^2\Big(3^{-\e n}  + \sum_{m=0}^n 3^{-\e (n-m)} \bar \tau_m\Big)
+K^{-2}\expec{J(U^{n+1})},
\end{equation*}
and thus, in view of~\eqref{J+5},
\begin{equation}\label{eq:preconcl-decJ}
\expecM{\fint_{U^{n}}|\!\D( \delta \psi(U^{n+1}))|^2}
\,\lesssim\,K^2\Big(3^{-\e n}  + \sum_{m=0}^n 3^{-\e (n-m)} \bar \tau_m\Big)
+K^{-2}(\expec{J(U^n)}+3^{-n}).
\end{equation}
Next, we argue that 
\begin{equation}\label{e.30.01}
\expec{J(U^n)}\,\lesssim\, \expecM{\fint_{U^{n}}|\!\D( \delta \psi(U^{n+1}))|^2}+\bar \tau_n.
\end{equation}
For that purpose, we first note that the definition of $J$ yields
\begin{eqnarray*}
\lefteqn{\expec{J(U^n)}-\expecM{\fint_{U^{n}}|\!\D( \delta \psi(U^{n+1}))|^2}}
\\
&=&\expecM{\fint_{U^n} \D(\delta \psi(U^n)-\delta\psi(U^{n+1})): \D(\delta \psi(U^n)+\delta\psi(U^{n+1}))}
\\
&\lesssim& \expecM{\fint_{U^n} |\!\D(\delta \psi(U^n)-\delta\psi(U^{n+1}))|^2}^\frac12 \Big( \expec{J(U^n)}+\expec{J(U^{n+1})}\Big)^\frac12.
\end{eqnarray*}
In order to control the first factor, we appeal to~\eqref{J+4} and~\eqref{J+4bonus} in form of 
\[\expecM{\sum_j\|\!\D(\delta \psi(U_j^n)-\delta \psi(U^{n+1}))\|_{\Ld^2(U_j^n)}^2} \, \lesssim \, 3^{nd} \bar\tau_n.\]
Further using the definition~\eqref{J+5} of $\bar \tau_n$ to reformulate the second factor, we deduce
\[\expec{J(U^n)}-\expecM{\fint_{U^n}|\!\D(\delta \psi(U^{n+1}))|^2}\,\lesssim\,(\bar\tau_n)^\frac12 \big(\expec{J(U^n)}+\bar \tau_n\big)^\frac12,\]
and the claim~\eqref{e.30.01} follows.

\medskip\noindent
Choosing $K\simeq1$ large enough, \eqref{eq:preconcl-decJ} and~\eqref{e.30.01} combine to 
\begin{equation*}
\expec{J(U^n)}\,\lesssim\,\expecM{\fint_{U^{n}}|\!\D( \delta \psi(U^{n+1}))|^2}+\bar\tau_n
\,\lesssim\,3^{-(\e\wedge1) n}  + \sum_{m=0}^n 3^{-\beta (n-m)} \bar \tau_m,
\end{equation*}
and the conclusion follows.
\end{proof}

We may now proceed to the proof of Proposition~\ref{prop:conv-BL}, which follows from Lemma~\ref{lem:AKM3} by iteration. 
\begin{proof}[Proof of Proposition~\ref{prop:conv-BL}]
Set $F_n:=3^{-\frac12\e  n}\sum_{m=0}^n 3^{\frac12\e m} \expec{J(U^m)}$.
In terms of $\tau_n$, recognizing a telescoping sum, we find
\begin{eqnarray*}
F_n-F_{n+1} &=&3^{-\frac12\e n}\sum_{m=0}^{n} 3^{\frac12\e m} \expec{J(U^m)}  -3^{-\frac12\e  (n+1)}\sum_{m=0}^{n+1} 3^{\frac12\e m} \expec{J(U^m)} 
\\
&=&3^{-\frac12\e  n}\sum_{m=0}^{n} 3^{\frac12\e m} \tau_m-3^{-\frac12 \e (n+1)} \expec{J(U^{0})},
\end{eqnarray*}
and thus, using~\eqref{J+5} and $\expec{J(U^0)}\lesssim 1$,
\begin{equation}\label{eq:preconcl-diffFnn1}
F_n-F_{n+1}
\,\ge \,3^{-\frac12\e  n}\sum_{m=0}^{n} 3^{\frac12\e m} \bar \tau_m-C 3^{-\frac12 \e n}.
\end{equation}
Similarly, we find
\begin{eqnarray*}
F_{n+1} &\le&3^{-\frac12 \e (n+1)} \sum_{m=1}^{n+1} 3^{\frac12\e m} \expec{J(U^{m})}+C3^{-\frac12 \e (n+1)}\\
&\le &  3^{-\frac12 \e (n+1)} \sum_{m=1}^{n+1} 3^{\frac12\e m}\big(\expec{J(U^{m-1})}+C 3^{-(m-1)}\big)+C3^{-\frac12 \e (n+1)}
\\
&\le &F_n+C3^{-\frac12 \e n},
\end{eqnarray*}
which, by Lemma~\ref{lem:AKM3}, turns into 
\begin{eqnarray*}
F_{n+1}
&\le&C3^{-\frac12\e  n}\sum_{m=0}^n 3^{\frac12\e  m}\Big(3^{-\e m}+\sum_{k=0}^m 3^{-\e(m-k)} \bar \tau_k\Big)+C3^{-\frac12 \e n}
\\
&\le&C3^{-\frac12\e  n}+C3^{-\frac12\e  n} \sum_{m=0}^nC3^{\frac12\e  m} \bar \tau_m.
\end{eqnarray*}
Combining this with~\eqref{eq:preconcl-diffFnn1}, we obtain
\[F_{n+1} \,\le\, C(F_n-F_{n+1})+C3^{-\frac12\e  n},\]
and thus
\[F_{n+1} \le \frac{C}{C+1} ( F_n +3^{-\frac12\e  n}).\]
By iteration, this yields $F_n \le C3^{-\gamma n}$ for some $\gamma>0$, and thus $\expec{J(U^n)} \le C3^{-\gamma n}$ and $\tau_n \le C3^{-\gamma n}$.
Since $\expec{J(U^n)}=E:(\Bt_{3^n,\circ}-\Bt_{3^n,*}) E$, this implies
\begin{equation}\label{eq:convBL-0}
0\,\le\, E:(\Bt_{L,\circ}-\Bt_{L,*}) E \,\le\, L^{-\gamma}.
\end{equation}
Combined with Lemma~\ref{lem:J}(i), this yields the conclusion.
\end{proof}

\section*{Acknowledgements}
The authors thank David G\'erard-Varet, Richard Höfer, and Jules Pertinand for comments on previous versions of this work.
MD acknowledges financial support from the CNRS-Momentum program and from F.R.S.-FNRS,
and AG from the European Research Council (ERC) under the European Union's Horizon 2020 research and innovation programme (Grant Agreement n$^\circ$~864066).

\bibliographystyle{plain}
\bibliography{biblio}

\end{document}